\documentclass[PhD,binding=0.6cm]{sapthesis}

\usepackage{microtype}
\usepackage[utf8]{inputenx}

\usepackage{hyperref}
\hypersetup{pdftitle={Lipschitz-homotopy invariants: $L^2$-cohomology, Roe index and $\rho$-class},pdfauthor={Stefano Spessato}}

\usepackage{lipsum}
\usepackage{curve2e}
\definecolor{gray}{gray}{0.4}

\usepackage{amsthm}
\usepackage{amsmath}
\usepackage{amssymb}
\usepackage{graphicx}
\usepackage{times}
\usepackage{fancyhdr}

\theoremstyle{plain} 
\newtheorem{thm}{Theorem}[section] 
\newtheorem{cor}[thm]{Corollary} 
\newtheorem{lem}[thm]{Lemma} 
\newtheorem{prop}[thm]{Proposition}

\theoremstyle{definition} 
\newtheorem{defn}{Definition}[section] 
\newtheorem{exem}{Example} [section]
\newtheorem{notation}{Notation}[section]

\theoremstyle{remark} 
\newtheorem{rem}{Remark} 
\newcommand{\numberset}{\mathbb}

\newcommand{\C}{\numberset{C}}

\title{Lipschitz-homotopy invariants:\\ $L^2$-cohomology, Roe index and $\rho$-class}
\author{Stefano Spessato}
\IDnumber{1806038}
\course[Mathematics]{Matematica}
\courseorganizer{Scuola di dottorato Vito Volterra}
\cycle{XXXIII}
\submitdate{March 2021}
\copyyear{2021}
\advisor{Prof. Paolo Piazza}
\authoremail{spessato@mat.uniroma1.it}
\examdate{26 July 2021}
\examiner{Prof. Daniele Guido}
\examiner{Prof. Éric Leichtnam}
\examiner{Prof. Paolo Antonini.}

\begin{document}
	
	\frontmatter
	
	\maketitle
	
		\chapter{Introduction}
	Cohomology Theories are one of the most powerful tools in Geometry. In this P.h.D. thesis we will concentrate in particular on one of them: the $L^2$-cohomology.
	\\The $L^2$-cohomology of a complete Riemannian manifold $(M,g)$ is defined as
	\begin{equation}
		H^l_2(M,g) := \frac{ker(d_l)}{im(d_{l-1})}
	\end{equation}
	where $d_l$ is the closure in $\mathcal{L}^2(M)$ of the exterior derivative defined on compactly supported differential $l$-forms. Moreover one can also define the reduced $L^2$-cohomology as
	\begin{equation}
		\overline{H}^l_2(M,g) := \frac{ker(d_l)}{\overline{im(d_{l-1})}}.
	\end{equation}
	It's a well-known fact that if two metrics $g$ and $h$ on $M$ are quasi-isometric, i.e. there are two constants $C_1$ and $C_2$ such that
	\begin{equation}
		C_1h \leq g \leq C_2h,
	\end{equation} 
	then the (un)-reduced $L^2$-cohomologies coincides. In particular if $M$ is a compact manifold we have that both, reduced or not, $L^2$-cohomologies coincide with the de Rham Cohomology and so they don't depend on the choice of the metric $g$ in $M$. 
	\\Differently from the case of de Rham cohomology, the pullback along a differentiable map $f$ between Riemannian manifolds, in general, doesn't induce a morphism between the $L^2$-cohomology groups (see Subsection \ref{probl} in Chapter 1). 
	\\Given a uniformly proper, lipschitz map between manifolds of bounded geometry we shall define an operator $T_f$ such that it induces a morphism in (un)-reduced $L^2$-cohomology. In particular, we will show that the usual functorial properties hold and that if $f^*$ does induces a map in (un)-reduced $L^2$-cohomology, then 
	\begin{equation}
		f^*[\alpha] = [T_f\alpha]
	\end{equation}
	for each square-integrable differential form. 
	\\Let us suppose, moreover, that there is a group $\Gamma$ of isometries that acts on the manifolds. Then, if $f$ is $\Gamma$-equivariant, also $T_f$ will be $\Gamma$-equivariant and the same will hold for the morphism in $L^2$-cohomology induced by $T_f$.
	\\The idea of how to define $T_f$ comes from the operator $T_f$ defined by Hilsum and Skandalis in \cite{hils}. The authors use it in order to build a perturbation for the signature operator of the disjoint union of compact manifold which makes it invertible. We can understand our $T_f$ as a \textit{bounded geometry} version of the operator of Hilsum and Skandalis.
	\\
	\\As consequence of our analysis, it follows that (un-reduced )$L^2$-cohomology is a lipschitz-homotopy invariant for manifolds of bounded geometry, i.e. if a map $f$ is a homotopy equivalence such that
	\begin{itemize}
		\item $f$ is lipschitz,
		\item there is a homotopy inverse $g$ of $f$ which is lipschitz and
		\item the homotopies between $g \circ f$ and $f \circ g$ with their respective identities are lipschitz,
	\end{itemize}
	then $T_f$ induces an isomorphism between the (un)-reduced $L^2$-cohomologies.
	\\
	\\
	\\In Chapter 2 we see an application of the existence of $T_f$. Our setting is the same as above. We have two Riemannian manifolds $(M,g)$ and $(N,h)$ of bounded geometry, a lipschitz-homotopy equivalence $f$ which preserves the orientation between them and a group $\Gamma$ acting by isometries so that $\frac{M}{\Gamma}$ and $\frac{N}{\Gamma}$ are manifolds of bounded geometry.
	\\Our goal, in this second Chapter, will be to prove that, given the Roe algebras $C^*(M)^\Gamma$ and $C^*(N)^\Gamma$, then
	\begin{equation}\label{ste}
		f_{\star}(Ind_{Roe}(D_M)) = Ind_{Roe}(D_N).
	\end{equation}
	In words: the Roe Index of the signature operator is a lipschitz-homotopy invariant for manifolds of bounded geometry.
	\\In order to establish (\ref{ste}) we check the existence of the Hilsum-Skandalis perturbation for the signature operator in our setting and we will adapt the work of Piazza and Schick \cite{piazzaschick}, the work of Wahl \cite{wahl}, of Fukumoto \cite{fukum} and Zenobi \cite{vito} to the non-compact case.
	\\Finally we will define a class
	\begin{equation}
		\rho(M, f) \in K_{n + 1}(D^*(N)^\Gamma)
	\end{equation}
	that only depends on $M$ and on the class of lipschitz-homotopy of $f$.
	\\In what follows there is a more detailed description of the Sections of this Thesis.		
	\subsection*{Maps between manifolds of bounded geometry}
	In the first Section of the first Chapter, we introduce the setting of this Thesis.
	\\The objects of our study are manifolds of bounded geometry, which are manifolds with strictly positive injectivity radius and some uniform bounds on the covariant derivatives of the curvature. Some examples of manifolds of bounded geometry can be compact manifolds, coverings of manifolds of bounded geometry (so, for example, the euclidean space and the hyperbolic space),
	and products of manifolds of bounded geometry. Moreover Greene, in \cite{Greene}, proved that every differential manifold admits a metric of bounded geometry. The main results that we will use about manifolds of bounded geometry can be found in \cite{bound} and in \cite{flow}.
	\\On the other hand, the maps that we consider are the \textit{uniformly proper lipschitz} maps between metric space. These are lipschitz maps such that the diameter of the preimage of a subset $A$ is controlled by the diameter of $A$ itself. We show, in particular, that every lipschitz-homotopy equivalence is an uniformly proper map.
	\\We conclude the first section by introducing a particular family of actions of a group $\Gamma$ on a metric space. We call these \textit{uniformly proper, discontinuous and free} actions (or \textit{FUPD} actions). We will see that if the metric space is a manifold of bounded geometry, then an action of a group $\Gamma$ is FUPD  if and only if the quotient is a manifold of bounded geometry.
	\subsection*{$C^k_b$-maps}
	In this second section of first Chapter we introduce the notion of $C^k_b$-map between manifolds of bounded geometry. One can also find the definition and the main properties of these maps in \cite{bound}. A $C^k_b$-map is a map such that, in local normal coordinates, has uniformly bounded derivatives of order $l \leq k$. In particular one can check that $C^1_b$-maps are lipschitz maps. The most important property, for us, about the $C^k_b$-map is proved in subsection \ref{approx}: we use a result of Eldering \cite{bound} to prove that every lipschitz map $f$ between manifolds of bounded geometry can be approximated by a$C^{\infty}$-map $F$ which is a $C^k_b$-map. Moreover if we have two FUPD actions of a group $\Gamma$ on $M$ and $N$ and $f$ is $\Gamma$-equivariant, then also the approximation is $\Gamma$-equivariant. In particular, we have that the morphisms induced between the K-theory groups of the Roe algebra by $f$ and $F$ are the same. This means that, in the next sections, we consider $f$ as a smooth $C^k_b$-map for each $k \in \numberset{N}$.
	\subsection*{$L^2$-cohomology and pull-back}
	In this third section of the first Chapter we introduce the notions of $L^2$-cohomology and reduced $L^2$-cohomology. In subsection \ref{probl} we show that the pullback along a map $f$, in general, doesn't induce a map in (un)-reduced $L^2$-cohomology, even if it is a lipschitz-homotopy equivalence. Then, given a map $f: (M,g) \longrightarrow (N,h)$, we define the \textit{Fiber Volume} of $f$ as the Radon-Nicodym derivative
	\begin{equation}
		Vol_f := \frac{\partial f_\star \mu_M}{\partial \mu_N},
	\end{equation}
	where $\mu_M$ and $\mu_N$ are the measures on $M$ and $N$ induced by their metrics. We call \textit{R.-N.-lipschitz} map, a lipschitz map $f$ which has bounded Fiber Volume. Then we show that the pullback along a R.-N.-lipschitz map induces an $\mathcal{L}^2$-bounded operator. In particular, we also show that R.-N.-lipschitz maps are exactly the \textit{v.b.-maps} defined by Thomas Schick in his Ph.D. Thesis.
	\\We conclude the third section of first Chapter by showing that the Fiber Volume of a submersion can be expressed as the integration along the fibers of some particular differential forms. In particular we choose the name \textit{Fiber Volume} because in the case of a totally geodesics Riemannian submersion we have that $Vol_f$ in a point $q$ in $N$ is exactly the volume of its fiber $F_q$ respect to the metric induced by $g$.
	\subsection*{Submersion related to lipschitz maps}
	This is the fourth section of first Chapter. Consider a uniformly proper lipschitz map $f:(M, g) \longrightarrow (N,h)$, where $(N,h)$ is an oriented manifold of bounded geometry and $(M,g)$ is an oriented manifold with bounded scalar curvature. Let $\Gamma$ be a group which acts by isometries on $M$ and $N$ and suppose that $f$ is $\Gamma$-equivariant. Let us define $T^\delta N$ as the subset of $TN$ given by the vectors with norm less or equal to $\delta := inj_N$. 
	\\The main goal of this section is to define a $\Gamma$-equivariant lipschitz submersion 
	\begin{equation}
		p_f: (f^*T^\delta N, g_S) \longrightarrow (N,h)
	\end{equation}
	such that $p_f(x,0) = f(x)$.
	The metric $g_S$ on $f^*T^\delta N$ is called \textit{Sasaki metric} and it depends on the metrics $g$ and $h$. Moreover we also have that if $f$, in local normal coordinates, has some uniform bounds on the derivatives of order less or equal $j$, then also $p_f$ has some bounds on the derivatives of the same order. In particular these bounds are important in Chapter 2. 
	\subsection*{The pull-back functor}
	The last section of the first Chapter starts by proving that the submersion $p_f$ has bounded Fiber Volume and so it is a R.-N.-lipschitz map. Then we define the operator $T_f:\mathcal{L}^2(N) \longrightarrow \mathcal{L}^2(M)$ as
	\begin{equation}
		T_f(\alpha) := \int_{B^\delta} p_f^*\alpha \wedge \omega,
	\end{equation}
	where $\omega$ is a particular Thom form of $f^*(TN)$. It follows the proof of the $\mathcal{L}^2$-boundedness of $T_f$: it will be easily proved since $p_f$ is a R.-N.-lipschitz map.
	\\In the subsection \ref{lemhom} we prove two important lemmas: the first one gives a formula which relates $f$ and $p_f$ when $f$ is a R.-N.-lipschitz map.  The second one gives a formula which relates the submersion $p_{f \circ g}$ with the submersions $p_f$ and $p_g$, where $f: (M,m) \longrightarrow (N,h)$ and $g:(S,l) \longrightarrow (M,m)$ are uniformly proper lipschitz map between manifolds of bounded geometry. 
	\\We conclude the section by showing that, for every $z$ in $\numberset{N}$
	\begin{equation}
		\begin{cases}
			\mathcal{F}_z(M,g) = H^z_2(M) \\
			\mathcal{F}_z((M,g) \xrightarrow{f} (N,h)) = H^z_2(N) \xrightarrow{T_f} H^z_2(M)
		\end{cases}
	\end{equation}
	is a controvariant functor between the category $\mathcal{C}$ of the manifolds of bounded geometry with uniformly proper lipschitz maps and the category $\textbf{Vec}$ of the vector spaces. In particular, we will show that if two maps $f_1$ and $f_2$ are uniformly proper lipschitz maps such that they are homotopic with a lipschitz homotopy $H$, then $T_{f_1} = T_{f_2}$ in $L^2$-cohomology. Moreover if $f$ is a R.-N.-lipschitz map, then $f^* = T_f$ in $L^2$-cohomology. Finally we show that the same holds considering the reduced $L^2$-cohomology instead of the $L^2$-cohomology.
	\subsection*{Coarse geometry}
	In the first section of the second Chapter, we introduce the notions of coarse structure and coarse map. This definitions and the main properties come from the book \textit{Analytic K-homology} of Higson and Roe \cite{Anal}. In particular we use two coarse structures. The first one is a \textit{metric} coarse structure, which is a coarse structure which can be defined on a metric space. It easily follows that uniformly proper lipschitz map are coarse maps respect to metric coarse structures.
	\\The second one is a coarse structure which can be defined on the disjoint union of metric spaces $M$ and $N$ such that there is a map $f: M \longrightarrow N$. Indeed, in the following sections, we will need a coarse structure on $M \sqcup N$, which is not a metric space.
	\\After that, given a coarse space $X$ and given an action of a group $\Gamma$ on $X$, we define the Roe (or coarse) algebras $C^*(X)^\Gamma$ and the structure algebra $D^*(X)^\Gamma$. In the case $X = M \sqcup N$ we will denote the algebras by $C^*_f(M \sqcup N)^\Gamma$ and by $D^*_f(M \sqcup N)^\Gamma$ to recall the dependence on $f$. We have that $C^*(X)^\Gamma$ is an ideal of $D^*(X)^\Gamma$ and so this means that the sequence
	\begin{equation}
		0 \longrightarrow C^*(X)^\Gamma \longrightarrow D^*(X)^\Gamma \longrightarrow \frac{D^*(X)^\Gamma}{C^*(X)^\Gamma} \longrightarrow 0
	\end{equation}
	is exact. Thus we obtain the long exact sequence in K-Theory
	\begin{equation}
		... K_{n}(C^*(X)^\Gamma) \longrightarrow K_n(D^*(X)^\Gamma) \longrightarrow K_n(\frac{D^*(X)^\Gamma}{C^*(X)^\Gamma}) \longrightarrow K_{n-1}(C^*(X)^\Gamma) \longrightarrow ...
	\end{equation}
	\\Finally, if $dim(M) = n$, we define the \textit{fundamental class} $[D_{M}] \in K_{n+1}(\frac{D^{\star}(M)^{\Gamma}}{C^{\star}(M)^\Gamma})$ of the signature operator $D_M$ of a Riemannian manifold and we conclude the section defining the Roe Index of a connected Riemannian manifold as the class
	\begin{equation}
		Ind_{Roe}(D_{M}) := \delta[D_{M}]
	\end{equation}
	in $K_{n}(C^{*}(M)^{\Gamma})$, where $\delta$ the connecting homomorphism in K-Theory.
	\subsection*{Smoothing operators between manifolds of bounded geometry}
	This is the second section of second Chapter. Consider $f: (M,g) \longrightarrow (N,h)$ a lipschitz-homotopy equivalence with preserves the orientations. Then, in this section we prove that there is an $\mathcal{L}^2$-bounded, $\Gamma$-invariant, operator $y$ such that
	\begin{equation}
		1 - T_f^\dagger T_f = d y + y d
	\end{equation}
	on $dom(d)$ where $T_f^\dagger = \tau_N T_f^* \tau_M$ and $\tau_X$ is the so-called \textit{chiral operator} of the Riemannian manifold $X$.
	\\After that we prove that $T_f$ is a smoothing operator, i.e. an operator $A$ such that for each smooth $\alpha$
	\begin{equation}
		A\alpha (p) = \int_N K(p,q)\alpha(q)d\mu_N,
	\end{equation}
	where $K$ is a smooth section of the fiber bundle $\Lambda^*(M) \boxtimes \Lambda(N)$. Then, using the results reported in Appendix \ref{smoothing}, we prove that also $dT_f$ is a smoothing operator, and, in particular, we can see that if $f$ is a $C^k_b$-map where $k \geq 2$, then $dT_f$ is also $\mathcal{L}^2$-bounded.
	\\We conclude the section proving that if $f$ is a $C^{k \geq 2}_b$-map, then the operators $T_fy$, $yT^\dagger_f$, $dT_f y$ and all their adjoints and their algebraic combinations are all operators in $C^*_f(M \sqcup N)^\Gamma$.
	\subsection*{Lipschitz-homotopy invariance of the Roe Index}
	In the third section of Chapter 2 we use $T_f$ and $y$ to define the Hilsum-Skandalis perturbation  $\mathcal{P}$ of \cite{hils} for the signature operator on the manifold of bounded geometry $M \sqcup N$. It follows, since $T_f$ induces an isomorphism in $L^2$-cohomology, that the perturbed signature operator $D_{M \sqcup N} + \mathcal{P}$ is $L^2$-invertible. Moreover, since $T_f$, $T_fy$, $dT_fy$ etc.., are all operators in $C^*_f(M \sqcup N)^\Gamma$, we are able to prove that also $\mathcal{P}$ is in $C^*_f(M \sqcup N)^\Gamma$.
	\\Similarly to what we did for Riemannian connected manifolds, we define the \textit{fundamental class} $[D_{M \sqcup N}]$ and the \textit{Roe Index} $Ind_{Roe}(D_{M \sqcup N})$. The using that $\mathcal{P}$ is in $C^*_f(M \sqcup N)^\Gamma$ and a result of \cite{higroe}, we can find a class $\rho_f$ in $K_{n+1}(D_f^*(M \sqcup N)^\Gamma)$ such that its image in the long exact sequence in K-Theory coincides with $[D_{M \sqcup N}]$. This fact implies that
	\begin{equation}
		Ind_{Roe}(D_{M \sqcup N}) = 0,
	\end{equation}
	and, using some lemmas proved in Appendix \ref{K-theory}, we conclude that
	\begin{equation}
		f_{\star}(Ind_{Roe}(\mathcal{D}_M)) = Ind_{Roe}(\mathcal{D}_N).
	\end{equation}
	\subsection*{Lipschitz-homotopy stability of $\rho_f$ and $\rho(f,M)$}
	In the last section of this Thesis we prove that the class $\rho_f$ just depend on the class of lipschitz-homotopy of $f$, i.e. if $f$ and $f'$ are two $C^k_b$ lipschitz-homotopy equivalences and there is a lipschitz homotopy $H$ between them, then
	\begin{equation}
		D^*_f(M \sqcup N)^\Gamma = D^*_{f'}(M \sqcup N)^\Gamma \mbox{           and             }
		C^*_f(M \sqcup N)^\Gamma = C^*_{f'}(M \sqcup N)^\Gamma.
	\end{equation}
	Moreover we also have that
	\begin{equation}
		\rho_f = \rho_{f'}.
	\end{equation}
	In order to prove this fact we consider for each $s \in [0,1]$ the $f_s := H(\cdot, s)$ and we prove that $\mathcal{P}_s$, which is the Hilsum-Skandalis perturbation related to $f_s$ defines a curve in $C^*_f(M \sqcup N)^\Gamma$.
	\\We conclude this Thesis defining the projected $\rho$-class in $K_{n+1}(D^*(N)^\Gamma)$
	\begin{equation}
		\rho(f,M) := f_\star[\rho_f]
	\end{equation}
	and showing that it depends only on $M$ and on the lipschitz-homotopy class of $f$.
	\section*{Acknowledgments}
	I am grateful to Paolo Piazza, my advisor, for the interesting reasearch project he proposed to me and for his assistance at every stage of the research project.
	\\I wish to thank Vito Felice Zenobi for the several discussions that we had, for his help and for its competence. 
	\\I also would like to thank Francesco Bei and Thomas Schick for their insightful comments and suggestions.
	\\Another thanks goes to the referees for patiently reported all glitches and imperfections to me.
	\\Finally I would like to thank my family and Laura for their constant support.
	\tableofcontents

	\mainmatter
	
	\chapter{A pull-back functor for (un)-reduced $L^2$-cohomology}
	\section{Maps between manifolds of bounded geometry}
	\subsection{Lipschitz maps}
	In this section our goal is to introduce families of maps between metric spaces. The first one is the one of lipschitz maps.
	Let $(X, d_X)$ and $(Y, d_Y)$ be two metric spaces. Let us consider a map $f: (X, d_X) \longrightarrow (Y, d_Y)$.
	\begin{defn}
		We say that $f$ is \textbf{lipschitz} if for all $p$ and $q$ in $X$ there is a constant $C_f$ such that
		\begin{equation}
			d_Y(f(p), f(q)) \leq C_fd_X(p,q).
		\end{equation}
	\end{defn}
	\begin{rem}
		If $f$ is a smooth map between Riemannian manifolds there are also two other equivalent definitions: the first one is that $f$ is a \textbf{lipschitz} map if there is a constant $C_f$ such that for all $p$ in $M$ an for all $v_p$ in $T_pM$
		\begin{equation}
			||f_{\star,p}(v_p)|| \leq C_f ||v_p||.
		\end{equation}
		The second one is the following: $f$ is a \textbf{lipschitz} map if there is a constant $C_f$ such that for every $p$ in $M$, fixed some normal coordinates $\{x^i\}$ centered in $p$ and $\{y^j\}$ centered in $f(p)$, then
		\begin{equation}
			|\frac{\partial y^j \circ f}{\partial x^i}(p)|\leq C_f.
		\end{equation}
		The above inequality doesn't depend on the choice of normal coordinates, indeed we have that
		\begin{equation}
			|\frac{\partial y_2^j}{\partial y_1^i}(f(p))| = |\frac{\partial x_2^j}{\partial x_1^i}(p)| = 1.
		\end{equation}
	\end{rem}
	\subsection{$\Gamma$-Lipschitz-homotopy equivalence}
	Let us consider now a group $\Gamma$ acting on $M$ and $N$. We will concentrate on the maps which preserves the action of $\Gamma$.
	\begin{defn}
		Two maps $f_0$ and $f_1: (X, d_X) \longrightarrow (Y,d_Y)$ are \textbf{$\Gamma$-lipschitz-homotopic}, denoted by
		\begin{equation}
			f_0 \sim_{\Gamma} f_1,
		\end{equation}
		if they are $\Gamma$-equivariant maps and there is a $\Gamma$-invariant lipschitz homotopy\footnote{Here we are considering $d_X \times d_{[0,1]}$ as the product distance between $d_X$ and the euclidean distance on $[0,1]$.} $H: (X\times[0,1], d_X \times d_{[0,1]}) \longrightarrow (Y, d_Y)$ such that
		\begin{equation}
			H(x,0) = f_0(x) \mbox{              and               } H(x,1) = f(x)
		\end{equation}
	\end{defn}
	\begin{defn}
		A map $f:(X, d_X) \longrightarrow (Y,d_Y)$ is a \textbf{$\Gamma$-lipschitz-homotopy equivalence} if $f$ is $\Gamma$-equivariant, lipschitz and there is a $\Gamma$-equivariant map $g$ such that
		\begin{itemize}
			\item $g$ is a homotopy inverse of $f$,
			\item $g$ is lipschitz,
			\item $f\circ g \sim_{\Gamma} id_Y$ and $g \circ f \sim_\Gamma id_X$.
		\end{itemize}
	\end{defn}
	\subsection{Uniformly proper maps}
	Let us introduce the notion of uniformly proper map.
	\begin{defn}
		Consider two metric spaces $(X,d_X)$ and $(Y, d_Y)$. Let $f:(X,d_X) \longrightarrow (Y,d_Y)$ be a map. Then $f$ is \textbf{uniformly (metrically) proper}, if there is a continuous function $\alpha: [0, +\infty) \longrightarrow [0,+\infty])$ such that for every subset $A\subseteq Y$ we have
		\begin{equation}
			diam(f^{-1}(A)) \leq \alpha(diam(A))
		\end{equation}
	\end{defn}
	\begin{rem}
		If $f$ is a uniformly proper lipschitz map, then $\alpha$ can be chosen such that for all $s,t \geq 0$
		\begin{equation}
			\alpha(t) < \alpha(t +s).
		\end{equation}
		Indeed if $A \subseteq B$ then $f^{-1}(A) \subseteq f^{-1}(B)$. This means that the composition of two uniformly proper maps is still uniformly proper, indeed
		\begin{equation}
			\begin{split}
				diam((g \circ f)^{-1}(A)) &= diam(f^{-1}(g^{-1}(A))) \\
				&\leq \alpha(diam(g^{-1}(A))) \\
				&\leq \alpha(\beta(diam(A))) = (\alpha \circ \beta)(diam(A)).
			\end{split}
		\end{equation}
	\end{rem}
	\begin{rem}\label{uniformlyp}
		Let $(X,d_X)$ and $(Y,d_Y)$ be two metric spaces and let $\Gamma$ be a group acting on $M$ and on $N$. Consider two lipschitz maps $F,f: (X, d_X) \longrightarrow (Y,d_Y)$ such that $f \sim_\Gamma F$ trought a $\Gamma$-equivariant lipschitz homotopy $h$. If $f$ is a uniformly proper map, then also the homotopy $h$ is uniformly proper. Then, in particular, also $F$ is uniformly proper.
		\\In order prove this fact we have to observe that 
		\begin{equation}
			d(f(x), h(x,t)) = d(h(x,0),h(x,t)) \leq C_h\cdot t.
		\end{equation}
		Then, if $A$ is a subset of $Y$ and $h_t: X \longrightarrow Y$ is defined as $h_t(p) :=h(p,t)$, we have that
		\begin{equation}
			\begin{split}
				h_t^{-1}(A) &= \{p \in X| h(p,t) \in A\} \\
				&\subseteq \{p \in X| f(p) \in  B_{C_h\cdot t}(A)\}\\
				&= f^{-1}(B_{C_h\cdot t}(A)) \label{eqno}
			\end{split}
		\end{equation}
		where $B_{C_h \cdot t}(A)$ are the points $y$ of $Y$ such that $d(y, A) \leq C_h\cdot t$. Then we have that
		\begin{equation}
			\begin{split}
				h^{-1}(A) &= \bigsqcup\limits_{t \in [0,1]} 	h_t^{-1}(A)\times \{t\} \\
				&\subseteq f^{-1}(B_{C_h}(A)) \times [0,1]
			\end{split}
		\end{equation}
		and so
		\begin{equation}
			diam(h^{-1}(A)) \leq diam(f^{-1}(B_1(A))) + diam([0,1]) \leq \alpha(diam(A) + 2{C_h}) + 1.
		\end{equation}
	\end{rem}
	\begin{lem}
		Let $f:(X,d_X) \longrightarrow (Y,d_Y)$ be a $\Gamma$-lipschitz-homotopy equivalence between metric spaces and let us denote by $g$ a lipschitz-homotopy inverse of $f$. Moreover we denote by $H$ the lipschitz homotopy between $g \circ f$ and $id_X$. Then $f$ is uniformly proper.
	\end{lem}
	\begin{proof}
		Since $H$ is a lipschitz map we have that if $y$ is a point in $Y$, then
		\begin{equation}
			\begin{split}
				d_X(x, g\circ f(x)) &= d_X(H(x,0),H(x,1)) \\
				&\leq C_H d_{X\times[0,1]}((x,0), (x,1)) = C_H.
			\end{split}
		\end{equation}
		Now, let $x_1$ and $x_2$ be two points in $f^{-1}(A)$, then \begin{equation}
			\begin{split}
				d_X(x_1,x_2) &\leq d_X(x_1, g\circ f (x_1)) + d_X( g\circ f (x_1),  g\circ f (x_2)) + d_X(x_1, g\circ f (x_1))\\
				&\leq 2C_H + C_gd_Y(f(x_1), f(x_2))\\
				&\leq 2C_H + C_gdiam(A).
			\end{split}
		\end{equation}
	\end{proof}
	\subsection{Manifolds of bounded geometry}
	In this section we will introduce the notion of manifolds of bounded geometry. All the definitions and propositions which follow can be found in  \cite{bound}.\\
	Let $(M,g)$ be a Riemannian manifold. 
	\begin{defn}
		The Riemannian manifold $(M,g)$ has \textbf{$k$-bounded geometry} if:
		\begin{itemize}
			\item the sectional curvature $K$ of $(M,g)$ and its first $k$-covariant derivatives are bounded, i.e. for each $ i = 0,...,k$ there is a constant $V_i$ such that $\forall x \in M$
			\begin{equation}
				|\nabla^i K(x) | \leq V_i.
			\end{equation}
			\item there is a number $C > 0$ such that for all $p$ in $M$ the injectivity radius $i_g(p)$ satisfies
			\begin{equation}
				i_g(p) \geq C.
			\end{equation}
			The maximal $C$ which satisfies this inequality will be denoted by $r_{inj}(M)$.
		\end{itemize}
		When we talk about a manifold $M$ with bounded geometry, without specifying the $k$, we mean that $M$ has $k$-bounded geometry for all $k$ in $\numberset{N}$.
	\end{defn}
	\begin{prop}\label{uno}
		Let $(M, g)$ be a Riemannian manifold of $k$-bounded geometry. Then there exists a $\delta > 0$ such that the metric up to its $k$-th order derivatives and the Christoffel symbols up to its $(k-1)$-th order derivatives are bounded in normal coordinates of radius $\delta$
		around each $x \in M$, with bounds that are uniform in $x$.
	\end{prop}
	\begin{prop}\label{due}
		Let $(M, g)$ be a Riemannian manifold of $k \geq 1$-bounded geometry.
		For every $C > 1$ there exists a $\delta > 0$ such that the normal coordinate charts $\phi_x$ are defined on $B_\delta(x)$ for each $x \in M$ and the Euclidean distance $d_E$ on the normal
		coordinates is uniformly $C$-equivalent to the metric distance $d$ induced by $M$, that
		is, $\forall x_1,x_2 \in B_\delta(x)$
		\begin{equation}
			C^{-1}d(x_1,x_2) \leq d_E (\phi_x(x_1),\phi_x(x_2)) \leq C d(x_1,x_2).
		\end{equation}
	\end{prop}
	\begin{prop}\label{charts}
		Let $(M, g)$ be a Riemannian manifold of $k$-bounded geometry with $k \geq 2$. There
		exists a $\delta$ with $0 < \delta < r_{inj}(M)$ and a constant $C > 0$ such that for all $x_1,x_2 \in M$ with $d(x_1,x_2) < \delta$ we have that the coordinate transition map
		\begin{equation}
			\phi_{2,1} = exp^{-1}_{x_2} \circ exp_{x_1}: U \longrightarrow T_{x_2}M \text{    with    } U = exp_{x_1}(B_{\delta}(x_1)\cap B_\delta(x_2)) \subset T_{x_1}M.
		\end{equation}
		is $C^{k-1}$-bounded\footnote{it means that all the derivatives of degree less or equal to $k-1$ are bounded.} with $|\phi_{2,1}|_{k-1} \leq C$.
	\end{prop}
	\begin{rem}\label{bvolume}
		Consider a Riemannian manifold $(M,g)$ such that
		\begin{equation}\label{riccibound}
			Ric(M) \geq (n-1)K
		\end{equation}
		where $K$ is a constant. In particular, if $M$ has $k_{\geq 0}$-bounded geometry we have that (\ref{riccibound}) is satisfied. Then, as consequence of Bishop-Gromov inequality, if $(M,g)$ satisfies (\ref{riccibound}) for each $p \in M$ and for each $r \geq 0$ we have that
		\begin{equation}
			\mu_M(B_r(p)) \leq C(r)
		\end{equation}
		for some function $C: \numberset{R}_{\geq 0} \longrightarrow  \numberset{R}_{\geq 0}$. This means that if $A \subset M$ has $diam(A) = r$, then, once we fix a $p$ in $A$, we have that
		\begin{equation}
			\mu_M(A) \leq \mu_M(B_{diam(A)}(p)) \leq C(diam(A)).
		\end{equation}
	\end{rem}
	\subsection{Uniformly proper and discontinuous actions}
	In this section we will introduce a particular family of actions of a group $\Gamma$ over a metric space $(X, d_X)$.
	\begin{defn}
		Consider $\Gamma$ a group which acts by isometries on a metric space $(X, d_X)$. We say that the action of $\Gamma$ is \textbf{uniformly properly discontinuous and free (FUPD)} if 
		\begin{itemize}
			\item  the action of $\Gamma$ is free and properly discontinuous,
			\item  there exists a number $\delta > 0$ such that
			\begin{equation}
				d_X(p, \gamma p) \leq \delta \implies p = \gamma p.
			\end{equation}
		\end{itemize}
	\end{defn}
	\begin{prop}\label{Gammaq}
		Let us suppose that $(M,g)$ is a manifold of bounded geometry. Then the following are equivalent:
		\begin{enumerate}
			\item $\Gamma$ induce a FUPD action,
			\item the quotient $\frac{M}{\Gamma}$ has bounded geometry.
		\end{enumerate}
	\end{prop}
	\begin{proof}
		$1 \implies 2$
		\\
		If $\delta$ is the constant of the FUPD action of $\Gamma$, then for each $p$ in $M$ we have that $B_\delta(p)$ is a trivializing open of $\frac{M}{\Gamma}$ and so $inj_{\frac{M}{\Gamma}} \geq \min\{inj_M, \delta\}$ and the curvature of $\frac{M}{\Gamma}$ has the same bounds of the curvature of $M$.
		\\$2 \implies 1$
		\\Let us suppose that for each $\delta > 0$ there is a point $p$ and a $\gamma \in \Gamma$ such that $d_M(p, \gamma p) \leq \delta$. Consider $\delta < \min\{inj_M, inj_{\frac{M}{\Gamma}}\}$. Then there is a vector $v$ in $T_pM$ with norm less then $\delta$ such that $exp_p(v) = \gamma p$. So, if we denote by $s$ the Riemannian covering $s: M \longrightarrow \frac{M}{\Gamma}$, we can apply the formula
		\begin{equation}
			s \circ exp_p = exp_{s(p)} \circ ds
		\end{equation}
		to $v$. We have that
		\begin{equation}
			s(p) = s \circ exp_p (v) = exp_{s(p)} \circ ds(v).
		\end{equation}
		Then, since $ds$ is an isometry, we have that the norm of $ds(v)$ is less or equal to the injectivity radius of $\frac{M}{\Gamma}$. This means that $ds(v)$ has to be null, then in particular $v$ is null and so
		\begin{equation}
			\gamma p = exp_p(v) = exp_p(0) = p.
		\end{equation} 
	\end{proof}
	\section{$C^k_b$-maps}
	\subsection{Maps with bounded derivatives}
	Consider $(M, g)$ be a Riemannian manifold of bounded geometry.
	\begin{defn}
		We define $\delta > 0$ to be
		\textbf{$M$-small} if the Propositions \ref{uno}, \ref{due} and \ref{charts} hold on all normal coordinate charts of
		radius $\delta$.
	\end{defn}
	\begin{defn}
		Let $X,Y$ be Riemannian manifolds of $k+1$-bounded geometry and $f \in C^k(X,Y)$. We
		say that $f$ is of class $C^k_b$ when there exist $X,Y$-small $\delta_X$ , $\delta_Y > 0$ such that for each $x \in X$ we have $f(B_{\delta_X}(x)) \subset B_{\delta_Y}(f (x))$ and the composition
		\begin{equation}
			F_x = exp^{-1}_{f(x)} \circ f \circ exp_x : B_{\delta_X}(0) \subset T_xX \longrightarrow T_yY
		\end{equation}
		in normal coordinates is of class $C^k_b$ and its $C^l$-norms ($l= 0,...,k$) as function from $T_xX$ to $T_yY$ are bounded uniformly in $x \in X$.
	\end{defn}
	\begin{rem}
		There are four remarks we wish to make:
		\begin{itemize}
			\item The condition of bounded geometry is necessary to define $C^k_b$-maps. One can see more detail in \cite{bound}.
			\item Composition of two $C^k_b$-maps is a $C^k_b$-map.
			\item Lipschitz maps between manifolds of bounded geometry are $C^0_{b}$-maps. 
			In order to prove this, observe that, since $M$ and $N$ are manifolds of bounded geometry, then the exponential maps and their inverses are lipschitz maps. In particular their lipschitz constants are uniformly bounded (see Proposition 2.5 of \cite{bound}). Then, in local normal coordinates,
			\begin{equation}
				|F_x(y)|= |F_x(y) - F_x(0)| \leq C|y| \leq C\cdot \delta_X
			\end{equation}
			where $C$ is a constant which uniformly bounds in $x$ the lipschitz constants of $F_x := exp^{-1}_{f(x)} \circ f \circ exp_x$. This means that the $C^0_b$-norm of $F_x$ is uniformly bounded in $x$.
			\item we have that if $f$ is a $C^k_{b}$-map, then, in particular, it is a $C^{k-1}_{bu}$-map.
		\end{itemize} 
	\end{rem}
	\subsection{$C^k_b$-approximation of a lipschitz map}\label{approx}
	Consider a lipschitz-homotopy equivalence $f: (M,g) \longrightarrow (N,h)$ between manifolds of bounded geometry. Let us suppose, moreover, that there is a group $\Gamma$ on $M$ and $N$ acting by isometries FUPD and assume that $f$ is $\Gamma$-equivariant. We want to show that for all $k$ there is a lipschitz-homotopy equivalence $F: (M,g) \longrightarrow (N,h)$ which is $C^k_{b}$ and that is $\Gamma$-lipschitz-homotopic to $f$.
	\\In order to prove this fact we need the following two proposition of \cite{bound}.
	\begin{prop}\label{cover}
		Let $(M, g)$ be a Riemannian manifold of $k$-bounded geometry for some $k$. Then for $\delta_2 > 0$ small enough and any $0 < \delta_1 \leq \delta_2$, $M$ has a countable cover $\{B_{\delta_1}(x_i)\}$ such that 
		\begin{itemize}
			\item $\forall i\neq j \implies d(x_i, x_j) \geq \delta_1$
			\item  There exists an explicit global bound $K$ such that for each $x \in M$ the ball $B_{\delta_2}(x)$ intersects at most $K$ of the $B_{\delta_2}(x_i)$.
		\end{itemize}
	\end{prop}
	\begin{prop}\label{primaappr}
		Let $r, \delta r > 0$, $g \in C^{k\geq 0}_{b}(B_{r + 2 \delta r}(0) \subset \numberset{R}^m, \numberset{R}^n)$ such that $g$ is uniformly continuous and fix $l > k$, and $\epsilon > 0$. Consider now a mollifier \footnote{A mollifier is a compactly supported, smooth, positive function $\phi: \numberset{R}^n \longrightarrow \numberset{R}$ such that its integral over $\numberset{R}^n$ equals to 1.} $\phi$ with support in the euclidean ball of $\numberset{R}^m$ and let us define $\phi_{\nu}(x):= \nu^{-m}\phi(\frac{x}{\nu})$. Then $g$ can be approximated by a function $G_\nu$ such that
		\begin{itemize}
			\item $G_\nu = g$ outside $B_{r + \delta r}(0)$;
			\item $G_\nu \in C^l_{b,u} \cap C^{\infty}$ on $B_r(0)$ and $G_\nu \in C^l_{b,u} \cap C^{\infty}$ on an open $A$ if $g \in C^l_{b,u} \cap C^{\infty}$ on $A$;
			\item $|g - G_\nu|_{k} \leq \epsilon$;
			\item $|G_\nu|_{l} \leq C(\nu, l)|g|_0$ on $B_r(0)$  for some $C_{\nu,l} >0$.
		\end{itemize}
		Note that $C(\nu,l)$ may grow unboundedly as $\nu \longrightarrow 0$ or $l \longrightarrow +\infty$.
	\end{prop}
	\begin{rem}
		The approximation above is defined in the following way. Let us consider a smooth map $\chi: \numberset{R} \longrightarrow [0,1]$ which is $1$ if $x \leq r$ and $0$ if $x\geq r+ \delta r$. Then we have that
		\begin{equation}
			G_\nu(x) := (1 - \chi(||x||))g(x) + \chi(||x||)\int_{\numberset{R}^m}g(x - y)\phi_{\nu}(y)dy
		\end{equation}
		where $\nu < \frac{\delta r}{2}$.
		\\moreover $G_{\nu}$ can be also written as
		\begin{equation}
			\begin{split}
				G_\nu(x) &= (1 - \chi(||x||))g_{i,\gamma}(x) + \chi(||x||)\int_{\numberset{R}^m}g_{i,\gamma}(x - y)\phi_{\nu}(y)dy \\
				&= (1 - \chi(||x||))g_{i,\gamma}(x) + \chi(||x||)\int_{B^m}g_{i,\gamma}(x - \nu y)\phi(y) dy.
			\end{split}
		\end{equation}
		Then, using this expression of $G_{\nu}$, it is also possible to consider $\nu = 0$. 
	\end{rem}
	\begin{rem}\label{epsnu}
		In particular following the proof of Lemma 2.34 of\cite{bound},we obtain that if $g$ is a lipschitz function of constant $C_g$ then
		\begin{equation}
			d(g(p), G_\nu(p)) \leq C_g \nu.
		\end{equation}
	\end{rem}
	We are ready to prove the existence of a $C^k_{b}$-approximation of a lipschitz map.
	\begin{prop}\label{appr}
		Consider two Riemannian manifolds $(M,g)$ and $(N,h)$ of bounded geometry and let $f:(M,g) \longrightarrow (N,h)$ be a lipschitz map. Fix $\nu >0$. Then there is map $F:(M,g) \longrightarrow (N,h)$ such that
		\begin{itemize}
			\item $d(F(p), f(p)) \leq K \nu$ where $K$ is a constant which depends on the lipschitz constant $C_f$ and by the uniform covering we choose to define $F$,	
			\item for all $l \geq 1$  the approximation $F$ is a $C^l_{b,u} \cap C^{\infty}$-map,
			\item Consider $\Gamma$ a group that acts FUPD by isometries on $M$ and $N$. Assume that $f$ is $\Gamma$-equivariant. Then also $F$ is $\Gamma$-equivariant.
			\item $f \sim_\Gamma F$.
		\end{itemize}
	\end{prop}
	\begin{proof}
		Observe that since $f$ is lipschitz, then in particular is uniformly continuous. Then there are two positive numbers $\sigma_1 < \sigma_2$ both $N$-small and there is a $M$-small number $\delta_f$ such that $f(B_{\delta_f}(q) \subseteq B_{\sigma_1}(f(q))$.\\
		Let us consider the Riemannian manifold $X := \frac{M}{\Gamma}$. Let us define 
		\begin{equation}
			R_0 := \frac{1}{4} min\{\delta_f, r_{inj}(X)\}.
		\end{equation}
		Following the Remark \ref{Gammaq}, we can see that $R_0 \leq inj_M$ and $R_0 \leq \delta_0$	where $\delta_0$ is the constant given by the FUPD action of $\Gamma$. Fix, then, $\delta_1 < \delta_2 < R_0$ where $\delta_2$ is $X$-small. 
		\\So, applying Proposition \ref{cover}, we know that there is a number $K$ and a countable cover of $X$ given by $\{B_{\delta_1}(x_i)\}$ such that for all $x$ in $X$ we have that $B_{\delta_2}(x)$ intersect at most $K$ of $\{B_{\delta_2}(x_i)\}_i$.
		\\Consider the preimage of this cover on $M$. Since $\delta_2 < \delta_0$, it has the form $\{ \bigsqcup_{\gamma \in \Gamma} B_{\delta_2}(\gamma x_i)\}_i$, where $\gamma x_i$ is an element of the fiber of $x_i$.
		\\Moreover we can observe that $f(B_{\delta_2}(q)) \subseteq B_{\sigma_1}(f(q))$ for all $q$ in $M$.
		\\
		Now, let us define $F_0 := f$ and  for all $i$ in $\numberset{N}_{>0}$ we have
		\begin{equation}
			F_{i+1}(p) := \begin{cases} \exp_{f(\gamma x_i)} \circ G_{i,\gamma, \nu} \circ exp_{\gamma x_i}^{-1} \mbox{  if   } p \in B_{R_0}(\bigcup_\gamma \{\gamma x_i\}) \\
				F_i(p) \mbox{   otherwise} \end{cases}
		\end{equation}
		where $G_{i,\gamma,\nu}$ is the $C^k_{b}$-approximation of $g_{i,\gamma} = \exp_{f(\gamma x_i)}^{-1} \circ F_i \circ exp_{\gamma x_i}$ using $\nu \leq (\frac{\sigma_2 - \sigma_1}{4K})$. Observe that $F_i$ is well-defined for all $i$ using the gluing lemma.
		\\Finally we can define
		\begin{equation}
			F(p) := \lim_{i \rightarrow + \infty} F_{i}(p).
		\end{equation}
		Let us check all the properties of $F$:
		\begin{itemize}
			\item $F$ is well-defined: consider $p$ a point in $M$. Let $s$ be the covering $s: M \longrightarrow X$, then $s(B_{\delta_2}(p))$ intersect at most $K$ ball $B_{\delta_2}(x_i)$. Moreover we have that, since $\delta_2 < \frac{1}{4}\delta$, then for all of these $x_i$ there exists only a $\gamma$ in $\Gamma$ such that $B_{\delta_2}(p)$ intersects $B_{\delta_2}(\gamma x_i)$.
			\\Then for all $p$ there are at most $K$ indexes $i_j$ such that $F_{i_j}(p) \neq F_{i_j + 1}(p)$ and so the limit exists since the sequence $\{F_i(p)\}_{i \geq i_K + 1}$ is constant.
			\item We have that
			\begin{equation}
				d(F(p), f(p)) \leq C_fK\nu.
			\end{equation}
			Indeed, applying the Proposition \ref{primaappr} and  Remark \ref{epsnu}, we have that 
			\begin{equation}
				d(f(p), F(p)) \leq d(f(p), F_{i_1 +1}(p)) + ... + d(F_{i_K}(p), F_{i_K + 1}(p)) \leq K\cdot C_f \nu
			\end{equation}
			\item $F$ is a $C^k_{bu}$-map. Consider a point $p$. As consequence of Proposition \ref{cover}, there is a neighborhood $U$ of $p$ and there is a finite sequence $x_{i_1},..., x_{i_k}$ of the $x_i$s such that $F_{i +1} \neq F_{i}$ on $U$. So, we have to study the boundedness of the derivatives of
			\begin{equation}
				\begin{split}
					F_p :&= exp^{-1}_{F_{i_k +1}(p)} \circ F_{i_k + 1} \circ exp_{p}\\
					&= exp^{-1}_{F_{i_k +1}(p)} \circ \exp_{f(\gamma x_i)} \circ G_{i_k,\gamma, \nu} \circ exp_{\gamma x_i}^{-1} \circ exp_{p}.
				\end{split}
			\end{equation}
			Now, we have that $exp^{-1}_{F_{i_k +1}(p)} \circ \exp_{f(\gamma x_i)}$ and $exp_{\gamma x_i}^{-1} \circ exp_{p}$ are just a change of normal coordinates, and so they are $C^k_{bu}$-maps as consequence of Proposition \ref{charts}. Then it is sufficient to study $G_{i_k,\gamma, \nu}$. Applying the Proposition \ref{primaappr}, we have that
			\begin{equation}
				|G_{i_k,\gamma, \nu}|_{k} \leq C(\nu , k)|G_{i_{k-1},\gamma, \nu}|_0 \leq C(\nu , k) \sigma_2.
			\end{equation}
			\item $F$ is $\Gamma$-equivariant. In order to prove this we have to check that all the $F_i$s are $\Gamma$-equivariant. Observe that $F_0 = f$ is $\Gamma$-equivariant. Moreover if $F_i$ is $\Gamma$-equivariant, then also $F_{i+1}$ is $\Gamma$-equivariant. Then we conclude observing that
			\begin{equation}
				\gamma F(p) = \gamma F_{i_{k}+1}(p) = F_{i_{k}+1} (\gamma p) = F(\gamma p).
			\end{equation}
		\end{itemize}
		In order to conclude the proof we have to show that $F$ and $f$ are $\Gamma$-lipschitz-homotopic. Let us define for each $\epsilon \in [0,1]$ the map
		\begin{equation}
			H_{i,\gamma, \epsilon} := G_{i,\gamma,\epsilon \nu}
		\end{equation}
		where $G_{i,\gamma,\epsilon \nu}$ is the $C^k_{bu}$-approximation of $g_{i,\gamma} = \exp_{f(\gamma x_i)}^{-1} \circ F_i \circ exp_{\gamma x_i}$ using $\epsilon \nu$. \\Let us define the map 
		\begin{equation}
			\begin{split}
				H_0: (M \times [0,1], g + g_{[0,1]}) &\longrightarrow (N,h) \\
				(p,\epsilon) \longrightarrow f(p)
			\end{split}
		\end{equation}
		and
		\begin{equation}
			H_{i+1}(p,\epsilon ) := \begin{cases} \exp_{f(\gamma x_i)} \circ H_{i,\gamma, \epsilon} \circ exp_{\gamma x_i}^{-1} \mbox{  if   } p \in B_{r_{inj}(M)}(\bigcup_\gamma \{\gamma x_i\}) \\
				H_i(p,\epsilon) \mbox{   otherwise} \end{cases}
		\end{equation}
		Then we define the map $H: (M \times [0,1], g + g_{[0,1]}) \longrightarrow (N,h)$ as
		\begin{equation}
			H(p, \epsilon) = \lim_{i \rightarrow +\infty} H_{i}(p,\epsilon).
		\end{equation}  
		Let us suppose that if $H_i$ lipschitz implies $H_{i+1}$ is lipschitz with a constant less or equal to $C \cdot C_{i}$ where $C_i$ is the costant of $H_i$. Then we will have that the constant of $H$ will be $C^{K}C_f$.
		\\
		\\To prove that if $H_{i-1}$ is lipschitz then $H_i$ is lipschitz it is sufficient to prove that $H_i$ is lipschitz on $B_{\delta_2}(\gamma x_i)$ for all $\gamma$ in $\Gamma$. Moreover we also know that the exponential map are lipschtiz with constant $C_{\sigma_2}$ and $C_{\delta_2}$ on $B_{\delta_2}(\gamma x_i)$ and on $B_{\sigma_2}(f(\gamma x_i))$. This means that we can study $H_{i, \gamma, \epsilon}$.
		\\
		Let us consider
		\begin{equation}
			d(H_{i,\gamma, \epsilon_1}(x_1), H_{i,\gamma, \epsilon_2}(x_2)) \leq d(H_{i,\gamma, \epsilon_1}(x_1), H_{i,\gamma, \epsilon_1}(x_2)) + d(H_{i,\gamma, \epsilon_1}(x_2), H_{i,\gamma, \epsilon_2}(x_2))
		\end{equation}
		One have to observe that $g_{i,\gamma}$ is a lipschitz map with costant less or equal to $C_{\delta_2} \cdot C_{\sigma_2} \cdot C_{F_i}$.
		\\Then we have
		\begin{equation}
			\begin{split}
				d(H_{i,\gamma, \epsilon_1}(x_1), H_{i,\gamma, \epsilon_1}(x_2)) &\leq \int_{B^m} |g_{i,\gamma}(x_2 - \epsilon_1 y) - g_{i,\gamma}(x_2 - \epsilon_1 y)| \phi(y) dy \\
				&\leq C_{\delta_2} \cdot C_{\sigma_2} \cdot C_{F_i}|x_1 - x_2| \\
				&\leq C_{\delta_2}^2 \cdot C_{\sigma_2} \cdot C_{F_i}d((x_1, \epsilon_1),(x_2, \epsilon_1)).
			\end{split}
		\end{equation}
		and, moreover
		\begin{equation}
			\begin{split}
				d(H_{i,\gamma, \epsilon_1}(x_1), H_{i,\gamma, \epsilon_2}(x_1)) &\leq \int_{B^m} |g_{i,\gamma}(x_2 - \epsilon_1 y) - g_{i,\gamma}(x_2 - \epsilon_2 y)| \phi(y) dy \\
				&\leq C_{\delta_2} \cdot C_{\sigma_2} \cdot C_{F_i}||x_2 - \epsilon_1y - x_2 + \epsilon_2y| \\ 
				&\leq C_{\delta_2} \cdot C_{\sigma_2} \cdot C_{F_i}|\epsilon_1 - \epsilon_2| \\
				&\leq C_{\delta_2} \cdot C_{\sigma_2} \cdot C_{F_i}d((x_1, \epsilon_1),(x_2, \epsilon_2)).
			\end{split}
		\end{equation}
		Then, we have that $H$ is a lipschitz map and so $f$ and $F$ are lipschitz-homotopic. The $\Gamma$-equivariance of $H$ follows exactly as for $F$.
	\end{proof}
	\begin{cor}\label{corappr}
		Consider a lipschitz map $f: (M,g) \longrightarrow (N,h)$ between manifolds of bounded geometry. Let us suppose that there is a closed set $C$ such that $f_{|_{M \setminus C}}$ is a $C^k_{bu}$-map. Then for all $\epsilon > 0$ there is a map $F:(M,g) \longrightarrow (N,h)$ such that $F$ is a $C^k_{bu}$-map and if $C_\epsilon$ is the $\epsilon$-neighborhood of $C$, then 
		\begin{equation}
			F_{|_{M \setminus C_\epsilon}} = f_{|_{M \setminus C_\epsilon}}.
		\end{equation}
		Moreover if $C$ is $\Gamma$-invariant and $f$ is $\Gamma$-equivariant then also $F$ is $\Gamma$-equivariant and they are $\Gamma$-lipschitz homotopic.
	\end{cor}
	\begin{proof}
		The proof is exactly the same of the previous proposition. Indeed it is sufficient to impose
		\begin{equation}
			\epsilon = 2 \delta_2,
		\end{equation}
		and to define the $F_i$s and the $H_i$s only on the $\bigsqcup\limits_{\gamma \in \Gamma}B_{\delta_2}(\gamma x_i)$ that intersect $C$. 
	\end{proof}
	\begin{lem}\label{C^k_bhomo}
		Let $f, f': (M,g) \longrightarrow (N,h)$ be two $C^k_b$ lipschitz-homotopy equivalences such that $f \sim_\Gamma f'$, where $\Gamma$ acts FUPD on $M$ and $N$. Let us fix some normal coordinates $\{x^i, s\}$ on $M \times [0,1]$ and $\{y^j\}$ on $N$. Then there is a lipschitz homotopy $H: (M \times [0,1], g + g_{eucl}) \longrightarrow (N,h)$ such that all its derivatives in normal coordinates of order minus or equal of $k$ are uniformly bounded.
	\end{lem}
	\begin{proof}
		Let us denote by $h$ the homotopy between $f$ and $f'$. Then we can define the lipschitz map $H_{0}: (M \times \numberset{R}, g + g_{0,1}) \longrightarrow (N,h)$ as
		\begin{equation}
			H_0(p,s) := \begin{cases} f(p) \mbox{      if     } s \leq \frac{1}{4} \\
				h(p,\frac{1}{2}(t + \frac{1}{2})) \mbox{      if     } s \in  [\frac{1}{4}, \frac{3}{4}] \\
				f'(p) \mbox{      if     } s \geq \frac{3}{4}. \end{cases}
		\end{equation}
		Let us consider the action of $\Gamma$ on $M \times \numberset{R}$ as the action $\gamma(p,t) := (\gamma(p), t)$. Then we can consider the $\Gamma$-equivariant closed subset given by $M \times [\frac{3}{8}, \frac{5}{8}]$. Since $M \times [0,1]$ is a manifold of bounded geometry, we can apply the Corollary \ref{corappr}, and we obtain a map $\tilde{H}$ which is a $C^k_b$-map such that $\tilde{H}(p, 0) = f(p)$ and $\tilde{H}(p, 1) = f'(p)$. Then it is sufficient to define
		\begin{equation}
			H := \tilde{H}_{|_{M \times [0,1]}}.
		\end{equation}
	\end{proof}
	\begin{cor}\label{C^k_bpH}
		Let $H:M \times [0,1] \longrightarrow N$ be the homotopy in the Lemma \ref{C^k_bhomo}. If $p_H: f^*(T^\delta N) \times [0,1] \longrightarrow N$ is the submersion related to $H$ then $p_H$ has, in normal coordinates, the derivatives of order $k$ uniformly bounded.
	\end{cor}
	\begin{proof}
		It is a consequence of Proposition \ref{tilde}.
	\end{proof}
	\section{$L^2$-cohomology and pull-back}
	\subsection{$\mathcal{L}^2$-forms}
	Let us consider an orientable Riemannian manifold $(M,g)$ and let us denote by $\Omega^k_{c}(M)$ the space of complex differential forms with compact support. The Riemannian metric $g$ induces for every $k \in \numberset{N}$ a scalar product on $\Omega^k_c(M)$ as follows: consider $\alpha$ and $\beta$ in $\Omega^*_{c}(M)$ then
	\begin{equation}
		<\alpha, \beta>_{\mathcal{L}^2(M)} := \int\limits_{M}\alpha \wedge \star \bar{\beta},
	\end{equation}
	where $\star$ is the Hodge star operator given by $g$.
	\\This hermitian product gives a norm on $\Omega^k_{c}(M)$:
	\begin{equation}
		\begin{split}
			\| \alpha \|^2_{\mathcal{L}^2(M)} :&= <\alpha, \alpha>_{\mathcal{L}^2(M)} \\ 
			&= \int\limits_{M}\alpha \wedge \star \bar{\alpha} < +\infty.
		\end{split}
	\end{equation}
	\begin{defn}
		We will denote by $\mathcal{L}^2\Omega^k(M)$ the Hilbert space given by the closure of $\Omega^*_{c}(M)$ respect the norm $| \cdot |_{\mathcal{L}^2\Omega^k(M)}$. Moreover we can also define the Hilbert space $\mathcal{L}^2(M)$ given by 
		\begin{equation}
			\mathcal{L}^2(M) := \bigoplus \limits_{k \in \numberset{N}} \mathcal{L}^2\Omega^k(M).
		\end{equation}
		The norm of $\mathcal{L}^2(M)$ it will be denoted by $| \cdot |_{\mathcal{L}^2}$ or $| \cdot |_{\mathcal{L}^2(M)}$.
	\end{defn}
	\begin{rem}
		Since $\Omega_c^k(M)$ is dense in $\mathcal{L}^2\Omega^k(M)$, then $\Omega_c^*(M)$ is dense in $\mathcal{L}^2(M)$.
	\end{rem}
	\begin{defn}
		Consider an orientable Riemannian manifold $(M,g)$ with $dim(M) = m$. Then, the \textbf{chirality operator} $\tau_M: \mathcal{L}^2(M) \longrightarrow \mathcal{L}^2(M)$ is the operator defined for each $\alpha$ in $\Omega_c^*(M)$ as
		\begin{equation}
			\tau_M(\alpha) := i^{\frac{m}{2}} \star \alpha
		\end{equation}
		if $m$ is even and
		\begin{equation}
			\tau_M(\alpha) := i^{\frac{m + 1}{2}} \star \alpha
		\end{equation}
		if $m$ is odd.
	\end{defn}
	\begin{rem}
		The chirality operator is a bounded operator. In particular, respect to $|| \cdot ||_{\mathcal{L}^2}$ it is an isometry and an involution and so we obtain that
		\begin{equation}
			\tau_M^2 =1
		\end{equation}
		and
		\begin{equation}
			||\tau_M|| =1.
		\end{equation}
	\end{rem}
	\subsection{Exterior derivative}
	Let us consider two orientable Riemannian manifolds $(M,g)$ and $(N,h)$ and let $F$ be a linear operator $F:dom(F) \subseteq \mathcal{L}^2(M)\longrightarrow \mathcal{L}^2(N)$.
	\begin{defn}
		The operator $F$ is \textbf{$\mathcal{L}^2$-bounded} if $F$ is a bounded operator respect to the $\mathcal{L}^2$-norms.
		\\Moreover we say that $F$ is a \textbf{closed} operator if for every $\alpha$ in $dom(F)$ and for each sequence $\{\alpha_n\}$ converging to $\alpha$ there is a $\beta$ in $\mathcal{L}^2(N)$ such that 
		\begin{equation}
			\lim\limits_{n \rightarrow + \infty} F(\alpha_n) = \beta \in \mathcal{L}^2(N)
		\end{equation}
		and, moreover, $F(\alpha) = \beta$.
	\end{defn}
	\begin{rem}
		Every bounded operator is closed. In particular we have that $F$ is a closed operator if and only if its graph is a closed subset of $\mathcal{L}^2(M) \times \mathcal{L}^2(N)$.
	\end{rem}
	Consider now a linear operator $F:\Omega^*_c(M)\longrightarrow \Omega^*_c(N)$. Since $\Omega_c^*(M)$ is dense in $\mathcal{L}^2(M)$, if $F$ is $\mathcal{L}^2$-bounded, then there is a unique extension $\overline{F}: \mathcal{L}^2(M) \longrightarrow \mathcal{L}^2(N)$ of $F$ and it is bounded.
	\\On the other hand, if $F$ is unbounded, then there could be some different ways to extend $F$. 
	\\Let us denote by $d: \Omega_c^*(M) \longrightarrow \Omega_c^*(M)$ the operator given by the exterior derivative of compactly supported complex differential forms. We introduce some \textit{closed extensions} i.e. some extensions $\overline{d}$ of $d$ such that $\overline{d}$ is a closed operator.
	\begin{defn}
		The \textbf{maximal extension of $d$} is the operator $d_{max}: dom(d_{max}) \subseteq \mathcal{L}^2(M) \longrightarrow \mathcal{L}^2(N)$ defined on
		\begin{equation}
			\begin{split}
				dom(d_{max}) := \{&\alpha \in \mathcal{L}^2(M)| \exists \alpha' \in \mathcal{L}^2(N) s.t. \forall \beta \in  \Omega_c^*(N)\\ & \langle \alpha', d(\beta) \rangle_{\mathcal{L}^2(\Omega^*(M))} = \langle \alpha, \beta \rangle_{\mathcal{L}^2(N)} \}
			\end{split}
		\end{equation}
		posing
		\begin{equation}
			d_{max}(\alpha)= \alpha'.
		\end{equation}
	\end{defn}
	Moreover we also have
	\begin{defn}
		The \textbf{minimal extension of $d$} is the operator $d_{min}: dom(d_{min}) \subseteq \mathcal{L}^2(M) \longrightarrow \mathcal{L}^2(N)$ defined on
		\begin{equation}
			\begin{split}
				dom(d_{min}) := \{&\alpha \in \mathcal{L}^2(M)| \exists \{\alpha_n\} \in \Omega_c^*(N) \mbox{ and} \exists y \in \mathcal{L}^2(N) s.t.\\
				& \lim\limits_{n\rightarrow +\infty} \alpha_n = \alpha \mbox{ and} \lim\limits_{n\rightarrow +\infty} d\alpha_n = y \}
			\end{split}
		\end{equation}
		as
		\begin{equation}
			d_{min}(\alpha)= y.
		\end{equation}
	\end{defn}
	Given two operators $F: dom(F) \subseteq \mathcal{L}^2(M) \longrightarrow \mathcal{L}^2(N)$ and $G: dom(G) \subseteq \mathcal{L}^2(M) \longrightarrow \mathcal{L}^2(N)$ we will say that
	\begin{equation}
		F \subseteq G \iff Graph(F) \subseteq Graph(G)
	\end{equation}.
	We can observe that
	\begin{equation}
		d \subseteq d_{min} \subseteq d_{max}.
	\end{equation}
	Moreover every closed extension $\overline{d}$ is such that
	\begin{equation}
		d_{min} \subseteq \overline{d} \subseteq d_{max}.
	\end{equation}
	However if the Riemannian manifold $(M,g)$ is complete then the following proposition of \cite{Gaf} holds.
	\begin{prop}
		Let $(M,g)$ a complete Riemannian manifold. Then the maximal and the minimal domain of the exterior derivative coincide and so
		\begin{equation}
			d_{max} = d_{min}.
		\end{equation}
	\end{prop}
	As consequence of this Proposition we also have that there is a unique closed extension for $d^*: \Omega_c^*(M) \longrightarrow \Omega_c^*(M)$, the adjoint of $d$ and also for the operator $d + d^*: \Omega_c^*(M) \longrightarrow \Omega_c^*(M)$.
	\\When $(M,g)$ is a manifold of bounded geometry, in particular is a complete Riemannian manifold and so, using an abuse of notation, we denote the closed extensions of these operators by $d$, $d^*$, $d + d^*$. 
	\subsection{$L^2$-cohomology and reduced $L^2$-cohomology}
	Let us consider a closed extension $\overline{d}$ of the exterior derivative on an orientable Riemannian manifold $(M,g)$. Let us suppose that $\overline{d}(dom(\overline{d})) \subseteq ker(\overline{d}) \subseteq dom(\overline{d})$, and so $\overline{d} \circ \overline{d}$ is well defined and equals to zero. In particular, this condition is satisfied for each closed extension $\overline{d}$ if $(M,g)$ is a complete Riemannian manifold.
	\\Let us consider now the sequence
	\begin{equation}
		0 \longrightarrow \mathcal{L}^2(\Omega^1(M)) \xrightarrow{\overline{d}} \mathcal{L}^2(\Omega^2(M)) \xrightarrow{\overline{d}} \mathcal{L}^2(\Omega^3(M)) \xrightarrow{\overline{d}}...
	\end{equation}
	It is, following the definition given in \cite{Bruning}, an Hilbert complex.
	\begin{defn}
		We will define \textbf{$i$-th group of $L^2$-cohomology } the group
		\begin{equation}
			H^i_{2, \overline{d}}(M) := \frac{ker(\overline{d}_i)}{im(\overline{d}_{i-1})}.
		\end{equation}
	\end{defn}
	\begin{defn}
		We will define \textbf{$i$-th group of reduced $L^2$-cohomology } the group
		\begin{equation}
			\overline{H}^i_{2, \overline{d}}(M) := \frac{ker(\overline{d}_i)}{\overline{im(\overline{d}_{i-1})}}.
		\end{equation}
	\end{defn}
	\begin{rem}
		If $(M,g)$ is a complete manifold, since \cite{Gaf}, the (reduced or not) $L^2$-cohomology is uniquely defined. In this case we denote the groups as $H^i_{2}(M)$ and $\overline{H}^i_{2}(M)$.
	\end{rem}
	\begin{prop} \label{boundedness}
		Let $(N,h)$ and $(M,g)$ be two orientable Riemannian manifolds. Let $B$ and $K: \mathcal{L}^2(N) \longrightarrow \mathcal{L}^2(M)$ be two $ \mathcal{L}^2$-bounded operators. Let us suppose that $B(\Omega^*_c(N)) \subseteq \Omega^*_c(M)$ and
		\begin{equation}
			d B = \pm Bd + K
		\end{equation}
		over $\Omega^*_c(N)$. Then we have that
		\begin{equation}
			B(dom(d_{min})) \subseteq dom(d_{min})
		\end{equation}
		and $d_{min}B = \pm Bd_{min} + K$ on the minimal domain of $d$.
	\end{prop}
	\begin{proof}
		Let $\alpha$ be an element in $dom(d_{min})$. This means that there is a sequence $\{\alpha_n \}$ in $\Omega^*_c(N)$ such that
		\begin{equation}
			\begin{cases}
				&\alpha = \lim\limits_{n \rightarrow +\infty} \alpha_n \\
				&d\alpha = \lim\limits_{n \rightarrow +\infty} d\alpha_n.
			\end{cases}
		\end{equation}
		Using that $B$ is continuous we have
		\begin{equation}
			B\alpha = \lim\limits_{n \rightarrow +\infty} B\alpha_n
		\end{equation}
		where $\{B\alpha_n\}$ is a sequence in $\Omega^*_c(M)$. Moreover one can see that the limit of $dB\alpha_n$ exists:
		\begin{equation}
			\begin{split}
				\lim\limits_{n\rightarrow + \infty} dB\alpha_n &= \lim\limits_{n\rightarrow + \infty} \pm Bd\alpha_n + K\alpha_n\\
				&= \pm B\lim\limits_{n\rightarrow + \infty} d\alpha_n + K\lim\limits_{n\rightarrow + \infty}\alpha_n\\
				&= \pm Bd\alpha + K\alpha.
			\end{split}
		\end{equation}
		So $d_{min}$ is well defined in $B\alpha$ and
		\begin{equation}
			d_{min}B \alpha = \pm Bd_{min}\alpha + K\alpha.
		\end{equation}
	\end{proof}
	\begin{rem}\label{altra}
		Let us suppose $(M,g)$ and $(N,h)$ are complete, orientable, Riemannian manifolds. In general an operator $A: dom(A) \subseteq \mathcal{L}^2(N) \longrightarrow \mathcal{L}^2(M)$ induces a map in $L^2$-cohomology if
		\begin{itemize}
			\item $dom(d_N) \subseteq dom(A)$,
			\item $A(dom(d_N)) \subseteq dom(d_M)$,
			\item $A$ and $d$ commute.
		\end{itemize}
		Moreover if we also want that $A$ induces a map in reduced $L^2$-cohomology, we need that $A$ is $\mathcal{L}^2$-bounded on $\overline{im(d)}$. Indeed, given $\alpha$ in $dom(d_N)$, we have that
		\begin{equation}
			\begin{split}
				A(\alpha + \lim\limits_{k \rightarrow + \infty}d\beta_k) &= A(\alpha) + A(\lim\limits_{k \rightarrow + \infty}d\beta_k) \\
				&=A(\alpha) + \lim\limits_{k \rightarrow + \infty} A(d \beta_k) \\
				&=A(\alpha) + \lim\limits_{k \rightarrow + \infty} dA(\beta_k)
			\end{split}
		\end{equation}
		and so, in reduced $L^2$-cohomology,
		\begin{equation}
			[A(\alpha + \lim\limits_{k \rightarrow + \infty}d\beta_k)] = [A(\alpha)].
		\end{equation}
	\end{rem}
	\begin{cor}\label{boundedness2}
		Consider two operators $B$ and $K$ as in Proposition \ref{boundedness}. Let us suppose that $K = 0$ and that
		\begin{equation}
			dB = Bd
		\end{equation}
		on $\Omega_c^*(M)$. Then $B$ induces a map in (un)-reduced $L^2$-cohomology.\footnote{We are considering $\overline{d} = d_{min}$.}
	\end{cor}
	\subsection{Problems with pullback}\label{probl}
	Let $f:(M, g) \longrightarrow (N,h)$ be a map between orientable Riemannian manifolds. In general the pullback $f^*$ is not an $\mathcal{L}^2$-bounded operator and it doesn't induce an operator in $L^2$-cohomology (reduced or not). In this section we will see some examples.
	\begin{exem}
		In this first example we will show that the pullback of a map $p: (E,g) \longrightarrow (N,h)$ between complete Riemannian manifolds is not a $\mathcal{L}^2$-bounded operator even if $p$ is a proper submersion. Let us consider $(N,h) = (\numberset{R}, g_{eucl})$, $E = (\numberset{R} \times S^1,g)$ where
		\begin{equation}
			g(x,\theta) = dx^2 + e^{\frac{x^2}{2}} d\theta^2.
		\end{equation}
		Consider $p: (E,g) \longrightarrow (N,h)$ given by the projection $(x,y) \longrightarrow x$. One can observe that $p$ is a proper submersion, but $p^*$ is not $\mathcal{L}^2$-bounded.\\
		Indeed if $\alpha = e^{-\frac{x^2}{2}}dx$ in $\mathcal{L}^2(N)$, then
		\begin{equation}
			\begin{split}
				\| p^*(\alpha)\|_{\mathcal{L}^2(E)} &= \int\limits_{\numberset{R}} \int\limits_{S^{1}} e^{-x^2} Vol_{\numberset{R} \times S^1} \\
				&= \int\limits_{\numberset{R}} \int_{S^1}  e^{-x^2} e^{x^2} d\theta dx \\
				&= \int\limits_{\numberset{R}} 2\pi e^{-x^2} e^{x^2} dx = +\infty,
			\end{split}
		\end{equation}
		and so $p^*(\mathcal{L}^2(N))$ is not in $\mathcal{L}^2(E)$. 
	\end{exem}
	The pullback is not a bounded operator between the $\mathcal{L}^2$-spaces even if $(M,g)$ and $(N,h)$ are manifolds of bounded geometry and $f$ is a smooth lipschitz-homotopy equivalence.
	\begin{exem}\label{nuovo}
		Let us consider $\alpha: \numberset{R} \longrightarrow \numberset{R}$ defined by
		\begin{equation}
			\alpha(x) = \begin{cases} \alpha_0e^{-\frac{1}{\tan^2(x)}} \mbox{ if $x \in [-1,1]$} \\
				|x| \mbox{ otherwise}
			\end{cases}
		\end{equation}
		where 
		\begin{equation}
			\alpha_0 := e^{-\frac{1}{\tan^2(1)}}.
		\end{equation}
		Consider the map $\psi: \numberset{R} \longrightarrow \numberset{R}$
		\begin{equation}
			\beta(x) = \begin{cases} \alpha(x) \mbox{ if $x \geq 0$} \\
				-\alpha(x) \mbox{ otherwise}.
			\end{cases}
		\end{equation}
		Observe that for each $\epsilon > 0$ we have that $\beta$ is a smooth map outside $C_\epsilon := [-1 - \epsilon, -1 + \epsilon] \sqcup [1 - \epsilon, 1 + \epsilon]$. Fix $\epsilon \leq \frac{1}{2}$. Consider $\psi$ a smooth approximation of $\beta$ which is equal to $\beta$ outside $C_\epsilon$. We know that such a approximation exists since Corollary \ref{corappr}.
		\\We have that $\psi$ is also lipschitz-homotopic to $id_\numberset{R}$, indeed the homotopy $h:\numberset{R} \longrightarrow \numberset{R}$
		\begin{equation}
			h(x,t) := \psi(x)\cdot t + (1-t) \cdot x
		\end{equation}
		is lipschitz. 
		We have that $\psi$ is smooth, indeed it is clearly smooth if $x \neq 0$ and
		\begin{equation}
			\frac{\partial^k}{\partial^k x} \psi(0) = 0,
		\end{equation}
		for each $k \geq 1$. Moreover $\psi$ is lipschitz-homotopic to the identity: the homotopy is $h: \numberset{R} \times [0,1] \longrightarrow \numberset{R}$ defined as follow
		\begin{equation}
			h(x,s) = s \psi(x) + (1-s)x.
		\end{equation}
		Consider now
		\begin{equation}
			f(x,y) = \begin{cases} (x, \psi(y-1)) \mbox{ if $y \geq 1$} \\
				(x,\psi(y+1)) \mbox{ if $y \leq -1$}\\
				(x,0) \mbox{ otherwise}.
			\end{cases}
		\end{equation}
		We have that $f$ is a uniformly proper and lipschitz map. In particular it is smooth since 
		\begin{equation}
			\frac{\partial^k}{\partial^k x} \psi(0) =  \frac{\partial^k}{\partial^k x} \beta (0) = 0.
		\end{equation}
		Finally $f$ is lipschitz-homotopic to the identity: indeed the homotopy
		\begin{equation}
			H_1(x,y, s) = \begin{cases} (x, \psi(y-s)) \mbox{ if $y \geq s$} \\
				(x,\psi(y+s)) \mbox{ if $y \leq -s$}\\
				(x,0) \mbox{ otherwise}.
			\end{cases}
		\end{equation}
		is a lipschitz homotopy between $f$ and $(id_{\numberset{R}}, \psi)$ and $H_2(x,y,s) := (x, h(y,s))$ is a lipschitz homotopy between $(id_{\numberset{R}}, \psi)$ and $id_{\numberset{R}^2}$.
		\\
		\\Let us consider the $0$-form given by $\chi_A$, where
		\begin{equation}
			A := \{(x,y) \in R^2 | y \in [-\frac{e^{-x^2}}{2},+\frac{e^{-x^2}}{2}]\}.
		\end{equation} 
		We have that $\chi_A$ is in $\mathcal{L}^2(\numberset{R}^2)$, but $f^*(\chi_A)$ is not squared-integrable.
	\end{exem}
	Finally we will show that given a lipschitz-homotopy equivalence $\phi$ between manifolds of bounded geometry, then $\phi^*(dom(d)) \nsubseteq dom(d)$ and so $\phi^*$ doesn't induce a morphism on $L^2$-cohomology (reduced or not).
	\begin{exem}
		Let us consider $\phi: (\numberset{R}^4, g_{eucl}) \longrightarrow (\numberset{R}^4, g_{eucl})$ such that $\phi(x,y,z,w) = (\sigma(x), \sigma(y), \sigma(z), w)$ where
		\begin{equation}
			\sigma(t) := \begin{cases}
				0 \mbox{    if   } |t| \leq 1 \\
				\psi(t-1) \mbox{    if   } t \leq -1 \\
				\psi(t+1) \mbox{    if   } t \geq 1.
			\end{cases}
		\end{equation}
		where $\psi$ is the map defined in Example \ref{nuovo}. Since the derivatives of each order of $\psi$ in $0$ are null, then $\sigma$ is a smooth map. So $\phi$ is a smooth map.
		\\Consider a compactly supported $0$-form  $\alpha$ such that around the origin is $\frac{1}{|(x,y,z,w)|^{k}}$ where $\frac{1}{2} < k < 2$, then $\alpha$ is in $dom(d)$, indeed we know that $\alpha$ is in the Sobolev space $W^{1,2}$.
		\\On the other hand $\phi^*\alpha$ is not an $\mathcal{L}^2$-form, indeed we have that on $[-1,1] \times [-1,1] \times[-1,1] \times U$,  where $U$ is a sufficiently small neighborhood of $0$ in $\numberset{R}$, we have that
		\begin{equation}
			\phi^*\alpha(x,y,z,w) = \frac{1}{w^k}
		\end{equation}
		and so 
		\begin{equation}
			\begin{split}
				\int_{\numberset{R}^4} (\phi^*\alpha)^2 d\mu_{\numberset{R}^4} &\geq \int_{[-1,1]^3} \int_{U} \frac{1}{w^{2k}} d\mu_{\numberset{R}} d\mu_{\numberset{R}^3}\\
				&\geq 2^3 \cdot \int_{U} \frac{1}{w^{2k}} d\mu_{\numberset{R}}\\ 
				&= + \infty,
			\end{split}
		\end{equation}
		since $2k > 1$.
		\\Observe that $\phi$ is a lipschitz-homotopy equivalence because there is a lipschitz homotopy $H$ between $\phi$ and the identity. This homotopy is defined as follow
		\begin{equation}
			\begin{split}
				H: (\numberset{R}^4 \times [0,1], g_{eucl} \times g_{[0,1]}) &\longrightarrow (\numberset{R}^4, g_{eucl}) \\
				(x,y,z,w, t) &\longrightarrow (h(x,t), h(y,t), h(z,t), w ).
			\end{split}
		\end{equation}
		where $h$ is defined in the Example \ref{nuovo}.
	\end{exem}
	\subsection{Radon-Nicodym derivative and Fiber Volume}
	In this section we will introduce some tools of Measure Theory, that we will use to define a family of maps whose pullbacks are $\mathcal{L}^2$-bounded operators.
	\begin{defn}
		Let $X$  be a set with a $\sigma$-algebra $\Sigma$ (i.e. a \textit{measurable set}). Let $\mu$ and $\nu$ two measures respect $\Sigma$. We will say that $\nu$ is \textbf{abolutely continuous with respect to} $\mu$ if and only if for all measurable $A$
		\begin{equation}
			\mu(A) = 0 \implies \nu(A)=0.
		\end{equation}
		We will denote it by $ \nu < < \mu$.
	\end{defn}
	\begin{defn}
		A measure $\mu$ over a measurable set $(X, \Sigma)$ is said \textbf{$\sigma$-finite} if $X$ is the countable union of substsets with finite measure.
	\end{defn}
	\begin{rem}
		The measure over a Riemannian manifold given by its metric is always $\sigma$-finite. It follows from para-compactness.
	\end{rem}
	\begin{thm}[Radon-Nyikodym]
		Let $(X, \Sigma)$ be a meausurable set. Given two measures $\mu$ and $\nu$ such that $\nu < < \mu$ and $\mu$ is $\sigma$-finite, then there is a non-negative $\mu$-measurable function $f: X \longrightarrow \numberset{R}$ such that
		\begin{equation}
			\nu(A) = \int_A f d\mu.
		\end{equation}
		We will call $f$ the \textbf{Radon-Nicodym derivative} of $\nu$ respect to $\mu$ and we will denote it by $\frac{\partial \nu}{\partial \mu}$.
	\end{thm}
	\begin{rem}
		If $g$ is an integrable function respect to $\nu$, then
		\begin{equation}
			\int_M g d\nu = \int_M g f d\mu.
		\end{equation}
	\end{rem}
	\begin{rem}
		The Radon-Nicodym derivative is unique (up to $0$-measure sets).
	\end{rem}
	\begin{rem}
		Let us consider on a measurable space $(X, \Sigma)$ and consider three $\sigma$-finite measure $\lambda$, $\mu$ and $\sigma$ such that $\lambda << \mu << \sigma$. Then
		\begin{equation}
			\frac{\partial \sigma}{\partial \lambda} = \frac{\partial \sigma}{\partial {\mu}} \frac{\partial \mu}{\partial \sigma}
		\end{equation}
	\end{rem}
	Let us consider now a measurable map $f: (M, \Sigma_M, \nu) \longrightarrow (N, \Sigma_N, \mu)$ 
	\begin{defn}
		The \textbf{pushforward measure} is the measure $f_\star(\nu)$ on $(N, \Sigma_N)$ defined for all $A$ in $\Sigma_N$ as
		\begin{equation}
			f_\star(\nu)(A) := \nu(f^{-1}(A)).
		\end{equation}
	\end{defn}
	\begin{rem}
		If $g$ is an integrable function on $N$, then
		\begin{equation}
			\int_M f^*(g) d\nu = \int_Ngd(f_\star\nu).
		\end{equation}
	\end{rem}
	\begin{defn}
		Let $(M, \Sigma_M, \nu)$ and $(N, \Sigma_N, \mu)$ be two measured spaces and let $f:(M, \Sigma_M, \nu) \longrightarrow (N, \Sigma_N, \mu)$ be a function such that $f_\star(\nu)$ is absolutely continuous with respect to $\nu$.
		Let $(N, \mu)$ be $\sigma$-finite, then the \textbf{Fiber Volume} is the function
		\begin{equation}
			Vol_{f, \nu, \mu} := \frac{\partial f_\star \nu}{\partial \mu}.
		\end{equation}
	\end{defn}
	\subsection{Radon-Nicodym-Lipschitz maps}
	\begin{defn}
		A \textbf{measured metric space} is a triplet $(M, d_M, \nu_M)$ where $d_M$ is a metric space, $M$ is a topological space with the topology induced by $d_M$ and $\nu_M$ is a measure defined on the $\sigma$-algebra of Borellians of $(M,d_M)$.
	\end{defn}
	\begin{defn}
		Let $(M, d_M, \nu_M)$ and $(N, d_N, \mu_N)$ be two measured and metric spaces. A map $f: (M, \nu) \longrightarrow (N, \mu)$ is \textbf{Radon-Nikodym-lipschitz} or \textbf{R.-N.-lipschitz} if
		\begin{itemize}
			\item $f$ is lipschitz
			\item $f$ has bounded Fiber Volume.
		\end{itemize}
	\end{defn}
	Consider $f: (M, d_M, \nu_M) \longrightarrow (N, d_N, \mu_N)$ an R.-N.-lipschitz map and let $C$ be the supremum of $\frac{\partial f_{\star} \mu_M}{\partial \mu_N}$. Then for all measurable set $A \subseteq N$, we have that
	\begin{equation}
		\label{follow}
		\begin{split}
			\mu_M(f^{-1}(A)) &= \int_A \frac{\partial f_{\star} \mu_M}{\partial \mu_N} d\mu_N \\
			&\leq C \int_A d\mu_N = C \mu_N(A).
		\end{split}
	\end{equation}
	A measurable, lipschitz map which satisfies the above inequality is called a \textit{volume bounded map} or \textit{v.b.-map}: these maps are defined by Thomas Schick in his Ph.D. thesis. 
	\begin{prop}
		Let us consider a map $f: (M, d_M, \nu_M) \longrightarrow (N, d_N, \mu_N)$ between measured metric spaces. Then we have that $f$ is a R.-N.-map if and only if it is a v.b.-map.
	\end{prop}
	\begin{proof}
		If $f$ is a R.-N.-lipschitz map, then it follows by \ref{follow} that it is a v.b.-map.
		To prove the other implication we can start observing that the Fiber Volume is well-defined since $\mu_N << \nu_M$. Consider $C$ the constant such that $\nu_M(f^{-1}A) \leq C\cdot \mu_N(A)$: then 
		\begin{equation}
			|\frac{\partial f_{\star} \nu_M}{\partial \mu_N}| < C.
		\end{equation}
		Indeed if the above inequality is not satisfied, then there is $A \subseteq N$ such that $\frac{\mu_M(f^{-1}A)}{\mu_N(A)} > C$. But this implies
		\begin{equation}
			\begin{split}
				\mu_M(f^{-1}(A)) &= \int_A \frac{\partial f_\star \nu_M}{\partial \mu_N}\\ 
				&\geq C \cdot \int_A d\mu_N\\
				&= C \cdot \mu_N(A)
			\end{split}
		\end{equation}
		and so we have the contradiction.
	\end{proof}
	Let us prove some properties of R.-N.-lipschitz maps.
	\begin{prop}\label{compo2}
		Consider $f:(M, d_M, \mu_M) \longrightarrow (N, d_N, \mu_N)$ and $g: (N, d_N, \mu_N) \longrightarrow (W, d_W, \mu_W)$ two R.-N.-lipschitz maps. Then $(g \circ f): (M, d_M, \mu_M) \longrightarrow (W, d_W, \mu_W)$ is a R.-N.-lipschitz map.
	\end{prop}
	\begin{proof}
		Since the equivalence of the definitions, we can check that the composition of two v.b.-maps is a v.b.-map. We can start observing that the composition of lipschitz map is lipschitz. Moreover we also have that
		\begin{equation}
			\mu_M((g \circ f)^{-1}A) \leq C_f \mu_N(g^{-1}(A)) \leq C_f \cdot C_g \mu_N(A)
		\end{equation}
		and so it means that $g \circ f$ is a v.b.-map.
	\end{proof}
	\begin{prop} \label{couple}
		Let us consider two R.-N.-lipschitz maps $f:(M, d_M, \mu_M) \longrightarrow (X, d_X, \mu_X)$ and $g:(N, d_N, \mu_N) \longrightarrow (Y, d_Y, \mu_Y)$. Then the map
		\begin{equation}
			(f,g): (M \times N, d_M \times d_N, \mu_M \times \mu_N) \longrightarrow (X \times Y, d_X \times d_Y, \mu_X \times \mu_Y)
		\end{equation}
		is a R.-N.-lipschitz map.
	\end{prop}
	\begin{proof}
		Again we will show that $(f,g)$ is a v.b.-map. We can observe that $(f,g)$ is lipschitz. Moreover we can also consider a subset $A \times B$ of $X \times Y$. Then we have that
		\begin{equation}
			\begin{split}
				\mu_{M \times N}((f,g)^{-1}(A \times B)) &= \mu_M(f^{-1}(A))\cdot \mu_N(g^{-1}(B)) \\
				&\leq  C_f \cdot C_g \mu_X(A)\cdot\mu_Y(B) \\
				&\leq C_f \cdot C_g \mu_{X \times Y} (A \times B).
			\end{split}
		\end{equation}
		Since the sets $\{A \times B\}$ are generators of the Borel $\sigma$-algebra of $X \times Y$ we can conclude.
	\end{proof}
	\begin{rem}
		We can also prove the Proposition \ref{couple} checking that the Fiber Volume of $(f,g)$ in a point $(p,q)$ is given by $Vol_f(p)\cdot Vol_g(q)$.
	\end{rem}
	\begin{prop}\label{mai}
		Let $(M,g)$ and $(N,h)$ be Riemannian manifolds. Let $f:(M,g) \longrightarrow (N,h)$ be a R.-N.-lipschitz map. Then $f$ induces an $\mathcal{L}^2$-bounded pullback $f^*:\mathcal{L}^2(N) \longrightarrow \mathcal{L}^2(M)$. In particular the norm of $f^*$ is less or equal to $K_f \cdot \sqrt{C_{Vol}}$, where $C_{Vol}$ is the maximum of the Fiber Volume.\footnote{In order to prove this Proposition we will use some inequalities between the norms of differential forms that can be found in Appendix \ref{ineqform}.}
	\end{prop}
	\begin{proof}
		Let $\omega$ be a smooth form with compact support in $\mathcal{L}^2(N)$. Then
		\begin{equation}
			\begin{split}
				||f^*\omega||_{\mathcal{L}^2(M)}^2 &= \int_M ||f^*\omega||^2 d\mu_g \\
				&\leq \int_M K_f^2 f^*(||\omega||^2) d\mu_g \\
				&= K_f^2 \int_N ||\omega||^2 d(f_\star \mu_g)\\
				&= K_f^2 \int_N ||\omega||^2 Vol_{f,g,h} d\mu_h\\
				&\leq K_f^2 C_{Vol}  \int_N ||\omega||^2 d\mu_h\\
				&=  K_f^2 C_{Vol} ||\omega||_{\mathcal{L}^2(N)}.
			\end{split}
		\end{equation}
	\end{proof}
	\begin{defn}
		Consider a map $f: (M,g) \longrightarrow (N,h)$ between Riemannian manifolds. We say that $f$ is a \textbf{$L^2$-map}, if it is a smooth, uniformly proper, R.-N.-lipschitz map.
	\end{defn}
	\begin{prop}
		Consider a $L^2$-map $f:(M,g) \longrightarrow (N,h)$ between complete Riemannian manifolds. Then $f^*$ induces a morphism between the (un)-reduced $L^2$-cohomology.
	\end{prop}
	\begin{proof}
		It is a consequence of Remark \ref{altra}.
	\end{proof}
	\subsection{Quotients of differential forms}
	In next sections we will focus on submersions. In particular we will study their Fiber Volumes. In order to do this we need the notion of quotient of differential forms
	\begin{defn}
		Let us consider a differentiable manifold $M$. Given two differential forms $\alpha \in \Omega^k(M)$, $\beta \in \Omega^n(M)$ we define \textbf{a quotient between $\alpha$ and $\beta$}, denoted by $\frac{\alpha}{\beta}$ as a section of $\Lambda^{k-n}(M)$ such that for all $p$ in $M$
		\begin{equation}
			\alpha(p) =  \beta(p) \wedge \frac{\alpha}{\beta}(p).\label{quotient}
		\end{equation}
	\end{defn}
	\begin{rem}
		There are no condition about the continuity, or smoothness, of $\frac{\alpha}{\beta}$.
	\end{rem}
	In general, given two differential forms, there isn't a quotient between them: for example we can consider $dx^1$ and $dx^2\wedge dx^3$ in $\numberset{R}^3$. Moreover if there is a quotient between $\alpha$ and $\beta$ it may not be unique: for example if we consider $\alpha = dx^1\wedge dx^2$ and $\beta = dx^1+dx^2$ in $\numberset{R}^3$ then $\frac{1}{2}(dx^2 - dx^1)$ and $-dx^1$ are both quotients.
	\\However there are some useful formulas concerning the quotients.
	\begin{prop}\label{quot}
		Let us consider a differentiable manifold $M$ and let $\alpha, \beta, \gamma, \delta \in \Omega^*(M)$.
		Then, if $\frac{\alpha}{\gamma}$, $\frac{\alpha}{\delta}$, $\frac{\beta}{\gamma}$, and $\frac{\beta}{\delta}$ are well-defined, then the following formulas hold\footnote{in these formulas the $=$ have to be read as \emph{"exists and one of the possible quotient is given by"}}
		\begin{enumerate}
			\item $\frac{\alpha + \beta}{\gamma} = \frac{\alpha}{\gamma} + \frac{\beta}{\gamma}$,
			\item if $\gamma$ as a $g$-form, $\beta$ a $b$-form and $\delta$ a $d$-form. Then $
			\frac{\alpha \wedge \beta}{\gamma \wedge \delta} = (-1)^{d(a-g)}\frac{\alpha}{\gamma}\wedge \frac{\beta}{\delta}.$
			\item If $\gamma$ is closed, and $\frac{\alpha}{\gamma}$ is a smooth form, then $\frac{d \alpha}{\gamma} = d(\frac{\alpha}{\gamma}).$
			\item if $\alpha$ is an $a$-form and $\beta$ is a $b$-form then $\frac{\alpha \wedge \beta}{\beta \wedge \gamma} = (-1)^{ab}\frac{\alpha}{\gamma}$.
			\item if $\frac{\frac{\alpha}{\delta}}{\gamma}$ exists then $ \frac{\alpha}{\delta \wedge \gamma } = \frac{\frac{\alpha}{\delta}}{\gamma}.$
			\item if $\frac{\alpha}{\delta \wedge \gamma}$ exists, then $\frac{\frac{\alpha}{\delta}}{\gamma} = \frac{\alpha}{\delta \wedge \gamma}.$
		\end{enumerate}
	\end{prop}
	\begin{proof}
		\begin{enumerate}
			\item It follows since 
			\begin{equation}
				\gamma \wedge (\frac{\alpha}{\gamma} + \frac{\beta}{\gamma}) = \gamma \wedge\frac{\alpha}{\gamma} + \gamma \wedge \frac{\beta}{\gamma} = \alpha + \beta
			\end{equation}
			\item We have that the degree of $\frac{\beta}{\delta}$ is $(b-d)$.\\
			Then we have that
			\begin{equation}
				\gamma \wedge \delta \wedge (-1)^{d(a-g)}\frac{\alpha}{\gamma}\wedge \frac{\beta}{\delta}  = 
				\gamma \wedge \frac{\alpha}{\gamma} \wedge \delta \wedge \frac{\beta}{\delta} = \alpha \wedge \beta.
			\end{equation}
			\item Since $\alpha = \gamma \wedge \frac{\alpha}{\gamma}$, then
			\begin{equation}
				d\alpha = \gamma \wedge d(\frac{\alpha}{\gamma}) + d\gamma \wedge  \frac{\alpha}{\gamma} = d(\frac{\alpha}{\gamma}) \wedge \gamma.
			\end{equation}
			\item Observe that 
			\begin{equation}
				\beta \wedge \gamma \wedge (-1)^{ab} \frac{\alpha}{\gamma}  = ((-1)^{ab} \beta \wedge \alpha = \alpha \wedge \beta.
			\end{equation}
			\item We have that
			\begin{equation}
				\delta \wedge \gamma \wedge \frac{\frac{\alpha}{\delta}}{\gamma} = \delta \wedge \frac{\alpha}{\delta} = \alpha.
			\end{equation}
			\item In oreder to prove this, observe that
			$\frac{\alpha}{\delta} = \gamma \wedge \frac{\alpha}{\delta \wedge \gamma}$, indeed
			\begin{equation}
				\delta \wedge (\gamma \wedge \frac{\alpha}{\delta \wedge \gamma}) = (\delta \wedge \gamma) \wedge \frac{\alpha}{\delta \wedge \gamma}  = \alpha,
			\end{equation}
			and then we have that
			$\frac{\frac{\alpha}{\delta}}{\gamma} = \frac{\gamma \wedge \frac{\alpha}{\delta \wedge \gamma}}{\gamma}$ is a form $F$ such that $\gamma \wedge \frac{\alpha}{\delta \wedge \gamma} = \gamma \wedge F$ and so, by the definition of quotient we have that
			\begin{equation}
				\frac{\frac{\alpha}{\delta}}{\gamma} = F = \frac{\alpha}{\delta \wedge \gamma}.
			\end{equation}
		\end{enumerate}
	\end{proof}
	Consider two oriented differentiable manifolds $X$ and $Y$.
	As proved in \cite{Dieu}, in particular in Proposition 16.21.7, when we have a submersion $f:X \longrightarrow Y$, then if $F_q$ is the fiber of $f$ in $q$ and $i_q: F_q \longrightarrow X$ is the immersion of the fiber, then $i_q^*(\frac{\beta}{f^*\alpha})$ doesn't depend on the choice of the quotient. Moreover in \cite{Dieu} it is proved that if $\beta$ is a smooth form in $X$, then $i_q^*(\frac{\beta}{f^*\alpha})$ is a smooth form on the fiber $F_q$.
	It means that for all $p$ in $Y$ we obtain the orientation of $F_p$ defined imposing
	\begin{equation}
		\int_{F_p} i_p^*\frac{Vol_X}{f^*Vol_Y} > 0.
	\end{equation}
	We can also observe that if $f:(X,g) \longrightarrow (Y,h)$ is a submersion between Riemannian manifolds, then it's possible to define a \textit{locally smooth} quotient between $Vol_X$ and $f^*Vol_Y$. Indeed it is sufficient to consider local fibered coordinates ${y^j, x^i}$ on $X$ and $\{y^j\}$ in $Y$ and we obtain that locally
	\begin{equation}
		\frac{Vol_X}{f^*Vol_Y} = \frac{det(g_{ij}(y,x))}{det(h_{rs})(y)}dx^1 \wedge ... \wedge dx^n,
	\end{equation}
	where $g_{ij}$ and $h_{rs}$ are the matrix related to the metrics $g$ and $h$. As consequence of this fact we have that the Projection Formula holds also for quotients of volume forms, i.e. given a differential form $\alpha$ in $Y$ we have that
	\begin{equation}
		\int_F f^*\alpha \wedge \frac{Vol_X}{f^*Vol_Y} = \alpha \wedge \int_F \frac{Vol_X}{f^*Vol_Y}.
	\end{equation}
	In order to prove this we have to consider a cover of coordinate charts and to decompose the integral using a partition of unity. Remember that, since \cite{Dieu}, the integral over the fiber in $p$ of a quotient between a form $\alpha$ and a form $f^*\beta$ doesn't depend on the choice of the quotient. This means that in each chart one can choose a smooth quotient and apply the usual Projection Formula (see Appendix \ref{B}).
	\begin{prop}
		Let $p$ be a submersion $p: X \longrightarrow Y$. Consider $\alpha \in \Omega_{vc}^*(X)$ and $\beta \in \Omega^*(X)$. Let us suppose that there is a well-defined $\frac{\int_F \alpha}{\beta}$ which is locally smooth. Then we have that\footnote{again $=$ have to be read as \emph{"exists and one of the possible quotient is given by": this property holds for all possible choice of the quotient.}, unless $\alpha$ and $\beta$ are top-degree forms: in that case the $=$ actually means the equality as differential forms.}
		\begin{equation}
			\frac{\int_F \alpha}{\beta} = \int_F \frac{\alpha}{p^*\beta}.
		\end{equation}
	\end{prop}
	\begin{proof}
		We have that
		\begin{equation}
			\beta \wedge	(\int_F \frac{\alpha}{p^*\beta}) = \int_F p^*\beta \wedge \frac{\alpha}{p^*\beta}  = \int_F \alpha.
		\end{equation}
	\end{proof}
	\subsection{Fiber Volume of a submersion}
	In this section we study the Fiber Volumes of lipschitz submersions between orientable manifolds. We use the notions of \textit{oriented fiber bundle}, \textit{integration along the fibers} and the \textit{Projection Formula}: we introduce them in Appendix \ref{B}.
	\begin{prop}\label{cosa}
		Let $(M,g)$ and $(N,h)$ two orientable, Riemannian manifolds. Let $p: (M,g) \longrightarrow (N,h)$ be a proper lipschitz submersion. Then we have that
		\begin{equation}
			Vol_{p,g,h}(q) = \int_F \frac{Vol_M}{p^*Vol_N}(q)
		\end{equation}
		if $q$ is in $p(M)$, $0$ otherwise.
	\end{prop}
	\begin{proof}
		Let $A$ a be measurable set of $N$. Then we have that
		\begin{equation}
			\begin{split}
				f_\star \mu_M (A) &= \int_{p^{-1}(A)} 1 d\mu_{M}\\
				&=\int_{p^{-1}(A)} Vol_{M} \\
				&= \int_{p^{-1}(A)} p^*(Vol_N)\wedge \frac{Vol_M}{p^*Vol_N} \\
				&= \int_{p^{-1}(A \cap p(M))} p^*(Vol_N)\wedge \frac{Vol_M}{p^*Vol_N} + \int_{A \cap [p(M)]^c} 0 d\mu_N \\\\
				&= \int_{A \cap p(M)} Vol_N (\int_F \frac{Vol_M}{p^*Vol_N})  + \int_{A \cap [p(M)]^c} 0 d\mu_N \\
				&= \int_{A \cap p(M)} (\int_F \frac{Vol_M}{p^*Vol_N}) Vol_N + \int_{A \cap [p(M)]^c} 0 d\mu_N\\
				&= \int_{A \cap p(M)} (\int_F \frac{Vol_M}{p^*Vol_N}) d\mu_N + \int_{A \cap [p(M)]^c} 0 d\mu_N
			\end{split}
		\end{equation}
	\end{proof}
	\begin{rem}
		The Fiber Volume doesn't depend on the choice of $\frac{Vol_M}{p^*Vol_N}$. It is coherent with the uniqueness of the Radon-Nicodym derivative.
	\end{rem}
	\begin{rem}\label{oss}
		If the submersion $f: X \rightarrow Y$ is in particular a diffeomorphism between oriented Riemannian manifold, then we have that the integration along the fibers of $f$ is the pullback operator $(f^{-1})^*$ with a $+$ if $f$ preserves the orientation and a $-$ otherwise. Indeed  the orientation of the fibers is given posing
		\begin{equation}
			\int_{f^{-1}(p)} \frac{Vol_X}{f^*Vol_X} > 0
		\end{equation}
		and we know that the integration over a $0$-chain is just the evaluation on the oriented points. So
		\begin{equation}
			\begin{split}
				f_\star(\alpha)_p(v_{1,p},..., v_{k,p}) &= \pm \alpha_{f^{-1}(p)}(df_{f^{-1}(p)}^{-1}(v_{1,p}),... , df_{f^{-1}(p)}^{-1}(v_{k,p}))\\
				&= \pm (f^{-1})^*\alpha(v_{1,p},..., v_{k,p}),
			\end{split}
		\end{equation}
		where we have a $+$ if $f$ preserves the orientation and $-$ otherwise. Then, this means that
		\begin{equation}
			f_\star = \pm(f^{-1})^*.
		\end{equation}
		and so, in particular, the Fiber Volume of $f$ is given by
		\begin{equation}
			\pm(f^{-1})^*\frac{Vol_X}{f^*(Vol_Y)} = |(f^{-1})^*\frac{Vol_X}{f^*(Vol_Y)}|.
		\end{equation}
	\end{rem}
	\begin{prop}\label{mai2}
		Given a fiber bundle $p: (M,g) \longrightarrow (N,h)$ where $M$ and $N$ are orientable, we have that if $p^*$ induces a map between the $\mathcal{L}^2$-space, then the same is true for $p_\star$, the operator of integration along the fibers of $p$. Moreover, if we denote by $(p_\star)^*$ the adjoint of $p_\star$, we have that, if $\tau_X$ and $\tau_Y$ are the chiral operators of $X$ and $Y$
		\begin{equation}
			(p_\star)^* = \tau_X \circ p^* \circ \tau_Y.
		\end{equation}
	\end{prop}
	\begin{proof}
		Let $\alpha$ be in $\Omega_{vc}^*(X)$ and let $\beta$ be in $\Omega^*(Y)$, if $n = dim(Y)$ we have that
		\begin{equation}	
			\langle p_\star \alpha, \beta\rangle _Y = \int\limits_{Y} [\int\limits_{F}\alpha] \wedge \star_Y \overline{\beta}
		\end{equation}
		Now, applying the Projection Formula, we have that	
		\begin{equation}
			\begin{split}
				\langle p_\star \alpha, \beta\rangle _Y &= \int\limits_{X} \alpha \wedge p^*(\overline{\star_Y \beta}) \\
				&= i^{-|\beta|(n - |\beta|)} \int\limits_{X} \alpha \wedge (\overline{ p^* \tau_Y \beta}) \\
				&= i^{-|\beta|(n - |\beta|)}i^{|\beta|(n - |\beta|)}
				\int\limits_{X} \alpha \wedge \star_X (\overline{\tau_M  p^* \tau_Y \beta}) \\
				&= \langle \alpha, \tau_X \circ p^* \circ \tau_Y (\beta)\rangle_X.
			\end{split}
		\end{equation}
	\end{proof}
	\begin{rem}\label{fibnorm}
		The norm  of $p_\star$ as operator between $\mathcal{L}^2$-spacesis the same of the norm of $p^*$.
	\end{rem}
	\begin{cor}\label{integr}
		Consider $p:(M,g) \longrightarrow (N,h)$ a R.-N.-lipschitz submersion. Then the operator $p_\star$ is a $\mathcal{L}^2$-bounded operator.
	\end{cor}
	We conclude this section giving a formula which allows to compute the Fiber Volume of the composition of two submersions.
	\begin{prop}\label{compo}
		Let $f:(M,g) \longrightarrow (N,h)$ and $g:(N,h) \longrightarrow (W, l)$ be two submersions between oriented Riemannian manifolds. Then we have that
		\begin{equation}
			Vol_{g \circ f, \mu_M, \mu_W}(q) = \int_{g^{-1}(q)}(\int_{f^{-1}g^{-1}(q)} \frac{Vol_M}{f^*(Vol_N)})\frac{Vol_N}{g^*Vol_N}
		\end{equation}
	\end{prop}
	\begin{proof}
		We can observe that, as quotients, we have that
		\begin{equation}
			\frac{Vol_M}{(g \circ f)^*Vol_W} = \frac{Vol_M}{f^*Vol_N} \wedge \frac{f^*Vol_N}{(g \circ f)^*Vol_W}
		\end{equation}
		and, in particular one can observe that we can choose as quotient
		\begin{equation}
			\frac{f^*Vol_N}{(g \circ f)^*Vol_W} = f^*(\frac{Vol_N}{g^*Vol_W}),
		\end{equation}
		indeed
		\begin{equation}
			(g \circ f)^*Vol_W \wedge f^*(\frac{Vol_N}{g^*Vol_W}) = f^*( g^*Vol_W \wedge  \frac{Vol_N}{g^*Vol_W}) = f^*Vol_N.
		\end{equation}
		Then we can conclude by applying the Projection Formula. Indeed
		\begin{equation}
			\begin{split}
				Vol_{g \circ f, \mu_M, \mu_W}(q) &= \int_{(g \circ f)^{-1}(q)} \frac{Vol_M}{(g \circ f)^*Vol_W}\\
				&= \int_{(g \circ f)^{-1}(q)} \frac{Vol_M}{f^*Vol_N} \wedge f^*(\frac{Vol_N}{g^*Vol_W})\\
				&= \int_{g^{-1}(q)}(\int_{f^{-1}g^{-1}(q)} \frac{Vol_M}{f^*(Vol_N)})\frac{Vol_N}{g^*Vol_N}.
			\end{split}
		\end{equation}
	\end{proof}
	\section{Submersion related to lipschitz maps}
	\subsection{The Sasaki metric}\label{sasaki}
	Let us consider a Riemannian manifold $(N,h)$ of dimension $n$, $\pi_E : E \longrightarrow N$ a vector bundle of rank $m$ endowed with a bundle metric $H_E \in \Gamma(E^*\otimes E^*)$ and a linear connection $\nabla_E$ which preserves $H_E$. Fix $\{s_j\}$ a local frame of $E$: we have that if $\{x^i\}$ is a system of local coordinate over $U \subseteq N$, then we can define the system of coordinates $\{x^i, \mu^j\}$ on $\pi_E^{-1}(U)$, where the $\mu^j$ are the components respect to $\{s_j\}$.
	\\Then we can denote by $K$ the map $K: TE \longrightarrow E$ defined as
	\begin{equation}
		K(b^i\frac{\partial}{\partial x^i }|_{(x_0, \mu_0)} + z^j \frac{\partial}{\partial \mu^j}|_{(x_0, \mu_0)}) := (z^l + b^iz^j \Gamma_{ij}^l(x_0))s_l(x_0),
	\end{equation}
	where the $\Gamma_{ij}^l$ are the Christoffel symbols of $\nabla_E$.
	\begin{defn}
		The \textbf{Sasaki metric} on $E$ is the Riemannian metric $h^E$ defined for all $A,B$ in $T_{(p, v_p)}E$ as
		\begin{equation}
			h^E(A,B) := h(d\pi_{E,v_p}(A), d\pi_{E,v_p}(B)) + H_E(K(A), K(B)).
		\end{equation}
	\end{defn}
	\begin{rem}\label{riemsub}
		Let us consider the system of coordinate $\{x^i\}$ on $N$ and $\{x^i, \mu^j\}$ on $E$. We have that the components of $h^E$ are given by
		\begin{equation}\label{metri}
			\begin{cases}
				h^E_{ij}(x,\mu) = h_{ij}(x) + H_{\alpha\gamma}(x)\Gamma^\alpha_{\beta i}(x)\Gamma^\gamma_{\eta j}(x)\mu^\beta\mu^\eta \\
				h^E_{i\sigma}(x, \mu) = H_{\sigma \alpha}(x)\Gamma^\alpha_{\beta i}(x)\mu^\beta\\
				h^E_{\sigma\tau}(x,\mu) = H_{\sigma, \tau}(x), 
			\end{cases}
		\end{equation}
		where $i,j = 1,..., n$ and $\sigma, \tau = n+1,...,n+m$. Consider a point $x_0 = (x^1_0, ..., x^n_0)$ in $M$. We have that all the Christoffel symbols are in $x_0$ are zero, then, in local coordinates the matrix of $h^E$ in a point $(x_0, \mu)$ is given by
		\begin{equation}
			\begin{bmatrix}
				h_{i,j}(x) && 0 \\
				0 && H_{\sigma, \tau}(x)
			\end{bmatrix}.
		\end{equation}
		Consider a fiber bundle $E$ over a Riemannian manifold $(M,g)$. Fix a bundle metric $H$ and a connection $\nabla$ on $E$. Then the projection $p: (E,g_S) \longrightarrow (M,g)$, where $g_S$ is a Sasaki metric related to $g$, $H$ and $\nabla$ is a Riemannian submersion.
	\end{rem}
	\begin{exem}
		Let us consider a Riemannian manifold $(N,h)$. We can consider as $E$ the tangent bundle $TM$ and as $h_E$ the metric $g$ itself. In this case the connection $\nabla_E$ is the Levi-Civita connection $\nabla_g^{LC}$.
	\end{exem}
	\subsection{Pull-back bundle and pull-back connection}
	Let us consider a smooth map $f:(M,g) \longrightarrow (N,h)$ between Riemannian manifolds.
	\begin{defn}
		The \textbf{pullback bundle} $f^*E$ is a bundle over $(M,g)$ given by
		\begin{equation}
			f^*(E) := \{(p,v) \in M \times E| \pi_E(t) = f(p)\}.
		\end{equation}
	\end{defn}
	\begin{rem}
		Given a smooth map $f:(M,g) \longrightarrow (N,h)$, then there is a bundle map $F:f^*E \longrightarrow E$ induced by $f$. This map is defined as
		\begin{equation}
			F(p,v) = (f(p),v).
		\end{equation}
		Using this map it's possible to define a metric bundle on $f^*(E)$ as follows
		\begin{equation}
			F^*(H)(A_p,B_p) := H(F(A_p), F(B_p))
		\end{equation}
	\end{rem}
	Given a section $\sigma$ of $E$, one can define the section
	\begin{equation}
		f^*\sigma := f^*\sigma(p) = \sigma \circ f(p).
	\end{equation}
	A consequence of the existence of the pullback of a section of $E$ is that it is possible to \emph{pullback} also a connection $\nabla_E$.
	\begin{defn}
		The \textbf{pullback connection} $f^*\nabla_E$ on $f^*(E)$ is uniquely defined imposing that
		\begin{equation}\label{pulb}
			(f^*\nabla_E)f^*\sigma = f^*(\nabla_E\sigma).
		\end{equation} 
	\end{defn}
	\begin{rem}
		The condition (\ref{pulb}) is sufficient to uniquely define a connection on $f^*(E)$. An equivalent condition is to impose that the local Christoffel symbols $\tilde{\Gamma}^\alpha_{\beta, i}$ are given by
		\begin{equation}\label{chri}
			\tilde{\Gamma}^\alpha_{\beta, i} := \frac{\partial f^l}{\partial x^i}f^*(\Gamma^\alpha_{\beta, l}).
		\end{equation}
	\end{rem}
	\begin{rem}
		Let us consider a map $f:(M,g) \longrightarrow (N,h)$ and consider, in particular, the bundle $\pi: f^*(TN) \longrightarrow M$.
		Fix on $f^*(TN)$ the bundle metric $f^*h$ and $f^*\nabla^{LC}_h$. Consider the Sasaki metric on $f^*(TN)$ induced by $g$, $f^*h$ and $f^*\nabla^{LC}_h$. Fix some normal coordinates $x^i$ around a point $p$ on $M$ and fix a frame $E_j$ of $f^*(TN)$ around $f(p)$. Then we can consider the fibered coordinates $\{x^i, y^j\}$ of $f^*(TN)$. Let us observe that in $(p, w_{f(p)})$ we have the orthogonal decomposition
		\begin{equation}\label{decom}
			\begin{split}
				T_{(p, w_{f(p)})}f^*(TN) &= Span\{\frac{\partial}{\partial x^i}|_{(p, w_{f(p)})}\} \oplus ker(d\pi_{(p, w_{f(p)})})\\
				&= Span\{\frac{\partial}{\partial x^i}|_{(p, w_{f(p)})}\} \oplus Span\{\frac{\partial}{\partial y^j}|_{(p, w_f(p))}\}
			\end{split}
		\end{equation}
		and, moreover, we have that $Span\{\frac{\partial}{\partial x^i}|_{(p, w_{f(p)})}\}$ is isometric to $T_pM$, while \\$Span\{\frac{\partial}{\partial y^j}|_{(p, w_f(p))}\}$ is isometric to $T_{f(p)}N$. This fact follows by (\ref{metri}) and (\ref{chri}) and observing that the Christoffel symbols of $\nabla^{LC}_N$ respect to normal coordinates around $f(p)$ are null in $f(p)$.
		\\This fact, in particular, implies that the projection $\pi: f^*TN \longrightarrow M$ is a Riemannian submersion. 
	\end{rem}
	\begin{rem}\label{Riccib}
		Consider $(M,g)$ and $(N,h)$ two manifolds of bounded geometry and let $f:(M,g) \longrightarrow (N,h)$ be a $C^{k}_{bu}$-map for each $k$ in $\numberset{N}$. Fix on $f^*(TN)$ the Sasaki metric $g_S$ induced by $g$, $f^*h$ and $f^*\nabla^{LC}_g$. Then, if we denote by $\nabla: = \nabla^{LC}_{g_S}$ and by $R$ the Riemann tensor on $f^*(TN)$, we have that for each $i$ in $\numberset{N}$ there is a continuous function $C_i: \numberset{R} \longrightarrow \numberset{R}$ such that
		\begin{equation}
			|\nabla^i R(v_p)| \leq C_i(||v_p||).
		\end{equation}
	\end{rem}
	\subsection{The map $p_f$}
	In this subsection we will define a submersion $p_f: f^*(T^\delta N) \longrightarrow N$ where 
	\begin{equation}
		f^*(T^\delta N) :=\{(p, w_{f(p)}) \in f^*(TN)| |w_{f(p)}| \leq \delta\}.
	\end{equation} 
	\begin{lem}\label{tilde}
		Let us consider $f:(M,g) \longrightarrow (N,h)$ a smooth lipschitz map between two oriented Riemannian manifolds. Suppose that $(N,h)$ has bounded geometry.
		\\Let us denote by $F: f^*TN \longrightarrow TN$ the bundle morphism induced by $f$ between $f^*(TN)$ and $TN$. Fix on $f^*TN$ the Sasaki metric $g_s$ defined using $f^*\nabla^LC_h$, $f^*h$ and $g$. Then there is a map $p_f: (f^*(T^\delta N), g_s) \longrightarrow (N,h)$ such that:
		\begin{enumerate}
			\item $p_f$ is a submersion,
			\item ${p}_f(x,0) = f(x)$,
			\item ${p}_f$ is $\Gamma$-equivariant,
			\item ${p}_f = {p}_{id_N} \circ F$,
			\item Fix some local normal coordinates $\{x^i\}$ around a point $p$ in $M$ and some local normal coordinates $\{y^j\}$ around $f(p)$ on $N$. Let us suppose that for each $k$ in $\numberset{N}$
			\begin{equation}
				\sup\limits_{s = 0,...,k}|\frac{\partial^s y^j \circ f}{\partial x^i_1 ... \partial x^{i_s}}(x)| \leq L_k
			\end{equation}
			for some $L_k$. Consider the frame $\{\frac{\partial}{\partial y^j}\}$ around $f(0)$ and define the fibered coordinates $\{x^i, \mu^j\}$ related to $\{\frac{\partial}{\partial y^j}\}$ on $f^*T^\delta N$. Then for each $k$ in $\numberset{N}$ there is a constant $C_k$ such that 
			\begin{equation}
				\sup\limits_{s + t = 0,...,k}|\frac{\partial^s y^j \circ p_f}{\partial x^i_1 ... \partial x^{i_s}\partial \mu^{j_1} ... \partial \mu^{j_t}}(x,\mu)| \leq C_k
			\end{equation}
			where $C_k$ doesn't depend on $x$ or $j$.
		\end{enumerate}
	\end{lem}
	\begin{proof}
		We can define
		\begin{equation}
			\begin{split}
				p_f: (f^*(T^\delta N), g_s) &\longrightarrow (N,h) \\
				(p, w_{f(p)}) &\longrightarrow exp_{f(p)}(w_{f(p)}).
			\end{split}
		\end{equation}
		Then we have that
		\begin{enumerate}
			\item $p_f$ is a submersion, indeed for each fixed $p$ in $M$ we have that ${p}_f(p, \cdot): f^*(T^\delta N)_p = T^\delta_{f(p)} N \longrightarrow N$ is the exponential map in $f(p)$. We know that the exponential map is a local diffeomorphism and so $\tilde{p}_f$ is a submersion,
			\item $p_f(p, 0_{f(p)}) = f(p)$: this follows immediately by the definition of exponential map,
			\item $p_f$ is $\Gamma$-equivariant, indeed, since $\Gamma$ acts by isometries,
			\begin{equation}
				\begin{split}
					p_f(\gamma p, d\gamma w_{f(p)}) &= exp_{f(\gamma p)}d\gamma w_{f(p)}\\
					&= exp_{\gamma f(p)}d\gamma w_{f(p)}\\
					&= \gamma exp_{f(p)}w_{f(p)} = \gamma {p}_f(p, w_{f(p)}).
				\end{split}
			\end{equation}
			\item It's obvious: we have that $F(p, w_{f(p)}) = w_{f(p)}$,
			\item Because of the last point it is sufficient to show the assertion just for $p_{id}$.
			\\Consider $V \subseteq N$ a normal coordinate chart. Let $\{x^i\}$ be the coordinates system on $V$. Consider $\pi:TN \longrightarrow N$ and let $\{\frac{\partial}{\partial x^i}\}$ be an orthonormal frame on $V$. Then we have the fibered coordinate system $\{x^i, \mu^j\}$ on $\pi^{-1}(V)$.  Let us study now $p_{id}$ restricted to $\pi^{-1}(V)$. We have that it can be seen as $\pi \circ \phi (x,\mu)$ where $\pi(x,\mu) = x$ and $\phi$ is the flow of the system of differential equations given by
			\begin{equation}
				\begin{cases}
					\dot{x}^k = c^k\\
					\dot{c}^k = - \Gamma^k_{ij}(x) x^ix^j
				\end{cases}
			\end{equation} 
			Then, applying the Lemma 3.4 of \cite{flow} we have that the partial derivatives of $\phi$ are uniformly bounded. Then we conclude that the derivatives of  $p_{id}$ are uniformly bounded. 
		\end{enumerate}
	\end{proof}
	\begin{cor}
		We have that $\tilde{p}_f$ is a lipschitz map.
	\end{cor}
	\begin{proof}
		It's a direct computation. Observe that $|| \frac{\partial}{\partial y^j}||$, $||\frac{\partial}{\partial x^i}||$ and $|| \frac{\partial}{\partial \mu^l}||$ are uniformly bounded. Moreover we also have that
		\begin{equation}
			dp_{f}(\frac{\partial}{\partial x^i}) = \frac{\partial p_f^j}{\partial x^i} \frac{\partial}{\partial y^j}
		\end{equation}
		and 
		\begin{equation}
			dp_{f}(\frac{\partial}{\partial \mu^l}) = \frac{\partial p_f^j}{\partial \mu^l} \frac{\partial}{\partial y^j}.
		\end{equation}
	\end{proof}	
	\section{The pull-back functor}\label{Pullback}
	\subsection{The Fiber Volume of $p_f$}
	\begin{lem}\label{svolta}
		Consider $f:(M,g) \longrightarrow (N,h)$ a smooth lipschitz map between Riemannian manifolds of bounded geometry. Then the map ${t}_f: f^*(T^\delta N) \longrightarrow M \times N$ defined as
		\begin{equation}
			{t}_f(p,w_{f(p)}) = (p, p_f(p, w_{f(p)}))
		\end{equation}
		is an R.-N.-lipschitz diffeomorphism with its image.
	\end{lem}
	\begin{proof}
		We start proving that ${t}_f$ is a diffeomorphism with its image. Observe that 
		\begin{equation}
			dim(f^*(T^\delta N)) = m+ n = dim(M) + dim(N) = dim(M \times N).
		\end{equation}
		Fix some normal coordinates $\{x^i\}$ around a point $p$ in $M$ and let $\{y^j\}$ be some normal coordinates around $f(p)$ in $N$. Consider the frame $\{\frac{\partial}{\partial y^j}\}$ and define the fibered coordinates $\{x^i, \mu^j\}$ related to $\{\frac{\partial}{\partial y^j}\}$ on $f^*(T^\delta N)$. On $M\times N$ we consider the normal coordinates $\{x^i, y^j\}$. Then the Jacobian of $\tilde{t}$ is given by
		\begin{equation}
			Jt_f(x,\mu) = \begin{bmatrix}
				1 && \star \\
				0 && Jexp_{x}(\mu)
			\end{bmatrix}
		\end{equation}
		Then, since the exponential map is a diffeomorphism for each $x_0$, we have that $Jt_f$ is invertible. Moreover ${t}_f$ is also injective, indeed if $(p, w_{f(p)})$ and $(q, v_{f(q)})$ have the same image, then $p = q$ and 
		\begin{equation}
			exp_{f(p)} w_{f(p)} = exp_{f(p)}v_p \implies w_p = v_p,
		\end{equation}
		since  their norm is less than $\delta$ and $\delta \leq inj_{N}$. We proved that ${t}_f$ is a diffeomorphism with its image.
		\\
		\\We know that ${t}_f$ is a lipschitz map because $p_f$ is a lipschitz map.
		\\So, in order to prove that ${t}_f$ is a R.-N.-lipschitz map, we have to show that it has bounded Fiber Volume. 
		\\Consider a point $(p,q)$ in $M \times N$. Then its fiber is empty or it is a singleton $\{(p, w_{f(p)})\}$. This means, following the Remark \ref{oss} that the Fiber Volume of ${t}_f$ is given by
		\begin{equation}
			|{t_f}^{-1^*}\frac{Vol_{T^\delta N}}{{t_f}^*Vol_{M \times N}}|
		\end{equation}
		on the image of $t_f$ and it null otherwise.
		\\Then if $\frac{Vol_{T^\delta N}}{{t}^*Vol_{M \times N}}$ is a bounded function, which is a-priori not clear, then we can conclude that ${t}_f$ is a R.-N.-lipschitz map. 
		\\Consider the fibered coordinates $\{x^i, \mu^j\}$ on $f^*(TN)$ and the coordinates $\{x^i, y^j\}$ on $M \times N$. Since the definition of exponential map we have that the image of ${t}_f$ is contained in a $\delta$-neighborhood of the $Graph(f) \in M \times N$. This means that we can cover all the image of ${t}_f$ using the normal coordinates $\{x^i, y^j\}$ around $(p,f(p))$. Out of the image of ${t}_f$ we already know since Proposition \ref{cosa} that the Fiber Volume is null. 
		\\Observe that, in this coordinates, we have
		\begin{equation}
			{t}_f(0, \mu^j) = (0, \mu^j).
		\end{equation}
		Consider 
		\begin{equation}
			Vol_{f^*T^\delta N}(x, \mu) = \sqrt{det(G_{ij})}(x,\mu) dx^1 \wedge... \wedge d\mu^n
		\end{equation}
		and 
		\begin{equation}
			Vol_{M\times N}(x,y) = \sqrt{det(H_{ij})}(x,y)dx^1 \wedge... \wedge dy^n,
		\end{equation}
		where $G_{ij}$ is the matrix of $g_S$ on $f^*(T^\delta N)$ with respect to $\{x^i, \mu^j\}$ and $H_{ij}$ is the matrix of the metric on $M \times N$. Then we have that
		\begin{equation}
			\frac{Vol_{T^\delta N}}{t^*Vol_{M \times N}}(x,\mu) = \frac{\sqrt{det(G_{ij})}}{t_f^*(\sqrt{det(H_{ij})})}(x, \mu) \cdot \frac{1}{det(Jexp_{f(x)}(\mu))}.
		\end{equation}
		Observe that in $(0, \mu)$ we have that $Jexp_{0}(\mu)$ is the identity matrix. Moreover we also have that
		\begin{equation}
			G_{ij}(0,\mu) = \begin{bmatrix}
				1 && 0 \\
				0 && 1
			\end{bmatrix}
		\end{equation}
		and so $\sqrt{det(G_{ij})}(0,y) = 1$. Finally we have that
		\begin{equation}
			H_{ij}(0,\mu) = \begin{bmatrix}
				1 && 0 \\
				0 && h_{ij}(0,\mu)
			\end{bmatrix}
		\end{equation}
		where $h_{ij}$ is the matrix related to the Riemannian metric $h$ in normal coordinates. Then we have that 
		\begin{equation}
			det(H_{ij})^{-1}(0,y) = det(h_{ij})^{-1}(0,y) \leq C
		\end{equation}
		because $N$ is a manifold of bounded geometry \cite{flow}. This means that
		\begin{equation}
			\frac{Vol_{T^\delta N}}{t^*Vol_{M \times N}}(0,\mu) = \sqrt{\frac{det(G_{ij})}{t_f^*det(H_{ij})}}(0,\mu) \leq C
		\end{equation}
		and so the Fiber Volume of $t_f$ is bounded.
	\end{proof}
	\begin{cor}\label{R.-N.-p}
		Let $f:(M,g) \longrightarrow (N,h)$ be a smooth uniformly proper lipschitz map between Riemannian manifolds of bounded geometry. Then $p_f$ is a R.-N.-lipschitz map and $p_f^*$ is $\mathcal{L}^2$-bounded.
	\end{cor}
	\begin{proof}
		We already know, since Lemma \ref{svolta}, that $t_f$ is a R.-N.-lipschitz map. Observe that $p_f = pr_N \circ t_f$.
		\\Consider $\overline{f}: f^*T^\delta N \longrightarrow N$ defined as
		\begin{equation}
			\overline{f}(w_{f(p)}) := f(p).
		\end{equation}
		We have that $p_f \sim_\Gamma \overline{f}$. Then, using Remark \ref{uniformlyp}, in particular (\ref{eqno}), we have that there is a $C > 0$ such that
		\begin{equation}
			p_f^{-1}(q) \subset A_q := \overline{f}^{-1}(B_C(q)) = \pi^{-1}f^{-1}(B_C(q)),
		\end{equation}
		where $\pi: f^*T^\delta N \longrightarrow M$ is the projection of the bundle.  
		\\This means that if we fix a $q$ in $N$, then the Fiber Volume of $t_f$ in a point $(p,q)$ can be different from zero only if $p \in f^{-1}(B_C(q))$.  
		\\Then, using the Proposition \ref{compo}, we have that the Fiber Volume of $p_f$ in a point $q$ is given by
		\begin{equation}
			\begin{split}
				Vol_{p_f}(q) &= \int_M Vol_{t_f}(p,q) d\mu_M \\
				&= \int_{f^{-1}(B_C(q))}  Vol_{t_f}(p,q) d\mu_M \\
				&\leq K \cdot \mu_M(f^{-1}(B_C(q))).
			\end{split}
		\end{equation}
		where $K$ is the maximum of the Fiber Volume of $t_f$. Let us observe that since $f$ is uniformly proper, then the diameter of $f^{-1}(B_C(q))$ is uniformly bounded and so there are a point $x_0$ in $M$ and a radius $R$ such that
		\begin{equation}
			f^{-1}(B_C(q)) \subseteq B_R(x_0).
		\end{equation}
		Then, since the Ricci curvature of $M$ is bounded, we also have that (applying Remark \ref{bvolume}) there is a constant $V$ such that
		\begin{equation}
			\mu_M(f^{-1}(B_C(q))) \leq \mu_M(B_R(x_0)) \leq V
		\end{equation}
		and so
		\begin{equation}
			Vol_{p_f}(q) \leq K \cdot V
		\end{equation}
		and $p_f$ is a R.-N.-lipschitz map.
	\end{proof}
	\begin{rem}\label{salva}
		Consider a smooth map $f:(M,g) \longrightarrow (N,h)$ which is lipschitz and uniformly proper. Let us suppose, moreover, that $(N,h)$ is a manifold of bounded geometry and $(M,g)$ has bounded Ricci curvature. Eventually suppose $M$ incomplete or $\partial M \neq \emptyset$. Then we can observe that the map $p_f: (f^*(T^\delta N), g_S) \longrightarrow (N,h)$ is well-defined and, moreover, using the same arguments we used in this section, it is also R.-N.-lipschitz.
	\end{rem}
	\subsection{A Thom form for $f^*(TN)$}\label{forma}
	Let $(M,g)$ and $(N,h)$ be two manifolds of bounded geometry, let $f:(M,g) \longrightarrow (N,h)$ be a smooth lipschitz map and consider $(f^*(TN), g_S)$ the pullback bundle $f^*(TN)$ endowed of the Sasaki metric defined in the subsection \ref{sasaki}. Let us introduce the notion of Thom form.
	\begin{defn}
		Let $\pi: E \longrightarrow M$ be a vector bundle. We say that $\omega$ in $\Omega_{cv}^*(E)$ is a \textbf{Thom form} if it is closed and its integral along the fibers of $\pi$ is equal to the constant function $1$.
	\end{defn}
	It's a well known fact that for each oriented vector bundle $(E, \pi, M)$ there is at least a Thom form $\omega$ (see for example \cite{bottu}). This, in particular, implies that the tangent bundle $(TN, \pi, N)$ always admits a Thom form. Then, if we consider a smooth map $f: M \longrightarrow N$, the same holds for $(f^*TN, \pi_1, M)$.
	\\ Given a Thom form $\omega$ of $f^*TN$ such that $supp(\omega)$ is contained in a $\delta_0 < \delta$ neighborhood of the null section, let us define an operator $e_\omega: \Omega^*(f^*(T^\delta N)) \longrightarrow \Omega^*(f^*(T^\delta N))$, for every smooth form $\alpha$ as
	\begin{equation}
		e_\omega(\alpha) := \alpha \wedge \omega.
	\end{equation}
	Our goal, in this subsection, is to find a Thom form $\omega$ such that the operator $e_\omega$ is $L^2$-bounded. To this end we follow the work of Mathai and Quillen \cite{mathai}. In their work, indeed, they compute a Thom form for a vector bundle endowed with a connection and a metric bundle.
	\\
	\\First we construct the Thom form of $f^*(TN)$ when $f = id_N$.
	\\Consider the bundle $\pi: (TN, h_S) \longrightarrow (N,h)$ and let $\pi^*(TN)$ be the pullback bundle over $TN$. Consider the bundle metric on $\pi^*(TN)$ given by $\pi^*h$ and we can also consider the connection $\pi^*\nabla^{LC}_h$. Then we can denote by $\Omega^{i,j}$ the algebra
	\begin{equation}
		\Omega^{i,j} := \Omega^{i}(TN, \Lambda^j\pi^*TN) = \Gamma(TN, \Lambda^i T^*(TN) \otimes \Lambda^j\pi^*TN).
	\end{equation}
	Let us consider the following section of $TN$ 
	\begin{equation}
		\begin{split}
			X: TN &\longrightarrow \pi^*TN \\
			v_p &\longrightarrow (v_p, v_p).
		\end{split}
	\end{equation}
	Observe that, identifying $\pi^*TN$ and $\Lambda^1(\pi^*TN)$ we obtain that $X$ is in $\Omega^{0,1}$.
	\\Let us fix some normal coordinates  $\{x^i\}$ on $N$ and the coordinates $\{x^i, \mu^j\}$ on $TN$, where $\{\mu^j\}$ are referred to the frame $\{\frac{\partial}{\partial x^j}\}$. Then we have that
	\begin{equation}
		X(x, \mu) = \mu^i \frac{\partial}{\partial x^i}.
	\end{equation}
	Let us consider 
	\begin{equation}
		\begin{split}
			\pi^*g(X,X) = |X|^2: TN &\longrightarrow \numberset{R} \\
			v_p &\longrightarrow h_p(v_p, v_p).
		\end{split}
	\end{equation}
	This map can be see as a differential form in $\Omega^{0,0}$. In coordinates it can be expressed as follow
	\begin{equation}
		|X|^2(x, \mu) = h_{i,j}(x)\mu^i\mu^j.
	\end{equation} 
	Consider $\pi^*\nabla^{LC}_h (X)$: this is a form in $\Omega^{1,1}$ and, in local coordinates, it is given by
	\begin{equation}
		\begin{split}
			\pi^*\nabla^{LC}_h (X) &=  d\mu^i \otimes \frac{\partial}{\partial x^i}  + \mu^i \nabla \frac{\partial}{\partial x^i}\\
			&= d\mu^i \otimes \frac{\partial}{\partial x^i} + \mu^i \Gamma^k_{ij}(x) dx^j \otimes \frac{\partial}{\partial x^k} \\
			&= (\delta^k_j + \mu^i \Gamma^k_{ij}(x))dx^j \otimes \frac{\partial}{\partial x^k}
		\end{split}
	\end{equation} 
	Finally let us consider $\Omega$ the curvature form of $TN$: this is a $2$-form on $N$ with values in $\Lambda^2(TN)$. We know that $\Omega$, in coordinates is given by
	\begin{equation}
		\Omega(x) = R^{ij}_{kl}(x) dx^k\wedge dx^l \otimes (\frac{\partial}{\partial x^i} \wedge \frac{\partial}{\partial x^k})
	\end{equation}
	where $R^{ij}_{kl} = g^{is}R^j_{kls}$ and $R^j_{kls}$ are the components of the Riemann tensor of $N$. Pulling back $\Omega$ along $\pi$, we obtain $\pi^*\Omega$ which is a differential form in $\Omega^{2,2}$.
	\\Let $\phi: \numberset{R} \longrightarrow \numberset{R}$ be a smooth function which support is contained in $[\delta_0, \delta_0]$. Then we define
	\begin{equation}
		\overline{\omega} := \sum_{k=0}^{n} \frac{\phi^{(k)}(\frac{|X^2|}{2})}{k!}(\pi^*\nabla X + \pi^*\Omega)^k,
	\end{equation}
	where $(\pi^*\nabla X + \pi^*\Omega)^k$ is the $k$-times wedge of $\pi^*\nabla X + \pi^*\Omega$.
	\\Observe that the support of $\overline{\omega}$ is strictly contained in a $\delta_0$-nighborhood of the zero section since $supp(f) \subseteq [- \delta_0, \delta_0]$.
	\\Let us consider $I,J,K$ some multi-index and let us denote $dx^I := dx^i_1 \wedge ... \wedge dx^i_{|I|}$, $d\mu^J := d\mu^i_1 \wedge ... \wedge d\mu^i_{|J|}$ and  $\frac{\partial}{\partial x^K} := \frac{\partial}{\partial x^{i_1}} \wedge ... \wedge \frac{\partial}{\partial x^{i_{|K|}}}$. Since $N$ is a manifold of bounded geometry, we have that $\overline{\omega}$ in local fibered coordinate $\{U, x^i, \mu^j\}$ is given by
	\begin{equation}
		\overline{\omega}(x, \mu) := \alpha_{IJ}^{K}(x, \mu) dx^I\wedge d\mu^J \otimes \frac{\partial}{\partial x^K}
	\end{equation}
	where $\alpha_{IJ}$
	\begin{equation}
		|\alpha_{IJ}^{K}(x,\mu)| \leq C
	\end{equation}
	and $C$ doesn't depend on the choice of $U$. Indeed the bounded geometry of $N$ implies that the components of the metric $h$, the Christoffel symbols of $\nabla$ and also the components of $h_S$ are uniformly bounded in the fibered coordinates $\{x^i, \mu^j\}$.
	\\
	We have to introduce an important tool: the \textit{Berezin integral} $\textbf{B}$. This is the isometry
	\begin{equation}
		\begin{split}
			\textbf{B}: \Lambda^n(TN) &\longrightarrow N \times \numberset{R}\\
			\alpha_p &\longrightarrow (p, Vol_p(\alpha_p))
		\end{split}
	\end{equation}
	where $n= dim(N)$ and $Vol_p$ is the volume form of $N$ in a point $p$ \footnote{Actually the definition of Berezin integral is much more general: this is the definition of Berezin integral for the fiber bundle $TN$.}. We can observe that the Berezin integral can be extended to $\Omega^{i,j}$
	\begin{equation}
		\mathcal{B}: \Omega^{i,j} \longrightarrow \Omega^i(TN)\\
	\end{equation}
	posing $\mathcal{B}(\alpha \otimes \beta) := \textbf{B}(\beta) \alpha$ if $j = n$, $\textbf{B}(\alpha \otimes \beta) := 0$ otherwise.
	\\As showed in  \cite{Gez} and in \cite{mathai}, we can define our Thom form as the differential form given by
	\begin{equation}\label{Thom}
		\omega := \mathcal{B}(\overline{\omega}).
	\end{equation}
	We can observe that, in coordinates $\{x,\mu\}$ on $TN$,
	\begin{equation}
		\mathcal{B}(\alpha^{I}_{J,L}(x, \mu)\frac{\partial}{\partial x^I} \otimes dx^J\wedge d\mu^L) = \alpha^{0}_{J,L}(x, \mu) dx^J\wedge d\mu^L,
	\end{equation}
	where $I$, $J$, $L$ are multi-index and $0$ is the multi-index $(1,..., n)$. This fact implies that, in local coordinates $(x, \mu)$ we can express
	\begin{equation}\label{Thombound}
		\omega(x, \mu) = \alpha_{IJ}(x, \mu) dx^I\wedge d\mu^J,
	\end{equation}
	where $|\alpha_{IJ}|$ are uniformly bounded. Moreover the support of $\omega$ is contained in a $\delta_0$-neighborhood of the null section. The differential form $\omega$ is a Thom form for $TN$: a proof of this fact can be find in \cite{mathai} or in \cite{Gez}.
	\begin{prop}
		Consider $(N,h)$ a manifold of bounded geometry. Let $\alpha$ be a differential form on $(T^\delta N, h_S)$. If in fibered coordinates $\{x^i, \mu^j\}$, where $\mu^j$ are referred to $\{\frac{\partial}{\partial x^i}\}$ the coefficients of $\alpha$ are uniformly bounded, then the punctual norm $|\alpha_p|$ is uniformly bounded.
	\end{prop}
	\begin{proof}
		It is a direct computation. We know that the norm of $dx^i$ and $d\mu^j$ are uniformly bounded in $f^* T^\delta N$. Moreover, applying the  Hadamard-Schwartz inequality (\cite{Sbordone} or Lemma \ref{spacc} in Appendix A) we have that there is a number $K$ such that $|dx^I\wedge d\mu^J|_p \leq K$. Then, since
		\begin{equation}
			|\alpha_p| \leq \sqrt{|\alpha_{IJ}(x, \mu)|^2 \cdot |\alpha_{RS}(x, \mu)|^2\cdot |dx^I\wedge d\mu^J|_p^2 \cdot |dx^R\wedge d\mu^S|_p^2} \leq C.
		\end{equation}
	\end{proof}
	Let $f:(M,g) \longrightarrow (N,h)$ be a smooth lipschitz map and consider $(f^*(TN), g_S)$ where $g_S$ is the Sasaki metric defined as in the previous section. Observe that if $\omega$ is the Thom form on $TN$ defined in (\ref{Thom}) and if we consider the map $F: f^*(TN) \longrightarrow TN$ given by $F(p,w_{f(p)}) =(f(p), w_{f(p)})$, then it's easy to check that $F^*\omega$ is a Thom form for $f^*(TN)$ and moreover, applying Lemma \ref{norm} in Appendix A, we also have a uniform bound on the norm of $|F^*\omega|_p$.
	\begin{prop}\label{eomega}
		Let us consider a form $\alpha$ on a Riemannian manifold $(M,g)$. Suppose that there is a number $C$ such that
		\begin{equation}
			|\alpha_p| \leq C
		\end{equation}
		for each $p$ in $(M,g)$. Then the operator $e_{\alpha}(\beta) := \beta \wedge \alpha$ defines an $\mathcal{L}^2$-bounded operator.
	\end{prop}
	\begin{proof}
		It's again a direct consequence of Hadamard-Schwartz inequality. Indeed
		\begin{equation}
			\begin{split}
				||e_{\alpha}(\beta)||_{\mathcal{L}^2(M)}^2 &= \int_{M} |\alpha \wedge \beta|_p^2 d\mu_M \\
				&\leq \int_M C \cdot |\alpha|_p^2 \cdot |\beta|_p^2 d\mu_M\\
				&\leq C_\alpha \int_M |\beta|_p^2 d\mu_M = C_\alpha ||\beta||_{\mathcal{L}^2(M)}^2.
			\end{split}
		\end{equation}
	\end{proof}
	\subsection{The $T_f$ operator}
	Let $(M,g)$ and $(N,h)$ be two manifolds of bounded geometry and let $f:(M,g) \longrightarrow (N,h)$ be a uniformly proper, smooth, lipschitz map.
	\\Let us denote by $pr_M$ the projection $pr_M: (f^*(T^\delta N),g_S) \longrightarrow (M,g)$, let $\omega$ be a Thom form defined as in the previous subsection and consider $p_f: (f^*(T^\delta N),g_S) \longrightarrow (M,g)$ the submersion related to $f$ that we introduced in the previous section. If $f$ is differentiable, then we can define the operator $T_f$ for each smooth $\mathcal{L}^2$-form as
	\begin{equation}
		\begin{split}
			T_f (\alpha) &= pr_{M, \star} \circ e_\omega \circ p_f^* (\alpha)\\
			&=\int_{B^\delta} p_f^*(\alpha) \wedge \omega \label{defT}.
		\end{split}
	\end{equation}
	where $B^\delta$ denotes the fibers of $pr_M$. If $f$ is not differentiable, then we can consider a differentiable lipschitz map $f'$ which is lipschitz-homotopic to $f$ \footnote{we know that such a $f'$ exists as consequence of Proposition \ref{appr}} and we impose 
	\begin{equation}
		T_f := T_{f'}.
	\end{equation}
	Actually,  if $f$ is not differentiable, the definition of $T_f$ depends on the choice of $f'$. We will not denote the choice of the $f'$ because, as we will see later, $T_f$ induces an operator in (un)-reduced $L^2$-cohomology which doesn't depend on the choice of $f'$. 
	\begin{prop}
		Let us consider a uniformly proper lipschitz map $f$ between Riemannian manifolds of bounded geometry. Then the operator $T_f$ is a bounded operator.
	\end{prop}
	\begin{proof}
		Let us suppose $f$ is smooth.
		We know, since Corollary \ref{R.-N.-p}, that $p_f^*$ is $\mathcal{L}^2$-bounded.
		\\Because of since Corollary \ref{eomega}, we also know that $e_\omega$ is a $\mathcal{L}^2$-bounded operator.  
		\\We just have to prove that $pr_{M, \star}$ is $\mathcal{L}^2$-bounded. 
		Observe that $pr_M$ is a Riemannian submersion (Remark \ref{riemsub}) and so, in particular, it is a lipschitz map. Moreover it's easy to check that its Fiber Volume is the Volume of an euclidean ball of radius $\delta$ in each point of $M$. Then, applying Proposition \ref{cosa}, we have that $pr_M$ is a R.-N.-lipschitz map and, in particular $pr_M^*$ is a $\mathcal{L}^2$-bounded operator. Then, beacause of Remark \ref{fibnorm},
		\begin{equation}
			||pr_{M, \star}|| =||pr_M^*|| = L < +\infty.
		\end{equation}
		Then we can conclude since $T_f$ is composition of bounded operators.
	\end{proof}
	\begin{rem}
		In particular the previous Proposition holds if $f$ is a lipschitz-homotopy equivalence.
	\end{rem}
	\begin{cor}
		Given a uniformly proper lipschitz map $f:(M,g) \longrightarrow (N,h)$ between two Riemannian manifolds of bounded geometry, then we have
		\begin{equation}
			T_f(dom(d_{min})) \subseteq dom(d_{min})
		\end{equation}
		and $T_fd = dT_f$. This, in particular, means that $T_f$ induces a morphism in $L^2$-cohomology. Moreover, since $T_f$ is $\mathcal{L}^2$-bounded, it also induces a morphism between the reduced $L^2$-cohomology groups. 
	\end{cor}
	\begin{proof}
		Let us suppose, again, that $f$ is smooth.
		In order to conclude the proof we will prove that the operator $T_f$ satisfies the hypothesis of Proposition \ref{boundedness2}. We already know that $T_f$ is $\mathcal{L}^2$-bounded. Then, in order to conclude the proof, we just have to prove that $T_f(\Omega^*_c(N)) \subseteq \Omega^*_c(M)$ and $dT_f\alpha = T_f d \alpha$ for each $\alpha$ in $\Omega^*_c(N)$
		\\
		Since $p_f$ is uniformly proper we have that for each smooth form $\alpha$ in $L^2(N)$ 
		\begin{equation}
			diam(supp(e_\omega \circ p_f^*(\alpha)) \leq C_\alpha.
		\end{equation}
		Moreover, since $supp(pr_{M, \star}\beta) \subset pr_M(supp(\beta))$ for each smooth $\beta$ in $L^2(f^*T^\delta N)$, we have that $supp(T_f \alpha)$ is bounded in $M$. This means that $supp(T_f \alpha)$ is compact since $(M,g)$ is complete and $T_f(\Omega_c(N)) \subseteq \Omega_c(M)$.
		\\
		\\We have that $e_\omega \circ p_f^*(\Omega^*_{c}(N)) \subseteq \Omega^*_{vc}(f^*(T^\delta N))$, where in $\Omega^*_{vc}(f^*(T^\delta N))$ we are considering the vertically compactly supported smooth forms related to the projection on the first component $pr_M: f^*(T^\delta(M)) \longrightarrow M$. 
		\\Then we can apply the Proposition \ref{commete} of \cite{Conn}, which allow us to say $pr_{M,\star}(\Omega^*_{vc}(f^*T^\delta N) \subseteq \Omega_c(M)$ and that
		\begin{equation}
			\int_{B^\delta} d \eta = d \int_{B^\delta} \eta,
		\end{equation}
		if $\eta$ is in $\Omega_{vc}^*(f^*T^\delta N)$. Then, using that $T_f$ is a $\mathcal{L}^2$-bounded operator we conclude applying Corollary \ref{boundedness2}.
	\end{proof}
	\begin{rem}
		Since $p_f$ is a $\Gamma$-equivariant map, then $p_f^*$ is $\Gamma$-equivariant and, in particular $T_f$ is $\Gamma$-equivariant.
	\end{rem}
	\subsection{R.-N.-lipschitz equivalences of vector bundles}
	Let us consider two smooth lipschitz maps $f_0, f_1:(M,g) \longrightarrow (N,h)$ between Riemannian manifolds and consider a vector bundle $E$ over $N$. Fix a group $\Gamma$ acting on $M$ and on $N$. Then it's a known fact (see for example \cite{fibbund}) that if $f_0 \sim_\Gamma f_1$ with a lipschitz homotopy $H$, then
	\begin{equation}
		f_0^*(E) \times [0,1] \cong f_1^*(E) \times [0,1] \cong H^*(E)
	\end{equation} 
	and
	\begin{equation}
		f_0^*(E) \cong f_1^*(E)
	\end{equation}
	where with $\cong$ we mean that they are isomorphic as fiber bundles. As consequence of this fact we also have that they are homeomorphic as manifolds.
	\begin{defn}
		Consider two vector bundles $E$, $F$ over a Riemannian manifold $(N,h)$. Let us suppose that $e$ is a Riemannian metric on $E$ and $f$ a Riemannian metric on $F$. Let us denote by $E^\delta$ and $F^\delta$ the disk bundles of radius $\delta$. We say that $(E,e)$ and $(F,f)$ are \textbf{R.-N.-lipschitz equivalent} if there is a number $\delta > 0$ and there is a bundle isomorphism $\phi: E \longrightarrow F$ such that $\phi(E^\delta) = F^\delta$ and $\phi$ restricted to $E^\delta$ is a R.-N.-lipschitz map.
	\end{defn}
	\begin{rem}\label{ossbundle}
		If $E$ and $F$ are R.-N.-lipschitz equivalent, then if $A: (F^\delta, f) \longrightarrow (M,g)$ is an R.-N.-lipschitz map, then $\phi^*A$ is again an R.-N.-lipschitz map. Moreover if $B: (E^\delta , e) \longrightarrow (M,g)$ is an R.-N.-lipschitz map it is also true that $(\phi^{-1})^{*}B$ is an R.-N.-lipschitz map, indeed, since $\phi$ is a diffeomorphism,
		\begin{equation}
			(\phi^{-1})^{*}B = \phi_\star B,
		\end{equation}
		and the $\mathcal{L}^2$-norm of $\phi_\star$ is the same of the norm of $\phi^*$.
	\end{rem}
	\begin{lem}\label{bundlelem}
		Let us consider two Riemannian manifold $(M,g)$ and $(N,l)$, where $(N,l)$ is a manifold of bounded geometry. Consider $f_0, f_1: (M,g) \longrightarrow (N,l)$ two smooth lipschitz maps and suppose that there is a lipschitz-homotopy $H:(M \times [0,1], g \times g_{[0,1]}) \longrightarrow (N,l)$ such that $H(p,i) := f_{i}(p)$ for $i= 0, 1$.
		\\Let us suppose, moreover, that
		\begin{equation}
			d(f_0(p), f_1(p)) \leq \frac{inj_N}{2}
		\end{equation}
		where $\delta$ is the radius of injectivity of $N$. Let us denote by $g_{S,f^*_i}$ the Sasaki metric on $f_{i}^*(TN)$ defined using $g$, $f^*_i(l)$ and $f^*_i(\nabla^{LC}_l)$  where $i = 0,1$.
		\\Then $(f^*_0(TN), g_{S,f^*_0})$ and $(f^*_1(TN), g_{S,f^*_1})$ are R.-N.-lipschitz equivalent. In particular $\delta$ can be choose less or equal to $inj_{N}$. 
	\end{lem} 
	\begin{proof}
		We already know that $f^*_0(TN)$ and $f^*_1(TN)$ are homeomorphic. Then we will consider $f^*_0(TN)$ and $f^*_1(TN)$ as the same bundle, but with different metrics. Then we have to check that the identity map $Id: f^*_1(T^\delta N) \longrightarrow f^*_0(T^\delta N)$ is lipschitz and that its Fiber Volume is bounded. Let us start with its Fiber Volume.
		\\Fix some normal coordinates $\{U, x^i\}$ on $M$. Fix, moreover, some normal coordinates $\{V, y^j\}$ on $N$ such that $f_0(U) \subseteq V$. We also suppose that $f_0(0) = 0$ in this coordinates.
		Then we can consider $\{x^i, \mu^j\}$ the fibered coordinates on $f_0^*(TN)$ related to the frame $\{\frac{\partial }{\partial y^j}\}$. Then we have that $Id(x, \mu) = (x, \mu)$ and so
		\begin{equation}
			Vol_{Id}(x, \mu) = \sqrt{\frac{det(G_{ij})}{det(L_{jk})}(x, \mu)}
		\end{equation}
		where $G_{ij}$ is the matrix related to the metric $g_{S,f^*_1}$ and $L_{jk}$ the matrix related to $g_{S,f^*_0}$. Observe that in $(0, \mu)$ we have that
		\begin{equation}
			det(L_{jk}(0, \mu)) = 1.
		\end{equation}
		On the other hand we have that $det(G_{ij}(0, \mu))$ is an algebraic combination of pullback of Christoffell symbols and terms of the metric on $N$ along the lipschitz map $f_1$, first derivatives of $f_1$ and $\mu^j\mu^i$. We obtain a uniform bound on $|h_{E, ij}|$ and $|\Gamma^k_{E,ij}|$ on a ball of radius less or equal to $inj_rN$. Moreover we have that $|\mu^i| \leq \delta$. Then we can conclude that the function
		\begin{equation}
			\sqrt{\frac{det(G_{ij})}{det(L_{jk})}}: U \times [0,1] \longrightarrow \numberset{R}
		\end{equation}
		is bounded and its bound doesn't depend on the choice of $U$.
		\\
		\\Let us verify that $Id$ is lipschitz. This time we fix the coordinates $\{x^i\}$ on $M$ and we consider some normal coordinates $\{y^j\}$ on $N$ such that $f_1(0) = 0$. Then, exactly as we did in the previous point, we choose the coordinates $\{x^i, \mu^j\}$ related to the frame $\{\frac{\partial }{\partial y^j}\}$.
		\\Then we have to check that
		\begin{equation}
			\frac{|\frac{\partial}{\partial x^i}|_{f_0^*(TN), (0, \mu)}}{|\frac{\partial}{\partial x^i}|_{f_1^*(TN), (0, \mu)}}
		\end{equation}
		and
		\begin{equation}
			\frac{|\frac{\partial}{\partial \mu^j}|_{f_0^*(TN), (0,\mu)}}{|\frac{\partial}{\partial \mu^j}|_{f_1^*(TN), (0, \mu)}}
		\end{equation}
		are uniformly bounded. Observe that $\frac{\partial}{\partial x^i}|_{f_1^*(TN), (0, \mu)}$ and $|\frac{\partial}{\partial \mu^j}|_{f_1^*(TN), (0,\mu)}$ are equal to $1$. 
		\\Let us concentrate on $|\frac{\partial}{\partial x^i}|_{f_0^*(TN), (0,\mu)}$ and on $|\frac{\partial}{\partial \mu^j}|_{f_0^*(TN), (0,\mu)}$. 
		\\Then, similarly to what we did proving the boundedness of the Fiber Volume, we can observe that the components of metric $g_{S, f^*_0}$, in these coordinates, are algebraic combinations of pullback of Christoffell symbols and terms of the metric on $N$ along the lipschitz map $f_0$, first derivatives of $f_0$ and $\mu^j\mu^i$. Then, since on a ball of radius less or equal to $\delta = inj_rN$ we have a bound on $|h_{E, ij}|$ and $|\Gamma^k_{E,ij}|$, we conclude that the lipschitz constant is bounded.
	\end{proof}
	\begin{prop}\label{equiv}
		Let us consider $f_0, f_1: (M,g) \longrightarrow (N,h)$ two lipschitz maps such that $f_0$ and $f_1$ are lipschitz-homotopic. Then $(f^*_0(TN), g_{S,0})$ and $(f^*_1(TN), g_{S,1})$ are R.-N.-lipschitz equivalent. In particular $\delta$ can be choose less or equal to $inj_{N}$. 
	\end{prop}
	\begin{proof}
		Let us consider the homotopy $H$ between $f_0$ and $f_1$. Consider a finite partition  $\{[s_i, s_{i+1}]\}$ of $[0,1]$ such that $s_{i+1} - s_i \leq \frac{inj_N}{2C_{H}}$, where $C_h$ is the lipschitz constant of $h$. Let us define the maps $H_{s}: (M,g) \longrightarrow (N,h)$ as $H_i(p) := H(p,s_i)$ then we can observe that
		\begin{equation}
			d(H_{i}(p), H_{i+1}(p)) \leq d(H(p, s_i),(p, s_{i+1})) \leq C_H\cdot(s_{i+1} - s_i) = \frac{inj_N}{2}.
		\end{equation}
		Then it is sufficient to apply the previous lemma to $H_i$ and $H_{i+1}$ and observe that composition of R.-N.-lipschitz map is a R.-N.-lipschitz map.
	\end{proof}
	\begin{cor}\label{equivhomo}
		Let us consider a lipschitz map $H: (M \times [0,1], g + g_{[0,1]}]) \longrightarrow (N,h)$. For each $i$ in $[0,1]$, let us denote by $f_i:(M,g) \longrightarrow (N,h)$ the map defined as $f_i(p) := H(p,i)$. Then we have that the vector bundles $(H^*(TN), g_{S,H})$ and $(f^*_i(TN) \times [0,1], g_S \times g_{[0,1]})$ are R.-N.-lipschitz equivalent and $\delta$ can choose less or equal to $inj_N$.
	\end{cor}
	\begin{proof}
		Let us define the map $\overline{f}_i: M \times [0,1] \longrightarrow N$ as $\overline{f}_{i}(p,s) := f_i(p)$. Observe that the Sasaki metric on $\overline{f}_i^*(E)$ defined using $g$, $\overline{f}^*_i(h)$ and $\overline{f}^*_i(\nabla^{LC}_h)$ is the product metric between the Sasaki metric on $f^*_i(E)$ and the metric $g_{[0,1]}$ on $[0,1]$.
		\\Then we can conclude observing that the map
		\begin{equation}
			\begin{split}
				\mathcal{H}: (M \times [0,1] \times [0,1], g \times g_{[0,1]} \times g_{[0,1]}) &\longrightarrow (N,h) \\
				(p,s,t) &\longrightarrow H(p, i \cdot t + s \cdot (1-t))
			\end{split}
		\end{equation}
		is a lipschitz-homotopy between $H$ and $\overline{f}_i$.
	\end{proof}
	\subsection{Lemmas about homotopy}\label{lemhom}
	In this section we study the $\mathcal{L}^2$-boundedness of the pullback of some homotopies.
	\begin{lem}\label{homo1}
		Let $f: (M,g) \longrightarrow (N,h)$ be a smooth lipschitz map and let $\delta \leq inj_N$. Consider the homotopy $H: (f^*(T^\delta N) \times [0,1], g_S \times g_{[0,1]}) \longrightarrow (N,l)$ defined as
		\begin{equation}
			H(p,w,t) := p_f(t\cdot w_{f(p)}).
		\end{equation}
		Then, if $f$ is a smooth R.-N.-lipschitz map, then also $H$ and $(H, id_{[0,1]}): (f^*(T^\delta N) \times [0,1], g_S \times g_{[0,1]}) \longrightarrow (N \times [0,1], l \times g_{[0,1]})$ are R.-N.-lipschitz maps.
	\end{lem}
	\begin{rem}
		The map $H$ above is important for us because it is a $\Gamma$-equivariant lipschitz-homotopy between $p_f$ and $\overline{f}:f^*(T^\delta N) \longrightarrow N$. In particular, proving this Lemma, we have that the maps $p_{id}: TN \longrightarrow N$ and the projection $pr_N: TN \longrightarrow N$ are $\Gamma$-lipschitz-homotopic with an homotopy $H$ which is a R.-N.-lipschitz map, and so $H$ induces a morphism in (un)-reduced $L^2$-cohomology.
	\end{rem}
	\begin{proof}
		Observe that $H = pr_{N} \circ (H, id_{[0,1]})$, where $pr_N: N \times [0,1] \longrightarrow N$. We have that $pr_N$ is a R.-N.-lipschitz map. This means that if $(H, id_{[0,1]})$ is a R.-N.-lipschitz map, then also $H$ is a R.-N.-lipschitz map.
		\\Moreover we can observe that $f^*(T^\delta N) \times \{0,1\}$ is a set of null measure an so we can consider the interval $[0,1]$ as open.
		\\
		\\Let us define the map 
		\begin{equation}
			\begin{split}
				(pr_M, H, id_{(0,1)}): (f^*T^\delta N \times (0,1), g_S \times g_{(0,1)}) &\longrightarrow (M \times N \times (0,1), g \times h \times g_{(0,1)})\\
				(v_{f(p)}, s) &\longrightarrow (p, p_f(s \cdot v_p), s).
			\end{split}
		\end{equation}
		This map is a submersion since we are considering $s \neq 0$.
		In particular we have that $(pr_M, H, id_{(0,1)})$ is lipschitz since it is composition of lipschitz maps.
		\\Moreover we have that $(H,id_{(0,1)}) = pr_{N \times (0,1)} \circ (pr_M, H, id_{(0,1)})$ and so, applying Proposition \ref{compo}, we have
		\begin{equation}
			Vol_{(H,id_{(0,1)})}(q,t) = \int_{M} Vol_{(pr_M, H, id_{(0,1)})}(p, q, t) d\mu_M.
		\end{equation}
		Let us calculate the Fiber Volume of $(pr_M, H, id_{(0,1)})$.
		\\Similarly to the Lemma \ref{svolta}, we can consider for each $p$ in $M$ some normal coordinates $\{x^i\}$ around $p$, some normal coordinates $\{y^j\}$ around $f(p)$ in $N$ and, considering the frame $\{\frac{\partial}{\partial x^i}\}$, we have the coordinates $\{x^i, \mu^j, t\}$ on $f^*(TN)\times [0,1]$ around $(p, 0, 0)$.
		\\Moreover we can also consider the coordinates $\{x^i, y^j, t\}$ on $M\times N \times [0,1]$. \\Observe that, the coordinates $\{x^i, y^j, t\}$ are enough to cover the image of $(id_M, H, id_{[0,1]})$ which is contained in a $\delta$-neighborhood of $Graph(f) \times [0,1]$.
		\\Finally, respect to these coordinates, we have that
		\begin{equation}
			(pr_M, H, id_{(0,1)})(0, \mu^j, t) = (0, t \cdot y^j, t).
		\end{equation}
		Similarly as we did in Lemma \ref{svolta} one can show that
		\begin{equation}
			Vol_{f^*(TN) \times [0,1]}(0,\mu,t) = dx^1 \wedge ... \wedge dx^m \wedge d\mu^1 \wedge ... \wedge d\mu^n \wedge dt
		\end{equation}
		and that
		\begin{equation}
			Vol_{M \times N \times [0,1]}(0, y, t) = \sqrt{det(H_{ij}(0,y))}dx^1 \wedge ... \wedge dx^m \wedge dy^1 \wedge ... \wedge dy^n \wedge dt,
		\end{equation}
		where
		\begin{equation}
			H_{ij} = \begin{bmatrix}
				1 && 0 && 0 \\
				0 && h_{l,s}(y) && 0 \\
				0 && 0 && 1
			\end{bmatrix}
		\end{equation}
		and $h_{l,s}$ are the components of the metric $h$ on $N$. Then we have that $(pr_M, H, id_{[0,1]})$ is a diffeomorphism with its image and so applying Remark \ref{oss}, its Fiber Volume on $im(pr_M, H, id_{[0,1]})$ is given by
		\begin{equation}
			|[(pr_M, \tilde{h}, id_{[0,1]})^{-1}]^* \frac{Vol_{f^*TN \times [0,1]}}{(pr_M, \tilde{h}, id_{[0,1]})^*Vol_{M \times N \times [0,1]}}|
		\end{equation}
		and it is null otherwise.
		\\Observe that, similarly to Lemma \ref{svolta}, in a point $(0, y^j, t)$ we have that the Fiber Volume of $(pr_M, \tilde{h}, id_{[0,1]})$ is
		\begin{equation}
			\frac{Vol_{f^*TN \times [0,1]}}{(pr_M, \tilde{h}, id_{[0,1]})^*Vol_{M \times N \times [0,1]}} = \frac{1}{t^n} (1 + C(t, y)),
		\end{equation}
		where $C$ is a bounded function.
		\\
		\\Let us consider now the projection $pr_{N \times [0,1]}: M \times N \times (0,1) \longrightarrow N \times (0,1)$. Let us recall that, since the Proposition \ref{compo}, we have
		\begin{equation}
			Vol_{(H,id_{(0,1)})}(q,t) = \int_M Vol_{(id_M, H, id_{[0,1]})}(p,q,t) d\mu_M.
		\end{equation}
		We know that outside the image of $(id_M, H, id_{[0,1]})$ the Fiber Volume $Vol_{(id_M, H, id_{[0,1]})}$ is null. Let us denote by
		\begin{equation}
			\begin{split}
				H_t: M \times N &\longrightarrow N \\
				(p,q) &\longrightarrow H(p,q,t).
			\end{split}
		\end{equation}
		we have that $Vol_{(id_M, H, id_{[0,1]})}$ on $M \times \{q\} \times \{t\}$ is null outside 
		\begin{equation}
			pr_M(H^{-1}_t(q)) \times \{q\} \times \{t\} \supseteq [im(id_M, H, id_{[0,1]}) \cap M \times \{q\} \times \{t\}]
		\end{equation}
		Since $H$ is lipschitz, applying Remark \ref{uniformlyp}, we have
		\begin{equation}
			H^{-1}_t(q) \subseteq \pi^{-1}(f^{-1}(B_{C_H\cdot t}(q))) \times \{t\}
		\end{equation}
		where $\pi: f^*TN \longrightarrow M$ is the projection of the bundle. So
		\begin{equation}
			pr_M(H^{-1}_t(q)) \times \{q\} \times \{t\} \subset f^{-1}(B_{C_H\cdot t}(q))\times \{q\} \times \{t\}.
		\end{equation}
		Then using that $f$ is a R.-N.-lipschitz map, we have that
		\begin{equation}
			\mu(f^{-1}(B_{C_H\cdot t}(q))) \leq K_0 \mu_N(B_{C_H \cdot t}(q)) \leq K_0C_H^nt^n(1 + L(t)),
		\end{equation}
		and so, since $N$ has bounded geometry,
		\begin{equation}
			\mu_N(B_{C_H \cdot t}(q)) \leq C_H^nt^n(1 + L(t))
		\end{equation}
		where $L$ is a bounded function which depends by $t$. Then we have that the Fiber Volume of $(H, id_{(0,1)})$ is given by
		\begin{equation}
			\begin{split}
				Vol_{(H, id_{(0,1)})}(q,t) &= \int_{f^{-1}(B_{C_H\cdot t}(q))} Vol_{(id_M, H, id_{[0,1]})}(p,q,t) d\mu_M \\
				&\leq \frac{1}{t^n}(1 + C(t, \tilde{y})) \cdot (K_0C_H^nt^n(1 + L(t))) \\
				&\leq K \cdot \frac{1}{t^n} \cdot t^n \leq K.
			\end{split}
		\end{equation}
		and so $(H, id_{(0,1)})$ is a R.-N.-lipschitz map and so also
		\begin{equation}
			H = pr_N \circ (H, id_{(0,1)})
		\end{equation}
		is an R.-N.-lipschitz map.
	\end{proof}
	Let us consider two uniformly proper lipschitz maps $g: (S,v) \longrightarrow (M,m)$ and $f:(M,m) \longrightarrow (N,l)$ between manifolds of bounded geometry. Let us define the maps
	\begin{equation}
		\begin{split}
			\overline{g}: g^*T^\sigma M &\longrightarrow M \\
			v_{g(p)} &\longrightarrow g(p).
		\end{split}
	\end{equation}
	and
	\begin{equation}
		\begin{split}
			p_g: g^*T^\sigma M &\longrightarrow M \\
			v_{g(p)} &\longrightarrow p_g(v_{g(p)}).
		\end{split}
	\end{equation}
	Consider the compositions $f \circ \overline{g}$ and $f \circ p_g:  g^*T^\sigma M \longrightarrow M$. Observe that these maps are lipschitz-homotopy equivalent. Because of Proposition \ref{equiv}, we have that the pullback bundles
	\begin{equation}
		(f \circ \overline{g})^*TN \longrightarrow S
	\end{equation}
	and 
	\begin{equation}
		(f \circ p_g)^*TN \longrightarrow S
	\end{equation}
	are R.-N.-lipschitz homotopy equivalent\footnote{for $\delta \leq inj_N$.}. Denote by $h$ the homotopy
	\begin{equation}
		\begin{split}
			h: g^*(T^\sigma M) \times [0,1] &\longrightarrow N \\
			(v_{g(p)}, t) &\longrightarrow f \circ p_g(t \cdot v_{g(p)})
		\end{split}
	\end{equation}
	between $f \circ p_g$ and $f \circ \overline{g}$. Because of Corollary \ref{equivhomo} we have that $(h^*(TN), g_{S,h})$ and 
	\begin{equation}
		(f \circ \overline{g}^*(T^\delta N) \times [0,1], g_{S} \times g_{[0,1]})
	\end{equation}
	are R.-N.-lipschitz equivalent. Then we have that
	\begin{equation}
		p_h: ((f \circ \overline{g})^*(T^\delta N) \times [0,1], g_{S} \times g_{[0,1]}) \longrightarrow (N,l)
	\end{equation}
	is a R.-N.-lipschitz map because $h$ is uniformly proper and lipschitz map and because of Remark \ref{salva}.
	\\Observe that $p_h$ is a lipschitz homotopy between the maps
	\begin{equation}
		p_{f \circ \overline{g}}: f \circ \overline{g}^*(T^\delta N) \longrightarrow (N,l)
	\end{equation}
	and
	\begin{equation}
		p_{f \circ p_{g}}: f \circ \overline{g}^*(T^\delta N) \longrightarrow (N,l).
	\end{equation} 
	Applying the point 4. of Proposition \ref{tilde}, we can obtain that
	\begin{equation}\label{cosse}
		\begin{split}
			p_{f \circ p_{g}} &= p_{id_M} \circ F \circ P_g,\\
			&= p_f \circ P_g
		\end{split}
	\end{equation}
	where $F: f^*(T^\delta N) \longrightarrow T^\delta N$ is the bundle map induced by $f$ and $P_g: (f \circ \overline{g})^*T^\delta N \longrightarrow f^*T^\delta N$ is the bundle map induced by $p_g$.
	\begin{prop}\label{homo2}
		The map $H$
		\begin{equation}
			\begin{split}
				H: ((f \circ \overline{g})^*T^\delta N \times [0,1], g_{S} \times g_{[0,1]}) &\longrightarrow (N, l) \\
				(v_{f \circ \overline{g}}, t) & \longrightarrow p_f(P_g(t\cdot v_{f \circ \overline{g}}))
			\end{split}
		\end{equation}
		is a R.-N.-lipschitz map.
	\end{prop}
	\begin{proof}
		Observe that $H = p_h$ defined above. Observe that since Remark \ref{Riccib} we have that $(f \circ \overline{g})^*T^\delta N \times [0,1]$ has bounded Ricci curvature. Then, since $h$ is a smooth, uniformly proper and lipschitz map and since Remark \ref{salva}, we have that $H = p_H$ is a R.-N.-lipschitz map.
	\end{proof}
	\subsection{Lipschitz-homotopy invariance of (un)-reduced $L^2$-cohomology}
	Consider the category $\mathcal{C}$ which has manifolds of bounded geometry as objects and uniformly proper lipschitz maps as arrows. Consider, moreover, the category $\textbf{Vec}$ which has complex vector spaces as objects and linear maps as arrows. In this section we will show that, for every $z$ in $\numberset{N}$, the association $\mathcal{F}_z: \mathcal{C} \longrightarrow \textbf{Vec}$ defined as
	\begin{equation}
		\begin{cases}
			\mathcal{F}_z(M,g) = H^z_2(M) \\
			\mathcal{F}_z((M,g) \xrightarrow{f} (N,h)) = H^z_2(N) \xrightarrow{T_f} H^z_2(M)
		\end{cases}
	\end{equation}
	is a controvariant functor.
	\\Moreover, we will show that if two maps $f_1$ and $f_2$ are uniformly proper lipschitz-homotopic then $T_{f_1} = T_{f_2}$ in (un)-reduced $L^2$-cohomology and that if $f$ is a $L^2$-map\footnote{i.e. a smooth R.-N.-lipschitz map which is also uniformly proper}, then $f^* = T_f$ in (un)-reduced $L^2$-cohomology.
	\\This fact will imply that if $f$ is a lipschitz-homotopy equivalence between manifolds of bounded geometry, then the $L^2$-cohomology groups are isomorphic. In other words, the $L^2$-cohomology is a lipschitz-homotopy invariant for manifolds of bounded geometry.
	\\
	\\Let us introduce, the operator $\int_{0 \mathcal{L}}^1$: it will be the main tool that we will use to study homotopies.
	\begin{lem}\label{GLee}
		Let $(M,g)$ be a Riemannian manifold and consider $([0,1], g_{[0,1]})$. Then there is a $\mathcal{L}^2$-bounded operator 
		\begin{equation}
			\int_{0 \mathcal{L}}^1: \mathcal{L}^2(M\times[0,1]) \longrightarrow \mathcal{L}^2(M)
		\end{equation}
		such that for all smooth $\alpha \in \mathcal{L}^2(M\times[0,1])$ we have, in $\Omega^*(M)$,
		\begin{equation}
			i_1^*\alpha - i_0^*\alpha = \int_{0 \mathcal{L}}^1d \alpha + d\int_{0 \mathcal{L}}^1 \alpha
		\end{equation}
		Moreover $\int_{0 \mathcal{L}}^1$ sends compactly-supported differential forms on $\Omega^*(M \times [0,1])$ into $\Omega^*_c(M)$.
	\end{lem}
	\begin{proof}
		Let $\alpha$ be in $\Omega^*_c(M\times [0,1])$ with compact support and let $p:M\times[0,1] \longrightarrow M$ be the projection on the first component. Then we can decompose every differential form on $M \times [0,1]$ as
		\begin{equation}
			\alpha = g(x,t)p^*\omega +  f(x,t)dt \wedge p^*\omega,
		\end{equation}
		for some $\omega$ in $\Omega_c(M)$ and for some $C^{\infty}$-class functions $g,f:M\times[0,1] \longrightarrow \numberset{C}$. Then we can define the linear operator $\int_{0,\mathcal{L}}^{1}$ as follow: if $\alpha$ is a $0$-form with respect to $[0,1]$, then
		\begin{equation}
			\int_{0,\mathcal{L}}^{1}\alpha := \int_{0,\mathcal{L}}^{1} g(x,t)p^*\omega =  0
		\end{equation}
		and, if $\alpha$ is a $1$-form with respect to $[0,1]$,
		\begin{equation}
			\int_{0,\mathcal{L}}^{1}\alpha := (\int_{0}^{1} f(x,t)dt)\omega.
		\end{equation}
		This operator is very similar to the operator $p_\star$ (the integration along the fiber of $p: M \times [0,1] \longrightarrow M$), but they are different for the signs. Indeed if we consider $\alpha = f(x,t)dt \wedge p^*\omega$ we have that
		\begin{equation}
			\begin{split}
				p_\star \alpha &= \int_{[0,1]} f(x,t)dt \wedge p^*\omega \\
				&=  \int_{[0,1]} (-1)^{deg(p^*\omega)} p^*\omega \wedge f(x,t)dt.
			\end{split}
		\end{equation}
		Then applying the Projection Formula we obtain
		\begin{equation}
			\begin{split}
				p_\star \alpha &= (-1)^{deg(p^*\omega)} \omega \cdot (\int_{[0,1]} f(x,t)dt)\\
				&= (-1)^{deg(p^*\omega)} \omega \cdot (\int_{0}^{1} f(x,t)dt)\\
				&= (-1)^{deg(p^*\omega)}\int_{0,\mathcal{L}}^{1}\alpha.
			\end{split}
		\end{equation}
		Moreover the operator $\int_{0,\mathcal{L}}^1$ doesn't commute with the exterior derivative $d$ on $\Omega_{vc}(M \times [0,1])$, indeed if we consider for example $M = U$ an open set of $\numberset{R}^n$
		\begin{equation}
			\begin{split}
				d \int_{0,\mathcal{L}}^{1}\alpha &= d (\int_{0}^{1} f(x,t)dt) \\
				&= (\int_0^1 \frac{\partial f}{\partial x^i}dt)\wedge dx^i
			\end{split} 
		\end{equation}
		but 
		\begin{equation}
			\begin{split}
				\int_{0,\mathcal{L}}^{1} d\alpha &= \int_{0\mathcal{L}}^1 (\frac{\partial f}{\partial x^i}dx^i\wedge dt) \\
				&= \int_{0\mathcal{L}}^1 (-\frac{\partial f}{\partial x^i}dt\wedge dx^i)\\
				&= - (\int_0^1 \frac{\partial f}{\partial x^i}dt) dx^i. 
			\end{split}
		\end{equation}
		Since the operator $\int_{0\mathcal{L}}^1$, up to signs, is the operator of integration along the fibers, then we have that the norm is the same of $p_\star=1$. We know from Lemma 11.4 of \cite{lee} that $\int_{0\mathcal{L}}^1: \Omega^*(M \times [0,1]) \longrightarrow \Omega^{*-1}(M)$ and that for all differential forms we have that
		\begin{equation}
			\int_{0,\mathcal{L}}^{1} d \alpha + d \int_{0,\mathcal{L}}^{1}\alpha = i_1^*\alpha - i_0^*\alpha.
		\end{equation}
	\end{proof}
	\begin{rem}
		Consider a map $f:(M,g) \longrightarrow (N,h)$ between manifolds of bounded geometry and let $H: (M \times[0,1], g + g_{[0,1]}) \longrightarrow (N,h)$ be a lipschitz map such that $H(x,0) = f(x)$.
		\\Because of Corollary \ref{equivhomo}, we have that the bundles $(H^*TN, g_H)$ and $(f^*(TN) \times [0,1], g_S \times g_{[0,1]})$ are R.-N.-lipschitz equivalent. Then we have that the map 
		\begin{equation}
			pr_{[0,1]}: (H^*T^\delta N, g_H) \longrightarrow (f^*T^\delta N, g_S)
		\end{equation}
		which is the projection on the component in $[0,1]$ is a R.-N.-lipschitz map. Then, using that,
		\begin{equation}
			||pr_{[0,1]}^*|| = ||pr_{[0,1], \star}|| = ||\int_{0\mathcal{L}}^1||
		\end{equation}
		we obtain that $\int_{0\mathcal{L}}^1: \mathcal{L}^2(H^*T^\delta N) \longrightarrow \mathcal{L}^2(f^*T^\delta N)$ is an $\mathcal{L}^2$-bounded operator.
	\end{rem}
	\begin{prop}
		Let $id: (N,h) \longrightarrow (N,h)$ be the identity map on a manifold of bounded geometry and consider $p_{id}: TN \longrightarrow N$ the submersion related to the identity as in Theorem \ref{tilde}. Then, if $pr_N: TN \longrightarrow N$ is the projection of the tangent bundle, then there is a $\mathcal{L}^2$-bounded operator $K_1: \mathcal{L}^2(N) \longrightarrow \mathcal{L}^2(T^\delta N)$ such that for every smooth form $\alpha$
		\begin{equation}\label{K1}
			p_{id}^*\alpha - pr_N^*\alpha = d \circ K_1\alpha + K_1\circ d\alpha.
		\end{equation}
		Moreover, if $\alpha$ is in $\Omega^*_c(N)$, then $K_1\alpha \in \Omega^*(T^\delta N)$ has compact support.
	\end{prop}
	\begin{proof}
		Observe that $pr_N$ is a lipschitz submersion with bounded Fiber Volume and so $pr_N^*$ is a bounded operator.
		\\Then we have that $p_{id}$ and the projection $pr_N: (TN, g_S) \longrightarrow (N, h)$ are lipschitz-homotopic. Indeed, if we consider
		\begin{equation}
			H(p,t,s) = p_{id}(p,st)
		\end{equation}
		this is a lipschitz homotopy. Moreover, by Lemma \ref{homo1} we know that $H$ is R.-N.-lipschitz and so $H^*$ is an $\mathcal{L}^2$-bounded operator.
		\\This means that for all $\alpha$ in $\Omega^*_c(N)$, using Lemma 11.4. of \cite{lee} we have that
		\begin{equation}
			p_{id}^* - pr_N^* (\alpha)= i_1^*H^*\alpha - i_0^*H^*\alpha = \int_{0\mathcal{L}}^1 H^* d\alpha+ d \int_{0\mathcal{L}}^1 H^*\alpha.
		\end{equation}
		Then
		\begin{equation}
			K_1:= \int_{0\mathcal{L}}^1 \circ H^*,
		\end{equation}
		satisfies (\ref{K1}) and it is $\mathcal{L}^2$-bounded.
	\end{proof}
	\begin{prop}\label{K2}
		Consider $g: (M,m) \longrightarrow (N,h)$ and $f:(N,h) \longrightarrow (S,r)$ two smooth uniformly proper lipschitz maps between manifolds of bounded geometry. There is a $\mathcal{L}^2$-bounded operator $K_2: \mathcal{L}^2(N) \longrightarrow \mathcal{L}^2((f \circ \overline{g})^*TS)$ such that for every smooth form $\alpha$
		\begin{equation}
			p_{f \circ p_g}^*\alpha - p_{f\circ \overline{g}}^*\alpha = d \circ K_2\alpha + K_2 \circ d\alpha.
		\end{equation}
		Moreover, if $\alpha$ is in $\Omega^*_c(S)$, then $K_2\alpha \in \Omega^*((f \circ \overline{g})^*T^\delta S)$ has compact support.
	\end{prop}
	\begin{proof}
		Consider $h: g^*f^*TS \times [0,1] \longrightarrow S$ the lipschitz homotopy between $f \circ p_g$ and $f \circ \overline{g}$. Let us define the submersion related to $h$ defined in Theorem \ref{tilde}
		\begin{equation}
			p_h: h^*(TS) \longrightarrow S
		\end{equation}
		Observe that $p_h$ is the homotopy between $p_f \circ P_g = p_{f \circ p_g}$ and $p_{f\circ \overline{g}}$ defined in Lemma \ref{homo2} and so $p_h$ is a R.-N.-lipschitz map. Let us define the operator
		\begin{equation}
			K_2 := \int_{0\mathcal{L}}^1 \circ p_h^*.
		\end{equation}
		It is a $\mathcal{L}^2$-bounded operator because it is composition of $\mathcal{L}^2$-bounded operators. Then we can observe that for every smooth form $\alpha$ we have
		\begin{equation}
			\begin{split}
				(p_{f \circ p_g})^* \alpha - p_{f\circ \overline{g}}^*\alpha &= (i_0^* - i_1^*)p_h^*\alpha \\
				&= (d \circ \int_{0\mathcal{L}}^1) p_h^* \alpha + ( \int_{0\mathcal{L}}^1 \circ d) p_h^*\alpha \\
				&= d \circ K_2 \alpha + K_2 \circ d \alpha.
			\end{split}
		\end{equation}
		Finally, since $p_h$ is a proper map (it is composition of proper maps), if $\alpha \in \Omega^*_c(S)$ the support of $K_2\alpha \in \Omega^*((f \circ \overline{g})^* TS)$ is compact.
	\end{proof}
	\begin{prop}\label{K3}
		Let $f_1$ and $f_2:(M,m) \longrightarrow (N,l)$ be two smooth, uniformly proper and lipschitz maps between manifolds of bounded geometry. Let us suppose that $f_1$ and $f_2$ are uniformly proper lipschitz-homotopic with a smooth lipschitz homotopy. Then there is a $\mathcal{L}^2$-bounded operator $K_3: \mathcal{L}^2(N) \longrightarrow \mathcal{L}^2(p_{f_1}^*TN)$ such that for all smooth form $\alpha$ 
		\begin{equation}
			p_{f_1}^* \alpha - p_{f_2}^* \alpha = d \circ K_3 \alpha + K_3\circ d \alpha.
		\end{equation}
		Moreover if $\alpha \in \Omega^*_c(N)$ then the support of $K_3\alpha \in \Omega^*(p_{f_1}^*TN)$ is compact.	
	\end{prop}
	\begin{proof}
		Observe that $h$, the homotopy such that $h(p, 0) = f_1(p)$ and $h(p,1) = f_2(p)$, is a uniformly proper lipschitz map. Let us define the metric $g_S$ on $f_1^*TN$ as the Sasaki metric defined using the Riemannain metric $m$, the bundle metric $f_1^*l$ and the connection $f^*\nabla^{LC, N}$.
		\\Applying Corollary(\ref{R.-N.-p}), we have that $p_h: (f_1^*TN \times  [0,1], g_S \times g_{[0,1]}) \longrightarrow (N,h)$ is a R.-N.-lipschitz map and so $p_h^*$ is a $\mathcal{L}^2$-bounded operator.
		\\Moreover $p_h$ is a lipschitz-homotopy between $p_{f_1}$ and $p_{f_2}$. This fact follows directly by the definition of submersion related to a lipschitz map in Theorem \ref{finsub}.
		\\So we can conclude exactly as we did in the Proposition \ref{K1} and Proposition \ref{K2} considering
		\begin{equation}
			K_3 := \int_{0\mathcal{L}}^1 \circ p_h^*
		\end{equation}
		and using that $p_h$ is a proper map.
	\end{proof}
	\begin{prop}
		Consider $(M,g)$,$(N,h)$ and $(S,l)$ three manifolds of bounded geometry and consider $f:(M,g) \longrightarrow (N,h)$, $F:(M,g) \longrightarrow (N,h)$ and $g:(S,l)  \longrightarrow (M,g)$ three uniformly proper lipschitz maps, possibly non-smooth. Then, in $L^2$-cohomology, we have that
		\begin{enumerate}
			\item $T_{id_M} = Id_{H^*_2(M)}$,
			\item if $f$ is not differentiable and $f'$ is a smooth lipschitz maps which is lipschitz-homotopic to $f$, then $T_{f}$ doesn't depend on the choice of $f'$,
			\item if $f \sim_\Gamma F$ then $T_f = T_F$,
			\item for each $\delta_0 \leq \delta = inj_{N}$ is possible to find a Thom form $\omega_0$ of $f^*TN$ with support in $f^*T^{\delta_0}N$. In particular if we denote by $T_f^0$ the $T_f$ operator defined using $\omega_0$ instead of $\omega$, we obtain that
			\begin{equation}
				T_f = T_{f}^0,
			\end{equation}
			\item $T_{f \circ g} = T_g \circ T_f$,
			\item if $f$ is a $L^2$-map then $f^* = T_f$,
		\end{enumerate}
		Moreover, since the operator $T_f$ is bounded, then the identities above also holds in reduced $L^2$-cohomology.
	\end{prop}
	\begin{proof}
		\textbf{Point 1.} Let us consider the standard projection $pr_M: TM \longrightarrow M$. Then, because of the previous proposition we have that for all smooth forms $\alpha$ in $\Omega^*_c(M)$,
		\begin{equation}
			pr_M^* - p_{id}^*(\alpha) = d \circ K_1 + K_1 \circ d (\alpha).
		\end{equation}
		Then we have the identity map in $\mathcal{L}^2(M)$ can be written as
		\begin{equation}
			1(\alpha) := \int_{B^k} pr^*_M \alpha \wedge \omega = pr_{M \star} \circ e_{\omega} \circ pr^*_M(\alpha)
		\end{equation}
		where $pr_{M \star}$ is the operator of integration along the fibers of $pr_M$.
		\\We have that for every $\alpha \in \Omega_c^*(M)$
		\begin{equation}
			\begin{split}
				1 - T_{id_M}(\alpha) &= pr_{M \star} \circ e_{\omega} \circ pr^*_M(\alpha) - pr_{M \star} \circ e_{\omega} \circ p_{id}^*(\alpha) \\
				&= pr_{M \star} \circ e_{\omega} \circ (pr^*_M - p_{id}^*)\alpha \\
				&= pr_{M \star} \circ e_{\omega} \circ (d \circ K_1 + K_1 \circ d) \alpha.
			\end{split}
		\end{equation}
		Observe that since $\omega$ is closed, we have that $d(\alpha \wedge \omega) = (d\alpha) \wedge \omega$. Moreover $\alpha \wedge \omega$ is in $\Omega^*_{vc}(TM)$. This means that the exterior derivative can be switched with $pr_{M \star}$ and then we have
		\begin{equation}
			\begin{split}
				1 - T_{id_M}(\alpha) &= d \circ  pr_{M \star} \circ e_{\omega} \circ K_1 + pr_{M \star} \circ e_{\omega} \circ K_1 \circ d (\alpha)\\
				&=d \circ Y_1 + Y_1 \circ d(\alpha) \label{Y_1},
			\end{split}
		\end{equation}
		where
		\begin{equation}
			Y_1 := pr_{M \star} \circ e_{\omega} \circ K_1.
		\end{equation}
		Observe that $Y_1$ is a $\mathcal{L}^2$-bounded operator because it is composition of $\mathcal{L}^2$-bounded operators. Moreover applying Proposition \ref{boundedness} choosing $B = Y_1$ and $K = 1 - T_{id_M}$ we obtain that $Y_1(dom(d)) \subseteq dom(d)$ and
		\begin{equation}
			1 \beta + T_{id_M}\beta = dY_1 \beta + Y_1 d\beta
		\end{equation}
		for every $\beta$ in $dom(d_{min})$. Then in $L^2$-cohomology we have that
		\begin{equation}
			Id_{H^*_2(M)} = T_{id_M}.
		\end{equation}
		\\
		\\
		\\
		\\
		\textbf{Point 2.} Consider two differentiable lipschitz approximation $f_1$ and $f_2$ of $f$. Then we have that $f_1 \sim_\Gamma f_2$. In order to prove this fact let us denote by $h_1:M \times[0,1] \longrightarrow N$ the lipschitz homotopy between $f_1$ and $f$ and by $h_2: M\times [0,1] \longrightarrow N$ the lipschitz-homotopy between $f$ and $f_2$. Then we can define a lipschitz homotopy $h: M\times [0,1] \longrightarrow N$ between $f_1$ and $f_2$ as
		\begin{equation}
			h(p,t) := \begin{cases} h_1(p, 2t) \mbox{                if              } t\in [0, \frac{1}{2}] \\
				h_2(p, 2t -1) \mbox{                if              } t\in [\frac{1}{2}, 1].
			\end{cases}
		\end{equation}
		Let us consider, now, a differentiable lipschitz homotopy $H$ between $f_1$ and $f_2$ as in Lemma \ref{C^k_bhomo}. We know, since Proposition \ref{K3} that there is a $\mathcal{L}^2$-bounded operator $K_3$, such that for every $\alpha$ in $\Omega_c^*(N)$
		\begin{equation}
			p_{f_1}^* - p_{f_2}^* (\alpha)= d \circ K_3 + K_3 \circ d (\alpha).
		\end{equation}
		This means that if we consider the $\mathcal{L}^2$-bounded operator\footnote{it is $\mathcal{L}^2$-bounded because it is composition of $\mathcal{L}^2$-bounded operators}
		\begin{equation}
			Y_3 := pr_{M\star} \circ e_\omega \circ K_3,
		\end{equation}
		then for all $\alpha$ in $\Omega_c^*(N)$ we have that
		\begin{equation}
			\begin{split}
				T_{f_1} - T_{f_2} (\alpha) &= pr_{M\star} \circ e_\omega \circ (p_{f_1}^* - p_{f_2}^*) \\
				&= pr_{M\star} \circ e_\omega \circ (d \circ K_3 + K_3 \circ d) (\alpha)\\
				&= d \circ Y_3 + Y_3 \circ d (\alpha).
			\end{split}
		\end{equation}
		Again, applying Proposition 1.2.2. we obtain that the identity above holds for all $\beta$ in $dom(d_{min})$
		\\Then, in $L^2$-cohomology, $T_f$ doesn't depend on the choice of the smooth lipschitz approximation of $f$.
		\\		
		\\				
		\\\textbf{Point 3.} Let us consider $f'$ and $F'$ two differentiable maps which are lipschitz-homotopic to $f$ and $F$. Let $H$ be a differentiable lipschitz homotopy between $f'$ and $F'$. Then we can conclude as the previous point.		
		\\		
		\\				
		\textbf{Point 4.} Let us consider $\omega$ the Thom form on $f^*(T^\delta N)$ defined in subsection \ref{forma}. Consider $\delta_0 \leq \delta$ and let $\phi$ be the map
		\begin{equation}
			\begin{split}
				\phi: f^*(T^{\delta}N) &\longrightarrow f^*(T^\delta N) \\
				(p, v_{f(p)}) &\longrightarrow (p, \frac{\delta_0}{\delta}v_{f(p)}).
			\end{split}
		\end{equation}
		Observe that $\phi$, respect to the Sasaki metric, is lipschitz. This means that the punctual norm of $\phi^*\omega$ is uniformly bounded and so $e_{\phi^*\omega}$ is a $\mathcal{L}^2$-bounded operator. Then if use $\phi^*\omega$ instead of $\omega$ to define $T_f$ we still have a $\mathcal{L}^2$-bounded operator.
		\\Moreover we can also observe that $\phi^*\omega$ is still a Thom form for $f^*(TN)$ and $supp(\phi^*\omega)$ is contained in $f^*T^{\delta_0}N$. The forms $\omega$ and $\phi^*\omega$ define on compactly supported cohomology on $f^*T^\delta N$ the same class. This means that if alpha is a closed form, then
		\begin{equation}
			\begin{split}
				\int_{B^\delta} (\omega - \phi^*\omega) \wedge \alpha &= \int_{B^\delta} d \eta \wedge \alpha \\
				&= \int_{B^\delta} d(\eta \wedge \alpha) + \int_{B^\delta} \eta \wedge d\alpha \\
				&= \int_{\partial B^\delta} \eta \wedge \alpha + d\int_{B^\delta} \eta \wedge \alpha + \int_{B^\delta} \eta \wedge d\alpha \\
				&= d\int_{B^\delta} \eta \wedge \alpha + \int_{B^\delta} \eta \wedge d\alpha
			\end{split}
		\end{equation}
		This means that if there is a form $\eta$ such that $\pi_\star \circ e_\eta$ is a $\mathcal{L}^2$-bounded operator and $\pi_\star \circ e_\eta(dom(d)) \subseteq dom(d)$ then we can conclude the proof of this point for both the unreduced and reduced $L^2$-cohomology\footnote{we are denoting by $\pi$ the projection $\pi: f^*T^\delta N \longrightarrow M$}.
		\\
		\\Let us introduce the map
		\begin{equation}
			\begin{split}
				\Phi: f^*(T^{\delta}N) \times [0,1] &\longrightarrow f^*(T^\delta N) \\
				(p, v_{f(p)}, s) &\longrightarrow (p, [\frac{\delta_0}{\delta}s + \delta (1-s)]v_{f(p)}).
			\end{split}
		\end{equation}
		Observe that
		\begin{equation}
			\omega\wedge \beta - \phi^*\omega \wedge \beta = \int_{0\mathcal{L}}^1 \Phi^* \omega \wedge d \beta + d  \int_{0\mathcal{L}}^1 \Phi^* \omega \wedge \beta.
		\end{equation}
		This means
		\begin{equation}\label{pino}
			\int_{B^\delta} \omega\wedge \beta - \int_{B^\delta} \phi^*\omega \wedge \beta = \int_{B^\delta}  \int_{0\mathcal{L}}^1 \Phi^* \omega \wedge d \beta + d  \int_{B^\delta}  \int_{0\mathcal{L}}^1 \Phi^* \omega \wedge \beta.
		\end{equation}
		We know that $\omega$, in local fibered coordinates, has the form $\alpha_{IJ}(x, \mu) dx^I \wedge d\mu^J$ where $|\alpha_{IJ}|$ uniformly bounded. Let us denote by $A(s) := [\frac{\delta_0}{\delta}s + 1-s]$. Then we have that
		\begin{equation}
			\begin{split}
				\Phi^* \omega &= \alpha_{IJ}(x, A(s)\mu) \frac{\partial A}{\partial s}(s) \cdot A(s)^{|J|-1} dx^I\wedge d\mu^J  \wedge ds\\
				&+\alpha_{IJ}(x, A(s)\mu) \cdot A(s)^{|J|} dx^I\wedge d\mu^J
			\end{split}
		\end{equation}
		Observe that $A$ has bounded derivative. Then we have that
		\begin{equation}
			\begin{split}
				\int_{0\mathcal{L}}^1 \Phi^* \omega &= [\int_{0}^1 \alpha_{IJ}(x, A(s)\mu) \frac{\partial A}{\partial s}(s) \cdot A(s)^{|J|-1} dx^I\wedge d\mu^J  \wedge ds \\
				&= \beta_{RS}(x, \mu)dx^I \wedge d\mu^J
			\end{split}
		\end{equation}
		where $|\beta_{IJ}|$ are uniformly bounded. This fact, as we showed in subsection \ref{forma}, implies $|\int_{0\mathcal{L}}^1 \Phi^* \omega|_{p, v_{f(p)}}$ is uniformly bounded. Let us denote by $\eta = \int_{0\mathcal{L}}^1 \Phi^* \omega$: we have that $e_{\eta}$ is a $\mathcal{L}^2$-bounded operator. Consider the projection of the bundle $\pi:f^*T^\delta N \longrightarrow N$.
		Then we have that the operator
		\begin{equation}
			\pi_\star \circ e_\eta
		\end{equation}
		is $\mathcal{L}^2$-bounded. We have to prove that $\pi_\star \circ e_\eta(dom(d)) \subseteq dom(d)$. Consider $\beta$ in $dom(d)$: we have that
		\begin{equation}
			\begin{cases}
				\beta = \lim\limits_{k} \beta_k \\
				d \beta = \lim\limits_{k} d\beta_k.
			\end{cases}
		\end{equation}
		where $\beta_k$ are compactly supported differential forms. We have to show that there is a sequence of compactly supported differential forms $\gamma_k$ and there is a $\mathcal{L}^2$-form $y$ such that
		\begin{equation}
			\begin{cases}
				\pi_{\star} \circ e_{\eta } \beta = \lim\limits_{k} \gamma_k \\
				y = \lim\limits_{k} d\gamma_k.
			\end{cases}
		\end{equation}
		Observe that if we consider $\gamma_k := \int_{B^\delta} \circ e_{\eta } \beta_k$  the first condition is satisfied because $\int_{B^\delta} \circ e_{\eta }$ is composition of $\mathcal{L}^2$-bounded operator.
		\\Moreover (\ref{pino}) implies that
		\begin{equation}
			\lim\limits_{k} \gamma_k = \int_{B^\delta} (\omega - \phi^*\omega) \wedge \beta - \int_{B^\delta} \eta \wedge d \beta
		\end{equation}
		which is a $\mathcal{L}^2$-form. This conclude the proof of this point.
		\\
		\\
		\\
		\textbf{Point 5.} Since the previous points we can consider $f$ and $g$ as differentiable maps. Consider the submersions $p_f: f^*T^\delta N \longrightarrow N$, $p_g: g^*T^\sigma S \longrightarrow M$, $p_{f \circ g}: (f \circ g)^*T^\delta N \longrightarrow N$ related to $f$, $g$, and $f \circ g$. Then, since the previous Proposition, we have that for every $\alpha$ in $\Omega^*_c(N)$ 
		\begin{equation}
			(P_g)^* \circ p_f^* - \overline{p}^*_{f \circ g}(\alpha)= d \circ K_2 + K_2 \circ d (\alpha). 
		\end{equation}
		Let us denote by $\omega$ and $\omega'$ are Thom form in $T S$ and in $TN$ defined as in subsection \ref{forma}. With a little abuse of notation we will confuse $\omega$ and $\omega'$ with their pullback on other vector bundles.
		Recall that
		\begin{equation}
			\overline{g}: g^*TM \longrightarrow M
		\end{equation}
		is the map defined as $\overline{g}(s, w_{g(s)}) := g(s)$, where $s \in S$.
		\\Then we have that
		\begin{equation}
			\begin{split}
				T_{f \circ g} &= pr_{S \star} \circ e_{\omega} \circ p_{f\circ g}^* \\
				&=pr_{S \star} \circ e_{\omega} \circ id_{(f \circ \overline{g})^* T^\delta N} \circ p_{f\circ g}^* \\
				&=pr_{S \star} \circ e_{\omega} \circ pr_{g^*TM, \star} \circ e_{\omega'} \circ pr_{g^*TM}^* \circ p_{f\circ g}^*,
			\end{split}
		\end{equation}
		where $pr_S: (f \circ \overline{g})T^\delta N \longrightarrow S$ and $pr_{g^*TM} : (f \circ \overline{g})^* T^\delta N \longrightarrow g^*T^\sigma M$ are the projections of the vector bundles. Observe that
		\begin{equation}
			p_{f \circ \overline{g}} = p_{f\circ g} \circ pr_{g^*TM},
		\end{equation}
		and so we have that
		\begin{equation}
			T_{f \circ g} = pr_{S \star} \circ e_{\omega} \circ pr_{g^*TM \star} \circ e_{\omega'} \circ p_{f \circ \overline{g}}^*. 
		\end{equation}
		Now we will focus on $T_g \circ T_f$. We have that
		\begin{equation}
			T_g \circ T_f = pr_{S \star} \circ e_{\omega'} \circ p_g^* \circ pr_{M \star} \circ e_\omega \circ p_f^*.
		\end{equation}
		It's possible to apply the Proposition VIII of Chapter 5 in \cite{Conn} to the fiber bundles 
		\begin{equation}
			(\overline{g}^*f^*TN , pr_{g^*TM}, g^*T^\sigma M, B^\delta)
		\end{equation}
		and $(TM, pr_{M}, M, B^\sigma)$ and the bundle morphism $P_g$ induced by $p_g$. \\We obtain that
		\begin{equation}
			p_g^* \circ pr_{M \star} = pr_{g^*TM, \star} \circ P_g^*.
		\end{equation}
		Finally we have that also $e_{\omega'}$ and $P_g^*$ commute. This means that
		\begin{equation}
			\begin{split}
				T_g \circ T_f &= pr_{S \star} \circ e_{\omega'} \circ p_g^* \circ pr_{M \star} \circ e_\omega \circ p_f^* \\
				&= pr_{S \star} \circ e_{\omega'} \circ pr_{g^*TM, \star} \circ P_g^* \circ e_\omega \circ p_f^*\\
				&= pr_{S \star} \circ e_{\omega'} \circ pr_{g^*TM, \star} \circ e_\omega \circ P_g^*\circ p_f^* 
			\end{split}
		\end{equation}
		This means that, on $dom(d_{min})$, because of Proposition \ref{K2}, we have
		\begin{equation}
			\begin{split}
				T_{f \circ g} - T_g \circ T_f &= pr_{S \star} \circ e_{\omega'} \circ pr_{g^*TM, \star} \circ e_\omega \circ (p_{f \circ \overline{g}}^* - (p_f\circ P_g)^*) \\
				&=pr_{S \star} \circ e_{\omega'} \circ pr_{g^*TM, \star} \circ e_\omega \circ (K_2 \circ d - d \circ K_2).
			\end{split}
		\end{equation}
		Now, since for each smooth form $\alpha$ we have that $e_\omega(\alpha)$ is a vertically compact smooth form, we have that
		\begin{equation}
			pr_{S \star} \circ e_{\omega'} \circ pr_{g^*TM, \star} \circ e_\omega \circ d = d \circ pr_{S \star} \circ e_{\omega'} \circ pr_{g^*TM, \star} \circ e_\omega.
		\end{equation}
		Let us define the $\mathcal{L}^2$-bounded operator
		\begin{equation}
			Y_2 := pr_{S \star} \circ e_{\omega'} \circ pr_{g^*TM, \star} \circ e_\omega \circ K_2.
		\end{equation}
		Then we have that for all $\alpha$ in $\Omega_c^*(N)$ the equality
		\begin{equation}
			T_{f \circ g} - T_g \circ T_f (\alpha)= d \circ Y_2 + Y_2 \circ d(\alpha)
		\end{equation}
		holds. In order to conclude the proof we have to show that the equation above also holds for every $\beta$ in $dom(d_{min})$. However the proof of this fact is a consequence of Proposition \ref{boundedness}. We obtain that
		\begin{equation}
			1 \beta + T_{id_M}\beta = dY_1 \beta + Y_1 d\beta
		\end{equation}
		for every $\beta$ in $dom(d_{min})$.
		Then, in $L^2$-cohomology,
		\begin{equation}
			T_{f \circ g} = T_g \circ T_f.
		\end{equation}
		\\
		\\
		\\
		\\
		\\\textbf{Point 6.} In order to prove this statement we have to observe that, since Theorem \ref{tilde}, we have that
		\begin{equation}
			p_f = p_{id} \circ F.
		\end{equation}
		Let us consider a form $\alpha \in \Omega_c^*(N)$: we have that
		\begin{equation}
			\begin{split}
				T_f \alpha &= pr_{M\star} \circ e_{\omega} \circ p_f^* \alpha \\
				&= pr_{M\star} \circ e_{\omega} \circ F^* \circ p_{id}^* \alpha\\
				&= pr_{M\star} \circ F^* \circ e_{\omega} \circ p_{id}^* \alpha.
			\end{split}
		\end{equation}
		Now, using the Proposition VIII of Chapter 5 of \cite{Conn} we have that
		\begin{equation}
			\begin{split}
				T_f \alpha &= pr_{M\star} \circ F^* \circ e_{\omega} \circ p_{id}^* \alpha \\
				&= f^* \circ pr_{N\star} \circ e_{\omega} \circ p_{id}^* \alpha \\
				&= f^* \circ T_{id_N} \alpha \label{impl}
			\end{split}
		\end{equation}
		Since $f$ is R.-N.-lipschitz, we have that $f^*$ is $\mathcal{L}^2$-bounded. This means that (\ref{impl}) implies that
		\begin{equation}
			T_f = f^* \circ T_{id_N}.
		\end{equation}
		Then we have that on $dom(d_{min})$ the following holds
		\begin{equation}
			\begin{split}
				f^* - T_f &= f^* \circ (1 - T_{id_N}) \\
				&= f^* \circ (d \circ Y_1 + Y_1 \circ d).
			\end{split}
		\end{equation}
		We can observe that $f^*(\Omega_c^*(N)) \subseteq \Omega_c^*(M)$ since $f$ is proper and for all smooth form $\alpha$ we have that $f^*d\alpha= d f^* \alpha$. Then we have that on $dom(d_{min})$
		\begin{equation}
			\begin{split}
				f^* - T_f &= f^* \circ (d \circ Y_1 + Y_1 \circ d) \\
				&= d \circ W + W \circ d,
			\end{split}
		\end{equation}
		where
		\begin{equation}
			W = f^* \circ Y_1.
		\end{equation}
		And so in $L^2$-cohomology we have that
		\begin{equation}
			f^* = T_f.
		\end{equation}
		\\
		\\
		\\In order to conclude the proof we have to show that all the identities above also holds in reduced $L^2$-cohomology. To this end it is sufficient to show that if $Q$ is an operator such that on $dom(d)$ 
		\begin{equation}
			Q = d Z + Zd
		\end{equation}
		where $Z$ is a $\mathcal{L}^2$-bounded operator, then, in reduced $L^2$ cohomology, $Q$ is the null operator.
		\\Consider a differential form $\alpha + \lim\limits_{k \rightarrow + \infty} d\beta_k$ on $ker(d)$. We have that
		\begin{equation}
			\begin{split}
				d Z + Zd (\alpha + \lim\limits_{k \rightarrow + \infty} d\beta_k) &= d Z + Zd ( \lim\limits_{k \rightarrow + \infty} \alpha +  d\beta_k)\\
				&= \lim\limits_{k \rightarrow + \infty}  	d Z + Zd (\alpha +  d\beta_k)\\
				&= \lim\limits_{k \rightarrow + \infty} dZ(\alpha + d \beta_k) \in \overline{im(d)}.
			\end{split}
		\end{equation}
		This means that on reduced $L^2$-cohomology $Q = d Z + Zd$ is the null operator.
	\end{proof}
	\begin{rem}
		Consider the category $\mathcal{B}$ which has manifolds of bounded geometry as objects and uniformly proper R.-N.-lipschitz maps as arrows and let us recall the category $\mathcal{C}$, the category of manifolds of bounded geometry with uniformly proper lipschitz maps as arrows.
		\\We can define the functor $\mathcal{I}: \mathcal{B} \longrightarrow  \mathcal{C}$, where  defined as
		\begin{equation}
			\begin{cases}
				\mathcal{I}(M,g) = (M,g) \\
				\mathcal{I}((M,g) \xrightarrow{f} (N,h)) = (M,g) \xrightarrow{f} (N,h).
			\end{cases}
		\end{equation}
		Consider moreover, for all $z \in \numberset{N}$, the functor $\mathcal{G}_z: \mathcal{B} \longrightarrow \textbf{Vec}$ defined as
		\begin{equation}
			\begin{cases}
				\mathcal{G}_z(M,g) = H^z_2(M) \\
				\mathcal{G}_z((M,g) \xrightarrow{f} (N,h)) = H^z_2(N) \xrightarrow{f^*} H^z_2(M).
			\end{cases}
		\end{equation}
		As consequence of the last point we have that
		\begin{equation}
			\mathcal{G}_z = \mathcal{F}_z \circ \mathcal{I}.
		\end{equation}
		All this also holds if we replace $H^z_2$ with $\overline{H}^z_2$.
	\end{rem} 
	\begin{cor}
		Let $(M,g)$ and $(N,h)$ be two manifolds of bounded geometry. Let $f:(M,g) \longrightarrow (N,h)$ be a lipschitz-homotopy equivalence. Then $T_f$ induces an isomorphism in (reduce or not) $L^2$-cohomology.
	\end{cor}
	\begin{proof}
		We will prove this fact for the unreduced $L^2$-cohomology. The reduced case can be proved exactly in the same way.
		\\
		\\Observe that if $f$ is a lipschitz-homotopy equivalence: then it is a uniformly proper lipschitz map. For the same reason also its homotopic-inverse $g$ is a uniformly proper lipschitz map. Moreover since $g \circ f$ is lipschitz-homotopic to $id_M$ then, using Lemma 7.2, we also have that $g \circ f$ is uniformly proper lipschitz homotopic to $id_M$. Then we have that in $L^2$-cohomology.
		\begin{equation}
			1_{H^*_2(M)} = T_{id_M} = T_{g \circ f} = T_f \circ T_g.
		\end{equation}
		Using the same arguments we also have that
		\begin{equation}
			1_{H^*_2(N)} = T_g \circ T_f.
		\end{equation}
	\end{proof}
	
	\chapter{Roe index and $\rho_f$-class}
	\section{Coarse geometry}
	\subsection{Coarse structures}
	In this section we introduce the notions of coarse structure, the Roe and the structure algebras of a coarse space. In particular we introduce the coarse \textit{metric} structure and a particular coarse structure defined on the disjoint union of Riemannian manifolds. Finally we conclude this section defining the Roe Index of the signature operator of an orientable Riemannian manifold.
	\\The next definitions can be found in \cite{Anal}. Let $X$ be a set
	\begin{defn}
		A \textbf{coarse structure} over $X$ is a provision, for each set $S$ of an equivalence relation on the set of maps form $S$ to $X$. If $p_1,p_2: S \rightarrow X$ are in relation then they are said to be \textbf{close} and it is also required that if 
		\begin{itemize}
			\item if $p_1$ and $p_2: S \longrightarrow X$ are close and $q:S' \longrightarrow S$ is another map then $p_1 \circ q$ and $p_2 \circ q$ are close too.
			\item If $S= S_1 \cup S_2$, $p_1, p_2:S\longrightarrow X$ are maps whose restrictions to $S_1$ and $S_2$ are close, then $p_1$ and $p_2$ are close,
			\item two constant maps are always close to each other.
		\end{itemize}
		If $X$ has a coarse structure, then is a \textbf{coarse space}.
	\end{defn}
	\begin{exem}\label{bru}
		When $(X,d_X)$ is a metric space, one can define a \textit{metric} coarse structure in the following way: let $p_1$ e $p_2: S \longrightarrow X$ be two functions, then
		\begin{equation}
			p_1 \sim p_2 \iff \exists C \in \numberset{R} \mbox{           such that           }\forall x \in S \mbox{                         }d_X(p_1(x), p_2(x)) \leq C.
		\end{equation}
		If the domain of $p_1$ and $p_2$ is a metric space $(Y, d_Y)$ and if $\Gamma$ is a group which acts on $X$ and on $Y$, then
		\begin{equation}
			p_1 \sim_{\Gamma} p_2 \implies p_1 \sim p_2.
		\end{equation}
		Indeed
		\begin{equation}
			d_X(p_1(x), p_2(x)) = d(h(x,0), h(x,1)) \leq C_H,
		\end{equation}
		where $C_H$  is the lipschitz constant of the homotopy equivalence $h$.
	\end{exem}
	\begin{exem}\label{coarseexem}
		Consider now a Riemannian manifold $X := M \sqcup N$ where $M$ and $N$ are two connected Riemannian manifolds and let $f:M \longrightarrow N$ be a lipschitz-homotopy equivalence.
		\\In this case $X$ is not a metric space in a natural way and so it has not \textit{a priori} a metric coarse structure. We will define a coarse structure using the map $f$.
		\\Let $p_1$ e $p_2: S \longrightarrow X$ be two maps, then
		\begin{equation}
			p_1 \sim p_2 \iff \exists C \in \numberset{R} \forall x \in S s.t. d_N((f \sqcup id_N)\circ p_1(x), (f \sqcup id_N)\circ p_2(x)) \leq C.
		\end{equation}
	\end{exem}
	\begin{defn}
		Let $X$ be a coarse space and let $S \subseteq X\times X$. Then $S$ is \textbf{controlled} if the projections $\pi_1, \pi_2: S \subseteq X$ are close. A family of subsets $\mathcal{U}$ of $X$ is \textbf{uniformly bounded} if $\cup_{U \in \mathcal{U}}U \times U$ is controlled.
		\\A subset $B$ of $X$ is \textbf{bounded} if $B\times B$ is controlled.
	\end{defn}
	\begin{defn}
		Let $X$ be a locally compact topological space. A coarse structure on $X$ is \textbf{proper} if
		\begin{itemize}
			\item $X$ has an uniformly bounded open covering,
			\item every bounded subset of $X$ has compact closure.
		\end{itemize}
	\end{defn}
	\begin{defn}
		A coarse space $X$ is \textbf{separable} if $X$ admits a  countable, uniformly bounded, open covering.
	\end{defn}
	\begin{rem}
		Given a metric space $(X, d_X)$, then the metric coarse structure is proper and separable. Moreover if $M$ and $N$ are two connected Riemannian manifolds and $f:M \longrightarrow N$ is a lipschitz-homotopy equivalence, then also the coarse structure defined in the Example \ref{coarseexem} is proper and separable.
	\end{rem}
	\subsection{Roe and Structure Algebras}
	Consider a locally compact topological space endowed with a proper and separable coarse structure. Moreover, let $\rho: C_0(X) \longrightarrow \numberset{B}(H)$ be a representation of $C_0(X)$, where $H$ is a separable Hilbert space. Then
	\begin{defn}
		The \textbf{support} of an element $v$ in $H$ is the complement in $X$ of all open sets $U$ such that $\rho(f)v = 0$ for all $f$ in $C_0(U)$.
	\end{defn}
	\begin{defn}
		The \textbf{support} of an operator $T$ in $\numberset{B}(H)$ is the complement in $X\times X$ of the union of $U\times V$ such that
		\begin{equation}
			\rho(f)T\rho(g) = 0,
		\end{equation}
		for all $f$ in $C_0(U)$ and for all $g$ in $C_0(V)$. An operator is \textbf{controlled} or \textbf{finite propagation} if its support is a controlled set.
	\end{defn}
	\begin{exem}
		Consider $(X,g)$ a Riemannian manifolds. Let $H_X$ be Hilbert space $H_X := \mathcal{L}^2(X)$ and let $\rho$ be the representation of $C_0(X)$ on $H_X$ given by the point-wise multiplication. With respect to the metric coarse structure, an operator $T$ is controlled if and only if
		\begin{equation}
			\exists C \in \numberset{R} \forall \phi, \psi \in C_0(X) s.t. d_X(supp(\phi), supp(\psi)) \geq C \implies \phi T \psi = 0.
		\end{equation}
	\end{exem}
	\begin{exem}
		Consider $M$ and $N$ as in Example \ref{coarseexem} and let us impose $H_X := \mathcal{L}^2(M\sqcup N)$. Fix the point-wise multiplication  as representation of $C_0(X)$ on $H_X$. Then an operator $T$ is \textit{controlled} if there is a number $C > 0$ such that for each $\phi, \psi \in C_0(X)$ we have that
		\begin{equation}
			d_N((f \sqcup id_N)(supp(\phi)), (f \sqcup id_N)(supp(\psi))) \geq  C \implies \phi T \psi = 0.
		\end{equation}
	\end{exem}
	The next two definition can be found in \cite{siegel}.
	\begin{defn}
		Let $X$ be a coarse space. Consider $H$ be a Hilbert space equipped with a representation $\rho: C_0(X) \longrightarrow \numberset{B}(H)$, a group $\Gamma$ and a unitary representation $U : \Gamma \longrightarrow \numberset{B}(H)$. We say that the the triple $(H, U, \rho)$ is a \textbf{$\Gamma$-equivariant $X$-module} or simply a \textbf{$(X, \Gamma)$-module} if 
		\begin{equation}
			U(\gamma)\circ \rho(f) = \rho(\gamma^*(f)) \circ U(\gamma),
		\end{equation}
		for every $\gamma \in \Gamma$, $f \in C_0(X)$.
	\end{defn}
	\begin{defn}
		A represention $\rho$ of a group $\Gamma$ in $\numberset{B}(H)$ for some Hilbert space $H$ is \textbf{nondegenerate} if the set
		\begin{equation}
			\{\rho(\gamma)h \in H | \gamma \in \Gamma, h \in H\}
		\end{equation}
		is dense in $H$.
	\end{defn}
	\begin{defn}
		A representation $\rho$ of a group $\Gamma$ in $\numberset{B}(H)$ for some Hilbert space $H$ is \textbf{ample} if it is nondegenerate and $\rho(g)$ is a compact operator if and only if $g = 0$.
		\\If the Hilbert space $H$ is separable and the representation $\rho$ is the countable direct sum of a fixed ample representation, then $\rho$ is said to be \textbf{very ample}.
	\end{defn} 
	\begin{exem}\label{perna}
		Let us consider a Riemannian manifold $X$ with $dim(X) > 0$ and consider the Hilbert space $H := \mathcal{L}^2(X)$. Let us define the representation $\rho_X: C_0^\infty(X) \longrightarrow B(\mathcal{L}^2(X))$ for each $\phi$ in $C_0^\infty(X)$ as 
		\begin{equation}
			\rho_X(\phi)(\alpha) := \phi \cdot \alpha
		\end{equation}
		where $\alpha$ is a differential form $\alpha$ in $\mathcal{L}^2(X)$. Observe that, because of Hadamard-Schwartz inequality \cite{Sbordone}, $\rho_X(\phi)$ is a $\mathcal{L}^2$-bounded operator. In particular we have that $\rho$ is an ample representation of $C_0^\infty(X)$.
		\\Let us consider a subgroup of the isometries of $X$ called $\Gamma$. Then, the representation
		\begin{equation}
			\begin{split}
				U : \Gamma &\longrightarrow \numberset{B}(\mathcal{L}^2(X))\\
				\gamma &\longrightarrow \gamma^*.
			\end{split}
		\end{equation}
		is well defined: as consequence of Proposition \ref{cosa} and Remark \ref{oss}, we have that $\gamma$ is a R.N.-lipschitz map and so $\gamma^*$ is a bounded operator.
		\\So the triple $(\mathcal{L}^2(X), U, \rho)$ is a $(X, \Gamma)$-module.
	\end{exem}
	\begin{exem}\label{serna}
		Consider $(X,g)$ an oriented Riemannian manifold (possibly non connected). Consider the Hilbert space $H_X := \mathcal{L}^2(X) \otimes l^2(\numberset{N})$. We have that an element in $H_X$ can be seen as sequence of $\{\alpha_i\}$ where $\alpha_i \in \mathcal{L}^2(X)$ such that
		\begin{equation}
			\sum_{n \in \numberset{N}} ||\alpha_i||^2 < + \infty.
		\end{equation}
		Consider, moreover, $\rho_X: C^{\infty}_{0}(X) \longrightarrow B(H_X)$ for each $\phi$ in $C^{\infty}_{0}(X)$ and for each $\{\alpha_i\}$ as
		\begin{equation}
			\rho_X(\phi)(\{\alpha_i\}) := \{\phi \cdot \alpha_i\}.
		\end{equation}
		Consider $\Gamma$ as a subgroup of isometries of $X$. Then we can define a representation $U_X: \Gamma \longrightarrow B(H_X)$ as
		\begin{equation}
			U_X(\gamma)(\{\alpha_i\}) := \{\gamma^* \alpha_i\}.
		\end{equation}
		Then $(H_X, \rho_X, U_X)$ defined in this way is a very ample $(X, \Gamma)$-module.
	\end{exem}
	\begin{exem}\label{terna}
		Let $(X,g)$ be an oriented Riemannian manifold such that $dim(X)$ is even. Let us consider the chirality operator $\tau$.
		We have that $\tau$ defines an involution $\tau_p: \Lambda^*_p(X) \longrightarrow \Lambda^*_p(X)$ in each point $p$ of $X$. Let us denote by $V_{\pm}$ as the bundles whose fibers are the $\pm1$-eigenspaces of $\tau_p$.
		\\Let us define the Hilbert space $\mathcal{L}^2(V_{\pm})$ as the closure of the space of compactly supported section $\Gamma_c(M, V_{\pm})$ respect to the bundle metric induced by the Riemannian metric $g$.
		\\We obtain an orthogonal splitting 
		\begin{equation}
			\mathcal{L}^2(M) = \mathcal{L}^2(V_{+}) \oplus \mathcal{L}^2(V_{-}).
		\end{equation}
		\\Let us define the vector space
		\begin{equation}
			H_X := \bigoplus_{i \in \numberset{Z}_{<0}} \mathcal{L}^2(V_{-}) \oplus \bigoplus_{i \in \numberset{N}} \mathcal{L}^2(V_{+}).
		\end{equation}
		An element in $H_X$ is a sequence $\{\alpha_i\}$ indexed by $i \in \numberset{Z}$ of $\alpha_i \in \mathcal{L}^2(V_{-})$ if $i < 0$ and $\alpha_i \in \mathcal{L}^2(V_{+})$ if $i > 0$ such that
		\begin{equation}
			\sum_{i \in \numberset{Z}} ||\alpha_i||_{\mathcal{L}^2(X)}^{2} < + \infty.
		\end{equation}
		Consider, moreover the representation $\rho_X: C^\infty_0(X) \longrightarrow B(H_X)$ defined for each $\phi$ in $C^{\infty}_0(X)$ as
		\begin{equation}
			\rho_X(\phi)(\{\alpha_i\}) := \{\phi \cdot \alpha_i\}
		\end{equation}
		and, given a subgroup $\Gamma$ of isometries of $X$, let $U: \Gamma \longrightarrow B(H_X)$ be
		\begin{equation}
			U_X(\gamma)(\{\alpha_i\}) := \{\gamma^* \alpha_i\}.
		\end{equation}
		Then we have that $(H_X, \rho_X, U_X)$ is a very ample $(X, \Gamma)$-module.
	\end{exem}
	\begin{defn}
		Let us consider a coarse space $X$ and let $(H, U, \rho)$ be a $\Gamma$-equivariant $X$-module. Then an operator $T$ in $\numberset{B}(H)$ is  \textbf{pseudo-local} if for all $f$ in $C_0(X)$ we have that $[\rho (f), T]$ is a compact operator.
	\end{defn}
	\begin{defn}
		An operator $T \in B(H)$ is \textbf{locally compact} if $T\rho(f)$ and $\rho(f)T$ are compact operators for all $f \in C_0(X)$.
	\end{defn}
	Fix a coarse space $X$ and a $\Gamma$-equivariant $X$-module $(H, U, \rho)$. We can define the following algebras.
	\begin{defn}
		The algebra $D^{\star}_{c, \rho}(X,H)$ is given by
		\begin{equation}
			\{T \in B(H)| T \textit{                    is pseudo-local and controlled               }\}.
		\end{equation}
	\end{defn}
	\begin{defn}
		We denote by $C^{\star}_{c, \rho}(X)$ the algebra
		\begin{equation}
			\{T \in D^{\star}_{c, \rho}(X.H)| T \textit{                    is locally compact               }\}.
		\end{equation}
	\end{defn}
	\begin{defn}
		Let $X$ be a coarse space, $(H, U, \rho)$ a $\Gamma$-equivariant $X$-module. Then we will denote by $D^{\star}_{c,\rho}(X,H)^{\Gamma}$ and $C^{\star}_{c, \rho}(X,H)^{\Gamma}$ the operators of $D^{\star}_{c,\rho}(X,H)$ and $C^{\star}_{c, \rho}(X,H)$ which commute with the action of $\Gamma$ on $H$.
	\end{defn}
	\begin{defn}
		The $C^*$-algebras $D^{\star}_\rho(X,H)$, $C^{\star}_\rho(X,H)$, $D^{\star}_\rho(X,H)^{\Gamma}$ and $C^{\star}_\rho(X,H)^{\Gamma}$ are the closure in $B(H)$, of $D^{\star}_{c,\rho}(X,H)$, $C^{\star}_{c,\rho}(X,H)$, $D^{\star}_{c,\rho}(X,H)^{\Gamma}$ e $C^{\star}_{c,\rho}(X,H)^{\Gamma}$. The algebra $D^{\star}_{\rho}(X,H)^{\Gamma}$ will be called \textbf{structure algebra} and $C^{\star}_{\rho}(X,H)^{\Gamma}$ will be called \textbf{coarse algebra} or \textbf{Roe algebra}.
	\end{defn}
	We have that $C^*_{\rho}(X,H)^{\Gamma}$ is an ideal of $D^*_{\rho}(X,H)^{\Gamma}$ and so the following sequence
	\begin{equation}
		{ 0 \longrightarrow C^{\star}_{\rho}(X,H)^\Gamma \longrightarrow D^{\star}_{\rho}(X,H)^{\Gamma} \longrightarrow \frac{D^{\star}_{\rho}(X,H)^{\Gamma}}{C^{\star}_{\rho}(X,H)^\Gamma} \longrightarrow 0 }
	\end{equation}
	is exact.
	\begin{rem}\label{boy}
		Let us consider a Riemannian manifold $(X,g)$ and let $\Gamma$ be a subgroup of isometries of $X$. In particular consider on $(X,g)$ the coarse metric structure if $X$ is connected. If $X = M \sqcup N$ and there is a lipschitz map $f:M \longrightarrow N$ consider the coarse structure defined in Example \ref{coarseexem}. Let us consider the $(X, \Gamma)$-modules defined in Exemple \ref{perna} and \ref{serna}. We can observe that $B(\mathcal{L}^2(X))$ can be embedded in $B(H_X)$ in the following way: for each $A$ in $B(\mathcal{L}^2(X))$, we define $\tilde{A} \in B(H_X)$ as follow
		\begin{equation}
			\tilde{A}(\{\alpha_i\}) := \{\beta_j\}
		\end{equation}
		where $\beta_j = 0$ if $j\neq 0$ and $\beta_0 := A \alpha_0$.
		\\Observe that if $A$ is in $C^*_{\rho}(X, \mathcal{L}^2(X))^\Gamma$ then $\tilde{A}$ is in $C^*_{\rho_X}(X, H_X)^\Gamma$ and that if $A$ is in $D^*_{\rho}(X, \mathcal{L}^2(X))^\Gamma$ then $\tilde{A}$ is in $D^*_{\rho_X}(X, H_X)^\Gamma$. In the following sections, with a little abuse of notation, we will denote $\tilde{A}$ by $A$.
	\end{rem}
	\begin{rem}\label{girl}
		Let $(X,g)$ be a Riemannian manifold and $\Gamma$ be a subgroup of isometries of $X$. Let us suppose $dim(X)$ is even. Again if $X$ is connected we consider the metric coarse structure on $X$,  if $X = M \sqcup N$ and there is a lipschitz-homotopy equivalence $f:M \longrightarrow N$ we consider the coarse structure defined in Example \ref{coarseexem}. \\Fix on $X$ the $(X, \Gamma)$-modules defined in Example \ref{perna} and in Example \ref{terna}. 
		\\Let $A$ be an operator in $B(\mathcal{L}^2(X))$ such that $A(\mathcal{L}^2(V_+)) \subseteq \mathcal{L}^2(V_-)$. Let us define the operator $\tilde{A}$ as the operator
		\begin{equation}
			\tilde{A}(\{\alpha_i\}) := \{\beta_{j}\}
		\end{equation}
		where $\beta_j := \alpha_{j+1}$ if $j \neq -1$ and $\beta_{-1} := A(\alpha_0)$. Then if $A$ is in $D^*_{\rho}(X, \mathcal{L}^2(X))^\Gamma$ then $\tilde{A}$ is in $D^*_{\rho_X}(X, H_X)^\Gamma$. Moreover we also have that if $A$ and $B$ are operators in $D^*_{\rho}(X, \mathcal{L}^2(X))^\Gamma$ such that $A - B$ are in $C^*_{\rho}(X, \mathcal{L}^2(X))^\Gamma$, then we have that also $\tilde{A} - \tilde{B}$, which is the operator
		\begin{equation}
			[\tilde{A} - \tilde{B}](\{\alpha_i\}) := \{\beta_j\}
		\end{equation}
		where $\beta_j = 0$ if $j \neq -1$ and $\beta_{-1} = (A-B)\alpha_0$, is in $C^*_{\rho_X}(X, H_X)^\Gamma$. 
		\\Moreover we also have that
		\begin{equation}
			||\tilde{A} - \tilde{B}|| = ||A - B||.
		\end{equation}
		In the following sections, with a little abuse of notation, we will denote $\tilde{A}$ with $A$.
	\end{rem}	
	\begin{notation}
		Given a Riemannian manifold $(X,g)$, if we write $C^*(X)^\Gamma$ or $D^*(X)^\Gamma$ without specify $H$ and $\rho$, we are considering $H$, $\rho$ and $\Gamma$ as in Example \ref{perna}.
		\\
		\\Moreover, if $X = M \sqcup N$, as in Example \ref{coarseexem}, then the coarse structure depends on $f$. Thus, for this reason, we will denote its algebras by $C^*_f(M \sqcup N)^\Gamma$ and by $D_f(M \sqcup N)^\Gamma$.
	\end{notation}
	\subsection{Coarse maps}
	Let us introduce a class of maps between coarse spaces. The following definitions and properties can be found in \cite{Anal} in Chapter 6.  
	\begin{defn}
		Let $X_1$ and $X_2$ be coarse spaces. A function $q:X_1 \longrightarrow X_2$ is called a \textbf{coarse map} if
		\begin{itemize}
			\item whenever $p$ and $p'$ are close maps into $X_1$, then so are the composition $q \circ p$ and $q \circ p'$,
			\item for every bounded (in the coarse sense of Definition 4.2) subset $B \subseteq X_2$ we have that $q^{-1}(B)$ is a bounded subset.
		\end{itemize}
	\end{defn}
	\begin{rem}
		If $X_i$ are complete metric space with the metric coarse structure, then lipschitz-homotopy equivalences and, more generally, uniformly proper lipschitz maps are coarse maps. Consider, indeed, a function $f: (X_1, d_2) \longrightarrow (X_2, d_2)$ be a lipschitz-homotopy equivalence. Let $p$ and $p': (Y, d_Y) \longrightarrow (X_1, d_1)$ be close maps, i.e.
		\begin{equation}
			d_1(p(y), p'(y)) \leq C
		\end{equation} 
		then also $f \circ p$ and $f \circ p'$ are close, indeed
		\begin{equation}
			d_2(f(p(y)), f(p'(y))) \leq C_f \cdot d_1(p(y), p'(y)) \leq C_f \cdot C,
		\end{equation}
		where $C_f$ is the lipschitz constant of $f$. Moreover recall that if $f$ is a lipschitz-homotopy equivalence, then, in particular, it is also uniformly proper map, i.e. there is a function $\alpha : \numberset{R} \longrightarrow \numberset{R}$ such that for each bounded subset $B$
		\begin{equation}
			diam(f^{-1}(B)) \leq \alpha(diam(B)).
		\end{equation}
		This concludes the proof.
	\end{rem}
	Our next step is to introduce a morphism between the structure algebras of coarse spaces related to a coarse map. 
	\begin{defn}
		Let $X$ and $Y$ be proper separable coarse spaces and suppose that $C_0(X)$ and $C_0(Y)$ are non-degeneratly represented on Hilbert spaces $H_X$ and $H_Y$. Consider a coarse map $q: X \longrightarrow Y$. A bounded operator $V: H_X \longrightarrow H_Y$ \textbf{coarsely covers $q$} if the maps $\pi_1$ and $q \circ \pi_2$ from $supp(V) \subseteq X \times Y$ to $Y$ are close.
	\end{defn}
	\begin{rem}\label{check}
		As showed in Remark 6.3.10 of \cite{Anal}, if $f_0$ and $f_1$ are close maps and $V$ coarsely covers $f_0$, then it also coarsely covers $f_1$.
	\end{rem}
	\begin{defn}
		Let $X$ and $Y$ be coarse spaces, let $\rho_X : C_0(X) \longrightarrow \numberset{B}(H_X )$ and $\rho_Y : C_0(Y ) \longrightarrow \numberset{B}_(H_Y )$ be ample representations on separable Hilbert spaces, and let $\phi: U \longrightarrow Y$ be a continuous proper map defined on an open subset $U \subseteq X$. An isometry $V : H_X \longrightarrow H_Y$ \textbf{topologically covers $\phi$} if for every $f \in C_0(Y)$ there is a compact operator $K$ such that in $\numberset{B}(H_Y )$
		\begin{equation}
			\rho_Y(f) = V\rho_X(\phi^*(f))V^* + K.
		\end{equation}
		An isometry $V$ which analytically and topologically covers a map $\phi$ then it \textbf{uniformly covers $\phi$}.
	\end{defn}
	Let us recall the Proposition 2.13 of \cite{siegel}
	\begin{prop}
		Let $H_X$ be an ample $(X, \Gamma)$-module and let $H_Y$ be a very ample $(Y, \Gamma)$-module. Then every equivariant uniform map $\phi: X \longrightarrow Y$ is uniformly covered by an equivariant isometry $V : H_X \longrightarrow H_Y$.
	\end{prop}
	As consequence of this fact we can prove the following Lemma.
	\begin{lem}
		Let $X$ and $Y$ be proper separable coarse spaces and fix $\Gamma$ a group. Let us consider a $\Gamma$-equivariant continuous coarse map $\phi: X \longrightarrow Y$ and consider an equivariant isometry $V : H_X \longrightarrow H_Y$ which uniformly covers $\phi$. Then the map
		\begin{equation}
			\begin{split}
				Ad_V: D_\rho^*(X, H_X)^\Gamma &\longrightarrow D_{\rho}^*(Y, H_Y)^\Gamma \\
				T &\longrightarrow VTV^*
			\end{split}
		\end{equation}
		is well defined, maps $C^*_\rho(X, H_X)^\Gamma$ in $C^*_\rho(Y, H_Y)^\Gamma$. Moreover the induced morphisms between the K-theory groups don't depend on the choice of $V$.
	\end{lem}
	\begin{proof}
		Let us start by proving that $Ad_V(T)$ is in $D^*_\rho(Y, H_Y)^\Gamma$. In Lemma 6.3.11 of \cite{Anal}, the authors prove that if $T$ is a controlled operator, then $Ad_V(T)$ is controlled. Moreover, since $T$ and $V$ are both $\Gamma$-equivariant, also $Ad_V(T)$ is $\Gamma$-equivariant. Then we just have to prove that, given $f$ in $C_0(Y)$, then
		\begin{equation}
			[\rho_Y(f), VTV^*]
		\end{equation}
		is a compact operator. We can observe that since $V$ topologically covers $\phi$, then
		\begin{equation}
			\rho_Y(f) = V\rho_X(\phi^*(f))V^* + K,
		\end{equation}
		where $K$ is a compact operator. This means that
		\begin{equation}
			\begin{split}
				[\rho_Y(f), VTV^*] &= [V\rho_X(\phi^*(f))V^*, VTV^*] + [K, VTV^*]\\ 
				&= V[\phi^*(f), T]V^* + [K, VTV^*],
			\end{split}
		\end{equation}
		which is a compact operator.
		\\
		\\That $Ad_V$ maps $C^*_\rho(X, H_X)^\Gamma$ in $C^*_\rho(Y, H_Y)^\Gamma$ is proved in Lemma 6.3.11 of \cite{Anal}. The $\Gamma$-equivariance, again, follows by the $\Gamma$-equivariance of $V$.
		\\Since Proposition 6.3.12 of \cite{Anal}, we have that the morphisms induced between $K_\star(C^*_\rho(X, H_X)^\Gamma)$ and $K_\star(C^*_\rho(Y, H_Y)^\Gamma)$ don't depend on the choice of $V$. Applying the same arguments used in the proof of Lemma 5.4.2. of \cite{Anal}, we obtain that the same holds for the morphisms between $K_\star(D^*_\rho(X, H_X)^\Gamma)$ and $K_\star(D^*_\rho(Y, H_Y)^\Gamma)$. Let us consider, indeed, a projection (or a unitary) $T$ in $D^*_\rho(X, H_X)^\Gamma$ and consider two isometries $V_1$ and $V_2$ that uniformly cover $f$.
		Consider the matrices 
		\begin{equation}
			\mathcal{V}_1 := \begin{bmatrix}
				V_1 TV_1^* && 0 \\
				0 && 0
			\end{bmatrix}
			\mbox{        and      } \mathcal{V}_2 :=\begin{bmatrix}
				0 && 0 \\
				0 && V_2 TV_2^*
			\end{bmatrix}
		\end{equation}
		if $T$ is a projection and
		\begin{equation}
			\mathcal{V}_1 := \begin{bmatrix}
				V_1 TV_1^* && 0 \\
				0 && 1
			\end{bmatrix}
			\mbox{        and      } \mathcal{V}_2 :=\begin{bmatrix}
				1 && 0 \\
				0 && V_2 TV_2^*
			\end{bmatrix}
		\end{equation}
		if $T$ is a unitary.
		Observe that if we define the matrix
		\begin{equation}
			C := \begin{bmatrix}
				0 && V_1V_2^* \\
				V_2V_1^* && 0
			\end{bmatrix}
			= \begin{bmatrix}
				0 && 1 \\
				1 && 0
			\end{bmatrix}
			\cdot \begin{bmatrix}
				V_1V_2^* && 0 \\
				0 && V_2V_1^*
			\end{bmatrix}
		\end{equation}
		then, we have that
		\begin{equation}
			\mathcal{V}_1 = C \mathcal{V}_2 C^*
		\end{equation}
		and each entry of $C$ is an isometry of $D^*_\rho(Y, H_Y)^\Gamma$. As consequence of Lemma 4.1.10 of \cite{Anal}, we have that there is a continuous curve of unitary elements connecting $C$ and the matrix
		\begin{equation}
			\begin{bmatrix}
				0 && 1 \\
				1 && 0
			\end{bmatrix}.
		\end{equation}
		Let us define the matrix $\mathcal{W}_2 $ as
		\begin{equation}
			\mathcal{W}_2 :=\begin{bmatrix}
				V_2 TV_2^* && 0 \\
				0 && 0
			\end{bmatrix}
		\end{equation}
		if $T$ is a projection and
		\begin{equation}
			\mathcal{W}_2 :=\begin{bmatrix}
				V_2 TV_2^* && 0 \\
				0 && 1
			\end{bmatrix}
		\end{equation}
		if $T$ is a unitary. As consequence of Lemma 4.1.10 there is a continuous curve of projection (or unitary) connecting $\mathcal{V}_1$ and $\mathcal{W}_2$.
		Then we conclude observing that
		\begin{equation}
			[A_{V_1}T] = [\mathcal{V}_1] = [\mathcal{W}_2] = [A_{V_2}T].
		\end{equation}
	\end{proof}
	\begin{defn}
		Consider two coarse spaces $X$ and $Y$ and let $(H_X, U_X, \rho_X)$ and $(H_Y, U_Y, \rho_Y)$ be a $(X, \Gamma)$-module and a $(Y, \Gamma)$-module. Suppose that $(H_X, U_X, \rho_X)$ and $(H_Y, U_Y, \rho_Y)$ are very ample. Then if $f: X \longrightarrow Y$ is a $\Gamma$-equivariant continuous coarse map, then
		\begin{equation}
			f_\star: K_n(C^*(X)^\Gamma) \longrightarrow K_n(C^*(Y)^\Gamma)
		\end{equation}
		is defined as the morphism induced by $Ad_V$ in K-Theory. We will use the same notation to denote the maps induced by $Ad_V$ between the K-theory groups of $D^*(\cdot)^\Gamma$ and $\frac{D^*(\cdot)^\Gamma}{C^*(\cdot)^\Gamma}$.
	\end{defn}
	\begin{rem}
		Because of Example \ref{serna}, we know that if $X$ is a coarse metric space or if $X = M \sqcup N$ and it has the coarse structure defined in Example \ref{coarseexem}, then they admit a very ample $(X, \Gamma)$-module. Then we have that if $f$ is a coarse map between coarse spaces that have one of these two coarse structures, then there is a well-defined map $f_\star$ in K-theory.
	\end{rem}
	\begin{rem}
		If we have two continuous coarse maps $f: X \longrightarrow Y$ and $g:Y \longrightarrow Z$, an isometry $V: H_X \longrightarrow H_Y$ which uniformly covers $f$ and an isometry $W: H_Y \longrightarrow H_Z$ then $W \circ V$ uniformly covers $g \circ f$. This fact implies that
		\begin{equation}
			f \longrightarrow f_\star
		\end{equation}
		respect the functorial properties.
	\end{rem}
	Consider a coarse map $f:X \longrightarrow Y$ between coarse spaces. As proved in Lemma 6.3.11 of \cite{Anal}, to induce a map between the coarse algebras it is sufficient that the isometry $V$ coarsely covers $f$. Moreover, as we said in Remark \ref{check}, if $f$ and $g: X \longrightarrow Y$ are close maps, then $V$ coarsely covers $f$ if and only if it coarsely covers $g$. 
	\\We already know, because Exemple \ref{bru}, that if $f$ and $g:(M,g) \longrightarrow (N,h)$ are two lipschitz-homotopy equivalent maps between connected Riemannian manifolds, then they are close. Then we have
	\begin{equation}\label{ops}
		f \sim_\Gamma g \implies f_\star = g_\star.
	\end{equation}
	where $f_\star, g_{\star}: K_n(C^*(X)^\Gamma) \longrightarrow K_n(C^*(Y)^\Gamma)$.
	\\Then the following Proposition holds.
	\begin{prop}
		Consider $(M,g)$ and $(N,h)$ two Riemannian manifolds. If they are lipschitz-homotopy equivalent, then
		\begin{equation}
			K_n(C^*(M)^\Gamma) \cong K_n(C^*(N)^\Gamma). 
		\end{equation}
	\end{prop}
	\subsection{Roe Index of the signature operator}
	Consider $(M,g)$ a connected, oriented and complete Riemannian manifold.
	\begin{defn}
		Let us denote by $d$ the closed extension in $\mathcal{L}^2(M)$ of exterior derivative operator. The \textbf{signature operator} $D_M: dom(d) \cap dom(d^*) \subset \mathcal{L}^2(M) \longrightarrow \mathcal{L}^2(M)$ is the operator defined as
		\begin{equation}
			D_M := d + d^* = d - \tau d \tau.
		\end{equation}
		if $dim(M)$ is even and as
		\begin{equation}
			D_M := \tau d + d \tau
		\end{equation} 
		if $dim(M)$ is odd. 
	\end{defn}
	Fix a group $\Gamma$ of isometries on $(M,g)$. Our goal, in this subsection, is to define a class in $K_\star(C^*(M)^\Gamma)$ related to $D_M$. We will call this class the \textit{Roe index} of $D_M$.
	\\Observe that $D_M$ is an unbounded, self-adjoint operator. Since all the operators in $C^*(M)^\Gamma$ are bounded, we need to define a bounded operator related to $D_M$. In order to do this we need the notion of chopping function.
	\begin{defn}
		Let $\chi: \numberset{R} \longrightarrow \numberset{R}$ be a smooth map. Then $\chi$ is a \textbf{chopping function} if it is odd, $\lim\limits_{x \to +\infty} \chi = 1$ and $\lim\limits_{x\to -\infty}\chi = -1$.
	\end{defn}
	Then, since $D_M$ is selfadjoint, the operator
	\begin{equation}
		\chi(D_M) \in B(\mathcal{L}^2(M))
	\end{equation}
	is well-defined. In particular, we have that $\chi(D_M)$ is in $D^*(M)^\Gamma$ (see \cite{Roeindex}).
	\\Let us denote by $C_0(\numberset{R})$ the vector space of continuous functions $h: \numberset{R} \longrightarrow \numberset{R}$ such that $\lim\limits_{t \rightarrow \pm \infty} h(t) = 0$. 
	\\Because of Proposition 3.6 \cite{Roeindex}, if $h$ is a function in $C_0(\numberset{R})$, then
	\begin{equation}
		h(D_M) \in C^*(M)^\Gamma.
	\end{equation} 
	Then, if $\chi_1$ and $\chi_2$ are two chopping functions, we have that
	\begin{equation}
		\chi_1 - \chi_2(D_M) \in C^*(M)^\Gamma
	\end{equation}
	This means that, given a chopping function $\chi$, the operator
	\begin{equation}
		\chi(D_M) \in \frac{D^*(M)^\Gamma}{C^*(M)^\Gamma}
	\end{equation}
	doesn't depend on the choice of $\chi$. Moreover, since
	\begin{equation}
		\chi^2 - 1 \in C_0(\numberset{R}),
	\end{equation}
	we also have that $\chi(D_M)$ is an involution of $\frac{D^*(M)^\Gamma}{C^*(M)^\Gamma}$. 
	\\Let us suppose that $dim(M)$ is odd and consider
	\begin{equation}
		\frac{1}{2}(\chi (D_{M})+1) \in \frac{D^{\star}(M)^{\Gamma}}{C^{\star}(M)^\Gamma}.
	\end{equation}
	This is a projection.
	\\
	\\Let us suppose $dim(M)$ is even. We have that $D_M$ anti-commute with the chirality operator $\tau$. Then, considering the orthogonal splitting 
	\begin{equation}
		\mathcal{L}^2(M) = \mathcal{L}^2(V_{+1}) \oplus \mathcal{L}^2(V_{-1}),
	\end{equation}
	we have that $D_M$ can be written as
	\begin{equation}
		D_M = \begin{bmatrix} 0 && D_{M-} \\
			D_{M+} && 0
		\end{bmatrix}.
	\end{equation}
	We have that
	\begin{equation}
		\chi(D_M) = \begin{bmatrix} 0 && \chi(D_{M})_- \\
			\chi(D_{M})_+ && 0
		\end{bmatrix}.
	\end{equation}
	Since $\chi(M)$ is a self-adjoint involution, then it is a unitary operator. The same holds also for $\chi(D_{M})_+$ and $\chi(D_{M})_-$.
	\\Consider $(H_M, \rho_M, U_M)$ the $(M, \Gamma)$-module defined in Example \ref{terna}: we have that
	\begin{equation}
		H_M = ... \mathcal{L}^2(V_{-1}) \oplus \mathcal{L}^2(V_{-1}) \oplus \mathcal{L}^2(V_{+1}) \oplus \mathcal{L}^2(V_{-1}) \oplus \mathcal{L}^2(V_{-1}) ...
	\end{equation}
	and, because of Remark \ref{girl}, we can see $\chi(D_M)_+$ as a bounded operator on $H_M$ defined for each $\{\alpha_j\}$ in $H_M$ as
	\begin{equation}
		\beta_i = \alpha_{i+1}
	\end{equation}
	if $i \neq -1$ and as $\beta_{-1} = \chi(D_{M})_+ (\alpha_0)$. Then we can see $\chi(D_{M})_+$ as an unitary operator in
	\begin{equation}
		\frac{D^*_{\rho_M}(M, H_M)^\Gamma}{C^*_{\rho_M}(M, H_M)^\Gamma}.
	\end{equation}
	\begin{defn}\label{index}
		The \textbf{fundamental class of $D_{M}$} is $[D_{M}] \in K_{n+1}(\frac{D^{\star}(M)^{\Gamma}}{C^{\star}(M)^\Gamma})$ given by
		\begin{equation}
			[D_{M}] := \begin{cases} 
				
				[\frac{1}{2}(\chi (D_{M})+1)] \text{if $n$ is odd,}
				\\ [\chi(D_{M})_+] \text{if $n$ is even.}
				
			\end{cases}
		\end{equation}
	\end{defn}
	\begin{rem}
		The definition of fundamental class in the even case is well-given: we are considering the $(M, \Gamma)$-module defined in Example \ref{terna} (remember that the K-theory groups of Roe algebra, structure algebra and their quotient don't depend on the Hilbert space or the representation).
	\end{rem}
	\begin{defn}
		The \textbf{Roe index} of $D_{M}$ is the class
		\begin{equation}
			{Ind_{Roe}(D_{M}) := \delta[D_{M}]}
		\end{equation}
		in $K_{n}(C^{*}(M)^{\Gamma})$, where $\delta$ the connecting homomorphism in the K-Theory sequence.
	\end{defn}
	\begin{rem}
		The definition of Roe index in the even case is coherent with the definition 12.3.5. given in \cite{Anal}. Indeed as the authors show in the proof of Proposition 12.3.7., our fundamental class is the image of the Kasparov class $[D] \in K_p(M)$ under Paschke duality.
	\end{rem}
	Our main goal, in this Chapter is to prove that given $(M,g)$ and $(N,h)$ two Riemannian manifolds of bounded geometry and given $f: (M,g) \longrightarrow (N,h)$ a lipschitz-homotopy equivalence which preserves the orientation, then we have that
	\begin{equation}
		f_\star [Ind_{Roe}(D_M)] = Ind_{Roe}(D_N).
	\end{equation}
	
	\section{Smoothing operators between manifolds of bounded geometry}
	\subsection{The operator $y$}
	Let us consider two manifold of bounded geometry $(M,g)$ and $(N,h)$. Consider a uniformly proper, lipschitz map $f: (M,g) \longrightarrow (N,h)$. Recall that in Proposition \ref{tilde} we introduced an uniformly proper R.-N.-lipschitz submersion $p_f: (f^*(T^\delta N), g_S) \longrightarrow (N,h)$ such that $p_f(0_{v_p}) =f(p)$. 
	\begin{lem}\label{yformula}
		Let $f:(M,g) \longrightarrow (N,h)$ be a lipschitz-homotopy equivalence between two complete oriented Riemannian manifolds. Let us consider the bundle\footnote{We use the numbers $1$ and $2$ just to distingush the first and the second summands.}
		\begin{equation}
			f^*(T N)_1 \oplus f^*(T N)_2 \longrightarrow M
		\end{equation}
		and consider $\mathcal{B} \subset f^*(T N)_1 \oplus f^*(T N)_2$ given by
		\begin{equation}
			\mathcal{B} =\{(v_{f_1(p)}, v_{f_2(p)}) \in f^*(T N)_1 \oplus f^*(T N)_2 | |v_{f_1(p)}| \leq \delta, |v_{f_2(p)}| \leq \delta \}.
		\end{equation}
		Let us define on $f^*(T N)_1 \oplus f^*(T N)_2$ the Sasaki metric $g_S$ induced by the metric $g$, the bundle metric $f^*h_1 \oplus f^*h_2$ and the connection $f^*\nabla^{N, LC}_1 \oplus \nabla^{N, LC}_2$. Then we consider on $\mathcal{B}$ the metric induced by $g_S$.
		\\Consider for $i= 1,2$ the projections
		\begin{equation}
			pr_i : \mathcal{B} \longrightarrow f^*(T^\delta N)_i
		\end{equation}
		and let $p_{f,i}$ be the maps $p_{f,i} = p_f \circ pr_i$. Then $p_{f,i}^*$ induce a $\mathcal{L}^2$-bounded map and there is a $\mathcal{L}^2$-bounded operator $y_0$ such that
		\begin{equation}
			p_{f,1}^* - p_{f,2}^* = d y_0 + y_0d.
		\end{equation}
		for all smooth forms in $\mathcal{L}^2(N)$.
	\end{lem}
	\begin{proof}
		We can start observing that
		\begin{equation}
			p_{f,i}^* = pr_i^* \circ p_f^*.
		\end{equation}
		It's easy to check that $pr_i$ is a lipschitz submersion with Fiber Volume equal to the volume of a $\delta$-euclidean ball in each point $(p,t)$. Then it is a R.-N.-lipschitz map. Moreover, applying the Remark \ref{uniformlyp} of the previous Chapter, we can prove that $f$ is uniformly proper lipschitz map and so $p_f$ is R.-N.-lipschitz. Then since composition of R.-N.-lipschitz maps is R.-N.-lipschitz, we can conclude that $p_{f,i}^*$ is R.-N.-lipschitz and so its pullback is a $\mathcal{L}^2$-bounded operator.
		\\
		\\To prove the second point we need a lipschitz-homotopy $H$ between $p_{f,0}$ and $p_{f,1}$ such that $H^*$ is a $\mathcal{L}^2$-continuous map:then the assertion will be proved considering
		\begin{equation}
			y_0 := \int_{0,\mathcal{L}}^1 \circ H^*.
		\end{equation}
		Let us define the map $a: B^\delta_1 \times B^\delta_2 \times [0,1] \longrightarrow B^\delta$ as
		\begin{equation}
			a(t_1, t_2, s) = t_1(1-s) + st_2.
		\end{equation}
		where $B^\delta$ and $B^\delta_i$ are euclidean balls of radius $\delta$ in $\numberset{R}^n$.
		\\Now, since $B^\delta$ and $[0,1]$ are compact spaces, we have that $a$ is a lipschitz surjective submersion with bounded Fiber Volume. Indeed we know that the Fiber Volume of a submersion is continuous on the image of the submersion (it is a consequence of Proposition \ref{cosa}). Then $a$ is a R.-N.-lipschitz map.
		\\Let us define the map $A: \mathcal{B} \times [0,1] \longrightarrow f^*T^\delta N$ as
		\begin{equation}
			A(v_{f(p),1}, v_{f(p), 2}, s) := v_{f(p),1}\cdot(1-s) + s\cdot v_{f(p), 2}
		\end{equation}
		and consider the homotopy $H: \mathcal{B} \times [0,1] \longrightarrow N$ defined as
		\begin{equation}
			H(v_{f(p),1}, v_{f(p), 2},s) = p_f \circ A (v_{f(p), 1}, v_{f(p),2},s).
		\end{equation}
		One can easily prove that $A$ is a lipschitz map and the Fiber Volume of $A$ has the same bound of the Fiber Volume of $a$. Then $H$ is an R.N.-lipschitz map because it is composition of R.-N.-lipschitz maps (Proposition \ref{compo2}). This means that $H^*$ is a $\mathcal{L}^2$-bounded operator. So, in particular, since $\int_{0\mathcal{L}}^1$ is a $\mathcal{L}^2$-bounded operator (Lemma \ref{GLee}), also the composition
		\begin{equation}
			y_0 := \int_{0\mathcal{L}}^1 \circ H^*
		\end{equation}
		is a $\mathcal{L}^2$-bounded operator.
		\\Observe that $H^*(\Omega_c^*(N)) \subset (\Omega^*(f^*(T^\delta N)_1 \oplus f^*(T^\delta N)_2 \times [0,1]))$. Since $d$ and $H^*$ commute for all smooth forms in $\mathcal{L}^2$, we can apply Proposition \ref{boundedness} and so the formula
		\begin{equation}
			p_{f,1}^*\alpha - p_{f,2}^*\alpha = d y_0\alpha + y_0d\alpha.
		\end{equation}
		holds for every $\alpha$ in $dom(d)$.
	\end{proof}
	\begin{lem}\label{y and Y}
		Consider an oriented manifold $(N,h)$ of bounded geometry. Consider $p_{id}: TN \longrightarrow N$ the submersion related to the identity map defined in Proposition \ref{tilde}. Consider a Thom form $\omega$ of the bundle $\pi:TN \longrightarrow N$, where $\pi(v_p) = p$, such that $supp(\omega) \subset T^\delta N$. Then for all $q$ in $N$ we have that
		\begin{equation}
			\int_{F_q} \omega = 1
		\end{equation}
		where $F_q$ is the fiber of the submersion $p_{id}:T^\delta N \longrightarrow N$ defined in Proposition \ref{tilde}.
	\end{lem}
	\begin{proof}
		For all $q$ in $N$ the fiber $F_q$ is an oriented compact submanifold with boundary. The same also holds for $B^\delta_q$ which is the fiber of the projection $\pi: T^\delta N \longrightarrow N$ defined as $\pi(v_q) := q$.
		\\Let us consider the map $H:T^\delta N \times [0,1] \longrightarrow N$ defined as
		\begin{equation}
			H(v_p,s) = p_{id}(s\cdot v_p).
		\end{equation}
		Since $H$ is a proper submersion, we have that the fiber along $H$ given by $F_{H,q}$ is submanifold of $T^\delta N \times [0,1]$. Its boundary, in particular is
		\begin{equation}
			\partial F_{H,q} = B^\delta_q \times \{0\} \sqcup F_q \times \{1\} \cup A
		\end{equation}
		where $A$ is contained in
		\begin{equation}
			S^\delta N := \{v_{p} \in TN | |v_p| = \delta\}.
		\end{equation}
		Then we have that, if $\omega$ is a Thom form of $TN$ whose support is contained in $T^\delta N$, then
		\begin{equation}\label{nano}
			0 = \int_{F_{H_q}} d \omega = d \int_{F_{H_q}} \omega + \int_{\partial F_{H_q}} \omega.
		\end{equation}
		We can observe that $\omega$ is a $k$-form and $dim(F_{H_q})= k+1$. Then the first integral on the right side of \ref{nano} is $0$. Moreover, we can observe that $\omega$ is null on $A$, and so
		\begin{equation}
			\int_A \omega = 0.
		\end{equation}
		Then, since \ref{nano}, we have that
		\begin{equation}
			0 = \mp \int_{B^\delta} \omega \pm \int_{F_q}\omega.
		\end{equation}
		and we conclude.
	\end{proof}
	\begin{prop}
		Let $f:(M,g) \longrightarrow (N,h)$ be a lipschitz-homotopy equivalence between two manifolds of bounded geometry which mantains the orientation. Then there is a bounded operator $y$ such that\footnote{We have that if $X$ and $Y$ are two Riemannian manifolds and $A: \mathcal{L}^2(X) \longrightarrow \mathcal{L}^2(Y)$ is a $\mathcal{L}^2$-buonded operator, then $A^\dagger = \tau_{X} \circ A* \circ \tau_Y$.}
		\begin{itemize}
			\item $y(dom(d_{min})) \subseteq dom(d_{min})$,
			\item on $dom(d_{min})$
			\begin{equation}
				1 - T^\dagger_fT_f = d y + y d \label{smo},
			\end{equation}
			\item $y^\dagger = y$.
		\end{itemize}
	\end{prop}
	\begin{proof}
		Let us consider two smooth $\mathcal{L}^2$-forms $\alpha$ and $\beta$ on $N$. First of all we will consider the case in which $\alpha$ and $\beta$ are both $\mathcal{L}^2$-smooth $j$-forms for some natural number $j$. Let us consider the fiber bundle $\mathcal{B} := f^*(T^\delta N)_1 \oplus f^*(T^\delta N)_2$ with the metric defined in Lemma \ref{yformula}. We denote, moreover with $B^\delta_1$ and with $B^\delta$ the fibers of $f^*(T^\delta N)_1$ and $f^*(T^\delta N)_2$.
		\\ Using that $(T_f^\dagger)^* = (\tau T_f^* \tau)^*$ and that $\tau$ is self-adjoint, we obtain
		\begin{equation}
			\begin{split}
				\langle T^\dagger_fT_f \alpha, \beta \rangle &=  \langle T_f \alpha, \tau T_f \tau \beta \rangle \\
				&= \int_{M} (\int_{B^\delta_2}p_{f,2}^*\alpha \wedge \omega_2) \wedge (\int_{B^\delta_1}p_{f,1}^*\tau \beta \wedge \omega_1) \\
				&= \int_{f^*(T^\delta N)_1} (\int_{B^\delta_2}p_{f,2}^*\alpha \wedge \omega_2) \wedge p_{f,1}^*\tau \beta \wedge \omega_1 \\
				&= (-1)^{(n + j)j} \int_{f^*(T^\delta N)_1} p_{f,1}^*\tau \beta \wedge \omega_1 \wedge (\int_{B^\delta_2}p_{f,2}^*\alpha \wedge \omega_2)  \\
				&= (-1)^{(n + j)j} \int_{\mathcal{B}} p_{f,1}^*\tau \beta \wedge \omega_1 \wedge p_{f,2}^*\alpha \wedge \omega_2  \\
				&= (-1)^{(n + j)j}(-1)^{(n + j)j} \int_{\mathcal{B}} p_{f,2}^*\alpha \wedge p_{f,1}^*\tau \beta \wedge \omega_1 \wedge \omega_2  \\
				&= \int_{\mathcal{B}} p_{f,2}^*\alpha \wedge p_{f,1}^*\tau \beta \wedge \omega_1  \wedge \omega_2
			\end{split}
		\end{equation}
		Consider the identity we proved in Lemma \ref{yformula}
		\begin{equation}
			p_{f,2}^* = p_{f,1}^* + dy_0 + y_0 d.
		\end{equation}
		We obtain
		\begin{equation}
			\begin{split}
				\langle T^\dagger_fT_f \alpha, \beta \rangle &= \int_{f^*(TN)_1}p_{f_1}^*(\alpha \wedge \tau \beta) \wedge \omega_1 \\
				&+ \int_{\mathcal{B}} (dy_0 + y_0 d) \alpha \wedge \omega_1 \wedge p_{f,2}^*\tau \beta \wedge \omega_2.
			\end{split}
		\end{equation}
		Since $\alpha\wedge \tau\beta \wedge \omega_1$ is a top-degree form, then its closed and, in particular, the first integral can be written as
		\begin{equation}
			\begin{split}
				\int_{f^*(TN)_1}p_{f_1}^*(\alpha \wedge \tau \beta) \wedge \omega_1 &= \int_{f^*(TN)_1} (f,id_{B^k})^* \circ p_{id}^*(\alpha \wedge \tau \beta) \wedge \omega_1 \\
				&= \int_{TN} p_{id}^*(\alpha \wedge \tau \beta) \wedge \omega_1 \\
				&= \int_N (\int_F \omega_1) \alpha \wedge \tau \beta \\
				&= \int_N \alpha \wedge \tau \beta \\
				&= \langle 1(\alpha), \beta \rangle,
			\end{split}
		\end{equation}
		where $F$ is the fiber of $p_{id}$.
		\\Consider the second integral: let us denote by $pr_1: \mathcal{B} \longrightarrow f^*TN_2$ the projection $(v_{f_1(p)}, v_{f_2(p)}) \rightarrow v_{f_2(p)}$. We have that
		\begin{equation}
			\begin{split}
				&\int_{\mathcal{B}} (dy_0 + y_0 d) \alpha \wedge p_{f,2}^*\tau \beta \wedge \omega_1 \wedge \omega_2\\
				&= (-1)^{n(n-j)} \int_{f^*(TN)_2} (\int_{f^*(TN)_1}(dy_0 + y_0 d)\alpha \wedge \omega_1) \wedge p_{f,2}^*\tau \beta \wedge \omega_2 \\
				&= (-1)^{n(n-j)} \int_N (\int_{F_2}[(\int_{f^*(TN)_1}(dy_0 + y_0 d)\alpha \wedge \omega_1)]\wedge \omega_2 \wedge \tau \beta \\
				&= \langle (d\circ Y \alpha + Y \circ d) \alpha, \beta \rangle
			\end{split}
		\end{equation}
		where
		\begin{equation}
			\begin{split}
				Y :&= (-1)^{n(n-deg(\cdot))} p_{f,2\star} \circ e_{\omega_2} \circ pr_{1,\star} \circ e_{\omega_1} \circ y_0 \\
				&= (-1)^{n(n-deg(\cdot))} p_{f,2\star} \circ e_{\omega_2} \circ pr_{1,\star} \circ e_{\omega_1} \circ \int_{0, \mathcal{L}}^1 \circ H^*.
			\end{split}
		\end{equation}
		Observe that $p_1$, respect to the Sasaki metric on $\mathcal{B}$, is a R.-N.-lipschitz map: the lipschitz constant is $1$ and the Fiber Volume is the volume of an euclidean ball with radius $\delta$. Since Proposition \ref{mai2}, we have that $p_{1,\star}$ is a $\mathcal{L}^2$-bounded operator. Moreover we also have that $p_{f,2\star}$, $e_{\omega_2}$, $e_{\omega_1}$ are $\mathcal{L}^2$-bounded. Then, using Lemma \ref{yformula}, we conclude that $Y$ is a $\mathcal{L}^2$-bounded operator.  
		\\Consider now $y := \frac{Y + Y^\dagger}{2}$. Then we have that $y$ is an $\mathcal{L}^2$-bounded operator. Observe that since $Y$, on $\Omega^*(N) \cap \mathcal{L}^2(N)$ satisfies the equation (\ref{smo}), then also $y$ satisfies it. Moreover we have that $[-1 + dy + yd]\alpha$ is a $j$-form if $\alpha$ is a $j$-form.
		\\Let us assume that $\alpha$ is a $j$-form and $\beta$ is a $r$-form with $j \neq r$, we have that
		\begin{equation}
			\langle T_f^\dagger T_f \alpha, \beta \rangle = \langle (-1 + dy + yd)(\alpha), \beta \rangle = 0.
		\end{equation}
		Now we just have to show that $y(dom(d_{min})) \subset dom(d_{min})$: using this fact the proof of the equation (\ref{smooth}) for every $\alpha$ in $dom(d_{min})$ will be immediate.
		\\
		\\First we observe that if $\beta$ is a differential form in $\Omega^*_c(N)$, then $Y \beta$ is in $\Omega^*_c(M)$. Indeed if $K_\beta$ is the compact support of $\beta$ and $pr_{\mathcal{B}}: \mathcal{B}\times [0,1] \longrightarrow \mathcal{B}$ is the projection on the first component, then 
		\begin{equation}
			supp(\int_{0, \mathcal{L}}^1 \circ H^* \beta) \subseteq pr_{M \times B^k_1 \times B^k_2}(H^{-1}(K_\beta)) =: K_{\beta}'
		\end{equation}
		which is compact since $H$ is proper. Then, if $pr_2: \mathcal{B} \longrightarrow f^*(T^\delta N)_2$ is the projection on the second component,
		\begin{equation}
			supp(pr_{1\star} \circ e_{\omega_1} \int_{0, \mathcal{L}}^1 \circ H^* \beta) \subseteq pr_{2} (K_{\beta}') =: K_{\beta}''
		\end{equation}
		which is compact because $K_{\beta}'$ is compact and $pr_{2}$ is continuous. Finally,
		\begin{equation}
			supp(Y \beta) \subseteq p_{f,2}(K_{\beta}'' \cap supp(\omega_2)).
		\end{equation}
		which is compact because $supp(\omega_2)$ is closed and $p_{f,2}$ is continuous. This means that $supp(Y \beta)$ is a closed subset of a compact set and so it is compact. We know that $Y\beta$ is smooth because $Y$ is a composition of operators which preserve the smoothness of differential forms.
		\\
		\\Moreover if $\beta$ is a differential form in $\Omega^*_c(M)$, then we also have that $Y^\dagger \beta$ is in $\Omega^*_c(N)$. Indeed we have that
		\begin{equation}
			Y^\dagger \beta = \pm(-1)^{n(n-deg(\beta))} H_\star \circ pr_{\mathcal{B}}^* \circ e_{\omega_1} \circ pr_{2}^* \circ e_{\omega_2} \circ p_{f,2}^* \beta
		\end{equation}
		Then we can prove, as we did for $Y$, that $Y^\dagger \beta$ is in $\Omega^*_c(N)$.
		\\
		\\Consider $\alpha$ in $dom(d_{min}) \subset \mathcal{L}^2(N)$: we have that 
		\begin{equation}
			\begin{cases}
				\alpha =  \lim\limits_{j \rightarrow +\infty} \alpha_j \\
				d\alpha = \lim\limits_{j \rightarrow +\infty} d \alpha_j
			\end{cases}
		\end{equation}
		where $\alpha_j \in \Omega^*_c(N)$. Then we have, for every $j$, that $Y\alpha_j$ is in $\Omega_c^*(M)$. Moreover,
		\begin{equation}
			Y\alpha = Y \lim\limits_{j \rightarrow +\infty} \alpha_j = \lim\limits_{j \rightarrow +\infty} Y \alpha_j
		\end{equation}
		and
		\begin{equation}
			\begin{split}
				\lim\limits_{j \rightarrow +\infty} dY (\alpha_j) &= \lim\limits_{j \rightarrow +\infty} - Y d(\alpha_j) + 1 \alpha_j + T_f^\dagger T_f \alpha_j\\
				&= - Y d \alpha + \alpha + T_f^\dagger T_f \alpha \in \mathcal{L}^2(M \times B^k_1 \times B^k_2).\label{form}
			\end{split}
		\end{equation}
		and so we have $Y(dom(d_{min}) \subseteq dom(d_{min})$. Using exactly the same argument one can easily check that $Y^\dagger (dom(d_{min}) \subseteq dom(d_{min})$ and $y(dom(d_{min}) \subseteq dom(d_{min})$.
		Finally if in the equation (\ref{form}) we replace $Y$ with $y$ we also prove the equation (\ref{smooth}) for all the elements of $dom(d_{min})$.
	\end{proof}
	\begin{lem}
		Let $(M,g)$ be an orientable Riemannian manifold and let $\Gamma$ be a group of isometries which preserve the orientation. Let $A$ un operator on $\mathcal{L}^2(M)$. If $A$ commute with $\Gamma$, then
		\begin{equation}
			A^* \Gamma = \Gamma A^*
		\end{equation}
		and 
		\begin{equation}
			A^\dagger \Gamma = \Gamma A^\dagger.
		\end{equation}
	\end{lem}
	\begin{proof}
		Let $\gamma$ be an element of $\Gamma$. Then, since $\gamma$ is an isometry which preserves the orientation, we have that
		\begin{equation}
			\gamma^{-1} = \gamma^* = \gamma^\dagger
		\end{equation}
		Then we have that
		\begin{equation}
			A^* \gamma = (\gamma^* A)^* = (\gamma^{-1} A)^* = (A\gamma^{-1})^* = (\gamma^{-1})^* A^* = \gamma A^*
		\end{equation}
		and
		\begin{equation}
			A^\dagger \gamma = (\gamma^\dagger A)^\dagger = (\gamma^{-1} A)^\dagger = (A\gamma^{-1})^\dagger = (\gamma^{-1})^\dagger A^\dagger = \gamma A^\dagger.
		\end{equation}
	\end{proof}
	\begin{cor}
		If there exists a group $\Gamma$ of isometries which commute with $f$, then also $y$ commute with $\gamma$ for all $\gamma$ in $\Gamma$.
	\end{cor}
	\begin{proof}
		One can easily check that the operator $Y$ commute with $\Gamma$: the previous lemma allows us to conclude.
	\end{proof}
	\subsection{$T_f$ as smoothing operator}
	\begin{prop}\label{Tsmooth}
		Consider $f:(M,g) \longrightarrow (N,h)$ a lipschitz-homotopy equivalence between manifolds of bounded geometry. Then the operator $T_f$ is a smoothing operator.
	\end{prop}
	\begin{proof}
		We divide the proof in two steps: in the first one we will show that for all smooth forms $\alpha$ in $\mathcal{L}^2(N)$ we have that
		\begin{equation}
			T_f\alpha(p) = \int_N K(p,q) \alpha(q) d\mu_N
		\end{equation}
		with $K(p,q) \in \Lambda^*(M) \boxtimes \Lambda(N)_{(p,q)}$. In the second one we will prove that the section $K$ is smooth.
		\\
		\\ \textbf{1.} The proof of the first step is similar to the proof given by Vito Zenobi in his Ph.D. thesis \cite{vito}.
		\\Let us denote the following maps as follows
		\begin{itemize}
			\item $p_f:f^*T^\delta N \longrightarrow N$ the submersion related to $f$,
			\item $t_f:f^*T^\delta N \longrightarrow M\times N$ is the map $t:=(\pi, p_f)$ defined in Lemma \ref{svolta},
			\item $\pi:f^*TN \longrightarrow M$ is the projection of the bundle $f^*TN$,
			\item $pr_M$ and $pr_N$ the projections of $M\times N$ over $M$ and $N$.
		\end{itemize}
		We already proved that $t_f$ is a diffeomorphism on its image, so, in particular, it is a submersion and the integration along the fibers of $t_f$ is the pullback along $t_f^{-1}: im(t_f) \longrightarrow f^*T^\delta N$.
		\\Observe that
		\begin{equation}
			p_f = pr_N \circ t_f
		\end{equation}
		and
		\begin{equation}
			\pi = pr_M \circ t_f.
		\end{equation}
		Consider a smooth form $\alpha$ in $\mathcal{L}^2(N)$. Applying the Projection Formula, we obtain that
		\begin{equation}
			\begin{split}
				T_f \alpha &= \pi_\star \circ e_\omega \circ p_f^* \alpha \\
				&= pr_{M, \star} \circ t_{f, \star} \circ e_\omega \circ t_f^* \circ pr_N^* \alpha \\
				&= pr_{M, \star} \circ e_{t_{f, \star} \omega} \circ pr_N^* \alpha,
			\end{split}
		\end{equation}
		where $t_{f, \star} \omega$ is the form $t_f^{-1}\omega$ on $im(t_f) \subset M \times N$ and $0$ in $im(t_f)^c$. Since $supp(\omega)$ is contained in $f^*T^{\delta_1}N$ where $\delta_1 < \delta$, then $t_{f, \star} \omega$ is a well-defined, smooth form.
		\\Observe that the Hodge star operator $\star_N$ on $N$ induces a bundle endomorphism on $\Lambda^*(M \times N)$ imposing that, for each couple of orthonormal frames $\{W^i\}$ of $\Lambda^*M$ and $\{\epsilon^j\}$ of $\Lambda^*N$, we have
		\begin{equation}
			\star_N (pr_M^* W^I \wedge pr_N^* \epsilon ^J (p,q)) := pr_M^* W^I \wedge (\star_N \epsilon ^J) (p,q)
		\end{equation}
		Consider $\Lambda^{\star, k}(M \times N)$ the subbundle of $\Lambda^*(M \times N)$ given by the forms which are $k$-forms with respect to $N$. Observe that
		\begin{equation}
			\Lambda^*(M \times N) = \bigoplus\limits_{k \in \numberset{N}}\Lambda^{\star, k}(M \times N).
		\end{equation}
		Let us denote by $n$ the dimension of $N$. We can define the bundle morphism
		\begin{equation}
			B: \Lambda^*(M \times N) \longrightarrow \Lambda^{\star, 0}(M \times N) = pr_M^*(\Lambda^* M)
		\end{equation}
		which is 
		\begin{equation}
			B\alpha_{(p,q)} := \star_N \alpha_{(p,q)} 
		\end{equation}
		if $\alpha_{(p,q)} \in \Lambda^{\star, n}(M \times N)_{(p,q)}$ and it is zero if $\alpha_{(p,q)}$ is in $\Lambda^{\star, q}(M \times N)_{(p,q)}$ where $q \neq n$.
		\\Then we have that for each
		\begin{equation}\label{piru}
			\begin{split}
				T_f \alpha(p) &= pr_{M, \star} \circ e_{t_{f, \star} \omega} \circ pr_N^* \alpha(p) \\
				&= \int_N B \circ e_{t_{f, \star} \omega} (\alpha)(p,q) d\mu_N
			\end{split}
		\end{equation}
		Let us define, for each $(p,q)$ in $M \times N$ the operators
		\begin{equation}
			\begin{split}
				E_{\omega, p,q}: pr_N^*(\Lambda^*N)_{(p,q)} &\longrightarrow  \Lambda^*(M \times N)_{(p,q)} \\
				\beta_{(p,q)} &\longrightarrow \beta_{(p,q)} \wedge t_{f, \star} \omega(p,q)
			\end{split}
		\end{equation}
		and
		\begin{equation}
			\begin{split}
				B_{p,q}: \Lambda^*(M \times N)_{(p,q)} &\longrightarrow pr_M^*(\Lambda^*M)_{(p,q)}  \\
				\gamma_{(p,q)} &\longrightarrow B\gamma_{(p,q)}.
			\end{split}
		\end{equation}
		Then the equation (\ref{piru}) can be read as
		\begin{equation}
			\begin{split}
				T_f \alpha(p) &= \int_N (B_{p,q} \circ E_{\omega, p,q})pr_N^*\alpha(q) d\mu_N \\
				\int_N K(p,q)\alpha(q) d\mu_N,
			\end{split}
		\end{equation}	   
		where
		\begin{equation}
			K(p,q) := B_{p,q} \circ E_{\omega, p,q}:pr_N^*(\Lambda^*N)_{(p,q)} \longrightarrow pr_M^*(\Lambda^*M)_{(p,q)}
		\end{equation}	   	   
		is, for each $(p,q)$ in $M \times N$ an element of $\Lambda^*M \boxtimes \Lambda TN_{(p,q)}$.	   
		\\
		\\\textbf{2.} 
		To prove that $K$ is a smooth section of $\Lambda^*M \boxtimes \Lambda N$ it is sufficient to show that, in local coordinates $\{x^i, y^j\}$, the kernel $K$ has the form
		\begin{equation}
			K(x,y) = K^J_I(x,y)dx^I \boxtimes \frac{\partial}{\partial y^J}
		\end{equation}
		where $K^J_I$ are smooth functions.
		\\Let us fix some normal coordinates $\{x,y\}$ on $M \times N$. We have that
		\begin{equation}
			t_{f, \star} \omega (x,y) = \beta_{IR}(x,y)dx^I \wedge dy^R
		\end{equation}
		for some smooth functions $\beta_{IR}$ and that $\alpha$ has the form 
		\begin{equation}
			\alpha(y) = \alpha_{J}(y)dy^J.
		\end{equation}
		Let us denote by $\frac{\partial}{\partial x^I} = \frac{\partial}{\partial x^{i_1}} \wedge ... \wedge \frac{\partial}{\partial x^{i_k}}$, the dual of $dx^I$.
		Let $J^c$ be the multindex such that, up to double switches, $(J, J^c)$ is the index $(1, ..., n)$. We have that
		\begin{equation}
			\begin{split}
				K^J_I(x,y) &= K(\frac{\partial}{\partial x^I}, dy^J) \\
				&= [B_{x,y}(\beta_{LR}(x,y)dx^L \wedge dy^R \wedge dy^J)](\frac{\partial}{\partial x^I})\\
				&= [\beta_{LJ^c}(x,y)dx^L](\frac{\partial}{\partial x^I})\\
				&= \beta_{IJ^c}(x,y).
			\end{split}
		\end{equation}
		This fact implies that if $K$ in local coordinates has the form 
		\begin{equation}
			K(x,y) = \beta_{IJ^c}(x,y) dx^I \boxtimes \frac{\partial}{\partial y^J}
		\end{equation}
		and so the kernel $K$ is smooth.
	\end{proof}
	\subsection{Some operators in $C^*_f(M \sqcup N)^\Gamma$}
	\begin{prop}
		The operator $T_f$ is in $C_f^*(M \sqcup N)^\Gamma$.
	\end{prop}
	\begin{proof}
		One can observe that
		\begin{equation}
			supp(T_f) \subseteq \{(x,x') \in M \times N | B^\delta_x \cap p_f^{-1}(x') \neq \emptyset \}
		\end{equation}
		where $B^\delta_x$ is the fiber of $\pi: f^*TN \longrightarrow M$.
		\\So, if $\psi$ and $\phi$ are two compactly supported functions on $N$ and $M$ respectively we have that
		\begin{equation}
			supp(\psi T_f \phi) \subseteq \pi^{-1}(supp(\psi)) \cap p_f^{-1}(supp(\phi)).
		\end{equation}
		And so
		\begin{equation}
			d(\pi^{-1}(supp(\psi)), p_f^{-1}(supp(\phi))) > 0 \implies \psi T_f \phi = 0.
		\end{equation}
		\\
		Now, if $t \in p_f(\pi^{-1}(supp(\psi)))$ then $t = p_f(s,q)$ for some $(s,q)$. Then we have
		\begin{equation}
			d(t, f(supp(\psi))) \leq d(p_f(s,q), p_f(s, 0)) \leq C_{p_{f}}d((s,q), (s, 0)) \leq C_{p_{f}}.
		\end{equation} 
		Let $R > C_{p_{f}}$. We have that
		\begin{equation}
			\begin{split}
				&d(f(supp(\psi)), supp(\phi)) > R \implies\\
				&d(p_f(\pi^{-1}(supp(\psi))), supp(\phi)) > 0 \implies \\
				&d(\pi^{-1}(supp(\psi)), p_f^{-1}supp(\phi)) > 0 \implies \\
				&\psi T_f \phi = 0.
			\end{split}
		\end{equation}
		It means that $T_f$ has propagation less or equal to $C_{p_{f}}$. Moreover if $\phi$ has compact support then $T_f\phi$ and $\phi T_f$ are integral operator with smooth, compactly supported kernel. Then, in particular they are compact.
		\\Since the $\Gamma$-invariance of $T_f$ it follows that $T_f$ is in $C_f^*(M\sqcup N)^\Gamma$.
	\end{proof}
	\begin{cor}
		Since $T_f$ is in $C_f^*(M\sqcup N)^\Gamma$, then also $T_f^\dagger$, $\tau T_f\tau$ and $\tau T_f^{\dagger}\tau$ are in $C_f^*(M\sqcup N)^\Gamma$.
	\end{cor}
	\begin{proof}
		It is sufficient to observe that $\tau$ is in $D_f^*(M \sqcup N)^\Gamma$ and that, since $C_f^*(M\sqcup N)^\Gamma$ is a $C^*$-algebra, if $T_f$ is in $C_f^*(M\sqcup N)^\Gamma$ then $T_f^*$ is in $C_f^*(M\sqcup N)^\Gamma$.
	\end{proof}
	\begin{prop}
		The operator $y$ is a $\Gamma$-equivariant operator with finite propagation. Moreover the operators $T_fy$, $yT_f$, $T_f^\dagger y$ $yT_f^\dagger$ are all operators in $C_f^*(M \sqcup N)^\Gamma$. 
	\end{prop}
	\begin{proof}
		Let us consider the operator $Y$ in Lemma \ref{y and Y}. If we show that $Y$ is an operator with finite propagation, then one can easily check that the same holds for $Y^\dagger$ and for $y:= \frac{Y + Y^\dagger}{2}$.
		\\
		\\First, we have to observe that, up to signs,
		\begin{equation}
			Y:= p_{f,2\star} \circ e_{\omega_2} \circ pr_{1,\star} \circ e_{\omega_1} \circ \int_{0, \mathcal{L}}^1 \circ H^*.
		\end{equation}
		where $H$ is the homotopy defined in Lemma \ref{yformula}. Then its pullbcak is given by
		\begin{equation}
			H^* = A^* \circ pr_1^* \circ p_{f,2}^*.
		\end{equation}
		Let $\alpha$ be a compactly supported differential forms on $N$. 
		\\Since the formula in Remark \ref{oss}, we have that
		\begin{equation}
			supp(p_{f,2}^*\alpha) = p_{f,2}^{-1}(supp(\alpha)) \subseteq \pi_2^{-1}(f^{-1}(B_1(supp(\alpha)))),
		\end{equation}
		where $\pi_2: f^*(T^\delta N)_2 \longrightarrow M$ is the projection of the fiber bundle and where $B_1(supp(\alpha))$ is the $1$-neighborhood of $supp(\alpha)$.
		\\This means that
		\begin{equation}
			supp(pr_{1,\star} \circ e_{\omega_2} \circ e_{\omega_1} \circ \int_{0\mathcal{L}} \circ H^*(\alpha)) \subseteq \pi_2^{-1}(f^{-1}(B_1(supp(\alpha)))).
		\end{equation}
		and, in particular
		\begin{equation}
			supp(Y(\alpha)) \subseteq p_{f,2}(\pi_2^{-1}(f^{-1}(B_1(supp(\alpha)))))
		\end{equation}
		Now we can observe that if $(x,t)$ is in $\pi_2^{-1}(f^{-1}(B_1(supp(\alpha))))$, then
		\begin{equation}
			d(p_f(x,t), B_1(supp(\alpha)) \leq C_{p_f}
		\end{equation}
		where $C_{p_f}$ is the lipschitz costant of $p_f$. Indeed it is sufficient to observe that
		\begin{equation}
			p_f(x,0) = f(x) \subseteq B_1(supp(\alpha))
		\end{equation}
		and
		\begin{equation}
			d(p_f(x,1), p_f(x,0)) \leq C_{p_f}d((x,0),(x,1)) = C_{p_f}.
		\end{equation}
		Then we have that if $\phi$ and $\psi$ are two functions on $N$ with
		\begin{equation}
			d(supp(\phi), supp(\psi)) > C_{p_f} + 1 \implies \phi Y \psi = 0.
		\end{equation}
		\\
		\\Now to conclude the proof it is sufficient to show that $T_fY$, $YT_f$, $T_f^\dagger Y$, $Y T_f^\dagger$ are all operator in $C_f^*(M \sqcup N)^\Gamma$. We will start by studying $T_fY$.
		\\We know that $Y$ and $T_f$ are bounded operator which have both finite propagation, so also $T_fY$ is a bounded operator with finite propagation.
		\\Moreover if $g$ is a function on $M \sqcup N$ with compact support than we have that $gT_f$ is a compact operator since $T_f$ is in $C_f^*(M \sqcup N)^\Gamma$ and son $gT_fY$ is compact.
		\\Now let us consider $T_fYg$: now we know that for all $\alpha$ we have that
		\begin{equation}
			supp(Yg(\alpha)) \subseteq B_1(supp(g)).
		\end{equation}
		This means that, if $\phi$ is a compactly supported function such that $\phi \cong 1$ on $B_1(supp(g))$, then we have that
		\begin{equation}
			T_fYg = T_f \phi Y g.
		\end{equation}
		Then we con observe that $T_f\phi$ is a compact operator. Then also holds for $T_fYg$. This means that $T_fY$ is in $C_f(M \sqcup N)^\Gamma$. Exactly in the same way one can check that $YT_f$, $T_f^\dagger Y$, $Y T_f^\dagger$ are operator in $C_f^*(M \sqcup N)^\Gamma$.
	\end{proof}
	\begin{rem}\label{FVdit}
		Let us consider the kernel $K$ of $T_f$. Let us fix some local normal coordinates $\{x,y\}$ on $M \times N$. Since Proposition \ref{Tsmooth}, we know that outside $im(t_f)$ the kernel is identically $0$. Morever, inside $im(t_f)$, we have that, the kernel is locally given by
		\begin{equation}
			K_J^I(x,y) = \beta_{IJ^c}(x,y), 
		\end{equation} 
		where the functions $\beta_{IJ^c}(x,y)$ are the components of the pullback of the Thom form $\omega$ along the map $t_{f}^{-1}: im(t_f) \subset M \times N \longrightarrow f^*TN$.
		\\
		\\Let us recall that, respect to some fibered coordinates $\{x^i, y^j\}$ on $f^*T^\delta N$, the components of $\omega$ are algebraic combinations of pullback of components of the metric $h$ on $N$ and of Christoffell symbols of the Levi-Civita connection $\nabla^E$ along the map $f$ and derivatives of $f$ (see subsection \ref{Thom}). This means that if $f$ is a $C^k_b$-map the local components of $\omega$ and their derivatives of order $k$ are uniformly bounded only .
		\\Moreover we also know, beacuse of Proposition \ref{tilde}, that if $f$ is a $C^k_b$-map then also $p_f$ and, in particular, $t_f$ have uniformly bounded derivatives of order $0, 1..., k$.
		\\
		\\Let us suppose that $t_f$ has bounded derivatives of order $k$. Then we know that $t_f$ is a diffeomorphism with its image. This fact, using also the inverse function Theorem, implies that $t_{f}^{-1}$ (up to consider $\delta$ small enough) has bounded derivatives of order $0, 1, ... k$.
		\\
		\\The bounds on the components of $\omega$ and on the derivatives of $t_f^{-1}$ imply that if $f$ is a $C^k_b$ map, then the components of the kernel $K$ of $T_f$ in normal coordinates have bounded derivatives of order $0, 1, ... k$.
	\end{rem}
	\begin{prop}\label{dT_fboun}
		Let $(M,g)$ and $(N,h)$ be two Riemannian manifolds of bounded geometry. Let $f:M \longrightarrow N$ be a lipschitz-homotopy equivalence. Then we have that $T_f$ and $dT_f =T_f d$ have uniformly bounded support (definition in Appendix \ref{smoothing}). Moreover if $f$ is a $C^2_{b}$-map, then $dT_f$ is a bounded operator.
	\end{prop}
	\begin{proof}
		We will first show that $T_f$ has uniformly bounded supported kernel $K$. Since $supp(d_M K) \subseteq supp(K)$ it will follow immediately that also $dT_f$ has uniformly bounded supported kernel. 
		\\
		\\We know, since Proposition \ref{Tsmooth} that $K(p,q) = 0$ if $(p,q)$ is not in $im(t_f)$, where $t_f = (\pi, p_f)$. \\Fix $p$, a point of $M$. We can observe that if $q$ is not in $p_f(B^\delta_p)$, then $(p,q) \notin im(t_f)$ and so $K(p,q) = 0$. We know, since $p_f$ is lipschitz and $diam(B^\delta) = 2\delta$, that
		\begin{equation}
			diam(supp(K(p, \cdot)) \leq  2\delta C_{p_f}.
		\end{equation}
		Fix now $q$ in $N$. We have that, becuse of Remark \ref{oss}, that
		\begin{equation}
			\pi(p_f^{-1}(q)) \subseteq B_1(f^{-1}(q))
		\end{equation}
		and if $p$ is not in $B_1(f^{-1}(q))$ then $K(p,q) = 0$. Then, since $f$ is a lipshitz-homotopy equivalence, then in particular we have that $f$ is a uniformly proper map. This means that
		\begin{equation}
			diam(f^{-1}(A)) \leq \alpha(diam(A))
		\end{equation}
		for a continuous function $\alpha: [0, +\infty] \longrightarrow [0, +\infty)$. Then we have that
		\begin{equation}
			diam(supp(K(\cdot, q)) \leq  \alpha(0) +1.
		\end{equation}
		\\
		\\
		\\To conclude the proof we want to apply Proposition \ref{smoothbound}. In order to apply this Proposition we have to show that
		\begin{equation}\label{true}
			|\frac{\partial}{\partial x^j}K^I_J(x,y)| \leq C.
		\end{equation}
		Because of Remark \ref{FVdit} and Remark \ref{derivsmooth} in Appendix \ref{smoothing} this fact imediatly follows if $f$ is a $C^2_b$-map.
	\end{proof}
	\begin{cor}\label{corn}
		The operators $dT_f = T_fd$, $T_f^\dagger d = d T_f^\dagger$, $T_fyd$, $dyT_f$, $T_f^\dagger yd$ $dyT_f^\dagger$ are all operators in $C^*_f(M \sqcup N)^\Gamma$.
	\end{cor}
	\begin{proof}
		Since $dT_f = T_fd$ is a bounded integral operator with $supp(dT_f) \subseteq supp(T_f)$ then $dT_f$ has finite propagation. Moreover, since it has finite propagation we have that for all $g$ in $C_c(M \sqcup N)$ then $dT_fg$ and $gdT_f$ are smoothing operator with compact support and so are compact operators (Appendix \ref{smoothing}).
		\\
		\\We can observe that
		\begin{equation}
			T_f^\dagger d = -T_f^\dagger d^\dagger = -(dT_f)^\dagger
		\end{equation}
		and since $dT_f$ is in $C^*(M\sqcup N)^\Gamma$ we have that the same holds for $T_f^\dagger d = dT_f^\dagger$.
		\\Finally we can observe that
		\begin{equation}
			\begin{split}
				T_fyd &= T_f(dy - 1 - T_f^\dagger T_f) \\
				&=(T_fd)y - T_f - T_f T_f^\dagger T_f
			\end{split}
		\end{equation}
		which is a combination of operators in $C_f^*(M\sqcup N)^\Gamma$. All the other combinations are in $C_f^*(M \sqcup N)^\Gamma$ for the same argument. 
	\end{proof}
	\section{Lipschitz-homotopy invariance of the Roe Index}\label{wai}
	\subsection{The perturbed signature operator}
	Let us consider now two orientable, connected, Riemannian manifolds $(M,g)$ and $(N,h)$.
	We will consider $\mathcal{L}^2(M \bigsqcup N) \cong \mathcal{L}^2(N) \oplus \mathcal{L}^2(-M)$, where $-M$ is the manifold $M$ with the opposite orientation. 
	\begin{defn}
		We will denote with $d_{M \sqcup N}$ the operator
		\begin{equation}
			d_{M \sqcup N} := \begin{bmatrix}
				d_N && 0 \\
				0 && -d_M
			\end{bmatrix}
		\end{equation}
		if the dimensions of $M$ and $N$ are even, and the operator
		\begin{equation}
			d_{M \sqcup N} := i\begin{bmatrix}
				d_N && 0 \\
				0 && -d_M
			\end{bmatrix}.
		\end{equation}
		if the dimensions are odd.
	\end{defn}
	Moreover we will say
	\begin{defn}
		\textbf{Signature operator} the operator
		\begin{equation}
			\mathfrak{D}_{M \sqcup N} := \begin{bmatrix}
				\mathfrak{D}_N && 0 \\
				0 && -\mathfrak{D}_M
			\end{bmatrix}
		\end{equation}
	\end{defn}
	In general, the signature operator is not $\mathcal{L}^2$-invertible: our next step is to define a perturbation $A$ such that $\mathfrak{D}_{M \sqcup N} +A$ is $\mathcal{L}^2$-invertible. To do this we will follow the construction given in \cite{wahl}.
	\\
	\\Let us introduce $\gamma: \mathcal{L}^2(\Omega^*(M \bigsqcup N)) \longrightarrow \mathcal{L}^2(\Omega^*(M \bigsqcup N))$ the operator defined for all $\alpha$ as $\gamma (\alpha) := (-1)^{|\alpha|}$. One can easily check that $\gamma ^{\dagger} = -\gamma$.
	Moreover $\gamma$ commute with $\tau$ and with $d_{M\sqcup N}$. 
	\\Let us define, now
	\begin{equation}
		R_\beta := \begin{bmatrix}
			1 && 0 \\
			\beta T_f && 1
		\end{bmatrix}.
	\end{equation}
	We have that $R_\beta$ is $\mathcal{L}^2$-invertible: the inverse is given by
	\begin{equation}
		\begin{bmatrix}
			1 && 0 \\
			-\beta T_{f}\gamma && 1
		\end{bmatrix}.
	\end{equation}
	We will denote by $\tau$ the operator
	\begin{equation}
		\tau :=
		\begin{bmatrix}
			\tau_N && 0 \\
			0 && -\tau_M
		\end{bmatrix}.
	\end{equation}
	\\Consider now $\alpha$ a real number: we can define
	\begin{equation}
		L_{\alpha, \beta} := \begin{bmatrix}
			1- T^{\dotplus}_fT_f && \beta (\gamma + \alpha y)T_f^{\dotplus} \\
			\beta T_f(-\gamma -\alpha y ) && 1
		\end{bmatrix}.
	\end{equation}
	Moreover, if the dimension of $M$ is even, we will define the operator $\delta_{\alpha}$ as follow
	\begin{equation}
		\delta_{\alpha} :=
		\begin{bmatrix}
			d_N && \alpha iT^{\dagger}_f\gamma \\
			0 && -d_M
		\end{bmatrix},
	\end{equation}
	in the other case we have
	\begin{equation}
		\delta_{\alpha} :=
		\begin{bmatrix}
			i\cdot d_N && -\alpha T^{\dagger}_f \gamma \\
			0 && -i \cdot d_M
		\end{bmatrix},
	\end{equation}
	It's easy to check that $\delta_{\alpha}^2 = 0$. Since $d_N \oplus (- d_M)$ is a closed operator and $T^{\dagger}$ is bounded, then we have that $\delta_\alpha$ is a closed operator with the same domain of $d_N - d_M$.
	\\
	\\It's possible to observe that $L_\alpha d_{M \sqcup N} = d_{M \sqcup N} L_\alpha$, $L_\alpha \delta_\alpha = \pm \delta_\alpha^{\dagger} L_\alpha$ (we have a plus if $dim(M)$ is odd, a minus otherwise) and that $L_\alpha^\dagger = L_\alpha$.
	\\Morover we have that $R^{\dagger}_\beta R_\beta = L_0$, so $L_0$ is $\mathcal{L}^2$-invertible and the same holds for $L_\alpha$ if $|\alpha|$ is small enough.
	\\Let us define now
	\begin{equation}
		S_\alpha := \frac{\tau L_\alpha}{|\tau L_\alpha|}.
	\end{equation}
	Again, if $|\alpha|$ is small enough $S_\alpha$ is invertible (it is in particular an involution). We also have that $S_\alpha^{\dotplus} = S_\alpha$.
	\\Consider, now $U_\alpha := (|\tau L_\alpha|)^{\frac{1}{2}}$. Since $|\tau L_\alpha|$ is invertible, also $U_\alpha$ is invertible.
	\begin{defn}
		If the dimension od $M$ is even, we will call \textbf{perturbed signature operator} the operator
		\begin{equation}
			D_{\alpha, \beta} := U_{\alpha, \beta}(\delta_\alpha - S_{\alpha, \beta} \delta_\alpha S_{\alpha, \beta} )U_{\alpha,\beta}^{-1}.
		\end{equation}
		if $M$ has dimension odd, we have that
		\begin{equation}
			D_{\alpha,\beta} := -i U_{\alpha,\beta}( S_{\alpha,\beta} \delta_\alpha +\delta_\alpha S_{\alpha, \beta} )U_{\alpha, \beta}^{-1}.
		\end{equation}
	\end{defn}
	\subsection{$L^2$-invertibility of $D_{\alpha,1}$}
	\begin{lem}
		Let us consider two Riemannian manifolds of bounded geometry $(M,g)$ and $(N,h)$ and let $f$ be a lipschitz-homotopic equivalence. Then we have that $ker(\delta_{\alpha}) = im(\delta_{\alpha})$ in $\mathcal{L}^2(M \sqcup N)$.
	\end{lem}
	\begin{proof}
		Since $T_f d = d T_f$, we have that $T_f$ gives a map between the complexes
		\begin{equation}
			0 \xrightarrow{d_N}....\xrightarrow{d_N}dom(d_N)^{k-1} \xrightarrow{d_N} dom(d_N)^k \xrightarrow{d_N} dom(d_N)^{k+1} \xrightarrow{d_N} ...
		\end{equation}
		and
		\begin{equation}
			0 \xrightarrow{-d_M}....\xrightarrow{-d_M}dom(d_M)^{k-1} \xrightarrow{-d_M} dom(d_M)^k \xrightarrow{-d_M} dom(d_M)^{k+1} \xrightarrow{-d_M} ...
		\end{equation}
		In particular, we have that $\delta_\alpha$ is a mapping cone over these chains. Now, since $T_f$ is an isomorphism in $\mathcal{L}^2$-cohomology, this means that $\delta_\alpha$ is acyclic, i.e.
		\begin{equation}
			ker(\delta_{\alpha}) = im(\delta_{\alpha}).
		\end{equation}
	\end{proof}
	\begin{prop}
		The operators 
		\begin{equation}
			D_\alpha: dom(d_N) \oplus dom(d_M)\subseteq \mathcal{L}^2(\Omega^*(M \sqcup N)) \longrightarrow \mathcal{L}^2(\Omega^*(M \sqcup N))
		\end{equation}
		are injective if $|\alpha|$ i small enough.
	\end{prop}
	\begin{proof}
		Let us consider the operator
		\begin{equation}
			\delta_\alpha \pm S_{\alpha,\beta} \delta_\alpha S_{\alpha,\beta},
		\end{equation}
		where we have a plus if $dim(M)$ is odd and a minus otherwise.
		If we check that this operator is injective, than also the perturbed signature operator is injective (in both cases: even or odd). 
		\\Let us consider on $\mathcal{L}^2(M \sqcup N)$ the scalar product
		\begin{equation}
			\langle \cdot , \cdot \rangle_{\alpha, \beta}:= \langle \cdot, |\tau L_{\alpha, \beta}| \cdot \rangle_{\mathcal{L}^2(M \sqcup N)}.
		\end{equation}	
		We have that $|\tau L_{\alpha,\beta}|$ is self-adjoint respect to this scalar product indeed given $\omega_1$ and $\omega_2$ two $\mathcal{L}^2$-differential forms, we have that
		\begin{equation}
			\begin{split}
				\langle |\tau L_{\alpha,\beta}| \omega_1 , \omega_2 \rangle_{\alpha,\beta} &= \langle |\tau L_{\alpha,\beta}| \omega_1, |\tau L_{\alpha,\beta}| \omega_2 \rangle_{\mathcal{L}^2(M \sqcup N)} \\
				&= \overline{ \langle |\tau L_{\alpha,\beta}| \omega_2, |\tau L_{\alpha,\beta}| \omega_1 \rangle}_{\mathcal{L}^2(M \sqcup N)} \\
				&= \overline{ \langle |\tau L_{\alpha,\beta}| \omega_2 , \omega_1 \rangle}_{\alpha,\beta} \\
				&= \langle \omega_1 , |\tau L_{\alpha,\beta}| \omega_2 \rangle_{\alpha,\beta}.
			\end{split}
		\end{equation}
		The operators $\tau L_{\alpha,\beta}$ and $S_{\alpha,\beta}$ are self-adjoint too, indeed
		\begin{equation}
			\begin{split}
				\langle \tau L_{\alpha,\beta} \omega_1 , \omega_2 \rangle_{\alpha,\beta} &= \langle \tau L_{\alpha,\beta} \omega_1, |\tau L_{\alpha,\beta}| \omega_2 \rangle_{\mathcal{L}^2(M \sqcup N)} \\
				&= \langle \omega_1, \tau L_{\alpha,\beta} |\tau L_{\alpha,\beta}| \omega_2 \rangle_{\mathcal{L}^2(M \sqcup N)} \\
				&= \langle \omega_1 , |\tau L_{\alpha,\beta}| \tau L_{\alpha,\beta} \omega_2 \rangle_{\alpha,\beta} \\
				&= \langle \omega_1 , \tau L_{\alpha,\beta} \omega_2 \rangle_{\alpha,\beta}.
			\end{split}
		\end{equation}
		and
		\begin{equation}
			\begin{split}
				\langle \omega_1 , S_{\alpha,\beta} \omega_2 \rangle_{\alpha,\beta} &= \langle \omega_1, |\tau L_{\alpha,\beta}| S_{\alpha,\beta} \omega_2 \rangle_{\mathcal{L}^2(M \sqcup N)} \\
				&= \langle \omega_1, \tau L_{\alpha,\beta} \omega_2 \rangle_{\mathcal{L}^2(M \sqcup N)} \\
				&= \langle  \tau L_{\alpha,\beta} \omega_1 , \omega_2 \rangle_{\mathcal{L}^2(M \sqcup N)} \\
				&= \langle  |\tau L_{\alpha,\beta}| |\tau L_{\alpha,\beta}|^{-1}\tau L_{\alpha,\beta} \omega_1 , \omega_2 \rangle_{\mathcal{L}^2(M \sqcup N)} \\
				&=  \langle  |\tau L_{\alpha,\beta}| S_{\alpha,\beta} \omega_1 , \omega_2 \rangle_{\mathcal{L}^2(M \sqcup N)} \\
				&= \langle S_{\alpha,\beta} \omega_1 , \omega_2 \rangle_{\alpha,\beta}\\
				&= \langle  \omega_2 , S_{\alpha,\beta} \omega_1 \rangle_{\alpha,\beta}.
			\end{split}
		\end{equation}
		Moreover, respect to this scalar product $\pm S_{\alpha,\beta} \delta_\alpha S_{\alpha,\beta}$ is the adjoint of $\delta_\alpha$, indeed
		\begin{equation}
			\begin{split}
				\langle \delta_\alpha \omega_1 , \omega_2 \rangle_{\alpha,\beta} &= \langle \delta_\alpha \omega_1, |\tau L_{\alpha,\beta}| \omega_2 \rangle_{\mathcal{L}^2(M \sqcup N)} \\
				&= \langle \delta_\alpha \omega_1, |\tau L_{\alpha,\beta}||\tau L_{\alpha,\beta}||\tau L_{\alpha,\beta}|^{-1} \omega_2 \rangle_{\mathcal{L}^2(M \sqcup N)} \\
				&= \langle \delta_\alpha \omega_1, \tau L_{\alpha,\beta} \frac{\tau L_{\alpha,\beta}}{|\tau L_{\alpha,\beta}|} \omega_2 \rangle_{\mathcal{L}^2(M \sqcup N)} \\
				&= \langle \omega_1, \tau \delta_\alpha^{\dotplus} \tau \tau L_{\alpha,\beta} \frac{\tau L_{\alpha,\beta}}{|\tau L_{\alpha,\beta}|} \omega_2 \rangle_{\mathcal{L}^2(M \sqcup N)} \\
				&= \langle \omega_1, \tau L_{\alpha,\beta} (\pm \delta_\alpha) \frac{\tau L_{\alpha,\beta}}{|\tau L_{\alpha,\beta}|} \omega_2 \rangle_{\mathcal{L}^2(M \sqcup N)} \\
				&= \langle \omega_1, |\tau L_{\alpha,\beta}| \pm \frac{\tau L_{\alpha,\beta}}{|\tau L_{\alpha,\beta}|} \delta_\alpha \frac{\tau L_{\alpha,\beta}}{|\tau L_{\alpha,\beta}|} \omega_2 \rangle_{\mathcal{L}^2(M \sqcup N)} \\
				&= \langle \omega_1 , \pm \frac{\tau L_{\alpha,\beta}}{|\tau L_{\alpha,\beta}|} \delta_\alpha \frac{\tau L_{\alpha,\beta}}{|\tau L_{\alpha,\beta}|} \omega_2 \rangle_{\alpha,\beta}.
			\end{split}
		\end{equation}
		We can denote with $\delta_\alpha^*$ the operators $\pm S_{\alpha,\beta} \delta_\alpha S_{\alpha,\beta}$.
		\\It's important to observe, now, that 
		\begin{equation}
			ker(\delta_\alpha^*) = im(\delta_\alpha)^\perp = ker(\delta_\alpha)
		\end{equation}
		We want to prove the injectivity of $\delta_\alpha + \delta_\alpha^*$ studying $(\delta_\alpha + \delta_\alpha^*)^2 = \delta_\alpha\delta_\alpha^* + \delta_\alpha^*\delta_\alpha$.
		\\Let us consider $\xi$ in $ker(\delta_\alpha\delta_\alpha^* + \delta_\alpha^*\delta_\alpha)$: we have that
		\begin{equation}
			\begin{split}
				0 &= \langle \xi, \delta_\alpha\delta_\alpha^* + \delta_\alpha^*\delta_\alpha\xi\rangle_{\alpha,\beta} \\
				&=  \langle \xi, \delta_\alpha\delta_\alpha^* \xi\rangle_{\mathcal{L}^2(M \sqcup N)} + \langle \xi,\delta_\alpha^*\delta_\alpha)\xi)\rangle_{\alpha,\beta} \\
				&=  \langle \delta_\alpha^* \xi, \delta_\alpha^* \xi \rangle_{\mathcal{L}^2(M \sqcup N)} + \langle \delta_\alpha \xi, \delta_\alpha \xi \rangle_{\alpha,\beta} \\
				&=  || \delta_\alpha^* \xi||_{\alpha,\beta} + || \delta_\alpha \xi||_{\alpha,\beta}.
			\end{split}
		\end{equation}
		Then we have that $\xi$ is in $ker(\delta_\alpha) \cap ker(\delta_\alpha^*) = 0$. Then $\delta_\alpha\delta_\alpha^* + \delta_\alpha^*\delta_\alpha$ is injective.
		\\Moreover we also have 
		\begin{equation}
			\langle \xi, \delta_\alpha\delta_\alpha^* + \delta_\alpha^*\delta_\alpha\xi\rangle_{\mathcal{L}^2(M \sqcup N)}  = \langle \delta_\alpha + \delta_\alpha^* \xi, \delta_\alpha + \delta_\alpha^* \xi \rangle_{\mathcal{L}^2(M \sqcup N)}.
		\end{equation}
		It means that if $\xi$ is in $ker(\delta_\alpha\delta_\alpha^* + \delta_\alpha^*\delta_\alpha)$ then it is in $ker(\delta_\alpha + \delta_\alpha^*)$ and so the first kernel is contained in the second one. Moreover the second kernel is contained in the first one beacuse $(\delta_\alpha + \delta_\alpha^*)^2 = \delta_\alpha\delta_\alpha^* + \delta_{\alpha}^*\delta_{\alpha}$. So the kernels coinde.
		\\This means that $\delta_\alpha + \delta_\alpha^*$ is injective.
	\end{proof}
	\begin{rem}
		The operators $D_{\alpha, \beta}$ are self-adjoint. Let us denote by $\hat{D}_{\alpha,\beta}$ the operator
		\begin{equation}
			\hat{D}_{\alpha,\beta} := \delta_\alpha \pm S_{\alpha,\beta} \delta_\alpha S_{\alpha,\beta},
		\end{equation}
		where we are considering a plus if $dim(M)$ is odd and a minus otherwise.\\
		We have that the even perturbed signature operator is
		\begin{equation}
			D_{\alpha, \beta} =  U_{\alpha, \beta} \hat{D}_{\alpha,\beta} U^{-1}_{\alpha, \beta}.
		\end{equation}
		We know that $\hat{D}_{\alpha,\beta}$ is a self-adjoint operator respect to the scalar product $\langle \cdot, \cdot \rangle_{\alpha,\beta}$. Then $D_{\alpha, \beta}$ is selfadjoint respect the standard scalar product on $\mathcal{L}^2(M \sqcup N)$.
		\begin{equation}
			\begin{split}
				\langle U_{\alpha,\beta} \hat{D}_{\alpha,\beta} U^{-1}_{\alpha,\beta} \omega_1, \omega_2 \rangle_{\mathcal{L}^2(M \sqcup N)} &=
				\langle \hat{D}_{\alpha,\beta} U^{-1}_{\alpha,\beta} \omega_1, U_{\alpha,\beta} \omega_2 \rangle_{\mathcal{L}^2(M \sqcup N)} \\
				&= \langle \hat{D}_{\alpha,\beta} U^{-1}_{\alpha,\beta} \omega_1, |\tau L_{\alpha,\beta}| U_{\alpha,\beta}^{-1} \omega_2 \rangle_{\mathcal{L}^2(M \sqcup N)}\\
				&= \langle U^{-1}_{\alpha,\beta} \omega_1, \hat{D}_{\alpha,\beta} |\tau L_{\alpha,\beta}| U_{\alpha,\beta}^{-1} \omega_2 \rangle_{\mathcal{L}^2(M \sqcup N)} \\
				&= \langle \omega_1, U^{-1}_{\alpha,\beta} |\tau L_{\alpha,\beta}| \hat{D}_{\alpha,\beta} U_{\alpha,\beta}^{-1} \omega_2 \rangle_{\mathcal{L}^2(M \sqcup N)} \\
				&= \langle \omega_1, U_{\alpha,\beta} \hat{D}_{\alpha,\beta} U_{\alpha,\beta}^{-1} \omega_2 \rangle_{\mathcal{L}^2(M \sqcup N)}.
			\end{split}
		\end{equation}
		Let us consider now the odd case. If we consider
		\begin{equation}
			\hat{D}_{\alpha,\beta} = (-i)(S_{\alpha,\beta} \delta_\alpha + \delta_\alpha S_{\alpha,\beta})
		\end{equation}
		we can observe then that
		\begin{equation}
			D_{\alpha,\beta} = U_{\alpha,\beta}\hat{D}_{\alpha,\beta}U_{\alpha,\beta}^{-1}
		\end{equation}
		Now, if we prove that $\hat{D}_{\alpha,\beta}$ is selfadjoint respect to $\langle \cdot, \cdot \rangle _{\alpha,\beta}$ we can conclude exactly how we did in the even case. We have
		\begin{equation}
			\begin{split}
				\langle (-i)(S_{\alpha,\beta} \delta_\alpha + \delta_\alpha S_{\alpha,\beta}) \omega_1, \omega_2 \rangle_{\alpha,\beta} &= (-i)\langle \delta_\alpha \omega_1, S_{\alpha,\beta} \omega_2 \rangle_{\alpha,\beta}\\ 
				&+ (-i)\langle S_\alpha \omega_1, ( - S_{\alpha,\beta} \delta_\alpha S_{\alpha,\beta})\omega_2 \rangle_{\alpha,\beta} \\
				&= \langle \omega_1, -i S_{\alpha,\beta} \delta_\alpha S_{\alpha,\beta} S_{\alpha,\beta} \omega_2 \rangle_{\alpha,\beta}\\
				& + \langle \omega_1, iS_{\alpha,\beta} ( - S_{\alpha,\beta} \delta_\alpha S_{\alpha,\beta})\omega_2 \rangle_{\alpha,\beta} \\
				&=\langle \omega_1, (-i)(S_{\alpha,\beta} \delta_\alpha + \delta_\alpha S_{\alpha,\beta}) \omega_2 \rangle_{\alpha,\beta}.
			\end{split}
		\end{equation}
	\end{rem}
	\begin{prop}
		The operators $D_{\alpha,\beta}$ are $L^2$-invertible.
	\end{prop}
	\begin{proof}
		Since $U_{\alpha,\beta}$ and $S_{\alpha,\beta}$ are $L^2$-invertible, it is sufficient to proof that 
		$\hat{D}_{\alpha,\beta} := \delta_\alpha + \delta_\alpha^*$ is invertible.
		\\Since $\hat{D}_{\alpha,\beta}$ is self-adjoint, its spectrum can be be decomposed in \textit{essential spectrum} given by the subset of $\lambda$s such that
		\begin{equation}
			\hat{D}_{\alpha,\beta} - \lambda Id
		\end{equation}
		is not a Fredholm operator and in \textit{discrete spectrum} which is the subset given by the eigenvalues with finite multiplicity.
		\\Since $\hat{D}_{\alpha,\beta}$ is injective, the zero can't be in in discrete spectrum.
		\\Now, using the Theorem 2.4. of \cite{Bruning} we have that, since $\delta_\alpha$ is acyclic, then   zero can't be in the essential spectrum of $\hat{D}_{\alpha,\beta}^2$. But the essential spectrum of $\hat{D}_{\alpha,\beta}^2$ is the same of $\hat{D}_\alpha$. Then $\hat{D}_{\alpha,\beta}$ is invertible and the same holds for $D_{\alpha,\beta}$.
	\end{proof}
	\subsection{The perturbation in $C^*_f(M \sqcup N)^\Gamma$}\label{pert}
	\begin{prop}
		The perturbation $D_{\alpha, \beta} - D$ is an operator of $C_f^*(M\sqcup N)^\Gamma$.
	\end{prop}
	\begin{proof}
		We can observe that
		\begin{equation}
			\begin{split}
				D_{\alpha, \beta} - D &= U_{\alpha, \beta}(\delta_\alpha - S_{\alpha, \beta} \delta_\alpha S_{\alpha, \beta})U_{\alpha, \beta}^{-1} - d + \tau d \tau\\
				&= U_{\alpha, \beta}(\delta_\alpha - d)U_{\alpha, \beta}^{-1} + (U_{\alpha, \beta}dU_{\alpha, \beta}^{-1} - d) \\
				&+U_{\alpha, \beta}(S_{\alpha, \beta}\delta_\alpha S_{\alpha, \beta} - \tau d \tau )U_{\alpha, \beta}^{-1} + (U_{\alpha, \beta}(\tau d \tau )U_{\alpha, \beta}^{-1} -\tau d\tau).
			\end{split}
		\end{equation}
		We know that $\delta_\alpha - \delta = \alpha T_f^\dagger$ is in $C_f^*(M\sqcup N)^\Gamma$.
		\\The first step is to show that 
		\begin{equation}
			S_{\alpha, \beta} = \tau + H_{\alpha, \beta},
		\end{equation}
		\begin{equation}
			U_{\alpha, \beta} = 1 + G_{\alpha, \beta}
		\end{equation}
		and
		\begin{equation}
			U_{\alpha, \beta}^{-1} = 1 + K_{\alpha, \beta}
		\end{equation}
		where $H_{\alpha, \beta}$, $G_{\alpha, \beta}$, $K_{\alpha, \beta}$ are operators in $C_f^*(M \sqcup N)^\Gamma$.
		\\To prove this first step we can follow Proposition 2.1.11. of \cite{vito}. We can observe that $L_{\alpha, \beta} = 1 + Q_{\alpha,\beta}$, with $Q_{\alpha,\beta}$ in $C_f^*(M \sqcup N)^\Gamma$.
		\\This means that
		\begin{equation}
			|\tau L_{\alpha, \beta}| = \sqrt{1 + R_{\alpha, \beta}},
		\end{equation}
		where $R_{\alpha, \beta} = 2Q_{\alpha,\beta} + Q_{\alpha,\beta}^2$ in $C_f^*(M \sqcup N)^\Gamma$. Now, let us consider the complex function $g(z) := \sqrt{1 + z} - 1$. Since $g(0) = 0$, we have that
		\begin{equation}
			g(z) = az + zh(z)z
		\end{equation}
		where $h$ is a holomorphic function and $a$ is a number. Then we have that
		\begin{equation}
			|\tau L_{\alpha, \beta}| = 1 + g(R_{\alpha, \beta}) = 1 + V_{\alpha, \beta}
		\end{equation}
		with $V_{\alpha, \beta}$ in $C^*(M \sqcup N)^\Gamma$. Now, we can observe that
		\begin{equation}
			U_{\alpha, \beta} = \sqrt{|\tau L_{\alpha, \beta}|}
		\end{equation}
		Applying the same argument we can obtain that
		\begin{equation}
			U_{\alpha, \beta} = 1 + G_{\alpha, \beta}.
		\end{equation}
		where $G_{\alpha, \beta}$ is in $C_f^*(M \sqcup N)^\Gamma$.
		\\Let us observe the operator
		\begin{equation}
			Z_{\alpha, \beta} = |\tau L_{\alpha, \beta}|^{-1} - 1
		\end{equation}
		is an operator of $D_f^*(M\sqcup N)^\Gamma$ since $|\tau L_{\alpha, \beta}|$ is in $D-f^*(M\sqcup N)^\Gamma$. In particular we have that
		\begin{equation}
			\begin{cases}
				|\tau L_{\alpha, \beta}| \circ |\tau L_{\alpha, \beta}| ^{-1} = (1 + V_{\alpha,\beta})(1 + Z_{\alpha, \beta}) =  1\\
				|\tau L_{\alpha, \beta}| ^{-1}  \circ |\tau L_{\alpha, \beta}| = (1 + Z_{\alpha, \beta})(1 + V_{\alpha,\beta}) = 1.
			\end{cases}
		\end{equation}
		and so we obtain that
		\begin{equation}
			Z_{\alpha, \beta} = - V_{\alpha, \beta} - V_{\alpha, \beta}Z_{\alpha, \beta} = - V_{\alpha, \beta} - Z_{\alpha, \beta} V_{\alpha, \beta}.
		\end{equation}
		Then we have that $Z_{\alpha, \beta}$ is in $C_f^*(M \sqcup N)^\Gamma$.
		\\Now, since
		\begin{equation}
			U_{\alpha, \beta}^{-1} = \sqrt{|\tau L_{\alpha, \beta}|^{-1}}
		\end{equation}
		then we have that
		\begin{equation}
			U_{\alpha, \beta}^{-1} = 1 + K_{\alpha, \beta}.
		\end{equation}
		where $K_{\alpha, \beta}$ is in $C_f^*(M \sqcup N)^\Gamma$.
		Moreover we also have that
		\begin{equation}
			\begin{split}
				S_{\alpha,\beta} &= \tau L_{\alpha, \beta} |\tau L_{\alpha, \beta}|^{-1} \\
				&= \tau(1 + Q_{\alpha, \beta})(1 + Z_{\alpha, \beta}) \\
				&= \tau + H_{\alpha, \beta},
			\end{split}
		\end{equation}
		where $H_{\alpha, \beta}$ in $C_f^*(M \sqcup N)^\Gamma$.
		\\
		\\We've completed the first step. In the second step, we have to proof that $G_{\alpha, \beta} d$, $H_{\alpha, \beta} d$, $K_{\alpha, \beta} d$, $dG_{\alpha, \beta}$, $dH_{\alpha, \beta}$, $dK_{\alpha, \beta}$,$G_{\alpha, \beta} d G_{\alpha, \beta}$ and $H_{\alpha, \beta} d K_{\alpha, \beta}$ are in $C^*(M\sqcup N)^\Gamma$.
		\\
		\\To prove this we have to observe that $L_{\alpha, \beta} = 1 + Q_{\alpha,\beta}$ where
		\begin{equation}
			Q_{\alpha, \beta} = \begin{bmatrix} \beta T_f^\dagger T_f && (1 - i \alpha \gamma y)\beta T_f^\dagger \\
				\beta T_f(1 + i \alpha \gamma y) && 0 \end{bmatrix}
		\end{equation}
		is an algebraic combination of $T_f$, $T_f^\dagger$, $T_fy$, $yT_f$, $T_f^\dagger y$, $yT_f^\dagger$. Then we have that $Q_{\alpha, \beta}d$ and $dQ_{\alpha, \beta}$ is an operator in $C_f^*(M \sqcup N)^\Gamma$ (Corollary \ref{corn}). The same property, obviously, holds for $R_{\alpha, \beta}$.
		\\Now we can observe that if $A$ is an operator in $C_f^*(M \sqcup N)^\Gamma$ such that $Ad$ and $dA$ are operators in $C^*(M \sqcup N)^\Gamma$ then also $g(A)d$ and $dg(A)$ are in $C_f^*(M \sqcup N)^\Gamma$, indeed
		\begin{equation}
			g(A) d = (aA + Ah(A)A) d = a(Ad) + Ah(A)(Ad)
		\end{equation}
		and
		\begin{equation}
			dg(A) = d(aA + Ah(A)A) = a(dA) + (dA)h(A)A.
		\end{equation}
		So we have that the compositions of $V_{\alpha, \beta}$ and $G_{\alpha,\beta}$ with $d$ are operators in $C_f^*(M \sqcup N)^\Gamma$. Moreover, since
		\begin{equation}
			Z_{\alpha, \beta} = - V_{\alpha, \beta} - V_{\alpha, \beta}Z_{\alpha, \beta} = - V_{\alpha, \beta} - Z_{\alpha, \beta} V_{\alpha, \beta}
		\end{equation}
		then this property also holds for $Z_{\alpha, \beta}$, $K_{\alpha, \beta}$ and $H_{\alpha, \beta}$.
		Now, since $D_{\alpha, \beta} - D$ is an algebraic combination of operators in $C_f^*(M \sqcup N)^\Gamma$, the perturbation is in $C_f^*(M \sqcup N)^\Gamma$.
	\end{proof}
	\subsection{Involutions}
	One can easily check that the operator $W_{\alpha, \beta} := U_{\alpha,\beta}S_{\alpha, \beta}U_{\alpha, \beta}^{-1}$ is a well-defined, selfadjoint, involution whenever $\alpha \neq 0$ and $\beta=1$ or when $\alpha = 0$ and $\beta \in [0,1]$.
	\\We have that $\mathcal{L}^2(M \sqcup N)$ can be split, for any $\alpha$ and $\beta$, in
	\begin{equation}
		\mathcal{L}^2(M \sqcup N) = V_{+, \alpha, \beta} \oplus V_{-, \alpha, \beta}
	\end{equation}
	where $V_{\pm, \alpha, \beta}$ are the $\pm1$-eigenvalues of $W_{\alpha, \beta} := U_{\alpha,\beta}S_{\alpha, \beta}U_{\alpha, \beta}^{-1}$. Respect to this decomposition we have that
	\begin{equation}
		D_{\alpha, \beta} = \begin{bmatrix} 0 && D_{\alpha, \beta +} \\
			D_{\alpha,\beta -} && 0 \end{bmatrix}.
	\end{equation}
	Consider a real value $\alpha_0$ such that $D_{\alpha_0, 1}$ is invertible, we can define $\gamma: [0,1] \longrightarrow \numberset{R}^2$ as $\gamma(t):= (\alpha(t), \beta(t))$ where
	\begin{equation}
		\alpha(t) = \begin{cases} 0 \text{   if    } t \in [0,\frac{1}{2}] \\
			2t\alpha_0 \text{    otherwise} \end{cases} \text{    and   } 
		\beta(t) = \begin{cases} 2t \text{   if    } t \in [0,\frac{1}{2}] \\
			1 \text{    otherwise} \end{cases}
	\end{equation}
	Then we can consider the map $W_\gamma : [0,1] \longrightarrow D^*(M \sqcup N)^\Gamma$ defined as
	\begin{equation}
		W_{\gamma(t)} := U_{\gamma(t)}S_{\gamma(t)}U_{\gamma(t)}^{-1}.
	\end{equation}
	We want to prove that $W_\gamma$ is a continuous function and that for all $t_1$, $t_2$ we have that
	\begin{equation}
		W_{\gamma(t_1)} - W_{\gamma(t_2)} \in C^*(M \sqcup N)^\Gamma \label{stat}.
	\end{equation}
	The statement (\ref{stat}) can easily checked observing that for all $t$ in $[0,1]$ we have that $W_{\gamma(t)} - \tau$ is in $C^*(M \sqcup N)^\Gamma$.
	\\To check the continuity of $W_{\gamma(t)}$ in $t$ we have to observe that for all $t$ in $[0,1]$ we have that
	\begin{equation}
		W_{t} := \sqrt{|\tau L_{\gamma(t)}|} \circ \frac{\tau L_{\gamma(t)}}{|\tau L_{\gamma(t)}|} \circ \sqrt{|\tau L_{\gamma(t)}|^{-1}}.
	\end{equation}
	Now, we know that $\tau L_{\gamma(t)}$ is continuous in $t$ and the same holds for its square. Now, we know that
	\begin{equation}
		|\tau L_{\gamma(t)}| = 1 + f(1 - \tau L_{\gamma(t)}^2)
	\end{equation}
	where $f(z) = az + zh(z)z$ and $h$ is an holomorphic function. Now, since the properties of holomorphic functional calculus on bounded operator, we have that if $T_k \rightarrow T$ then $f(T_k) \rightarrow f(T)$. In particular we have that $|\tau L_{\gamma(t)}|$ is continuous in $t$. Moreover exactly with the same argument we can see that also $\sqrt{|\tau L_{\gamma(t)}|}$ is continuous in $t$.
	\\The operator $|\tau L_{\gamma(t)}|^{-1}$ is continuous in $t$ because for all $t$ the operator $|\tau L_{\gamma(t)}|$ is invertible with bounded inverse and the function $z \rightarrow \frac{1}{z}$ is holomorphic in every bounded open set of $\numberset{C}$ which doesn't contains the $0$. Finally also $\sqrt{|\tau L_{\gamma(t)}|^{-1}}$ is continuous. Then we have that $W_t$ is continuous in $t$. Moreover, since $[0,1]$ is compact, we have that $W_t$ is uniformly continuous, i.e. exists $C > 0$ such that
	\begin{equation}
		||W_{t} - W_{t + \epsilon}|| \leq C \epsilon.
	\end{equation}
	\begin{lem}\label{Cos}
		Let $f: (M,g) \longrightarrow (N,h)$ be a lipschitz-homotopic equivalence between two Riemannian manifolds of bounded geometry. Consider the splitting given by
		\begin{equation}
			\mathcal{L}^2(M \sqcup N) = V_+ \oplus V_-
		\end{equation}
		where $V_\pm$ is the $\pm 1$-eigenspace of $\tau$. Then if $\alpha_0$ is such that $D_{\alpha_0, 1}$ is invertible, then there is an isometry $\mathcal{U}_{\alpha_0, \pm}: V_{\pm, \alpha_0, 1} \longrightarrow V_\pm$ (which implies $\mathcal{U}_{\alpha_0,\pm}^{\times}\mathcal{U}_{\alpha_0,\pm} = I$) and $\mathcal{U}_{\alpha_0, \pm}$ is the restriction to $V_{\pm, \alpha_0, 1}$ of the operator $\frac{I \pm \tau}{2} + P_{\alpha_0}$ where $P_{\alpha_0}$ is an operator in $C_f^*(M \sqcup N)^\Gamma$.
	\end{lem}
	\begin{proof}
		We will prove the assertion just for the $+$ case. The minus case is exactly the same.
		\\Consider the operator $W_{\gamma(t)}$. Since it is uniformly continuous in $t$, we can divide $[0,1]$ in $N_0$ intervals $[t_i, t_{i+1}]$ where $t_0 = 0$ and $t_{N_0}= 1$ and 
		\begin{equation}
			||W_{t_i} - W_{t_{i+1}}|| \leq 1.
		\end{equation}
		Now, we know that for all $t$ the operator $W_t$ is an involution and we have an orthogonal decomposition $\mathcal{L}^2(M \sqcup N) = V_{+, \alpha(t), \beta(t)} \oplus V_{-, \alpha(t), \beta(t)}$, where $V_{\pm, \alpha(t), \beta(t)}$ is the $\pm 1$-eigenspace of $W_t$.
		\\Now the projection on $V_{\pm, \alpha(t), \beta(t)}$ can be written as
		\begin{equation}
			\frac{I \pm W_t}{2}.
		\end{equation}
		Let us now consider $F_i$ the restriction of $\frac{I + W_{t_{i}}}{2}$ to $V_{+, \alpha(t_{i+1}), \beta(t_{i+1})}$. 
		\\Our next step is to prove that $F_i: V_{+, \alpha(t_{i+1}), \beta(t_{i+1})}\longrightarrow V_{+, \alpha(t_i), \beta(t_i)}$ is invertible. To prove this fact we have to consider the operator $G_i$ given by the restriction of $\frac{I + W_{t_{i+1}}}{2}$ to $V_{+, \alpha(t_{i}), \beta(t_{i})}$.
		\\Consider now $H_i := W_{t_i} - W_{t_{i+1}}$. Then we have that for all $v$ in $V_{+, \alpha(t_{i+1}), \beta(t_{i+1})}$
		\begin{equation}
			\begin{split}
				G_i \circ F_i (v) &=  (\frac{I + W_{t_{i+1}}}{2})(\frac{I + W_{t_{i}}}{2})v \\
				&=  (\frac{I + W_{t_{i+1}}}{2})(\frac{I + W_{t_{i+1}} + H_i}{2})v \\
				&=  \frac{I + W_{t_{i+1}}}{2} v + (\frac{I + W_{t_{i+1}}}{4})H_iv \\
				&=  Iv + (\frac{I + W_{t_{i+1}}}{4})H_iv.
			\end{split}
		\end{equation}
		Now, since
		\begin{equation}
			||(\frac{I + W_{t_{i+1}}}{4})H_i|| \leq ||\frac{I + W_{t_{i+1}}}{2}|| \cdot ||\frac{H_i}{2}|| \leq 1 \cdot \frac{1}{2}||W_{t_{i+1}} - W_{t_{i}}|| \leq \frac{1}{2},
		\end{equation}
		we have that  $G_i \circ F_i$ is invertible. Then $F_i$ is injective. With exactly with the same argument one can proof that $F_i \circ G_i$ is invertible and so $F_i$ is also surjective.
		\\Let us now consider the isometry $U_i: V_{+, \alpha(t_{i+1}) , \beta(t_{i+1})} \longrightarrow V_{+, \alpha(t_i), \beta(t_i)}$ as
		\begin{equation}
			U_i:= \frac{F_i}{|F_i|}.
		\end{equation}
		Finally we can define the isometry as the following composition
		\begin{equation}
			\mathcal{U}_{\alpha_0, +} := U_0 \circ U_{1} \circ ... \circ U_n : V_{+, \alpha_0, 1} \longrightarrow V_{+}.
		\end{equation}
		Now we have to prove that $\mathcal{U}_{\alpha_0, +}$ is the restriction to $V_{+, \alpha_0, 1}$ of an operator $I + P_{\alpha_0}$ where $P_{\alpha_0}$ is an operator in $C^*(M \sqcup N)^\Gamma$. First we have to observe that $G_i = F_i^\times$, indeed if $v$ is a vector of $V_{+, \alpha(t_{i+1}) , \beta(t_{i+1})}$ and $w$ is a vector of $V_{+, \alpha(t_i), \beta(t_i)}$ then we can observe that
		\begin{equation}
			\langle v, \frac{I + W_{t_{i+1}}}{2} w \rangle = \langle v, w \rangle = \langle \frac{I + W_{t_{i}}}{2} v, w \rangle.
		\end{equation}
		This means that $F_i^\times F_i(v) = Iv + (\frac{I + W_{t_{i+1}}}{4})H_iv$. Then, since $H_i$ is in $C^*(M \sqcup N)^\Gamma$, we have that
		\begin{equation}
			F_i^\times F_i = (I + L)_{|_{V_{+, \alpha(t_{i+1}) , \beta(t_{i+1})}}},
		\end{equation}
		where $L$ is an on operator in $C^*(M \sqcup N)^\Gamma$. So, exactly as we did in the last proposition, one can prove that
		\begin{equation}
			|F_i|^{-1} = (I + Q)_{|_{V_{+, \alpha(t_{i+1}) , \beta(t_{i+1})}}},
		\end{equation}
		where $Q$ is an operator in $C^*(M \sqcup N)^\Gamma$. This means
		\begin{equation}
			\begin{split}
				U_i = \frac{F_i}{|F_i|} &= (\frac{I + W_t}{2})(I + Q)\\
				&= (\frac{I + \tau}{2} + \frac{1}{2}H_{\alpha(t),\beta(t)})(I + Q)\\
				&= \frac{I + \tau}{2} + C^*(M \sqcup N)^\Gamma
			\end{split}
		\end{equation}
		Then we have that
		\begin{equation}
			\mathcal{U}_{\alpha_0} = (\frac{I + \tau}{2})^n + C^*(M \sqcup N)^\Gamma = \frac{I + \tau}{2} + C_f^*(M \sqcup N)^\Gamma.
		\end{equation}
	\end{proof}
	\subsection{Invariance of the Roe Index}
	Consider $(M,g)$ and $(N,h)$ two manifolds of bounded geometry and let $f: (M,g) \longrightarrow (N,h)$ be a $C^2_{b}$-map. Our first goal, in this subsection, is to defined a Roe index in $K_{n}(\frac{D_f^*(M \sqcup N)^\Gamma}{C^*_f(M \sqcup N)^\Gamma})$ of the signature operator on $M \sqcup N$.
	\\In order to do it, we have to introduce a lemma which relates the Roe and structure algebras of $M$ and $N$ with the Roe and structure algebras of $M \sqcup N$.
	\begin{lem}
		Consider $M$ and $N$ two Riemannian manifolds. Then if $X = M,N$ there is a $*$-homomorphism
		\begin{equation} 
			i_X: D^*(X)^\Gamma \longrightarrow D_f^*(M \sqcup N)^\Gamma 
		\end{equation}
		such that $i_X(C^*(X)^\Gamma) \subseteq C_f^*(M \sqcup N)^\Gamma$. We will denote by $\iota_X$ the map
		\begin{equation} 
			\iota_X: \frac{D^*(X)^\Gamma}{C^*(X)^\Gamma} \longrightarrow \frac{D^*_f(M \sqcup N)^\Gamma}{C^*_f(M \sqcup N)^\Gamma}
		\end{equation}
		induced by $i_X$.
	\end{lem}
	\begin{proof}
		We can start observing that $\mathcal{L}^2(M \sqcup N) = \mathcal{L}^2(N) \oplus \mathcal{L}^2(M)$, and so we can define $\pi_{M} : \mathcal{L}^2(M \sqcup N) \longrightarrow \mathcal{L}^2(M)$ and $\pi_{N} : \mathcal{L}^2(M \sqcup N) \longrightarrow \mathcal{L}^2(N)$. Then we have that 
		\begin{equation}
			D_c^*(M \sqcup N)^\Gamma = \bigoplus\limits_{Y,Z = M, N} D_c^*(M \sqcup N)_{Y,Z}^\Gamma,
		\end{equation}
		where $D_c^*(M \sqcup N)_{Y,Z}^\Gamma$ are the operators $A$ of $D_{c,f}^*(M \sqcup N)^\Gamma$ such that
		\begin{equation}
			W \neq Z \lor S \neq Y \implies \pi_{W} \circ A \circ \pi_{S} = 0. 
		\end{equation}
		The same decomposition can be considered for $C^*_{c,f}(M \sqcup N)^\Gamma$. As consequence we have that
		\begin{equation}
			\frac{D_f^*(M \sqcup N)^\Gamma}{C_f^*(M \sqcup N)^\Gamma} = \bigoplus\limits_{Y,Z = M, N} \frac{\overline{D_c^*(M \sqcup N)_{Y,Z}^\Gamma}}{\overline{C^*_c(M \sqcup N)_{Y,Z}^\Gamma}}.
		\end{equation}
		We will first consider the case $X=N$. We can define the map $i_N: D^*_c(N)^\Gamma \longrightarrow D^*_{c,f}(M \sqcup N)^\Gamma$ for all $A$ as the operator
		\begin{equation}
			i_N(A)(\alpha,\beta) = (A \alpha, 0).
		\end{equation}
		We have that $i_N(A)$ is a controlled operator. To prove this fact we consider two function $\phi, \psi \in C^{\infty}_c(M \sqcup N)$ (without loss of generality we can suppose that their supports are connected). Then if $supp(\phi)$ and $supp(\psi)$ are not contained in $N \subset M \sqcup N$ we have $\phi i_N(A) \psi = 0$. On the other hand if both are contained in $N$ we have that
		\begin{equation}
			d(f \sqcup id_N(supp(\phi)), f \sqcup id_N(supp(\psi))) = d(supp(\phi), supp(\psi)).
		\end{equation}
		Then if $A$ is a controlled operator, then same is true for $i_N(A)$. It is also true that if $A$ is pseudo-local or locally compact then also $i_1(A)$ is pseudo-local or locally compact. It's easy now to check that
		\begin{equation}
			D^*_c(N)^\Gamma \cong D^*_c(M \sqcup N)^\Gamma_{N,N}
		\end{equation}
		and
		\begin{equation}
			C^*_c(N)^\Gamma \cong C^*_c(M \sqcup N)^\Gamma_{N,N}.
		\end{equation}
		The isomorphism is $i_N$ and its inverse map is 
		\begin{equation}
			i^{-1}(A) = \pi_N \circ A_{|_{\mathcal{L}^2(N)}}.
		\end{equation}
		Since 
		\begin{equation}
			\numberset{B}(\mathcal{L}^2(M \sqcup N)) = \bigoplus_{Y, Z = M,N}\numberset{B}(\mathcal{L}^2(Y), \mathcal{L}^2(Z))
		\end{equation}
		and 
		\begin{equation}
			C^*_c(M \sqcup N)^\Gamma_{N,N} \subset \numberset{B}(\mathcal{L}^2(N), \mathcal{L}^2(N)) = \numberset{B}(\mathcal{L}^2(N))
		\end{equation}
		we have that
		\begin{equation}
			D^*(N)^\Gamma \cong \overline{D^*_c(M \sqcup N)^\Gamma_{N,N}}
		\end{equation}
		and
		\begin{equation}
			C^*(N)^\Gamma \cong \overline{C^*_c(M \sqcup N)^\Gamma_{N,N}}.
		\end{equation}
		Finally we can define the map $\iota_N$ just as
		\begin{equation}
			\iota_N([A]) := [i_N(A)].
		\end{equation}
		\\
		Let us consider the case $X = M$. We will define $i_M$ and $\iota_M$ exactly as we defined $i_N$ and $\iota_N$. We just have to prove that if $A$ is a controlled operator in $\numberset{B}(\mathcal{L}^2(M))$, then $i_M(A)$, which is the operator defined as
		\begin{equation}
			i_M(A) (\alpha, \beta) = (0, A\beta),
		\end{equation}
		is a controlled operator. Fix then $S  > C_f\cdot R$ where $C_f$ is the lipschitz constant of $f$ and $R$ is a constant such that if $\phi, \psi \in  C_c^\infty(M)$
		\begin{equation}
			d(supp(\phi), supp(\psi)) > R \implies \phi A \psi = 0.
		\end{equation}
		We want to show that if $\phi, \psi \in  C_c^\infty(M \sqcup N)$ (with out loss of generality we will consider $\phi$ and $\psi$ with connected support)
		\begin{equation}
			d((f \sqcup id_N) supp(\phi), (f \sqcup id_N) supp(\psi)) > S \implies \phi i_M(A) \psi = 0.
		\end{equation}
		We can observe that if $supp(\phi)$ and $supp(\psi)$ are not contained in $M \subset M \sqcup N$, then $\phi i_M(A) \psi = 0$ so we can consider $\phi, \psi \in  C_c^\infty(M)$. We have that
		\begin{equation}
			d(f(supp(\phi)), f(supp(\psi))) > S.
		\end{equation}
		Consider $x$ in $supp(\phi)$ and $y$ in $supp(\psi)$: then we have
		\begin{equation}
			S < d(f(x), f(y)) \leq C_fd(x,y)
		\end{equation}
		and so
		\begin{equation}
			d(x,y) > \frac{d(f(x), f(y))}{C_f} > \frac{S}{C_f} > R.
		\end{equation}
		Then it means that
		\begin{equation}
			\begin{split}
				&d(f(supp(\phi)), f(supp(\psi))) > S \implies \\ &d(supp(\phi), supp(\psi)) > R \implies \phi i_M(A) \psi = 0.
			\end{split}
		\end{equation}
	\end{proof}
	\begin{cor}
		The previous Lemma also holds considering the Roe and structure algebras defined using the $(M, \Gamma)$-module, the $(N,\Gamma)$-module and the $(M \sqcup N, \Gamma)$-module defined as in Example \ref{terna}. 
	\end{cor}
	\begin{cor}
		Let us consider a Riemannian manifold $X$. and let us denote by $\delta_{X}$ the transgression map 
		\begin{equation}
			\delta_{X}: K_n(\frac{D^*(X)^\Gamma}{C^*(X)^\Gamma}) \longrightarrow K_{n-1}(C^*(X)^\Gamma).
		\end{equation}
		Let us consider $X= M,N$. Then if $n$ is even
		\begin{equation}
			\delta_{M \sqcup N} \circ \iota_{X, \star} = i_{X,\star}^u \circ \delta_X,
		\end{equation}
		if $n$ is odd
		\begin{equation}
			\delta_{M \sqcup N} \circ \iota^u_{X, \star} = i_{X,\star} \circ \delta_X.
		\end{equation}
	\end{cor}
	\begin{proof}
		It is a consequence of the previous Lemma and Lemma \ref{K-comm} in Appendix C.
	\end{proof}
	Let us recall that the signature operator in $M \sqcup N$ is defined as
	\begin{equation}
		D_{M \sqcup N} := \begin{bmatrix}
			D_N && 0 \\
			0 && -D_M
		\end{bmatrix}.
	\end{equation}
	Let us consider a chopping function $\chi$. We have that 
	\begin{equation}
		\chi(D_{M \sqcup N}) = \begin{bmatrix} 
			\iota_N(\chi(D_{N})) && 0 \\
			0 && -\iota_M\chi(D_M)
		\end{bmatrix}
	\end{equation}
	is an operator in $\frac{D^*_f(M \sqcup N)^\Gamma}{C^*_f(M \sqcup N)^\Gamma}$.
	\\Exactly as in the connected case, we have the following definition.
	\begin{defn}
		The \textbf{fundamental class of $D_{M \sqcup N}$} is $[D_{M \sqcup N}] \in K_{n+1}(\frac{D_f^{*}(M \sqcup N)^{\Gamma}}{C_f^{*}(M \sqcup N)^\Gamma})$ given by
		\begin{equation}
			[D_{M \sqcup N}] := \begin{cases} 
				
				[\frac{1}{2}(\chi (D_{M \sqcup N})+1)] \text{if $n$ is odd,}
				\\ [\chi(D_{M \sqcup N})_+] \text{if $n$ is even.}
				
			\end{cases}
		\end{equation}
	\end{defn}
	\begin{rem}
		Again the definition in the even case is well-given since we are considering $\chi(D_{M \sqcup N})_+$ in $B(H_{M \sqcup N})$ where $H_{M \sqcup N}$ is defined as in Example \ref{terna}.
	\end{rem}
	\begin{defn}\label{doubleindex}
		The \textbf{Roe index} of $D_{M \sqcup N}$ is the class
		\begin{equation}
			{Ind_{Roe}(D_{M \sqcup N}) := \delta[D_{M \sqcup N}]}
		\end{equation}
		in $K_{n}(C_f^{*}(M \sqcup N)^{\Gamma})$, where $\delta$ is the connecting homomorphism in K-Theory.
	\end{defn}
	Let us now introduce a result of Higson and Roe, in particular from \cite{higroe}. The following paragraph, definitions and Proposition come directly from \cite{higroe}.
	\\
	\\Let us fix a $C^*$-subcategory $A$ of the category of all Hilbert spaces and bounded linear maps (this is an additive subcategory whose morphism sets $Hom_A(H_1, H_2)$ are Banach subspaces of the bounded linear operators from $H_1$ to $H_2$, and are closed under the adjoint operation). Let us also fix an ideal $J$ in $A$ (apart from the fact that we no longer require identity morphisms, this is a $C^*$-subcategory with the additional property that any composition of a morphism in $J$ with a morphism in $A$ is a morphism in $J$).
	\begin{defn}
		An unbounded, self-adjoint, Hilbert space operator $D$ is analytically controlled over $(A, J)$ if
		\begin{itemize}
			\item the Hilbert space on which it is defined is an object of $J$;
			\item the resolvent operators $(D \pm i)^{-1}$ are morphisms of $J$; and
			\item the operator $\frac{D}{\sqrt{1+D^2}}$ is a morphism of $A$.
		\end{itemize}
	\end{defn}
	\begin{prop}
		Let $D$ be an analytically controlled self-adjoint operator and let $g : [-\infty , + \infty] \longrightarrow \numberset{R}$ be a continuous function. If $C$ is a bounded self-adjoint operator in $A$ then the difference $g(D) - g(D + C)$ lies in $J$.
	\end{prop}
	Now, we will use this Proposition of \cite{higroe} to prove the following fact.
	\begin{prop}
		Let $(M,g)$ and $(N,h)$ be two manifolds of bounded geometry and let $f:(M,g) \longrightarrow (N,h)$ be a $C^2_{b}$-map which is a lipschitz-homotopy equivalence. Then if $n$ is odd we have that in $K^*(\frac{D_f^*(M \sqcup N)^\Gamma}{C_f^*(M\sqcup N)^\Gamma})$
		\begin{equation}
			[D_{M \sqcup N}] = [\frac{1}{2}(\chi(D_{M \sqcup N}) +1)] = [\frac{1}{2}(\chi(D_{\alpha,\beta}) +1)]
		\end{equation}
		and if $n$ is even
		\begin{equation}
			[D_{M \sqcup N}] = [Q\chi(D_{M \sqcup N})_+] = [ \mathcal{U}_{\alpha_0,-} \chi(D_{\alpha_0,1})\mathcal{U}^*_{\alpha_0,+}].
		\end{equation} 
	\end{prop}
	\begin{proof} 
		Let us start with the odd case. It's sufficient to apply Propostion of \cite{higroe} in the case $A$ is $D_f^*(M \sqcup N)^\Gamma$, $J$ is $C_f^*(M \sqcup N)^\Gamma$, $D$ is the signature operator, $g= \chi$ and $C = D_{\alpha, \beta} - D$. Indeed since
		\begin{equation}
			\chi(D_{\alpha,\beta}) - \chi(D) \in C^*_f(M \sqcup N)^\Gamma
		\end{equation}
		then they are the same element in $\frac{D_f^*(M \sqcup N)^\Gamma}{C_f^*(M \sqcup N)^\Gamma}$ and so they induce the same element in $K$-theory.
		\\
		\\For the even case it is sufficient, because of \ref{Cos}, to remind that
		\begin{equation}
			\mathcal{U}_{\alpha_0,-} - \frac{I - \tau}{2} \in C_f^*(M \sqcup N)^\Gamma
		\end{equation}
		and so
		\begin{equation}
			\mathcal{U}_{\alpha_0,+}^* - \frac{I + \tau}{2} \in C_f^*(M \sqcup N)^\Gamma.
		\end{equation}
		Now, applying \cite{higroe}, we also know that
		\begin{equation}
			\chi(D_{\alpha_0,1}) - \chi(D) \in C_f^*(M \sqcup N)^\Gamma
		\end{equation}
		and this means that
		\begin{equation}\label{poi}
			\chi(D_{M \sqcup N})_+ - \mathcal{U}_{\alpha_0,-} \chi(D_{\alpha_0,1})\mathcal{U}^*_{\alpha_0,+}
		\end{equation}
		is in $C_f^*(M \sqcup N)^\Gamma$ since all the operators here are in $D^*_f(M \sqcup N)^\Gamma$. Then we can conclude applying Remark \ref{girl}, indeed we have that  implies that the perturbation (\ref{poi}) can be seen in
		\begin{equation} 
			C^*_{\rho_{M \sqcup N}}(M \sqcup N, H_{M \sqcup N})^\Gamma
		\end{equation}
		where $H_{M \sqcup N}$ and $\rho_{M \sqcup N}$ are defined as in Example \ref{terna}.
	\end{proof}
	We can observe that since $D_{\alpha_0, 1}$ is invertible, one can choose as function $\chi$ a function which is $1$ on $spec(D) \cap (0, + \infty)$ e $\chi \equiv -1$ on $spec(D) \cap (-\infty, 0)$. This means that the following definition is well-given.
	\begin{defn}
		The \textbf{$\rho_f$-class} of $D$ is the class of $K_{n+1}(D_f^*(M \sqcup N)^\Gamma)$
		\begin{equation}
			\rho_f = \begin{cases} 
				[\frac{1}{2}(\chi (D_{\alpha_0, 1})+1)] \text{if n is odd,}
				\\ [(\mathcal{U}_{\alpha_0,-}) \chi(D_{\alpha_0,1})(\mathcal{U}^*_{\alpha_0,+})] \text{if n is even.}
			\end{cases}
		\end{equation} 
	\end{defn}
	\begin{prop}
		Let $(M,g)$ and $(N,h)$ be two manifolds of bounded geometry. Let $f:(M,g) \longrightarrow (N,h)$ be a lipschitz-homotopy equivalence which preserves the orientation. Then if $f$ is a $C^2_{b,u}$-map, then the Roe index of $M \sqcup N$ is zero.
	\end{prop}
	\begin{proof}
		Consider the long exact sequence in $K$-theory
		\begin{equation}
			.... \xrightarrow{i_\star} K_n(D_f^*(M \sqcup N)^\Gamma) \xrightarrow{p_\star} K_n(\frac{D_f^*(M \sqcup N)^\Gamma}{C_f^*(M \sqcup N)^\Gamma}) \xrightarrow{\delta} K_{n+1}(C_f^*(M \sqcup N)^\Gamma) \longrightarrow ...
		\end{equation}
		Now, we know that $Ind_{Roe}(D_{M \sqcup N}) = \delta([D_{M \sqcup N}]$. But in this case we know that
		$\delta([D_{M \sqcup N}] = \delta(p_\star(\rho_f(D_{M \sqcup N}))) = 0$. Then the Roe index vanishes.
	\end{proof} 
	We are ready now to prove the lipschitz-homotopy invariance of the Roe index.
	\begin{thm}
		Let $(M,g)$ and $(N,h)$ be two manifolds of bounded geometry. Let $\Gamma$ be a group acting uniformly proper, discontinuous and free on $M$ and $N$ by orientation-preserving isometries. Consider $f:(M,g) \longrightarrow (N,h)$ a $\Gamma$-equivariant lipschitz-homotopy equivalence which preserves the orientations. Then
		\begin{equation}
			f_\star(Ind_{Roe}(\mathcal{D}_M)) = Ind_{Roe}(\mathcal{D}_N).
		\end{equation}
	\end{thm}
	\begin{proof}
		Suppose that $\Gamma$ acts FUPD. Since all $\Gamma$-invariant lipschitz map are coarsely approximable by $\Gamma$-invariant $C^2_b$-maps (Proposition \ref{appr}), we can consider $f$ as a $C^2_b$-lipschitz-homotopy equivalence. 
		\\Consider a $\Gamma$-invariant isometry $A: \mathcal{L}^2(N) \longrightarrow \mathcal{L}^2(M)$ which covers $f$. We know that it exists by Proposition 4.3.5. of \cite{siegel}. Then we know that
		\begin{equation}
			f_\star([D_M]) = Ad_{V,\star}^{(u)}[D_M] = [V D_M V^*],
		\end{equation}
		where we have the $u$ if $n$ is odd and we haven't it if $n$ is even.
		\\Let us consider now $C_f^*(M \sqcup N)^\Gamma$. We already know the decomposition
		\begin{equation}
			C_f^*(M \sqcup N)^\Gamma = \bigoplus\limits_{X,Y = M,N} C^*(M \sqcup N)_{X,Y}^\Gamma.
		\end{equation}
		Let us define now the operator $H: C_f^*(M \sqcup N)^\Gamma \longrightarrow C^*(N)^\Gamma$ defined for all $T$ in $C^*(M \sqcup N)_{M,M}^\Gamma$ as
		\begin{equation}
			Ad_V(T) = V T V^*,
		\end{equation}
		as $i_N^{-1}$ on $C^*(M \sqcup N)_{N,N}^\Gamma$ and zero for all the others operators. It's easy to see that $H$ is a $*$-homomorphism, and so it means that $H$ induce a map between the $K$-groups. In particular we have that if $n$ is odd, then
		\begin{equation}
			\begin{split}
				H_\star^{u} (\delta_{M \sqcup N}[\mathcal{D}_{M\sqcup N}])
				&= H_\star \circ i_{M, \star}^{u}(\delta_M [\mathcal{D}_M]) - H_\star^u \circ i_{N, \star}^{u}(\delta_{N}[\mathcal{D}_N]) \\
				&= (H \circ i_M)^u_\star (Ind_{Roe}(\mathcal{D}_M)) - (H \circ i_N)^u_\star (Ind_{Roe}(\mathcal{D}_N)) \\
				&= Ad_{V,\star}^u(Ind_{Roe}(\mathcal{D}_M)) - (id_N)^u_\star(Ind_{Roe}(\mathcal{D}_N))\\
				&= f_\star (Ind_{Roe}(\mathcal{D}_M)) - Ind_{Roe}(\mathcal{D}_N).
			\end{split} 
		\end{equation}
		and if $n$ is even we have that
		\begin{equation}
			\begin{split}
				H_\star (\delta[\mathcal{D}_{M\sqcup N}]) &= H_\star(\delta_{M \sqcup N} (i_{M, \star}^{u}[\mathcal{D}_M])) - H_\star(\delta_{M \sqcup N} (i_{N, \star}^{u}[\mathcal{D}_N])) \\
				&= H_\star\circ i_{M, \star}(\delta_M [\mathcal{D}_M]) - H_\star \circ i_{N, \star}(\delta_{N}[\mathcal{D}_N]) \\
				&= (H \circ i_M)_\star (Ind_{Roe}(\mathcal{D}_M)) - (H \circ i_N)_\star (Ind_{Roe}(\mathcal{D}_N)) \\
				&= Ad_{V,\star}(Ind_{Roe}(\mathcal{D}_M)) - (id_N)_\star(Ind_{Roe}(\mathcal{D}_N))\\
				&= f_\star (Ind_{Roe}(\mathcal{D}_M)) - Ind_{Roe}(\mathcal{D}_N).
			\end{split} 
		\end{equation}
		Finally we know that $\delta_{M \sqcup N}[\mathcal{D}_{M\sqcup N}] = 0$ which implies
		\begin{equation}
			H^{(u)}_\star(\delta_{M \sqcup N}[\mathcal{D}_{M\sqcup N}]) = 0,
		\end{equation}
		and so
		\begin{equation}
			f_\star (Ind_{Roe}(\mathcal{D}_M)) = Ind_{Roe}(\mathcal{D}_N).
		\end{equation}
	\end{proof}
	\section{Lipschitz-homotopy stability of $\rho_f$}
	\subsection{Lipschitz-homotopy stability of $D^*_f(M \sqcup N)^\Gamma$ and $C^*_f(M \sqcup N)^\Gamma$}
	In the previous sections we saw that given a lipschitz-homotopy equivalence, one can define a coarse structure on $M \sqcup N$ and a class $\rho_f$ in the K-theory of $D_f^*(M \sqcup N)^\Gamma$. Starting by now, we would like to study lipschitz-homotopy equivalence which are lipschitz-homotopic.
	\\The main goal of this section is to prove that if $f \sim_\Gamma f'$ then they induce the same coarse structure on $M \sqcup N$ and, in particular,
	\begin{equation}
		f \sim_\Gamma f' \implies \rho_{f}=\rho_{f'},
	\end{equation}
	\begin{lem}
		Let us consider $f$ and $f':(M,g) \longrightarrow (N,h)$ two lipschitz map. Then if $f \sim_\Gamma f'$ for some $\Gamma$, then they are close and so they induce the same coarse structure on $M \sqcup N$.
	\end{lem}
	\begin{proof}
		We will say that two maps $\phi_1, \phi_2: S \longrightarrow M \sqcup N$. Given a lipschitz-homotopy equivalence $g:M \longrightarrow N$, we will say that $\phi_1$ and $\phi_2$ are $g$-close if they are close respect to the coarse structure induced by $g$. We will denote this relation as $\phi_1 \sim_f \phi_2$.
		\\In particular we know that
		\begin{equation}
			\phi_1 \sim_g \phi_2 \iff \exists R \geq 0 \mbox{     such  that    } d(g \sqcup id_N \circ \phi_1(x), g \sqcup id_N \circ \phi_2(x)) \leq R.
		\end{equation}
		If we are able to show that $\phi_1 \sim_f \phi_2$ implies $\phi_1 \sim_{f'} \phi_2$ then we conclude the proof.
		\\Consider $H: (M\times [0,1], g \times g_{eucl}) \longrightarrow (N,h)$ the lipschitz homotopy between $f$ and $f'$. Then we have that $H \sqcup id_N$ is a lipschitz-homotopic equivalence between $f \sqcup id_N$ and $f' \sqcup id_N$. Then we have that for every $y$ in $M \sqcup N$
		\begin{equation}
			\begin{split}
				d(f \sqcup id_N (y), f' \sqcup id_N (y)) &= d(H \sqcup id_N (y,0), H \sqcup id_N (y,1)) \\
				&\leq C_H d((y,0), (y,1)) = C_H,
			\end{split}
		\end{equation}
		where $C_H$ is the lipschitz constant of $H$. Let us suppose now $\phi_1 \sim_f \phi_2$: we have that
		\begin{equation}
			\begin{split}
				d(f' \sqcup id_N(\phi_1(x)), f' \sqcup id_N(\phi_2(x))) &\leq d(f' \sqcup id_N(\phi_1(x)), f \sqcup id_N(\phi_1(x))) \\
				&+ d(f \sqcup id_N(\phi_1(x)), f \sqcup id_N(\phi_2(x)))\\ 
				&+ d(f \sqcup id_N(\phi_2(x)), f' \sqcup id_N(\phi_2(x))) \\
				&\leq C_H + R + C_H \leq K
			\end{split}
		\end{equation}
		and so it means that $\phi_1 \sim_{f'} \phi_2$.
	\end{proof}
	\begin{cor}\label{curvaph}
		Given two lipschitz-homotopy equivalences $f, f': (M,g) \longrightarrow (N,h)$, if $f \sim_\Gamma f'$ then
		\begin{equation}
			C^*_f(M \sqcup N)^\Gamma = C^*_{f'}(M \sqcup N)^\Gamma
		\end{equation}
		and
		\begin{equation}
			D^*_f(M \sqcup N)^\Gamma = D^*_{f'}(M \sqcup N)^\Gamma.
		\end{equation}
		Obviously the same also holds considering the Roe and structure algebras defined using the $(M \sqcup N, \Gamma)$-module defined in Example \ref{terna}.
	\end{cor}
	\subsection{The curves $T_s$ and $dT_s$}
	\begin{prop}
		Let us consider two manifolds of bounded geometry $(M,g)$ and $(N,h)$. Consider, moreover, a group $\Gamma$ which acts FUPD on $(M,g)$ and $(N,h)$. Let $f$ and $f':(M,g) \longrightarrow (N,h)$ be two $C^2_b$-maps which are lipschitz-homotopy equivalences. Let us suppose, moreover, that $f \sim_{\Gamma} f'$. Then there is a homotopy $h$ such that the curve in $C^*_f(M \sqcup N)^\Gamma$ defined as
		\begin{equation}
			\begin{split}
				T: [0,1] &\longrightarrow C^*_f(M \sqcup N)^\Gamma \\
				s &\longrightarrow T_s := T_{h(\cdot, s)}
			\end{split}
		\end{equation}
		and the curve $dT$ defined for each $s$ as $dT(s) :=dT_s$ are continuous.
	\end{prop}
	\begin{proof}
		We already know that for each homotopy $h$ and for all $s \in [0,1]$ we have that $T_s \in C^*_f(M \sqcup N)^\Gamma$. Then we just have to prove that
		\begin{equation}
			\lim\limits_{\epsilon \rightarrow 0}||T_s - T_{s + \epsilon}|| = 0.
		\end{equation}
		To show this fact we consider a homotopy $h$ between $f$ and $f'$ such that its first and second derivatives in normal coordinates are uniformly bounded: we know that it exists since Lemma \ref{C^k_bhomo}.
		\\Consider the submersion $p_h$ defined as in Proposition \ref{tilde}. Because of Remark \ref{salva}, $p_H^*$ is a R.-N.-lipschitz map. Let us consider the operator $T_h$: following the proof we give for $T_f$ (Proposition \ref{Tsmooth}) one can check that $T_h$ is a smoothing operator. 
		\\Then if we check that for each $s\in [0,1]$ the operators $T_{s}$, $T_{s + \delta}$ and $T_h$ satisfy the conditions of Lemma \ref{giga}, and the same happens to $dT_{s}$, $dT_{s + \delta}$ and $dT_h$ then we can conclude.
		Consider $\omega$ a Thom form of $TN$ defined as in Subsection \ref{forma}. Let us denote by $H: h^*TN \longrightarrow TN$ and by $F_s: f^*_sTN \longrightarrow TN$ the bundle maps induced by $h, f_s$. Then the operators $T_h$ and $T_s$ are defined using $H^*\omega$ and $F_s^*\omega$. Define, moreover, the maps
		\begin{equation}
			t_h = (id_M \times [0,1], p_h)_: h^*TN \longrightarrow M \times [0,1] \times N
		\end{equation}
		and for each $w$ in $[0, 1]$ consider
		\begin{equation}
			t_s = (id_M, p_{f_{s}}): f^*_sTN \longrightarrow M \times N
		\end{equation}
		We have that the kernel of $T_h$, around $im(t_h)$, is locally given by
		\begin{equation}
			K(x, y, s) = \beta_{I,J}(x,y,s) dx^I \boxtimes \frac{\partial}{\partial y^J} + \beta_{IJ0}(x,y,s) dx^I\wedge ds \boxtimes  \frac{\partial}{\partial y^J}
		\end{equation}
		where $\beta_{IJ}$ and $\beta_{IJ0}$ are the components of $(t^{-1}_h)^*H^*\omega$. Observe that $F_s = H \circ I_s$ where $I_s: f^*_sTN \longrightarrow h^*TN$ is the map induced by $i_s(p) := (p, s)$. Then we have that the kernel
		of $T_s$, around $im(t_s)$, is given by
		\begin{equation}
			K_s(x,y) = \alpha_{IJ}(x,y) dx^I \boxtimes \frac{\partial}{\partial y^J}
		\end{equation}
		where $\alpha_{IJ}$ are the components of 
		\begin{equation}
			(t_s^{-1})^*F_s^*\omega = (t_s^{-1})^* I_s^* H^* \omega = i_s^*(t_h^{-1})^*H^*\omega.
		\end{equation}
		This means that
		\begin{equation}
			\alpha_{IJ} = i_s^*(\beta_{I,J}).
		\end{equation}
		This means that the first condition of Lemma \ref{giga} on the kernels of $T_s$, $T_{s + \delta}$ and $T_h$ is satisfied.
		\\We know that $T_h$ and $dT_h$ have uniformly bounded support (which is the first condition), because of Proposition \ref{dT_fboun}.
		\\Moreover, the second condition, which is the bound on the derivative of the kernel of $T_h$ and $dT_h$
		\begin{equation}\label{condiz}
			|\frac{\partial}{\partial s} K^I_J(x,s)| \leq C,
		\end{equation} 
		is satisfied for both $T_h$ and $dT_h$ since the first and second derivatives in normal coordinates of $h$ uniformly bounded. 
		\\Consider, indeed, the map $t_h := (\pi, p_h)$: we know, because of Proposition \ref{Tsmooth}, that the components of the kernel of $T_H$ in local normal coordinates are given by an algebraic combination of the pullback along $t_{h}^{-1}$ of the components of the Thom form $\omega$ written in fibered coordinates. Then, since $h$ is a lipschitz map, the inequality (\ref{condiz}) is satisfied for $T_h$.
		\\Moreover, applying Lemma \ref{impo}, we also have that the inequality (\ref{condiz}) is satisfied for $dT_h$ if the second derivatives of $h$ are uniformly bounded in normal coordinates.
	\end{proof}
	\begin{cor}
		Under the same assumptions, also the operator $dT^\dagger_{s} = T^\dagger_{s}d$ define a curve in $C_f^*(M \sqcup N)^\Gamma$.
	\end{cor}
	\begin{proof}
		We can observe that 
		\begin{equation}
			T_{s}^\dagger d = (dT_{s})^\dagger = \tau \circ (dT_{s})^* \circ \tau
		\end{equation}
		Then, since $dT_{f_s}$ is a continuous curve in $C^*_f(M \sqcup N)^\Gamma$, that $\tau$ is an operator in $D^*_f(M \sqcup N)^\Gamma$ and that the adjoint
		\begin{equation}
			\begin{split}
				* : C^*_f(M \sqcup N)^\Gamma &\longrightarrow C^*_f(M \sqcup N)^\Gamma\\
				A &\longrightarrow A^*
			\end{split}
		\end{equation}
		is continuous we conclude that $T_{s}^\dagger d = d T_{s}^\dagger$ defines a continuous curve in $C^*_f(M \sqcup N)^\Gamma$.
	\end{proof}
	\subsection{The curves $T_sy_s$ and $dT_sy_s$}
	Let us consider for each $s \in [0,1]$ a lipschitz-homotopy equivalence $f_s:(M, g) \longrightarrow (N,h)$ such that the map $h:(M \times [0,1], g + g_{eucl}) \longrightarrow (N,h)$ is a lipschitz map. Then we can denote by $y_s$ the $y$ operator defined for each $f_s$. Our next goal is to prove that  $T_sy_s$ and $dT_sy_s$ define two continuous curves in $\numberset{B}(\mathcal{L}^2(N))$.
	\\We can start observing that
	\begin{equation}
		y_s := \frac{Y_s + Y^\dagger_s}{2}
	\end{equation}
	where
	\begin{equation}
		Y_s := (p_{f_s, 2})_{\star} \circ e_{\omega_2} \circ pr_{2\star} \circ e_{\omega_1} \circ \int_{0,\mathcal{L}}^1 \circ C \circ p_{f_s,1},
	\end{equation}
	and\footnote{In what follows we identify the bundles $f_s^*(TN)$. Actually, the only
		difference between $f_s^*(TN)$ and $f_{s + \delta}^*TN$ concern the Sasaki metrics on them, but it is irrelevant in this
		Subsection since Proposition \ref{bundlelem}.}
	\begin{itemize}
		\item $p_{f_s,i}: f^*(T^\delta N)_i \longrightarrow N$ is the map $p_{f_s}$,
		\item $pr_{i}: \mathcal{B}= f^*(T^\delta N)_1 \oplus f^*(T^\delta N)_2 \longrightarrow f^*(T^\delta N)_i$, is the projection on the $i$th component,
		\item $e_{\omega,i}$ is the operator of wedging by $pr_{i}^*F_s^*\omega$ where $pr_i: \mathcal{B} \longrightarrow f^*(T^\delta N)_i$ is the projection on the $i$ component,
		\item $C: f^*(T^\delta N)_1 \oplus f^*(T^\delta N)_2 \times [0,1] \longrightarrow f^*(T^\delta N)_1$ is the map defined as
		\begin{equation}
			C(v_{f(p)1}, v_{f(p)2}, w) := wv_{f(p)1} + (1-w)v_{f(p)2}.
		\end{equation}
	\end{itemize}
	We can observe that $ \int_{0,\mathcal{L}}^1$ commutes with $e_{\omega_i}$, $pr_{2, \star}$. This happens because $ \int_{0,\mathcal{L}}^1$, up to signs, is an integration along fibers and so we can apply the Proposition VII of Chapter 9 of \cite{Conn}. Then we can see $Y_s$, up to signs, as
	\begin{equation}
		Y_s =  \int_{0,\mathcal{L}}^1 \circ (p_{f_s, 2}, id_{[0,1]})_{\star} \circ e_{\omega_2} \circ (pr_{2}, id_{[0,1]})_{\star} \circ e_{\omega_1} \circ \tilde{P}_{f_s}^*
	\end{equation}
	where $(pr_{2}, id_{[0,1]}): \mathcal{B} \times [0,1] \longrightarrow f^*(T^\delta N) \times [0,1]$ is the standard projection and $\tilde{P}_{f_s} :=  p_{f_s} \circ C$.
	\begin{lem}\label{torchio}
Let $p_H: (M \times [0,1], g + g_{[0,1]}) \longrightarrow (N,h)$ be a lipschitz submersion. Suppose, moreover that
\begin{equation}
\begin{split}
(p_H, id_[0,1]): (M \times [0,1], g + g_{[0,1]}) &\longrightarrow (N \times [0,1], h + g_{[0,1]}) \\
(p,s) &\longrightarrow (p_H(p,s),s)
\end{split}
\end{equation}
is a R,-N.-lipschitz submersion. Let us denote, for each $t \in [0,1]$, by $p_t$
the map 
\begin{equation}
p_t:(M,g) \longrightarrow (N,h)
\end{equation}
defined as $p_t(p) := p_H(p,t)$. We have that $p_t$ is a R.-N.-lipschitz submersion for each $t$. Moreover if $\omega \in \Omega^*_{cv}(M)$ is a closed form such that $|\omega|_p$ is uniformly bounded, then for each smooth $\alpha$ in $dom(d_M)$ we have that
\begin{equation}
p_{t\star} \circ e_\omega - p_{t + \delta\star}\circ e_\omega = A_\delta d \alpha,
\end{equation}
where $A_\delta$ is a bounded operator with norm less or equal to $\delta \cdot C$ for some $C$ in $\numberset{R}$. This, in particular implies that
\begin{equation}
||p_t \circ e_\omega - p_{t+\delta} \circ e_\omega|| \leq \delta C ||d\alpha||.
\end{equation}
Moreover we also have that
\begin{equation}
e_\omega \circ p_t^* - e_\omega \circ p_{t+\delta}^*(\alpha) = \pm A\dagger d\alpha.
\end{equation}
and so
\begin{equation}\label{eccolo2}
||e_\omega \circ p_t^* - e_\omega \circ p_{t+\delta}^*(\alpha)|| \leq \delta \cdot C \cdot d\alpha.
\end{equation}
\end{lem}
\begin{proof}
Observe that if $(p_h, id_{[0,1]})$ is a lipschitz submersion then also $p_t$ is a lipschitz submersion. Moreover the Fiber Volume of $p_t$ is the Fiber Volume of $(p_h, id_{[0,1]})$ in $(p, t)$. Then we have that $p_t$ is a R.-N.-lipschitz submersion.
Consider $s \in [t, t + \delta]$. We denote by $F_s$ the fiber of $p_s$. Observe that, since the Stokes Theorem, we have that
\begin{equation}
	\begin{split}
	p_t \circ e_\omega - p_{t + \delta} \circ e_\omega(\alpha)  &= \int_{F_{t}} \alpha \wedge \omega - \int_{F_{t+\delta}} \alpha \wedge \omega \\
	&= \int_{F_{t}\times \{t\}} \alpha \wedge \omega - \int_{F_{t+\delta}\times \{t\}} \alpha \wedge \omega\\
	&= \int_{t}^{t + \delta} \int_{\partial F_s \times \{s\}} \alpha \wedge \omega + d \int_{t}^{t + \delta} \int_{F_s \times \{s\}} \alpha \wedge \omega \\
	&+ \int_{t}^{t + \delta} \int_{F_s \times \{s\}} d\alpha \wedge \omega.
	\end{split}
\end{equation}
Observe that the first integral is null since $\omega$ is null on $\partial F$. The second integral is also null
since $\alpha$ is $0$-form respect to $[t, t + \delta]$. Then we obtain that
\begin{equation}\label{eccolo}
	p_t \circ e_\omega - p_{t + \delta} \circ e_\omega(\alpha) 	= pr_{M\star} \circ (p_h, id_{[t,t + \delta]})_\star \circ e_\omega (d\alpha),
\end{equation}
where $pr_M : M \times [t, t + \delta ] \longrightarrow M$ is the projection on the first component. Then $||pr_{M\star}|| \leq \delta \cdot C_1$. Moreover, since $|\omega_p|$ is uniformly bounded then $e_\omega$ is $\mathcal{L}^2$-bounded. This means that
\begin{equation}
	||p_t \circ e_\omega - p_{t + \delta} \circ e_\omega(\alpha)|| 	\leq \delta \cdot C_1 \cdot C_2 \cdot C_3 \cdot ||d\alpha||.
\end{equation}
Then, observing that
\begin{equation}
(p_t \circ e_\omega - p_{t + \delta} \circ e_\omega)^\dagger = \pm e_\omega \circ p_t^* \mp e_\omega \circ p_{t + \delta}^*
\end{equation}
We have, since (\ref{eccolo})
\begin{equation}
	\begin{split}
	e_\omega \circ p_t^* - e_\omega \circ p_{t + \delta}^*\alpha &= \pm (d \circ e_\omega \circ (p_h, id_{[t, t + \delta]})^* \circ pr_M^*) \alpha \\
	&= \pm (e_\omega \circ (p_h, id_{[t, t + \delta]})^* \circ pr_M^*) d\alpha
	\end{split}
\end{equation}
Then we also have (\ref{eccolo2}).
\end{proof}
\begin{prop}
The curves $y_sT^\dagger_s$, $T_sy_s$, $dT_sy_s$, $dy_sT^\dagger_s$, $y_sT^\dagger_sd$ and $T_sy_sd$ are continuous in $s$.
\end{prop}
\begin{proof}
Observe that for each $s$ in $[0, 1]$ we have that
\begin{equation}
||Y_{s + \delta}T_{s + \delta}^\dagger - Y_sT^\dagger_{s}|| \leq ||Y_{s + \delta}|| \cdot ||T^\dagger_{s + \delta} - T^\dagger_s|| + ||(Y_{s + \delta} - Y_s)T^\dagger_s||.
\end{equation}
Observe that the norm of $Y_s$ is uniformly bounded because the norm of $p_{f_s}$ is uniformly bounded in $s$. Indeed the norm of $p_{f_s}$ is bounded by a constant which continuously depends on the lipschitz constant and the Fiber Volume of $p_{f_s}$. However, these are both controlled by the lipschitz constant and the Fiber Volume of $p(h,id_{[0,1]}$. This means that, in order to show that $Y_sT^\dagger_s$ is continuous in $s$ we have to prove that
\begin{equation}
||(Y_{s + \delta} - Y_s)T^\dagger_s|| \leq C \cdot \delta.
\end{equation}
Observe that
\begin{equation}
	\begin{split}
	Y_{s + \delta} - Y_s &= \int_{0,\mathcal{L}}^1 \circ [(p_{f_{s + \delta},2}, id_{[0,1]})_\star \circ e_{\omega_2} - (p_{f_{s},2}, id_{[0,1]})_\star \circ e_{\omega_2}] \circ (pr_2, id_{[0,1]})^* \circ e_{\omega_1} \circ \tilde{P}_{f_s}^* \\
	&+ \int_{0,\mathcal{L}}^1 \circ (p_{f_{s},2}, id_{[0,1]})_\star \circ e_{\omega_2} (pr_2, id_{[0,1]})^* \circ [ e_{\omega_1} \circ \tilde{P}_{f_{s + \delta}}^* - e_{\omega_1} \circ \tilde{P}_{f_{s}}^*].
	\end{split}
\end{equation}
Let us denote by $A$ the operator $(pr_2, id_{[0,1]})* \circ e_{\omega_1} \circ  \tilde{P}_{f_{s}}^*$. Observe that
\begin{equation}
Ad\alpha = dA\alpha
\end{equation}
for each smooth form $\alpha$. Observe that, applying Lemma \ref{torchio},
\begin{equation}
\begin{split}
&|| \int_{0,\mathcal{L}}^1 \circ [(p_{f_{s + \delta},2}, id_{[0,1]})_\star \circ e_{\omega_2} - (p_{f_{s},2}, id_{[0,1]})_\star \circ e_{\omega_2}]\circ A(\alpha) ||\\
 &\leq C_1 \cdot ||[(p_{f_{s + \delta},2}, id_{[0,1]})_\star \circ e_{\omega_2} - (p_{f_{s},2}, id_{[0,1]})_\star \circ e_{\omega_2}]\circ A(\alpha) || \\
&\leq C_1 \cdot C \cdot ||dA\alpha|| \\
&\leq  C_1 \cdot C \cdot ||A d\alpha||\\
&\leq  C_1 \cdot C \cdot K \cdot ||d\alpha||.
\end{split}
\end{equation}
Moreover we also have that
\begin{equation}
\begin{split}
&||\int_{0,\mathcal{L}}^1 \circ (p_{f_{s},2}, id_{[0,1]})_\star \circ e_{\omega_2} (pr_2, id_{[0,1]})^* \circ [ e_{\omega_1} \circ \tilde{P}_{f_{s + \delta}}^* - e_{\omega_1} \circ \tilde{P}_{f_{s}}^*](\alpha)|| \\
&\leq C \cdot ||e_{\omega_1} \circ \tilde{P}_{f_{s + \delta}}^* - e_{\omega_1} \circ \tilde{P}_{f_{s}}^*(\alpha)||\\
&\leq \delta \cdot C \cdot ||d\alpha||. 
\end{split}
\end{equation}
This means that
\begin{equation}
||(Y{s+\delta} - Y_s)\alpha|| \leq \delta \cdot R \cdot ||d\alpha||.
\end{equation}
and so
\begin{equation}
	\begin{split}
	||(Y{s+\delta} - Y_s)T^\dagger_s\alpha|| &\leq \delta \cdot K_1 \cdot ||dT^\dagger_s\alpha||\\
	&\leq \delta \cdot K_2 \cdot ||\alpha||
	\end{split}
\end{equation}
Then we have that
\begin{equation}
	||(Y{s+\delta} - Y_s)T^\dagger_s|| \leq \delta \cdot K_2
\end{equation}
and the curve $Y_sT^\dagger_s$ is continuous in $s$. Observe that, up to signs, we have that
\begin{equation}
\begin{split}
Y_s^\dagger &= \tilde{P}_{f_{s}\star} \circ   e_{\omega_1} \circ (pr_2, id_{[0,1]})^* \circ e_{\omega_2} \circ (p_{f_{s},2}, id_{[0,1]})^* \circ [\int_{0,\mathcal{L}}^1]^\dagger\\
&= \tilde{P}_{f_{s}\star} \circ   e_{\omega_1} \circ (pr_2, id_{[0,1]})^* \circ e_{\omega_2} \circ (p_{f_{s},2}, id_{[0,1]})^* \circ pr_{[0,1]}^*
\end{split}
\end{equation}
where $pr_{[0,1]} N \times [0,1] \longrightarrow N$is the projection on the first component. Then, observe
that, using the same arguments we used for $Y_s$, we obtain that
\begin{equation}	
||Y^\dagger_{s + \delta} - Y^\dagger_s (\alpha)|| = \delta \cdot K_3 ||d \alpha||
\end{equation}
and so
\begin{equation}	
	||(Y^\dagger_{s + \delta} - Y^\dagger_s)T^\dagger_s \alpha|| = \delta \cdot C ||\alpha||
\end{equation}
This means that
\begin{equation}	
	||(Y^\dagger_{s + \delta} - Y^\dagger_s)T^\dagger_s|| = \delta \cdot C 
\end{equation}
and so $Y_s^\dagger T^\dagger_s$, $y_sT^\dagger_s$ and $T_s y_s$ are continuous in $s$. 
\\Finally, observe that
\begin{equation}
(y_s - y_{s + \delta})dT^\dagger_s = 0
\end{equation}
indeed for each smooth $\alpha$ we have that
\begin{equation}
||(y_s - y_{s + \delta})dT^\dagger_s \alpha|| \leq \delta \cdot C \cdot ||d^2T^\dagger_s \alpha|| = 0.
\end{equation}
Then we obtain that also $y_sdT^\dagger_s$ and $T_sdy_s$ are continuous on $s$. Moreover also
\begin{equation}
dy_sT^\dagger_s = (1 - T^\dagger_sT_s + y_sd)T^\dagger_s
\end{equation}
and
\begin{equation}
T_sy_sd = T_s(1 - T^\dagger_sT_s + dy_s)
\end{equation}
are continuous in s.
\end{proof}	
	\subsection{Lipschitz-homotopy stability of $\rho_f$}
	Starting by now we will denote by $L_{\alpha, \beta, s}$ the operator that we called $L_{\alpha, \beta}$ in subsection \ref{pert}, but defined using $T_s$. We know that if $\alpha = 0$, then $L_{0, \beta, s}$ is invertible. It follows by the the definition of $L_{\alpha, \beta, s}$ that $L_{0,\beta, s}$ can be seen as a compact continuous curve in $\numberset{B}(\mathcal{L}^2(M \sqcup N))$. Then there is an $\alpha_0 > 0$ such that $L_{\alpha_0, \beta, s}$ are invertible for all $s$.
	\\Then we will denote by $\delta_{\alpha_0, s}$, $U_{\alpha_0, \beta, s}$ and $S_{\alpha_0, \beta,s}$ the operators $\delta_{\alpha_0}$, $U_{\alpha_0, \beta}$ and $S_{\alpha_0, \beta}$ introduced in subsection \ref{pert} defined using $T_s$. \\Finally we denote by $D_{\alpha_0, \beta, s}$ the perturbed signature operator defined using $T_s$.
	\begin{prop}
		Let us consider two orientated Riemannian manifolds of bounded geometry $(M,g)$ and $(N,h)$. Consider a group $\Gamma$ acting by isometries on $(M,g)$ and on $(N,h)$ that preserve the orientations. Let $f_0, f_1: (M,g) \longrightarrow (N,h)$ two lipschitz-homotopy equivalence which preserves the orientation and let us suppose that $f_0 \sim_\Gamma f_1$. Then the difference between
		\begin{equation}
			D_{\alpha_0, \beta, s} - D
		\end{equation}
		defines a continue curve in $C^*_f(M \sqcup N)^\Gamma$.
	\end{prop}
	\begin{proof}
		We know, beacause of Proposition 14.5, that for all $s \in [0,1]$
		\begin{equation}
			D_{\alpha_0, \beta, s} - D \in C^*_f(M \sqcup N)^\Gamma,
		\end{equation}
		then we just have to prove that 
		\begin{equation}
			D_{\alpha_0, \beta, s} = D + \gamma(s),
		\end{equation}
		where $\gamma: [0,1] \longrightarrow C^*_f(M \sqcup N)^\Gamma$. We know that
		\begin{equation}
			S_{\alpha, \beta, s} = \frac{\tau L_{\alpha,\beta, s}}{|\tau L_{\alpha,\beta, s}|}.
		\end{equation}
		It's easy to check that since $T_s$ is continuous in $s$, then the same happens to $\tau L_{\alpha,\beta, s}$. Then, using the boundedness of $f(x) = \frac{x}{|x|}$ and the uniformly boundedness of the operators $\tau L_{\alpha,\beta, s}$, then we have that $S_{\alpha, \beta, s}$ is continuous in $S$. This means that if we consider as in Proposition 14.5
		\begin{equation}
			S_{\alpha, \beta, s} = \tau + H_{\alpha,\beta, s},
		\end{equation}
		where $H_{\alpha,\beta, s}$ is a continuous curve on $C^*_f(M \sqcup N)^\Gamma$. Using the same arguments one can easily prove that
		\begin{equation}
			U_{\alpha, \beta, s} = 1 + G_{\alpha, \beta, s}
		\end{equation}
		and
		\begin{equation}
			U_{\alpha, \beta, s}^{-1} = 1 + K_{\alpha, \beta, s}
		\end{equation}
		where $G_{\alpha, \beta, s}$ and $K_{\alpha, \beta, s}$ are operators in $C^*_f(M \sqcup N)^\Gamma$ continuous in $s$.
		\\The next step is to prove that $G_{\alpha, \beta, s}d$, $H_{\alpha,\beta, s} d$, $K_{\alpha, \beta, s}d$, $dG_{\alpha, \beta, s}$,  $d H_{\alpha,\beta, s}$, $dK_{\alpha, \beta, s}$ are continuous operators in $s$.
		\\Let us now consider the operators $y_s$, $T_s$ and $dT_s$: we proved that they give a curve in $C^*_f(M \sqcup N)^\Gamma$. The same obviously holds also for $dT_s y_s$. Moreover it is also true for $T_s y_s d$, indeed 
		\begin{equation}
			T_s y_s d = - T_s d y_s + T_s + T_s T_s^\dagger T_s
		\end{equation}
		We can observe that
		\begin{equation}
			L_{\alpha,\beta, s} = 1 + Q_{\alpha, \beta, s},
		\end{equation}
		where $Q_{\alpha, \beta, s}$ is an algebraic combination of $T_s$ , $T^\dagger_s$, $T_s y$, $yT_s$, $T^\dagger_s y$ and $yT^\dagger_s$. Then we have that $Q_{\alpha, \beta,s}$, $Q_{\alpha, \beta, s}d$ and $dQ_{\alpha, \beta, s}$ give a curve in $C^*(M \sqcup N)^\Gamma$. The same property holds for $R_{\alpha, \beta, s} = 2Q_{\alpha, \beta, s} + Q_{\alpha, \beta, s}^2$.
		\\Now we can observe that if $A_s$ is a curve in $C^*(M \sqcup N)^\Gamma$ such that $Ad$ and $dA$ are also two curves in $C^*(M \sqcup N)^\Gamma$ then also $g(A)d$ and $dg(A)$ are curves in $C^*(M \sqcup N)^\Gamma$, indeed
		\begin{equation}
			g(A_s) d = (aA_s + A_sh(A_s)A_s) d = a(A_sd) + A_sh(A_s)(A_sd)
		\end{equation}
		and
		\begin{equation}
			dg(A_s) = d(aA:s + A_sh(A_s)A_s) = a(dA_s) + (dA_s)h(A_s)A_s.
		\end{equation}
		and since the continuity of the functional calculus on bounded operators we have that the compositions of $V_{\alpha, \beta}$ and $G_{\alpha,\beta}$ with $d$ are operators in $C^*(M \sqcup N)^\Gamma$. Moreover, since
		\begin{equation}
			Z_{\alpha, \beta, s} = - V_{\alpha, \beta,s} - V_{\alpha, \beta,s}Z_{\alpha, \beta,s} = - V_{\alpha, \beta,s} - Z_{\alpha, \beta,s} V_{\alpha, \beta,s}
		\end{equation}
		then this property also holds for $Z_{\alpha, \beta,s}$, $K_{\alpha, \beta,s}$ and $H_{\alpha, \beta,s}$.
		Now, since $D_{\alpha, \beta,s} - D$ is an algebraic combination of operators in $C^*(M \sqcup N)^\Gamma$, the perturbation is a continuous curve in $C^*(M \sqcup N)^\Gamma$.
	\end{proof}
	\begin{lem}
		Let $f_0, f_1: (M,g) \longrightarrow (N,h)$ twolipschitz homotopy-equivalences between two riemannian manifolds with bounded geometry. Consider the splitting given by
		\begin{equation}
			\mathcal{L}^2(M \sqcup N) = V_+ \oplus V_-
		\end{equation}
		where $V_\pm$ is the $\pm 1$-eigenspace of $\tau$. Then if $\alpha_0$ is such that $D_{\alpha_0, 1, s}$ is invertible for each $s$, then the isometry $\mathcal{U}_{\alpha_0, s, \pm}: V_{\pm, \alpha_0, 1} \longrightarrow V_\pm$ of Lemma 14.6 is the restriction to $V_{\pm, \alpha_0, 1, s}$ of the operator $\frac{I \pm \tau}{2} + P_{\alpha_0,s}$ where $P_{\alpha_0,s}$ is continuous curve in $C_f^*(M \sqcup N)^\Gamma$.
	\end{lem}
	\begin{proof}
		Let us define the operator
		\begin{equation}
			W_{\alpha, \beta, s} = U_{\alpha, \beta, s}S_{\alpha, \beta, s}U_{\alpha, \beta, s}^{-1}.
		\end{equation}
		In the previous proposition we proved that $U_{\alpha, \beta, s}$, $S_{\alpha, \beta, s}$ and $U_{\alpha, \beta, s}^{-1}$ are continuous in $s$ and so the same holds also for $W_{\alpha, \beta, s}$. Then, exactly as we did in Remark 48, we can define $W_{\gamma(t), s}$ that will be continuous in $s$ and in $t$. This means that there is a $C$ for each $s$ and for each $\epsilon$ such that
		\begin{equation}
			||W_{t,s} - W_{t + \epsilon, s}|| \leq C \epsilon.
		\end{equation}
		Since $W_{t,s}$ is uniformly continuous on $t$, we can divide $[0, 1]$ in $N_0$ intervals $[t_i, t_{i+1}]$ where $t_0 = 0$ and $t_{N_0} = 1$ and
		\begin{equation}
			||W_{t_i, s} - W_{t_i+1, s} || \leq 1.
		\end{equation}
		Then we can define the curves $H_{i,s} := W_{t_i, s} - W_{t_i+1, s}$ and
		\begin{equation}
			L_{i,s} := \frac{I + W_{t_i,s}}{4}H_{i,s}.
		\end{equation}
		As we proved in Lemma 14.6, one can easily check that $I +L_{i,s}$ is invertible and if we define $Q_{i,s}$ as the operator such that
		\begin{equation}
			I + Q_{i,s} = (I + L_{i,s})^{-1}
		\end{equation}
		and, using the continuity of the functional calculus, we have that $Q_{i,s}$ is continuous in $s$. Then the operator
		\begin{equation}
			\begin{split}
				U_{i,s} &= \frac{I + W_{t_i,s}}{2}(I + Q_{i,s}) \\
				&= \frac{I + W_{0,s} - W_{0,s} + W_{t_i,s}}{2}(I + Q_{i,s})\\
				&= [\frac{I + \tau}{2} + \frac{1}{2}(\sum\limits_{j = 0}^i H_{i,s})](I + Q_{i,s}) \\
				&= \frac{I + \tau}{2} + S_{i,s}
			\end{split}
		\end{equation}
		where $S_{i_s}$ is a continuous curve in $C^*_f(M \sqcup N)^\Gamma$. Then we can conclude observing that
		\begin{equation}
			\begin{split}
				\mathcal{U}_{\alpha_0, s} &= U_{0,s} \circ ... \circ U_{N_0,s}\\ &= \frac{I + \tau}{2}^{N_0} + P_{\alpha_0, s} \\ &= \frac{I + \tau}{2} + P_{\alpha_0, s},
			\end{split}
		\end{equation}
		where $P_{\alpha_0, s}$ is a continuous curve in $C^*_f(M \sqcup N)^\Gamma$.
	\end{proof}
	\begin{prop}
		Let us consider $D$ an operator which is analytically controlled over $(D^*_f(M \sqcup N)^\Gamma, C^*_f(M \sqcup N)^\Gamma)$ on $\mathcal{L}^2(M \sqcup N)$. Suppose that $D$ has bounded inverse and let $S$ be a continuous map $S:[0,1] \longrightarrow C^*_f(M \sqcup N)^\Gamma$ such that for all $t \in [0,1]$ the operator $D + S(t)$ has a bounded inverse.
		\\Then if $\chi: \numberset{R} \longrightarrow \numberset{R}$ is a bounded continuous function, we have that
		\begin{equation}
			\chi(D+S): [0,1] \longrightarrow D^*_f(M \sqcup N)^\Gamma
		\end{equation}
		is a continuous map.
	\end{prop}
	\begin{proof}
		We know, because Lemma 5.8 of \cite{higroe}, that for all $t \in [0,1]$ the operator $\chi(D+S)(t)$ is in $C^*_f(M \sqcup N)^\Gamma$. This means that we just have to prove that for all $\epsilon > 0$ and for all $t_0$ in $[0,1]$ there is a $\delta > 0$ such that for all $t$ such that $|t - t_0| \leq \delta$
		\begin{equation}
			||\chi(D + S(t)) - \chi(D + S(t_0))|| \leq \epsilon.
		\end{equation}
		It's easy to see that if the Proposition holds for $\chi(x) = \arctan(x)$ then it also holds for every bounded continuous $f$. Indeeed we can observe that 
		\begin{equation}
			\bigcup\limits_{t \in [0,1]} Spec(arctan(D + S(t))) \subseteq [-b, -a] \sqcup [a, b]
		\end{equation}
		for some $a,b >0$ and so, since $f = \frac{f}{\arctan} \arctan$ on $[-b, -a] \sqcup [a, b]$, we have that
		\begin{equation}
			f(D + S(t)) = \frac{f}{\arctan}(\arctan(D + S(t)))
		\end{equation} 
		and since $\arctan(D + S(t))$ are bounded, then
		\begin{equation}
			\lim\limits_{t \rightarrow t_0} f(D + S(t)) = f(D + S(t_0)).
		\end{equation}
		Let us suppose $\chi(x) =  \arctan(x)$. We will use the complex logarithm form
		\begin{equation}
			\arctan(x) = \frac{i}{2}\log(\frac{i + x}{i - x}).
		\end{equation}
		In particular, since $\arctan$ is an odd function, we will use
		\begin{equation}
			\arctan(x) = -\frac{i}{2}\log(-\frac{i + x}{i - x}).
		\end{equation}
		We know, since Proposition 5.9 of \cite{higroe}, that for all $t$ the operator $D + S(t)$ is analitically controlled and so $(D + S(t) \pm i)$ has bounded inverse. We can observe then 
		\begin{equation}
			\begin{split}
				\frac{i + D+S(t)}{i - D - S(t)} &= \frac{i + D+S(t)}{i - D - S(t)}\cdot  \frac{i + D+S(t)}{i + D + S(t)}\\
				&= [-1 + 2i(D+S(t)) + (D+S(t))^2]((D+S(t))^2 + 1)^{-1}\\
				&= [-1 + 2i(D+S(t))^{-1} + (D+S(t))^{-2}]\frac{(D+S(t))^2}{(D+S(t))^{2}}((D+S(t))^{-2} + 1)^{-1}\\
				&= [-1 + 2i(D+S(t))^{-1} + (D+S(t))^{-2}]((D+S(t))^{-2} + 1)^{-1}.
			\end{split} 
		\end{equation}
		Then for all $t$ we have that $-1 + 2i(D+S(t))^{-1} + (D+S(t))^{-2}$ is bounded since it is a sum of bounded operators. In particular, applying the Theorem 1.16 of \cite{kato}, we obtain that  
		$-1 + 2i(D+S(t))^{-1} + (D+S(t))^{-2}$ is a continuous curve in $\numberset{B}(\mathcal{L}^2(M \sqcup N))$ and its norm is a continuous function $\alpha:[0,1] \longrightarrow \numberset{R}$.
		\\Moreover also $((D+S(t))^{-2} + 1)^{-1}$ is bounded, indeed we know that
		\begin{equation}
			Spec((D+S(t))^{-2}) \subseteq [\epsilon, + \infty)
		\end{equation}
		and so
		\begin{equation}
			Spec((D+S(t))^{-2} + 1) \subseteq [1 + \epsilon, + \infty)
		\end{equation}
		Then, using that $\lambda$ is in the spectrum of $((D+S(t))^{-2})^{-1}$ if and only if $\lambda + 1$ is in the spectrum of $Spec((D+S(t))^{-2} + 1)$, we have that the spectral radius (i.e. the norm) of $((D+S(t))^{-2} + 1)^{-1}$ is strictly less of $1$.
		\\Then we have that
		\begin{equation}
			||[-1 + 2i(D+S(t))^{-1} + (D+S(t))^{-2}]((D+S(t))^{-2} + 1)^{-1}|| \leq \alpha(t) \leq M_0.
		\end{equation}
		Observe that the spectral radius of $\frac{i - D-S(t)}{i + D + S(t)}$ is less or equal to its norm. Then we have that
		\begin{equation}
			Spec(\frac{i + D+S(t)}{i - D-S(t)}) \subseteq B_{M_0}(0),
		\end{equation}
		where $B_{M_0}(0)$ is the ball centered in $0$ with radius $M_0$ in $\numberset{C}$.
		Moreover, exactly as we proved that the spectrum of $\frac{i + D+S(t)}{i - D-S(t)}$ is uniformly bounded in $t$, then we can check the same for
		\begin{equation}
			\frac{i - D-S(t)}{i + D+S(t)} =(\frac{i + D+S(t)}{i - D-S(t)})^{-1}.
		\end{equation}
		This means that there is a disk of radius $M_1$ centered in zero, such that
		\begin{equation}
			Spec(\frac{i + D+S(t)}{i - D-S(t)}) \cap B_{M_1}(0) = \emptyset.
		\end{equation}
		Consider the corona $C_{M_0,M_1}$ given by the intersection of $B_{M_1}(0)^c$ and $B_{M_0}(0)$.
		We obtain that
		\begin{equation}\label{tretuno}
			Spec(-\frac{i + D+S(t)}{i - D-S(t)}) \subset C_{M_0,M_1}.
		\end{equation}
		\\Fix the polar coordinates $(r, \theta)$ in $\numberset{C} \setminus \{0\}$ where $\theta = (-\pi, +\pi)$. Since the spectrum of $D+S(t)$ is real and there is a gap around $0$ which doesn't depends on $t$, then there is an angle $\theta_0$ such that
		\begin{equation}\label{treta}
			Spec(-\frac{i + D+S(t)}{i - D-S(t)}) \subset F_{\theta_0}
		\end{equation}
		where
		\begin{equation}
			F_{\theta_0} := \{z \in \numberset{C}\setminus \{0\}| \theta(z) \in [- \pi + \theta_0, \pi - \theta_0 ] \}.
		\end{equation}
		Indeed we have that 
		\begin{equation}
			Spec(D+S(t)) \subset (- \infty, -r_0) \times \{-\pi\} \sqcup  (r_0, + \infty) \times \{0\} = A_- \sqcup A_+
		\end{equation}
		for some $r_0 >0$. Then, if $z$ is in $A_-$, then
		\begin{equation}
			\theta(i + z) = \pi - \alpha
		\end{equation}
		where
		\begin{equation}
			\alpha = \arccos(\frac{z}{\sqrt{z^2 + 1} })
		\end{equation}
		is contained in $[0, \theta_1] \subset [0, \frac{\pi}{2})$, where
			\begin{equation}
			\theta_1 = \arccos(\frac{r_0}{\sqrt{r_0^2 + 1} }).
		\end{equation}
		Moreover we also have that
		\begin{equation}
			\theta(i - z) = \alpha > 0.
		\end{equation}
		Then we have that
		\begin{equation}
			\theta(\frac{i + z}{i - z}) = \pi - 2\alpha \in [\pi - 2\theta_1, \pi) \subset (0, \pi)
		\end{equation}
		and so
		\begin{equation}
			\theta(-\frac{i + z}{i - z}) \in [-\pi + 2\theta_1, 0) \subset (-\pi, 0)
		\end{equation}
		If $z$ is in $A_+$, then $\theta(i + z) = \alpha$ and $\theta(i - z) = \pi - \alpha$. This means that if $z$ is in $A_+$
		\begin{equation}
			\theta(\frac{i + z}{i - z}) = -\pi + 2\alpha \in (-\pi , - 2\theta_1] \subset(-\pi, 0)
		\end{equation}
		and so
		\begin{equation}
			\theta(- \frac{i + z}{i - z}) \in (0 , \pi - 2\theta_1] \subset(0, \pi)
		\end{equation}
		We proved (\ref{treta}): it is sufficient consider $\theta_0 := 2\theta_1$.
		\\Since (\ref{treta}) and (\ref{tretuno}) we obtain that
		\begin{equation}
			Spec( -\frac{i + D+S(t)}{i - D-S(t)}) \subset F_{\theta_0} \cap C_{M_0, M_1}.
		\end{equation}
		Observe that $K := F_{\theta_0} \cap C_{M_0, M_1}$ is a compact set where the logarithm is a well-defined, continuous function. This implies that for all $\epsilon$ there is a polinomial $p_\epsilon$ which is an  $\epsilon$-approximation of logarithm on $K$. Then if we denote by $a := \frac{i + D+S(t)}{i - D+S(t)}$ and $b:= \frac{i + D+S(t_0)}{i - D+S(t_0)}$ we have that
		\begin{equation}
			\begin{split}
				||\log(a) - \log(b)|| &= ||\log(a) - p_\epsilon(a)|| + ||p_\epsilon(a) - p_\epsilon(b)|| + ||\log(b) - p_\epsilon(b)|| \\
				&\leq \epsilon + ||p_\epsilon(a) - p_\epsilon(b)|| + \epsilon.
			\end{split}
		\end{equation}
		Observe that if $||a - b||$ is small enough, then also $||p_\epsilon(a) - p_\epsilon(b)||$ is small. This means that to conclude the proof we just have to show that for each $\epsilon$ there is a $\delta$ such that if $|t - t_0| \leq \delta$, then
		\begin{equation}
			||\frac{i + D+S(t)}{i - D-S(t)} - \frac{i + D+S(t_0)}{i - D-S(t_0)}||\leq \epsilon.
		\end{equation}
		Let us denote by $\tilde{D} := D + S(t_0)$ and by $S_t:= S(t) - S(t_0)$. Consider, then
		\begin{equation}
			\begin{split}
				\frac{i+ \tilde{D} + S_t}{i - \tilde{D} - S_t} &= \frac{i+ \tilde{D}}{i - \tilde{D} - S_t} + \frac{S_t}{i - \tilde{D} - S_t}\\
				&= [\frac{i - \tilde{D}}{i + \tilde{D}} + \frac{S_t}{i + \tilde{D}}]^{-1} + \frac{S_t}{i - \tilde{D} - S}_t.
			\end{split}
		\end{equation}
		We know that
		\begin{equation}\label{tilde0}
			\lim\limits_{t \rightarrow t_0} S_t = 0.
		\end{equation}
		Moreover, applying the Theorem 1.16 of \cite{kato} we have that
		\begin{equation}
			||(i - \tilde{D} - S_t)^{-1} - (i - \tilde{D})^{-1}|| \leq \frac{||(i - \tilde{D})^{-1}||^2 \cdot ||S_t||}{1 - ||(i - \tilde{D})^{-1}|| \cdot ||S_t||}
		\end{equation}
		and so
		\begin{equation} \label{tilde1}
			\lim\limits_{t \rightarrow t_0} (i - \tilde{D} - S_t)^{-1} = (i - \tilde{D})^{-1}.
		\end{equation}
		Finally we have that (always applying \cite{kato}) 
		\begin{equation}
			||[\frac{i - \tilde{D}}{i + \tilde{D}} + \frac{S_t}{i + \tilde{D}}]^{-1} - [\frac{i - \tilde{D}}{i + \tilde{D}}^{-1}]^{-1}|| \leq \frac{||\frac{i + \tilde{D}}{i - \tilde{D}}^{-1}||^2 \cdot ||(i + \tilde{D})^{-1} \cdot ||S_t||}{1 - ||\frac{i + \tilde{D}}{i - \tilde{D}}^{-1}|| \cdot ||(i + \tilde{D})^{-1} \cdot ||S_t||}
		\end{equation}
		and so
		\begin{equation}\label{tilde2}
			\lim\limits_{t \rightarrow t_0} [\frac{i - \tilde{D}}{i + \tilde{D}} + \frac{S_t}{i + \tilde{D}}]^{-1} = \frac{i - \tilde{D}}{i + \tilde{D}}.
		\end{equation}
		Then, using (\ref{tilde0}), (\ref{tilde1}) and (\ref{tilde2}) we have that
		\begin{equation}
			\begin{split}
				\lim\limits_{t \rightarrow t_0 } \frac{i + D+S(t)}{i - D+S(t)} &= \lim\limits_{t \rightarrow t_0 } \frac{i + \tilde{D} +S_t }{i - \tilde{D} +S_t} \\
				&= \lim\limits_{t \rightarrow t_0 } \frac{i + \tilde{D}}{i - \tilde{D}} \\
				&= \frac{i + D+S(t_0)}{i - D+S(t_0)}.
			\end{split}
		\end{equation}
		And so we can conclude that
		\begin{equation}
			\lim\limits_{t \rightarrow t_0} \arctan(D + S(t)) = \arctan(D + S(t_0)).
		\end{equation}
	\end{proof}
	\begin{thm}\label{ultimo}
		Let us consider two manifolds of bounded geometry $(M,g)$ and $(N,h)$ and a group $\Gamma$ which acts uniformly proper and free by isometries on $M$ and $N$. Let $f_0, f_1: (M, g) \longrightarrow (N,h)$ two $C^k_b$-map which are $\Gamma$-equivariant lipschitz-homotopy equivalences. Suppose, moreover, that $f_0 \sim_\Gamma f_1$. Then we have
		\begin{equation}
			[\rho_{f_0}] = [\rho_{f_1}] \in D^*_f(M \sqcup N)^\Gamma.
		\end{equation}
		As consequence of this fact we have that for each lipschitz-homotopy equivalence $f: (M,g) \longrightarrow (N,h)$ we can define a class in $D^*_f(M \sqcup N)^\Gamma$ which only depends on the lipschitz-homotopy class of $f$.
	\end{thm}
	\begin{proof}
		Consider the odd case: if we are able to find a continuous curve $\gamma$ in $D^*_f(M \sqcup N)^\Gamma$ such that for all $t \in [0,1]$ is a projection,
		\begin{equation}
			\gamma(0) = \frac{1}{2}(\chi(D_{\alpha_0,1,0}) + 1 \mbox{       and      } \gamma(1) = \frac{1}{2}(\chi(D_{\alpha_0,1,1}) + 1.
		\end{equation}
		Then obviously we can define
		\begin{equation}
			\gamma(t) = \frac{1}{2}(\chi(D_{\alpha_0,1,t}) + 1
		\end{equation}
		and the continuity of $\gamma$ is given by the continuity of $\chi(D_{\alpha_0,1,t}$.
		\\
		\\Consider the even case: we have to find a continuous curve in $D^*_{\rho}(M \sqcup N, H_{M \sqcup N})^\Gamma$, where $H_{M \sqcup N}$ and $\rho_{M \sqcup N}$ are defined as in Example \ref{terna} such that for all $t \in [0,1]$ is a unitary,
		\begin{equation}
			\gamma(0) = \mathcal{U}_{\alpha_0, 0, -} \chi(D_{\alpha_0,1,0}) \mathcal{U}_{\alpha_0, 0, +}^* \mbox{       and      } \gamma(1) = \mathcal{U}_{\alpha_0, 1, -} \chi(D_{\alpha_0,1,1}) \mathcal{U}_{\alpha_0, 1, +}^*.
		\end{equation}
		Let us consider the for each $t \in [0,1]$
		\begin{equation}
			\gamma(t) = \mathcal{U}_{\alpha_0, t, -} \chi(D_{\alpha_0,1,t}) \mathcal{U}_{\alpha_0, t, +}^*
		\end{equation}
		This is a curve in $D^*_f(M \sqcup N)^\Gamma$. The continuity of $\gamma$ is given by the continuity of $\chi(D_{\alpha_0,1,t})$ and $\mathcal{U}_{\alpha_0, t, -} $. Let us recall Remark \ref{girl}: we have that, since $\gamma$ is continuous on $D^*_f(M \sqcup N)^\Gamma$ and unitary for each $t$ in $[0,1]$, then it is also continuous in $D^*_{\rho}(M \sqcup N, H_{M \sqcup N})^\Gamma$ and it is still unitary for each $t$ in $[0,1]$.
	\end{proof}
	\begin{defn}
		Let $(M,g)$ and $(N,h)$ be two manifolds of bounded geometry. Consider $\Gamma$ a group that acts uniformly proper and free on $M$ and on $N$ by isometries. Suppose that $f:(M,g) \longrightarrow (N,h)$ is a $\Gamma$- invariants lipschitz-homotopy equivalence. Then, if $dim(N) = n$ we can define the \textbf{projected secondary index of} $(M,f)$ as the class in $K_{n+1}(D^*(N)^\Gamma)$ defined as 
		\begin{equation}
			\rho(M, f) := f_{\star}(\rho_{\tilde{f}}),
		\end{equation}
		where $\tilde{f}$ is a $C^k_b$-approximation of $f$, where $k \geq 2$.
	\end{defn}
	Observe that, if $\tilde{f}_1$ and $\tilde{f}_2$ are two $C^k_b$ approximations of $f$, then $\tilde{f}_1 \sim_\Gamma \tilde{f}_2$. Then applying Theorem \ref{ultimo} we have that the definition of $\rho(M,f)$ doesn't depend on the choice of the $C^k_b$ approximation. Moreover we also have the following Corollary.
	\begin{cor}
		Consider $f$, $(M,g)$ and $(N,h)$ as above. Then the projected secondary index $\rho(M,f)$ only depends on $M$ and on the lipschitz-homotopy class of $f$.
	\end{cor}
	
	\appendix
	\chapter{Inequalities between norms of differential forms}\label{ineqform}
	\begin{lem}
		Let $(M,g)$ be a Riemannian manifold. Consider $p$ in $M$: we have that on $T_p^*M$ the norm $|| \cdot||_{T^*_p(M)}$ given by $g$ and the operatorial norm
		\begin{equation}
			||\alpha_p||_{Op} = \sup \limits_{v_p \in B_1(0) \subset T_pM} |\alpha_p(v_p)|,
		\end{equation}
		are the same norm.
	\end{lem}
	\begin{proof}
		Let $e_i$ be an orthonormal basis of $T_pM$ and let $\epsilon^i = g(e_i, \cdot)$ be the dual basis of $T_p^*M$. The basis $\{\epsilon^i \}$ is an orthonormal basis. Indeed if $G$ is the matrix of $g$ related to $\{e_i\}$ then
		\begin{equation}
			\begin{split}
				g(\epsilon^i, \epsilon^j) &= (Ge_i)^{\tau}G^{-1}Ge_j \\
				&= e_i^{\tau}G^{\tau}e_j = e_i^{\tau}Ge_j.
			\end{split}
		\end{equation}
		So, if $\omega = \alpha_j \epsilon^j$ then its $g$-norm is $\sqrt{\sum\limits_{j}\alpha_j^2}$.
		\\Let $w = a^ie_i$ be an unitary tangent vector (i.e. $\sqrt{\sum\limits_{i}(a^i)^2}=1$), then
		\begin{equation}
			\begin{split}
				||\omega||_{Op} &= (\sup \limits_{w \in B_1(0) \subset T_pM} |\omega(w)|) \\
				&=\sup \limits_{w \in B_1(0)\subset T_pM} |\alpha_j a^i \epsilon^j(e_i)| \\
				&=\sup \limits_{w \in B_1(0)\subset T_pM} |\alpha_i a^i| \\
				&= \sup \limits_{w \in B_1(0)\subset T_pM} (\alpha, a),
			\end{split}
		\end{equation}
		where $(\alpha, a)$ is the euclidean scalar product in $\numberset{R}^n$ between $\alpha = (\alpha_i)$ amd $a = (a^j)$. Following now the proof of Cauchy-Schwartz we have that if $a$ is proportional to $\alpha$ (i.e. $a = \frac{\alpha}{||\alpha||}$), then 
		\begin{equation}
			\sup \limits_{w \in B_1(0)\subset T_pM} (\alpha, a) = (\alpha, \frac{\alpha}{||\alpha||})
		\end{equation}
		So
		\begin{equation}
			\begin{split}
				||\omega||_{Op} &= ||\alpha||_{eucl}\cdot||\frac{\alpha}{||\alpha||}||_{eucl} \\
				&= ||\alpha||_{eucl}1 = ||\omega||_{g}.
			\end{split}
		\end{equation}
	\end{proof}
	\begin{lem}\label{spacc}
		Consider two differential forms $\alpha$ and $\beta$ on a Riemannian manifold $(M,g)$. Then there is a constant $C$ such that
		\begin{equation}
			|\alpha \wedge \beta|_p \leq C |\alpha|_p \cdot |\beta|_p
		\end{equation}
	\end{lem}
	\begin{proof}
		It's a direct consequence of Hadamard-Schwartz inequality \cite{Sbordone}, which states that given $\alpha_1, ..., \alpha_k$ arbitrary exterior forms in $\numberset{R}^n$ of degree $l_1,..., l_k$, then there is a constant $C$ that depends only by $n$ such that
		\begin{equation}
			|\alpha_1 \wedge ... \wedge \alpha_k| \leq C |\alpha_1| \cdot ... \cdot |\alpha_k|.
		\end{equation}
	\end{proof}
	\begin{lem}\label{norm}
		Let $f:(M,g) \longrightarrow (N,h)$ be a lipschitz map between Riemannian manifolds. Then for all $\alpha_p$ in $T^*_p(N)$ and for all $q$ in $f^{-1}(p)$ we have that there is a number $L$ 
		\begin{equation}
			||f^*\alpha_p|^2_{\Lambda^{k}_qM} \leq L||\alpha_p||^2_{\Lambda^{k}_pN},
		\end{equation}
		where $||\cdot||^2_{\Lambda^{k}_{x} X}$ is the norm induced by the metric of $X$ on $\Lambda^k_xX$.
	\end{lem}
	\begin{proof}
		Let us suppose that $\alpha$ is a $1$-form. We have that
		\begin{equation}
			\begin{split}
				||f^*\alpha_p||_{T^*_qM} &= \sup_{||v|| =1 } f^*\alpha_p(v) \\
				&= \sup_{||v|| =1 } \alpha_p(f_\star v) \\
				&\leq ||\alpha_p||_{T^*_pN}\sup_{||v|| = 1} ||f_\star v||\\
				&\leq C_f||\alpha_p||_{T^*_pN}.
			\end{split}
		\end{equation}
		Let us recall that a $k$-form in a point $q$ can be always seen as sum of wedges of $k$ orthogonal $1$-forms.
		Then consider $\alpha = a^{i_1 ... i_k} \epsilon_{i_1} \wedge... \wedge \epsilon_{i_k}$ an element of $\Lambda^k_p(N)$. We have that
		\begin{equation}
			||\alpha||_{\Lambda^k_p(N)} = \sum\limits_{i_1 <... < i_k} |a^{i_1 ... i_k}|.
		\end{equation}
		Observe that
		\begin{equation}
			\begin{split}
				||f^*\alpha||_{\Lambda^k_q(M)} &\leq \sum\limits_{i_1 <... < i_k} |a^{i_1 ... i_k}| \cdot ||f^*(\epsilon_{i_1}) \wedge... \wedge f^*(\epsilon_{i_k})|| \\
				&\leq \sum\limits_{i_1 <... < i_k} |a^{i_1 ... i_k}| \cdot C\cdot ||f^*(\epsilon_{i_1})|| \cdot... \cdot ||f^*(\epsilon_{i_k})|| \\
				&\leq C\cdot C_f^k \sum\limits_{i_1 <... < i_k} |a^{i_1 ... i_k}| = C \cdot C_f^k||\alpha||_{\Lambda^k_p(N)}.
			\end{split}
		\end{equation}
	\end{proof}
	\begin{cor}
		Given a differential form $\alpha$ on $N$ and a lipshitz map $f: (M,g) \longrightarrow (N,h)$. Consider the function
		\begin{equation}
			\begin{split}
				N &\longrightarrow \numberset{R} \\
				q &\longrightarrow ||f^*\alpha||_{\Lambda_q^*M}.
			\end{split}
		\end{equation}
		In the same way $||\alpha||_{\Lambda_p^*(N)}$ can be seen as a function $M \longrightarrow \numberset{R}$.
		Then we have that
		\begin{equation}
			||f^*\alpha||_{\Lambda_q^*M} \leq K_f f^*(||\alpha||_{\Lambda_p^*(N)})(q).
		\end{equation}
		as functions on $M$ to $\numberset{R}$, where $K_f$ is the maximum between $1$, $C_f$, $C_f^n$, where $n =dim(N)$.
	\end{cor}
	\chapter{Orientation and integration along the fibers}\label{B}
	In this Appendix we will introduce the notion of orientation for a fiber bundle and the integration along fibers. We are interested in this notions because we have the following result of Ehresmann:
	\begin{thm}[Ehresmann lemma]
		Given a smooth map $p: E \longrightarrow N$ between smooth manifolds, if $p$ is 
		\begin{itemize}
			\item a surjective submersion
			\item a proper map
		\end{itemize}
		then there is a smooth manifold $F$ such that $(E, p, N, F)$ is a fiber bundle.
	\end{thm}
	The following part is taken from \cite{Conn}. Let $\mathcal{B} = (E, \pi, B, F)$ be a smooth fiber bundle with $dim B = n$, $dimF = r$. Recall that the fibre $F_x$, at $x \in B$ is a submanifold of $E$ and denote the inclusion by $j_x: F \longrightarrow E$.
	\begin{defn}
		The bundle $B$ is called orientable if there exists an $r$-form $\psi$ on $E$ such that $j_x^*\psi$ orients $F$, for every $x \in B$. An equivalence class of such $r$-forms is called an \textbf{orientation} for the bundle and a member of the equivalence class is said to represent the orientation.
	\end{defn}
	Let now $f:X \longrightarrow Y$ be an orientable fiber bundle with fiber $F$. If $\partial X = \emptyset$, then we will denote by $\Omega^*_{vc}(X)$ the space of differential forms $\alpha$ such that if $K$ is a compact set of $Y$ then $supp(\alpha) \cap f^{-1}(K)$ is a compact set. If $X$ has a boundary, then we define $\Omega^*_{vc}(X)$ as $\Omega^*_{vc}(X \setminus \partial X)$.
	\\Consider $\alpha_p \in \Lambda^*(X)$ and let $v_p$ be a tangent vector in $p \in X$. We denote by $i_{v_p} \alpha_p$ the contraction of $\alpha_p$ with $v_p$. 
	\begin{defn}
		Consider a submersion between differentiable manifolds $p: X \longrightarrow Y$. We denote by $\Omega_{vc}^*(X)$ the space of \textbf{vertically compactly supported smooth forms} the set of differential form $\alpha$ on $X$ such that, for each compact set $K$, we have that $p^{-1}(K) \cap supp(\alpha)$ is a compact set.
	\end{defn}
	\begin{defn}
		Consider a fiber bundle $f: X \longrightarrow Y$ and let $F$ be a generic fiber over a point $p$ in $Y$. The \textbf{integration along the fibers} is the map $f_\star: \Omega^*_{vc}(X) \longrightarrow \Omega^{*-r}(Y)$ defined for all $\alpha$ in $\Omega^*_{vc}(X)$ as
		\begin{equation}
			f_\star(\alpha)_p(v_1,... v_{*-r}) := \int_{F_p} j_p^*(i_{\tilde{v_1}} ... i_{\tilde{v}_{*-r}} \alpha),
		\end{equation}
		where $\tilde{v}_j$ are some vectors such that $f_\star(\tilde{v}_j) = v_j$. We will denote $f_\star(\alpha)$ also as
		\begin{equation}
			\int_F \alpha.
		\end{equation}
	\end{defn}
	\begin{rem}
		The definition above doesn't require a choice for the lifts of $v_j$s. Indeed if we use $\tilde{v_j} + w_j$ where $w_j \in ker(f_\star)$ instead of $v_j$, one can easily check that
		\begin{equation}
			j_p^*((\tilde{v_1} + w_1) \dashv... (\tilde{v}_{*-r} + w_{*-r}) \dashv \alpha) = j_p^*(\tilde{v_1} \dashv... \tilde{v}_{*-r} \dashv \alpha) + j_p^*(\beta)
		\end{equation}
		where the differential form $j_p^*(\beta)$ is a form in $\Omega^*(F)$ which is not a top-degree form. This means that
		\begin{equation}
			\int_{F} j_p^*(\beta) = 0.
		\end{equation}
	\end{rem}
	\begin{rem}
		If $f: X \longrightarrow Y$ is a proper surjective submersion, then $f_\star$ is well defined for each $\alpha$ in $\Omega^*(X)$ even if $\partial X \neq \emptyset$.
	\end{rem}
	Now we introduce some important properties of the integration along fibers. One can find all the proofs in \cite{Conn}, in particular in chapter 7, section 5.
	\begin{prop}[Projection Formula and Fubini Theorem]
		The map $f_\star$ is surjective and, in particular if $\alpha$ is in $\Omega^*(Y)$ and $\beta$ in $\Omega^*_{vc}(X)$ then
		\begin{equation}
			\int_F f^*\alpha \wedge \beta = \alpha \wedge \int_F \beta,
		\end{equation}
		and
		\begin{equation}
			\int_X \beta = \int_Y(\int_F \beta).
		\end{equation}
	\end{prop}
	\begin{rem}
		The proof of the Projection Formula for a vector bundle $(E, \pi M)$ given \cite{bottu} holds even if $(E, \pi, M)$ is a fiber bundle. In particular, following the proof, it's clear that if $\beta$ is in $\Omega^*(X) \setminus \Omega^*_{vc}(X)$, but $f_\star\beta$ is well-defined, then Projection Formula still holds.
	\end{rem}
	\begin{prop}\label{commete}
		The integration along fibers $f_\star: \Omega^*_{vc}(X) \longrightarrow \Omega^{*-r}(Y)$ satisfies the following equalities for all  vector field $W$ in $Y$ and for all vector field $Z$ in $X$ which is $f$-related to $W$
		\begin{itemize}
			\item $i_W \circ q_{F} = q_F \circ i_Z$
			\item $\Theta_W \circ q_F = q_F \circ \Theta_Z$
			\item $d \circ q_F = q_F \circ d$
		\end{itemize}
		where $i_{W}$ (and $i_Z$) is the inner product with $W$ (resp. $Z$), $\Theta_{W}$ (or $\Theta_Z$) is the Lie derivative along $W$ (resp. $Z$) and $d$ is the exterior derivative.
	\end{prop}
	\chapter{K-Theory}\label{K-theory}
	Consider a $*$-algebra $A$. 
	\begin{defn}
		A $C^*$-algebra is a Banach algebra such that 
		\begin{equation}
			||x^*x|| = ||x||^2
		\end{equation}
		for each $x$ in $A$.
	\end{defn}
	There is an equivalent definition of $C^*$-algebra.
	\begin{defn}
		A \textbf{$C^*$-algebra} is a $*$-algebra which is isometrically $*$-isomorphic to a norm-closed $*$-subalgebra of $B(H)$ where $H$ is an Hilbert space.
	\end{defn}
	Our goal, in this Appendix, is to introduce the notion of K-theory groups of a $C^*$-algebra $A$ and to show the main properties.
	\begin{defn}
		Let $A$ be a $C^*$-algebra. Consider $A \subseteq B(H)$ for some Hilbert space $H$. A $C^*$-algebra $A$ is said to be \textbf{unital} if it contains the identity operator.
	\end{defn}
	\begin{defn}
		We define the \textbf{unitalization of } $A$ as the $C^*$-algebra $\tilde{A}$ generated by $A$ and the identity operator of $B(H)$.
	\end{defn}
	\begin{rem}
		Let $A$ and $B$ be two $C^*$-algebras, then a $*$-morphism $\phi:A \longrightarrow B$ can be extended in a unique way to a $*$-morphism $\Phi: \tilde{A} \longrightarrow \tilde{B}$ imposing $\Phi(1) = 1$.
	\end{rem}
	Observe that if $A$ already contains the identity operator, then $\tilde{A} = A$.
	\\As shown in \cite{Anal}, we have a short exact sequence
	\begin{equation}\label{qui}
		0 \longrightarrow A \longrightarrow \tilde{A} \longrightarrow \numberset{C} \longrightarrow 0.
	\end{equation}
	We use this sequence later, to define the K-theory groups of a non-unital $C^*$-algebra.
	\begin{defn}
		Let us consider a ring $A$ and the module $M_n(A)$ given by the matrices with coefficient in $A$. An element $p$ of $M_n(A)$ is a \textbf{projection} if $p =p^2$.
	\end{defn}
	\begin{defn}
		Let us consider a unital $C^*$-algebra. The \textbf{K-Theory of $A$} is the group $K_0(A)$ whose generators are $[p]$ where $p$ is a projection of $M_n(A)$ for some $n \in \numberset{N}$ and
		\begin{itemize}
			\item if $p$ and $q$ are projections in $M_n(A)$ and if there is a continuous path $\alpha:[0,1] \longrightarrow M_n(A)$ such that $\alpha(0) = p$, $\alpha(1) =q$ and $\alpha(t)$ is a projection for every $t$ in $(0,1)$ then $[p]=[q]$.
			\item $[0_n] = 0_{K_0(A)}$ for all $n \in \numberset{N}$, where $0_n$ is the null matrix of $M_n(A)$ and $0_{K_0(A)}$ is the identity element of $K_0(A)$.
			\item if $[p] \in M_n(A)$ and $[q]\in M_m(A)$ then  $[p] + [q] = [p\oplus q]$ where $[p\oplus q]$ is the matrix in $M_{n+m}(A)$ given by
			\begin{equation}
				{\begin{pmatrix}
						p & 0 \\ 0 & q
				\end{pmatrix}}.
			\end{equation}
		\end{itemize}
	\end{defn}
	Given a $*$-homomorphism $\phi: A \longrightarrow B$ between $C^*$-algebras, there is a map $\phi_\star: K_0(A) \longrightarrow K_0(B)$ such that for all $a$ in $M_n(A)$ we have that
	\begin{equation}
		\phi_\star[a] = [\phi_{n}(a)],
	\end{equation}
	where $\phi_{n}(a)$ is the $n \times n$ matrix defined as $\phi_{n}(a)_{i,j} := \phi(a_{i,j})$. This map induces a morphism $\phi_\star$ between the $K_0$ groups of $A$ and $B$.
	\begin{defn}
		Let $A$ a non-unital $C^*$-algebra. Then we define the \textbf{K-Theory of $A$} as the group $K_0(A)$ defined as the kernel of
		\begin{equation}
			\phi_{\star}: K_0(\tilde{A}) \longrightarrow K_{0}(\numberset{C})
		\end{equation}
		where $\phi: \tilde{A} \longrightarrow \numberset{C}$ is the map in (\ref{qui})
	\end{defn}
	The proof of the following proposition can be find in Remark 4.1.2 of \cite{Anal}.
	\begin{prop}
		In particular we have a covariant functor $\mathcal{K}$ between the category $\mathring{A}$ which has the unital $C^*$-algebras as objects and unital $*$-homomorphism as arrows and the category \textbf{Grp} which has groups has objects and morphisms of groups as arrows. This covariant functor is defined as
		\begin{equation}
			\begin{cases}
				\mathcal{K}(A) := K_0(A) \\
				\mathcal{K}(A \xrightarrow{\phi} B) := K_0(A) \xrightarrow{\phi_\star} K_0(B).
			\end{cases}
		\end{equation}
	\end{prop}
	\begin{rem}
		The functor can be extend to non-unital $C^*$-algebras defining for each $*$-morphism $\phi: A \longrightarrow B$, the group morphism
		\begin{equation}
			\Phi_{\star|_{K_{0}(A)}}: K_0(A) \longrightarrow K_0(B),
		\end{equation}
		where $\Phi: \tilde{A} \longrightarrow \tilde{B}$ is the $*$-morphism induced by $\phi$.
	\end{rem}
	Consider an unital $C^*$-algebra $A$ and $J$ an ideal of $A$. If we have the short exact sequence
	\begin{equation}
		0 \longrightarrow J \longrightarrow A \longrightarrow \frac{A}{J} \longrightarrow 0
	\end{equation}
	Then by Proposition 4.3.15 of \cite{Anal} also the sequence
	\begin{equation}
		K_0(J) \longrightarrow K_0(A) \longrightarrow K_0(\frac{A}{B}) \label{k-short}
	\end{equation}
	is exact. 
	\begin{defn}
		Let us consider a $C^*$-algebra $A$. The \textbf{suspension of $A$} is the set $S(A)$ of continuous functions $f:[0,1] \longrightarrow A$ such that $f(0)=f(1)=0$ and the structure of $C^*$-algebra induced by $A$.
	\end{defn}
	\begin{defn}
		Let us consider a $C^*$-algebra $A$ and $p \in \numberset{N}$. The \textbf{$p$-th group of K-Theory} is the group
		\begin{equation}
			K_p(A) = K_0(S^p(A)),
		\end{equation}
		where $S^1(A) = S(A)$ and $S^p(A) = S(S^{p-1}(A))$.
	\end{defn}
	\begin{rem}	
		Consider two $C^*$-algebras $A$ and $B$ and let $\psi:A \longrightarrow B$ a $*$-homomorphism. Then we can define the map
		\begin{equation}
			\begin{split}
				S(\psi): S(A) &\longrightarrow S(B) \\
				\gamma &\longrightarrow \phi \circ \gamma.
			\end{split}
		\end{equation}
		Moreover, we define $S^n(\psi)$ as $S(S^{n-1}(\psi))$.
		\\We can observe that if $\psi$ is injective, then also $S(\psi)$ is injective. Moreover, if $J$ is an ideal of $A$ and $i: J \rightarrow A$ is the inclusion, then also $S(i): S(J) \longrightarrow S(A)$ is the inclusion of $S(J)$ in $S(A)$.
		\\Let us denote by $\pi: A \longrightarrow \frac{A}{J}$ be the quotient map.
		\\Then we have that
		\begin{equation}
			\frac{S(A)}{S(J)} \cong S(\frac{A}{J}).
		\end{equation}
		To prove this fact we can consider the map $\phi: \frac{S(A)}{S(J)} \longrightarrow S(\frac{A}{J})$ as follow: given a class of path $[\alpha]$ in $\frac{S(A)}{S(J)}$ we have that
		\begin{equation}
			\phi([\alpha]) := \pi \circ \alpha.
		\end{equation}
		The map $\phi$ is well defined, indeed if 
		\begin{equation}
			[\beta] = [\alpha] \implies \alpha = \beta + \eta
		\end{equation}
		where $\eta \in S(I)$ but this means that $\pi \circ \alpha = \pi \circ \beta + \pi \circ \eta = \pi \circ \beta$.
		\\Moreover $\phi$ is also injective: if $[\alpha] \neq [\beta]$ then it means that there is a $t_0 \in [0,1]$ such that
		\begin{equation}
			\alpha(t_0) - \beta(t_0) \notin J.
		\end{equation}
		This implies that 
		\begin{equation}
			\pi \circ \alpha (t_0) - \beta(t_0) \neq 0
		\end{equation}
		and so
		\begin{equation}
			\phi([\alpha]) \neq \psi([\beta]).
		\end{equation}
		Finally $\phi$ is surjective. Consider a curve 
		\begin{equation}
			\gamma: [0,1] \longrightarrow \frac{A}{J}
		\end{equation}
		such that $\gamma(0) = \gamma(1) = [0]$: by Lemma 4.3.13. of \cite{Anal} we have that it can be lifted to a path 
		\begin{equation}
			\tilde{\gamma}: [0,1] \longrightarrow A
		\end{equation}
		such that $\tilde{\gamma}(0) = 0$. We also know, since $\pi \circ \tilde{\gamma} = \gamma$ that $\tilde{\gamma}(1) \in J$. Let us define the curve $\sigma: [0,1] \longrightarrow J$ as
		\begin{equation}
			\sigma(t) = t\cdot \tilde{\gamma}(1).
		\end{equation}
		Then we can observe that 
		\begin{equation}
			\eta = \tilde{\gamma} - \sigma
		\end{equation}
		is a curve in $A$ such that $\eta(0) = \eta (1) = 0$ and $\pi(\eta) = \pi(\tilde{\gamma}) = \gamma$. Then we have that 
		\begin{equation}
			\phi([\eta]) = \gamma.
		\end{equation}
		Finally, if $P: S(A) \longrightarrow \frac{S(A)}{S(J)}$ is the quotient map, then one can easily check that
		\begin{equation}
			S(\pi) = \psi \circ P
		\end{equation}
		This means that, identifying $\frac{S^n(A)}{S^n(J)}$ with $S^n\frac{A}{J}$, then we can also identify $S^n(\pi)$ and the projection $S^n(A) \longrightarrow \frac{S^n(A)}{S^n(J)}$. Then if
		\begin{equation}
			0 \rightarrow J \xrightarrow{i} A \xrightarrow{\pi} \frac{A}{J} \rightarrow 0
		\end{equation}
		is an exact sequence, then also the sequence
		\begin{equation}
			0 \rightarrow S^n(J) \xrightarrow{S^n(i)} S^n(A) \xrightarrow{S^n(\pi)} S^n(\frac{A}{J}) \rightarrow 0
		\end{equation}
		is exact. Then using (\ref{k-short}) we have that
		\begin{equation}
			K_n(J) \xrightarrow{S^n(i)_\star}  K_n(A) \xrightarrow{S^n(\pi)_\star} K_n(\frac{A}{J})
		\end{equation}
		is exact.
	\end{rem}
	There is another way to define $K_1(A)$.
	\begin{defn}
		Given a $C^*$-algebra $A$, a matrix $u$ in $M_n(A)$ we define the \textbf{adjoint} of $u$ as the matrix $u^*$ given by the transposition $u^\tau$ of $u$ and applying the involution of $A$ in each component of $u^\tau$.
		\\We have that the matrix $u$ is \textbf{unitary} if it is invertible and 
		\begin{equation}
			u^{-1} = u^*.
		\end{equation}
	\end{defn}
	\begin{defn}
		Let $A$ be an unital $C^*$-algebra. Then \textbf{$K_1^u(A)$} is the group whose generators are $[u]$ where $u$ is a unitary matrix of $M_n(A)$for some $n \in \numberset{N}$ and
		\begin{itemize}
			\item if there is a continuous path $\alpha:[0,1] \longrightarrow M_n(A)$ such that $\alpha(0) = u$, $\alpha(1) = v$ and for all $t$ in $[0,1]$, $\alpha(t)$ is a unitary matrix then $[u]=[v]$.
			\item the identity matrices are all in the class which is the identity element of $K_1^u(A)$.
			\item if $[u] \in M_n(A)$ and $[v]\in M_m(A)$ then  $[u] + [v] = [u\oplus v]$ where $[u\oplus v]$ is the matrix in $M_{n+m}(A)$ given by
			\begin{equation}
				{\begin{pmatrix}
						u & 0 \\ 0 & v
				\end{pmatrix}}
			\end{equation}
		\end{itemize}
	\end{defn}
	\begin{rem}
		As the authors say  on pg.107 in \cite{Anal}, we have that also the $K_1$ groups of non unital $C^*$-algebras can be descripted by unitaries. In particular one have to consider the unitary matrixes over $\tilde{A}$ which are equal to the identity matrix modulo $A$.
	\end{rem}
	\begin{rem}
		If $\phi: A \longrightarrow B$ is a $*$-homomorphism, then one can define for all $n \in \numberset{N}$ the map between the space of unitary $n \times n$ matrices
		\begin{equation}
			\begin{split}
				\phi^u: U(A)_n &\longrightarrow U(B)_n \\
				v = (v_{i,j}) &\longrightarrow (\phi(v_{i,j}))
			\end{split}
		\end{equation}
		where $(\phi(v_{i,j}))$ is defined applying $\phi$ to each entry of $v$.
		\\This maps induce a morphism
		\begin{equation}
			\phi^u_\star: K_1^u(A) \longrightarrow K_1^u(B)
		\end{equation}
		defined on the generators $[v]$ of $K_1^u(A)$ as $[u_\phi(v)]$. Obviously the functorial properties
		\begin{equation}
			\phi^u_\star \circ \sigma^u_\star = (\phi \circ \sigma)^u_\star
		\end{equation}
		and
		\begin{equation}
			id^u_{A,\star} = id_{K_0^u(A)}
		\end{equation}
		hold.
	\end{rem}	
	\begin{prop}
		Let $A$ be an unitary $C^*$-algebra. The groups $K_1(A)$ and $K_1^u(A)$ are isomorphic.
	\end{prop}
	One can find the proof of this fact in \cite{Anal}. The isomorphism, in particular, is in Proposition 4.8.2. In this Proposition they consider, without loss of generality, as generators of $K_0(S(A))$ the classes $[p]$ where $p$ is a loop of projections in $M_{2m}(A)$ based in $1_m \oplus 0_m$. Then they prove that any such loop $p(t)$ can be written as
	\begin{equation}
		p(t) = u(t)p(0)u^*(t)
	\end{equation}
	where $u: [0,1] \longrightarrow M_n(A)$ is, for all $t \in [0,1]$ unitary. 
	Since $u(0)$ commute with $p(0) = 1_m \oplus 0_m$, then $u(0)$ has the form
	\begin{equation}
		u(0) = \begin{bmatrix} v & 0 \\ 0 & w \end{bmatrix}.
	\end{equation}
	At this point they define the map $\mathcal{A}_A: K_0(S(A)) \longrightarrow K_0^u(A)$ posing
	\begin{equation}
		\mathcal{A}_A[p(t)] = [v].
	\end{equation}
	\begin{lem}\label{kkk}
		Consider a $*$-homomorphism $\phi: A \longrightarrow B$. Then we have that
		\begin{equation}
			\phi^u_\star = \mathcal{A}_B \circ (S\phi)_{\star} \circ \mathcal{A}_A^{-1}. 
		\end{equation}	
	\end{lem}
	\begin{proof}
		It is enough to observe that
		\begin{equation}
			\phi^u_\star[v] = [\phi^u(v)] = [(\phi(v_{i,j})_{i,j})]
		\end{equation}
		and 
		\begin{equation}
			\mathcal{A}_B \circ (S\phi)_{\star} \circ \mathcal{A}_A^{-1} [v] = \mathcal{A}_B \circ (S\phi)_{\star} [V]
		\end{equation}
		where $V$ is the constant curve based on the matrix
		\begin{equation}
			\begin{bmatrix} vv^* & 0 \\ 0 & 0 \end{bmatrix} = \begin{bmatrix} v & 0 \\ 0 & 0 \end{bmatrix}\begin{bmatrix} 1 & 0 \\ 0 & 0 \end{bmatrix}\begin{bmatrix} v^* & 0 \\ 0 & 0 \end{bmatrix}.
		\end{equation}
		Then we can observe that 
		\begin{equation}
			(S\phi)_{\star} [V] = [\phi \circ V] = \begin{bmatrix} \phi(v)\phi(v)^* & 0 \\ 0 & 0 \end{bmatrix} = \begin{bmatrix} \phi(v) & 0 \\ 0 & 0 \end{bmatrix}\begin{bmatrix} 1 & 0 \\ 0 & 0 \end{bmatrix}\begin{bmatrix} \phi(v)^* & 0 \\ 0 & 0 \end{bmatrix}
		\end{equation}
		and so we can conclude
		\begin{equation}
			\mathcal{A}_B \circ (S\phi)_{\star} \circ \mathcal{A}_A^{-1} [v] = \mathcal{A}_B \circ (S\phi)_{\star}[V] = [(\phi(v_{i,j})_{i,j})] = \phi^u_\star[v].
		\end{equation} 
	\end{proof}
	Similarly to the singular homology case, there is a connecting homomorphism $\delta:K_{n}(\frac{A}{B}) \longrightarrow K_{n-1}(J)$ which given a long exact sequence associated to the short exact sequence (\ref{k-short}).
	\\To define the connecting homomorphism we have to introduce a particular $C^*$-algebra.
	\begin{defn} 
		Let $A$ be a $C^*$-algebra. Then the \textbf{mapping cone of the projection $\pi: A \longrightarrow \frac{A}{J}$} is the set
		\begin{equation}
			C(A, A/J) := \{(a,f)| a \in A, f:[0,1] \longrightarrow \frac{A}{J} \mbox{     is continuous,      } f(0)= 0, f(1) = \pi(a) \}
		\end{equation}
		with the $C^*$-algebra structure induced by $A$.
	\end{defn}
	\begin{rem}
		The Proposition 4.5.3. of \cite{Anal} states
		\begin{equation}
			K_0(J) \cong K_0(C(A, A/J)).
		\end{equation}
		The isomorphism between the $K_0$-groups is induced by the $*$-homomorphisms
		\begin{equation}
			\begin{split}
				\psi_{J}: J &\longrightarrow C(A, \frac{A}{J})\\
				b &\longrightarrow (b, 0).
			\end{split}
		\end{equation}	
		Consider now another $C^*$-algebra $B$, an ideal $J$ of $B$ and a $*$-homomorphism
		\begin{equation}
			\phi: B \longrightarrow A
		\end{equation}
		such that $\phi(I) \subseteq J$ and denote by $\Phi: \frac{B}{I} \longrightarrow \frac{A}{J}$ the $*$-homomorphism induced by $\phi$.
		\\We can define the map
		\begin{equation}
			\begin{split}
				\mathcal{I}_{\phi} :C(B, \frac{B}{I}) &\longrightarrow C(A, \frac{A}{J}) \\
				(a, \gamma) &\longrightarrow (\phi(a), \Phi \circ \gamma).
			\end{split}
		\end{equation}
		Then it's easy to see that
		\begin{equation}
			\psi_{J} \circ \phi = \mathcal{I}_{\Phi} \circ \psi_{I}.
		\end{equation}
		This means that in K-Theory, considering $\phi_\star: K_0(I) \longrightarrow K_0(J)$, we have that
		\begin{equation}
			\phi_{\star} = \psi_{J, \star}^{-1} \circ \mathcal{I}_{\Phi, \star} \circ \psi_{I, \star}. \label{tre}
		\end{equation}
	\end{rem}
	\begin{rem}
		Let us define the $*$-homomorphism $\Delta$ given by
		\begin{equation}
			\begin{split}
				\Delta: S(\frac{A}{J}) &\longrightarrow C(A, A/J) \\
				f &\longrightarrow (0,f).
			\end{split}
		\end{equation}
		Then we can define the transgression map $\delta$ as the map
		\begin{equation}
			\delta := \Delta_\star: K_1(\frac{A}{J}) =  K_0(S(\frac{A}{J})) \longrightarrow  K_0(C(A, A/J)) \cong K_0(J).
		\end{equation}
		Moreover, using that $\frac{S(A)}{S(J)} \cong S(\frac{A}{J})$, finally we define $\delta$ as 
		\begin{equation}
			\delta := \Delta_\star: K_{n+1}(\frac{A}{J}) \longrightarrow K_n(J).
		\end{equation}
		Now, following Proposition 4.5.9 of \cite{Anal}, one can prove that
		\begin{equation}
			... \rightarrow K_n(A) \rightarrow K_{n}(\frac{A}{J}) \rightarrow K_{n-1}(J) \rightarrow K_{n-1}(A) \rightarrow ... \rightarrow K_0(\frac{A}{J})
		\end{equation}
		is a long exact sequence. Moreover we also have the following classical result
	\end{rem}
	\begin{prop}[Bott periodicity]
		Given a $C^*$-algebra $A$, then $K_{0}(A) \cong K_2(A)$.
	\end{prop} 
	This means that the long exact sequence above can be seen as the sequence
	\begin{equation}
		\begin{matrix}
			K_0(A) && \longrightarrow K_0(B)  \longrightarrow  &&K_0(C) \\
			\uparrow &&                                      && \downarrow\\
			K_1(C)&&  \longleftarrow  K_1(B)  \longleftarrow && K_1(A)
		\end{matrix}
	\end{equation}
	where the vertical arrows are the connecting homomorphisms.
	\begin{rem}
		Starting by now, given a $C^*$-algebra $A$, we will consider $K_1(A) = K_1^u(A)$ and not as $K_0(S(\frac{A}{J}))$. This means that the transgression map $\delta: K_1(\frac{A}{J}) \longrightarrow K_0(J)$, for us, will be
		\begin{equation}
			\delta = \psi_{J,\star}^{-1} \circ \Delta_{\star} \circ \mathcal{A}^{-1}_{\frac{A}{J},\star} \label{delta1},
		\end{equation}
		where $\mathcal{A}_{\frac{A}{J},\star}$ is the isomorphism between $K_0(S(\frac{A}{J})) \longrightarrow K_0^u(\frac{A}{J})$  and $\psi_\star$ is the isomorphism $\psi_{J,\star}: K_0(J) \longrightarrow K_0(C(A, A/J))$.
		\\
		\\Moreover, following pg. 110 of \cite{Anal}, the map which gives Bott periodicity
		\begin{equation}
			\beta: K_0(\frac{A}{J}) \longrightarrow K_0^u(S(\frac{A}{J})) = K_1(S(\frac{A}{J}))
		\end{equation}
		may be defined sending a projection $p \in M_n(A)$ in the unitary loop given by
		\begin{equation}
			\beta([p]) := [p\overline{z} + 1 - p],
		\end{equation}
		where $\overline{z} : S^1 \longrightarrow \numberset{C}$ is the identity function on the unit circle.
		\\This means that, for us, the transgression map $\delta: K_0(\frac{A}{J}) \longrightarrow K_1(J)$ will be
		\begin{equation}
			\delta = \mathcal{A}_{J,\star} \circ \psi_{S(A),S(J),\star}^{-1} \circ \Delta_{S(A), S(J), \star} \circ \mathcal{A}_{S(\frac{A}{J}),\star}^{-1} \beta \label{delta2}
		\end{equation}
		Indeed we have that
		\begin{equation}
			\mathcal{A}_{S(\frac{A}{J}),\star}^{-1} : K_1^u(S(\frac{A}{J})) = K_1(S(\frac{A}{J})) \longrightarrow K_0(SS\frac{A}{J}),
		\end{equation} 
		\begin{equation}
			\Delta_{S(A), S(J), \star}: K_0(SS\frac{A}{J}) = K_0(S(\frac{SA}{SJ}))\longrightarrow  K_0(C(S(A), S(A)/S(J)))
		\end{equation}
		and
		\begin{equation}
			\psi_{S(A), S(J), \star}: K_0(C(S(A), S(A)/S(J))) \longrightarrow K_0(S(J)).
		\end{equation}
		Then applying $\mathcal{A}_{J,\star}$ we obtain a class in $K_1(J) = K_0^u(J)$.
		\\Finally if $i: B \longrightarrow A$ is such that $i(I) \subseteq J$ and there is a map $\iota: \frac{B}{I} \longrightarrow \frac{A}{J}$ induced by $i$ which is injective, then we have
		\begin{equation}
			\beta_A \circ \iota_\star = \iota^u_\star \circ \beta_B \label{quattro}.
		\end{equation}
	\end{rem}
	\begin{lem}\label{K-comm}
		Consider two $C^*$-algebras $A$ and $B$ and let $J$ be an ideal of $A$ and $I$ be an ideal of $B$. Then consider a $*$-homomorphism $i: B \longrightarrow A$ such that $i(I) \subseteq J$. Then, if $\iota: \frac{B}{I} \longrightarrow \frac{A}{J}$ is the map induced by $i$ and we denote by $\delta_A$ and $\delta_B$ the transgression maps, we have that
		\begin{equation}
			\delta_A \circ \iota^u_\star = i_\star \circ \delta_B, \label{equno}
		\end{equation}
		if we consider $\delta_A: K_1(\frac{A}{J}) \longrightarrow K_0(J)$ and
		\begin{equation}
			\delta_A \circ \iota_\star = i^u_\star \circ \delta_B \label{eqdue}
		\end{equation}
		if we consider $\delta_A: K_0(\frac{A}{J}) \longrightarrow K_1(J)$.
	\end{lem}
	\begin{proof}
		First we will prove (\ref{equno}). Recall that because of (\ref{delta1}), that
		\begin{equation}
			\delta_A = \psi_{J,\star}^{-1} \circ \Delta_{A, \star} \circ \mathcal{A}^{-1}_{\frac{A}{J},\star}.
		\end{equation}
		Moreover we can observe that, because of Lemma \ref{kkk}, we have that
		\begin{equation}
			\iota^u_\star = \mathcal{A}_{\frac{A}{J},\star} \circ (S\iota)_{\star} \circ \mathcal{A}^{-1}_{\frac{B}{I},\star}.
		\end{equation}
		Finally, applying the equality (\ref{tre}) we obtain
		\begin{equation}
			i_{\star} = \psi^{-1}_{J, \star} \circ \mathcal{I}_{i, \star} \circ \psi_{I, \star}.
		\end{equation}
		Then we have that
		\begin{equation}
			\begin{split}
				\delta_A \circ \iota^u_\star &= \psi_{J,\star}^{-1} \circ \Delta_{A, \star} \circ \mathcal{A}^{-1}_{\frac{A}{J},\star} \circ \mathcal{A}_{\frac{A}{J},\star} \circ (S\iota)_{\star} \circ \mathcal{A}^{-1}_{\frac{B}{I},\star}\\
				&= \psi_{J,\star}^{-1} \circ \Delta_{A, \star} \circ (S\iota)_{\star} \circ \mathcal{A}^{-1}_{\frac{B}{I},\star}
			\end{split}
		\end{equation}
		and
		\begin{equation}
			\begin{split}
				i_\star \circ \delta_B &= \psi^{-1}_{J, \star} \circ \mathcal{I}_{i, \star} \circ \psi_{I, \star} \circ \psi_{I,\star}^{-1} \circ \Delta_{B, \star} \circ \mathcal{A}^{-1}_{\frac{B}{I},\star} \\
				&= \psi^{-1}_{J, \star} \circ \mathcal{I}_{i, \star} \circ \Delta_{B, \star} \circ \mathcal{A}^{-1}_{\frac{B}{I},\star} 
			\end{split}
		\end{equation}
		Then we have that
		\begin{equation}
			\begin{split}
				&\delta_{A} \circ \iota^u_\star = i_\star \circ \delta_B \iff \\
				&\psi_{J,\star}^{-1} \circ \Delta_{A, \star} \circ (S\iota)_{\star} \circ \mathcal{A}^{-1}_{\frac{B}{I},\star} = \psi^{-1}_{J, \star} \circ \mathcal{I}_{i, \star} \circ \Delta_{B, \star} \circ \mathcal{A}^{-1}_{\frac{B}{I},\star}  \iff \\
				&\Delta_{A, \star} \circ (S\iota)_{\star} = \mathcal{I}_{i, \star} \circ \Delta_{B, \star}.\label{fin}
			\end{split}
		\end{equation}
		Consider $\alpha$ in $S(\frac{B}{I})$: we have that
		\begin{equation}
			\Delta_{A, \star} \circ (S\iota)_{\star} (\alpha) = (0, \iota \circ \alpha) = \mathcal{I}_{i, \star} \circ \Delta_{B, \star} \alpha.
		\end{equation}
		This means that all the identities in (\ref{fin}) hold and we proved (\ref{equno}).
		\\
		\\
		\\Let us prove (\ref{eqdue}). Since (\ref{delta2}) we have that
		\begin{equation}
			\delta_A = \mathcal{A}_{J,\star} \circ \psi_{S(A),S(J),\star}^{-1} \circ \Delta_{A,\star} \circ \mathcal{A}_{S(\frac{A}{J}),\star}^{-1} \circ  \beta_A ,
		\end{equation}
		where $\beta$ is the Bott isomorphism. We know, since (\ref{quattro}), that
		\begin{equation}
			\beta_A \circ \iota_\star = \iota^u_\star \circ \beta_B.
		\end{equation}
		This means that
		\begin{equation}
			\begin{split}
				\delta_A \circ \iota_\star &= \mathcal{A}_{J,\star} \circ \psi_{S(A),S(J),\star}^{-1} \circ \Delta_{A,\star} \circ \mathcal{A}_{S(\frac{A}{J}),\star}^{-1} \circ \beta_A \circ \iota_\star \\
				&= \mathcal{A}_{J,\star} \circ \psi_{S(A),S(J),\star}^{-1} \circ \Delta_{A,\star} \circ \mathcal{A}_{S(\frac{A}{J}),\star}^{-1} \circ  \iota^u_\star \circ \beta_B \\
				&= \mathcal{A}_{J,\star} \circ \psi_{S(A),S(J),\star}^{-1} \circ \Delta_{A,\star} \circ S(\iota)_\star \circ  \mathcal{A}_{S(\frac{B}{I}),\star}^{-1} \circ \beta_B.
			\end{split}
		\end{equation}
		and 
		\begin{equation}
			\begin{split}
				i^u_\star \circ \delta_B &= \mathcal{A}_{J, \star} \circ (Si)_\star \circ \mathcal{A}_{I,\star}^{-1} \circ \mathcal{A}_{I,\star} \circ \psi_{S(B),S(I),\star}^{-1} \circ \Delta_{B,\star} \circ \mathcal{A}_{S(\frac{B}{I}),\star}^{-1} \circ \beta_B\\
				&= \mathcal{A}_{J, \star} \circ S(i)_\star \circ \psi_{S(B),S(I),\star}^{-1} \circ \Delta_{B,\star} \circ \mathcal{A}_{S(\frac{B}{I}),\star}^{-1} \circ \beta_B.
			\end{split}
		\end{equation}
		We can observe that $S(i)(S(I)) \subseteq S(J)$. This means that we can apply the Remark 6 and we have that
		\begin{equation}
			S(i)_\star = \psi^{-1}_{S(A),S(J), \star} \circ \mathcal{I}_{S(i), \star} \circ \psi_{S(B), S(I), \star}.
		\end{equation}
		This means that
		\begin{equation}
			\begin{split}
				i^u_\star \circ \delta_B &= \mathcal{A}_{J, \star} \circ \psi^{-1}_{S(A),S(J), \star} \circ \mathcal{I}_{S(i), \star} \circ \psi_{S(B), S(I), \star} \circ \psi_{S(B),S(I),\star}^{-1} \circ \Delta_{B,\star} \circ \beta_B \\
				&= \mathcal{A}_{J, \star} \circ \psi^{-1}_{S(A),S(J), \star} \circ \mathcal{I}_{S(i), \star} \circ \Delta_{B,\star} \circ \mathcal{A}_{S(\frac{B}{I}),\star}^{-1} \circ \beta_B.
			\end{split}
		\end{equation}
		Then we have that
		\begin{equation}
			\begin{split}
				&\delta_A \circ \iota_\star = i^u_\star \circ \delta_B \iff \\
				&\Delta_{A,\star} \circ S(i)_\star = \mathcal{I}_{S(i), \star} \circ \Delta_{B,\star}.
			\end{split}
		\end{equation}
		Using then a simple computation, as we did proving (\ref{equno}), one can check that the last equality holds. So we proved (\ref{eqdue}).
	\end{proof}
	\chapter{Smoothing operators}\label{smoothing}
	In this section we will focus on integral operators. In particular, given two complete Riemannian manifolds $(M,g)$ and $(N,h)$ and give an integral operator $A: \mathcal{L}^2(N) \longrightarrow \mathcal{L}^2(M)$ we introduce conditions on its kernel that make $A$ a bounded operator. Moreover we will study the kernel of the operator $dA$ and $Ad$, where $d$ is the closed extension of the exterior derivative. We have used this tools in section \ref{wai} to prove that $T_f$, $dT_f$ and $T_f y$ and other combinations of these operators are in $C_f^*(M \sqcup N)^\Gamma$.
	\\
	\\In this section we will denote by
	\begin{equation}
		\int_N f(x) d\mu_X
	\end{equation}
	the Bochner integral of a measurable function $f:(X, \mu_X) \longrightarrow B$, where $(X, \mu_X)$ is a measured space and $B$ is a Banach space, and we will denote by $a_\star$ the operator of integration along the fibers of a submersion $a$.
	\\
	\\
	\begin{defn}
		Consider two complete Riemannian manifolds $(M,g)$ and $(N,h)$. Let us denote by $pr_N$ the projection $pr_N: M\times N \longrightarrow N$ and by $pr_M$ the projection on the first component. Let us define the bundle on $M\times N$ given by
		\begin{equation}
			\Lambda^*(M) \boxtimes \Lambda(N) := pr_M^*(\Lambda^*(M)) \otimes pr_N^*(\Lambda(N)),
		\end{equation}
		where $\Lambda(N) := \Lambda(TN)$ is the dual bundle of $\Lambda^*(N)$.
	\end{defn}
	\begin{rem}
		Consider two differential manifolds $X_1$ and $X_2$ and consider for $i=1,2$ the projections $pr_i: X_1 \times X_2 \longrightarrow X_i$. Then we have that
		\begin{equation}
			pr_1^*(\Lambda^*(X_1)) \cong \Lambda^*_{\star,0}(X_1 \times X_2) \subset \Lambda^*(X_1 \times X_2)
		\end{equation}
		where $\Lambda^*_{\star,0}(X_1 \times X_2)$ is the subbundle of $\Lambda^*(X_1 \times X_2)$ given by the forms
		\begin{equation}
			\alpha_{(p,q)} \in\Lambda^*_{(p,q)}(X_1 \times X_2)
		\end{equation}
		such that
		\begin{equation}
			ker(\alpha_{(p,q)}) \supseteq (\{0\}\oplus T_qX_2)\times ... \times (\{0\}\oplus T_qX_2)
		\end{equation}
		The isomorphism is the following: for all $(s,q)$ in $X_1 \times X_2$ if 
		\begin{equation}
			\alpha_{s,q}: T_sX_1 \times ... \times T_sX_1 \longrightarrow \numberset{C}
		\end{equation}
		is an element of $pr_1^*(\Lambda^*(X_1))$, the we can define $\overline{\alpha}$ as the application
		\begin{equation}
			\overline{\alpha}_{s,q}: (T_sX_1 \oplus T_qX_2) \times ... \times (T_sX_1 \oplus T_qX_2) \longrightarrow \numberset{C}
		\end{equation}
		defined as
		\begin{equation}
			\overline{\alpha}_{s,q}((v_{1,1} + v_{1,2}), ..., (v_{k,1} + v_{k,2})) = \begin{cases} \alpha_{s,q}(v_{1,1},... ,v_{k,1}) \mbox{     if all    } v_{j,2}=0 \\
				0 \mbox{    otherwise.}
			\end{cases}
		\end{equation}
		The inverse map is given by the restriction of $\overline{\alpha}_{s,q}$ to $(T_sX_1 \oplus \{0\})\times ... \times (T_sX_1 \oplus\{0\}) \cong T_sX_1 \times ... \times T_sX_1$.
		\\This means that for all $k$ in $\numberset{N}$ the sections of $pr_1^*(\Lambda^k(X_1))$ can be seen as $(k, 0)$-differential forms on $X_1 \times X_2$ and vice-versa. In particular we have that we can see the pullback using the projection $pr_i$ as
		\begin{equation}
			pr_1^*: \Omega^*(X_1) \longrightarrow \Gamma(pr_1^*(\Lambda^*(X_1))) \subseteq \Omega^*(X_1 \times X_2).
		\end{equation}
		The same holds also for $pr_2$.
	\end{rem}
	\begin{defn}
		An operator $A: dom(A) \subseteq \mathcal{L}^2(N) \longrightarrow \mathcal{L}^2(M)$ is an \textbf{integral operator} if there is a section $K$ of the fiber bundle $\Lambda^*(M) \boxtimes \Lambda(N)$ such that, for $\alpha \in dom(A) \cap \Omega^*(N)$ and for almost all $p \in M$ we have that
		\begin{equation}
			A(\alpha)(p) := \int_N K(p,q)\alpha(q) d\mu_{N}.
		\end{equation}
		where $K(p,q)\alpha(q)$ is the inner product $i_{pr_N^*\alpha} K$ valued in $(p,q)$. Moreover, if the kernel $K$ is a smooth section of $\Lambda^*(M) \boxtimes \Lambda(N)$ then $A$ is a \textbf{smoothing operator}.
	\end{defn}
	\begin{rem}
		The integral operators are well-defined. Indeed we have that $K(p,q)\alpha(q)$ is a section of $pr_M^*(\Lambda^*(M))$, and so it can also be seen for all $p$ in $M$ as a function
		\begin{equation}
			\begin{split}
				N &\longrightarrow \Lambda^*_p(M) \\
				q &\longrightarrow K(p,q) \alpha(q).
			\end{split}
		\end{equation}
		Then, for all $p \in M$, the Bochner integration respect to the measure on $N$ is well-defined.
		\\
		\\It is also possible to see the integral over $N$ as an integration along the fibers. Indeed $K(p,q) \alpha(q) = i_{pr_N^*\alpha} K(p,q) \in pr_M^*(\Lambda^*(M))_{(p,q)}$ can be seen as sum of $(k_i,0)$-differential form on $M\times N$. This means that
		\begin{equation}
			\begin{split}
				A(\alpha)(p) &= \int_N K(p,q) \alpha(q) d\mu_{N} \\
				&= pr_{M\star}(i_{pr_N^*\alpha} K \wedge pr_N^*Vol_N)(p).
			\end{split}
		\end{equation}
		To show this equality it is sufficient to compute the integrals using local coordinates on $M$.
	\end{rem}
	\begin{rem}
		Given two coordinate charts $\{U, x^s\}$ on $M$ and $\{V, y^l\}$ on $N$, we have that a smooth section of $\Lambda^*(M) \boxtimes \Lambda(N)_{(p,q)}$ is locally given by
		\begin{equation}
			f(p,q)_S^L dx^S \otimes \frac{\partial}{\partial y^L},
		\end{equation}
		where $S = (s_1,... ,s_m)$ and $L = (l_1, ..., l_n)$ are multi-index, $dx^S = dx^{s_1}\wedge ... \wedge dx^{s_m}$, $\frac{\partial}{\partial y^L} = \frac{\partial}{\partial y^{l_1}}\wedge... \wedge \frac{\partial}{\partial y^{l_n}}$ and $f(p,q)_S^L$ is a function in $C^{\infty}(U \times V)$.
	\end{rem}
	\begin{rem}\label{fuori}
		We know that for all $p$ in $M$ and for all $q$ in $N$ we have that $\Lambda_p^*(M)$ and $\Lambda_q(N)$ have a norm induced by their metrics (one can easily prove that the norm on $\Lambda_q(N)$ is equal to the dual norm of $\Lambda_q^*(N)$). We have that for all $k \in \Lambda^*(M)_p \boxtimes \Lambda(N)_q$ there are some $\beta_{i,p} \in \Lambda^*(M)$ and some $\gamma_{i,q} \in \Lambda(N)$ such that
		\begin{equation}
			k = \sum_i \beta_{i,p}  \otimes \gamma_{i,q}.
		\end{equation}
		Then, imposing for all $\beta_p \in \Lambda_p^*(M)$ and for all $\gamma_{q} \in \Lambda_q(N)$ that
		\begin{equation}
			|\beta_p \otimes \gamma_q| := |\beta_p|_{\Lambda_p^*(M)}\cdot|\gamma_q|_{\Lambda_q(N)},
		\end{equation}
		we can induce a norm on $\Lambda^*(M) \boxtimes \Lambda(N)_{(p,q)}$. Moreover, if the $\beta_{i,p}$ and the $\gamma_{i,q}$ are choosen such that for all $i \neq j$
		\begin{equation}
			\langle \beta_{i,p}, \beta_{j,p} \rangle_{\Lambda_p^*(M)} = \langle \gamma_{i,p}, \gamma_{j,p} \rangle_{\Lambda_q(N)} = 0,
		\end{equation}
		then we have that
		\begin{equation}
			|k|^2 = \sum_i |\beta_{i,p}|^2  \cdot |\gamma_{i,q}|^2.
		\end{equation}
		Consider a smooth kernel of an integral operator $K$ and a differential form $\alpha$ in $\Omega^*(N)$. Then, for each point $(p,q)$ we have that
		\begin{equation}
			K(p,q)\alpha(q) = \sum_i \gamma_{i,q}(\alpha(q)) \beta_{i,p}
		\end{equation}
		and so
		\begin{equation}
			\begin{split}
				|K(p,q)\alpha(q)|^2 &= \sum_i |\gamma_{i,q}(\alpha(q))|^2 |\beta_{i,p}|^2\\
				&\leq \sum_i |\beta_{i,p}|^2|\gamma_{i,q}|^2)\alpha(q)|^2 \\
				&\leq (\sum_i |\beta_{i,p}|^2|\gamma_{i,q}|^2)|\alpha(q)|^2 \\
				&= |K(p,q)|^2 |\alpha(q)|^2.
			\end{split}
		\end{equation}
	\end{rem}
	\begin{defn}
		Let us consider two Riemannian manifolds $(N,h)$ and $(M,g)$ and let $A: dom(A) \mathcal{L}^2(N) \longrightarrow \mathcal{L}^2(M)$ be an integral operator with kernel $K$. We have that $A$ is an in integral operator with \textbf{compactly supported} kernel if the support of $K$ as section of $\Lambda^*(M) \boxtimes \Lambda(N)$ is compact.
	\end{defn}
	\begin{defn}
		Consider $N$ and $M$ two Riemannian manifolds. Let $A:\mathcal{L}^2(N) \longrightarrow \mathcal{L}^2(M)$ be an integral operator with kernel $K$. We say that 
		\begin{itemize}
			\item $A$ has \textbf{left-uniformly bounded support} if there is $R \geq 0$ such that for all $q \in N$
			\begin{equation}
				diam(supp(K(\cdot, q))) \leq R,
			\end{equation}
			\item $A$ has \textbf{right-uniformly bounded support} if there is $S \geq 0$ such that for all $p \in M$ for all $q \in N$ we have that 
			\begin{equation}
				diam(supp(K(p, \cdot))) \leq S.
			\end{equation}
			Finally we will say that $A$ has \textbf{uniformly bounded support} if it has both right and left uniformly bounded support.
		\end{itemize}
	\end{defn}
	\begin{lem}[Jensen inequality]
		Consider a probability space $(X, \sigma_X, \mu_X)$ and consider a measurable function $f: X \longrightarrow \numberset{R}$. Let $\phi: \numberset{R} \longrightarrow \numberset{R}$ be a convex function. Then
		\begin{equation}
			\phi(\int_X f(x) d\mu_X) \leq \int_X \phi(f(x)) dx.
		\end{equation}
	\end{lem}
	\begin{prop} \label{smoothbound}
		Let $(N,n)$ and $(M,m)$ be two Riemannian manifolds of bounded geometry. Let $A: dom(A) \mathcal{L}^2(N) \longrightarrow \mathcal{L}^2(M)$ be an integral operator with kernel $K$. If $A$ has uniformly bounded support with constant $Q$ and
		\begin{equation}
			\sup_{(p,q) \in M \times N} |K(p,q)| < L,
		\end{equation}
		then the operator $A$ is bounded and its norm satisfies
		\begin{equation}
			||A|| \leq W,
		\end{equation}
		where $W$ is a constant which depends on $L$ on curvature of $M$.
	\end{prop}
	\begin{proof}
		Consider a smooth form $\alpha$ in $\mathcal{L}^2(N)$. We have that
		\begin{equation}
			||A(\alpha)||^2 = \int_M |\int_N K(p,q)\alpha(q)d\mu_N|^2 d\mu_M.
		\end{equation}
		\textbf{Step 1.} Fix $p$ in $M$, fix an orthonormal basis $\{\epsilon^I\}$ on $\Lambda^*_p(M)$ and consider the coordinates $\{y_I\}$ related to $\{\epsilon^I\}$. Consider, moreover, the maps $y_{I,r}:\Lambda^*_p(M) \longrightarrow \numberset{R}$ and  $y_{I,i}:\Lambda^*_p(M) \longrightarrow \numberset{R}$ such that
		\begin{equation}
			y_{I} = y_{I,r} + i\cdot y_{I,i}.
		\end{equation}
		We have that
		\begin{equation}\label{sbo}
			\begin{split}
				|\int_N K(p,q)\alpha(q) d\mu_N|^2 &= \sum\limits_{I, r} (\int_N y_{I,r}[K(p,q)\alpha(q)]d\mu_N)^2\\
				&+ \sum\limits_{I, i} (\int_N y_{I,i}[K(p,q)\alpha(q)] d\mu_N)^2.
			\end{split}
		\end{equation}
		Since $A$ is right-uniformly bounded, then there is a $S\geq 0$ such that, for each fixed $p$ in $M$, there is a ball $B_S(q_p) \subset N$ such that
		\begin{equation}
			K(\tilde{p},q)\alpha(q) = 0
		\end{equation}
		if $\tilde{p}$ is not in $B_{S}(q_p)$.
		\\For each fixed $p$ consider the measure $\mu_p$ on $B_S(q_p)$ defined for each $\mu_N$-measurable set as
		\begin{equation}
			\mu_p(B) := \frac{\mu_N(B)}{\mu_N(B_S(q_p))}.
		\end{equation}
		Observe $B_S(q_p)$ is a probability space. Let us define the map
		\begin{equation}
			\begin{split}
				F_p:B_S(q_p) &\longrightarrow \Lambda^*_p(M) \\
				q &\longrightarrow \mu_N(B_S(q_p)) \cdot K(p,q)\alpha(q)
			\end{split}
		\end{equation}
		Then, for each $I$, define the maps 
		\begin{equation}\label{bruno}
			y_{I,r} \circ F_p: B_S(q_p) \longrightarrow \numberset{R}.
		\end{equation}
		and
		\begin{equation}\label{brue}
			y_{I,i} \circ F_p: B_S(q_p) \longrightarrow \numberset{R}.
		\end{equation}
		Consider the function $\phi: \numberset{R} \longrightarrow \numberset{R}$ defined as $\phi(x) = x^2$. Then we can apply the Jensen inequality to (\ref{bruno}) and (\ref{brue}) considering on $B_S(q_p)$ the measure $\mu_p$. We obtain
		\begin{equation}\label{ciccone}
			\begin{split}
				\phi(\int_{B_S(q_p)} y_{I,r} \circ F_p d\mu_p) \leq \int_{B_S(q_p)} \phi(y_{I,r} \circ F_p) d\mu_p
			\end{split}
		\end{equation}
		and
		\begin{equation}\label{cic}
			\begin{split}
				\phi(\int_{B_S(q_p)} y_{I,i} \circ F_p d\mu_p) \leq \int_{B_S(q_p)} \phi(y_{I,i} \circ F_p) d\mu_p.
			\end{split}
		\end{equation}
		Observe that the left-hand part of (\ref{ciccone}) is
		\begin{equation}
			\begin{split}
				\phi(\int_{B_S(q_p)} y_{I,r} \circ F_p d\mu_p) &= (\int_{B_S(q_p)} y_{I,r}(K(p,q)\alpha(q)) \frac{\mu_N(B_S(q_p))}{\mu_N(B_S(q_p))} d\mu_N)^2\\
				&(\int_{N} y_{I,r}(K(p,q)\alpha(q)) d\mu_N)^2
			\end{split}
		\end{equation}
		and the left-hand part of (\ref{cic}) is
		\begin{equation}
			\phi(\int_{B_S(q_p)} y_{I,i} \circ F_p d\mu_p) = (\int_{N} y_{I,i}(K(p,q)\alpha(q)) d\mu_N)^2.
		\end{equation}
		The right-hand of (\ref{ciccone}) is
		\begin{equation}
			\begin{split}
				\int_{B_S(q_p)} \phi(y_{I,r} \circ F_p) d\mu_p &= \int_{B_S(q_p)} \frac{\mu_N(B_S(q_p))^2}{\mu_N(B_S(q_p))} (y_{I,r}[K(p,q)\alpha(q)])^2 d\mu_N \\
				&= \mu_N(B_S(q_p)) \int_{B_S(q_p)} (y_{I,r}[K(p,q)\alpha(q)])^2 d\mu_N\\
				&\leq \mu_N(B_S(q_p)) \int_{N} (y_{I,r}[K(p,q)\alpha(q)])^2 d\mu_N
			\end{split}
		\end{equation}
		while the right-hand of (\ref{cic}) is
		\begin{equation}
			\int_{B_S(q_p)} \phi(y_{I,i} \circ F_p) d\mu_p = \mu_N(B_S(q_p)) \int_{N} (y_{I,i}[K(p,q)\alpha(q)])^2 d\mu_N
		\end{equation}
		Then, since (\ref{sbo}), we have that
		\begin{equation}
			\begin{split}
				|\int_N K(p,q)\alpha(q) d\mu_N|^2 &= \sum\limits_{I, r} (\int_N y_{I,r}[K(p,q)\alpha(q)]d\mu_N)^2\\
				&+ \sum\limits_{I, i} (\int_N y_{I,i}[K(p,q)\alpha(q)] d\mu_N)^2 \\
				&\leq \mu_N(B_S(q_p)) \sum\limits_{I, r} \int_{N} (y_{I,r}[K(p,q)\alpha(q)])^2 d\mu_N \\
				&+ \mu_N(B_S(q_p)) \sum\limits_{I, i} \int_{N} (y_{I,i}[K(p,q)\alpha(q)])^2 d\mu_N\\
				&\leq \mu_N(B_S(q_p)) \int_N |K(p,q)\alpha(q)|^2d\mu_N.
			\end{split}
		\end{equation}
		Observe that since $N$ has buonded geometry, then there is a constant $C$ such that
		\begin{equation}
			\mu_N(B_S(q_p)) \leq C.
		\end{equation}
		Then we have that
		\begin{equation}
			|\int_N K(p,q)\alpha(q) d\mu_N|^2 \leq C\cdot \int_N |K(p,q)\alpha(q)|^2d\mu_N.
		\end{equation}
		And so we conclude the first step.
		\\
		\\
		\\
		\textbf{Step 2.}
		Observe that 
		\begin{equation}
			|K(p,q)\alpha(q)|^2 \leq |K(p,q)|^2 |\alpha(q)|^2.
		\end{equation}
		We proved in Remark \ref{fuori}. Then we have that
		\begin{equation}
			\begin{split}
				||A(\alpha)||^2 &= \int_M |\int_N K(p,q)\alpha(q)d\mu_N|^2 d\mu_M \\
				&\leq C \cdot \int_M \int_N |K(p,q)\alpha(q)|^2 d\mu_N d\mu_M\\
				&\leq C \cdot  \int_M \int_N |K(p,q)|^2\cdot |\alpha(q)|^2 d\mu_N d\mu_M\\
				&\leq C \cdot  \int_N \int_M |K(p,q)|^2\cdot |\alpha(q)|^2  d\mu_M d\mu_N\\
				&\leq C \cdot  \int_N (\int_M |K(p,q)|^2 d\mu_M |\alpha(q)|^2 d\mu_N.
			\end{split}
		\end{equation}
		We can observe that, since $A$ has uniformly bounded support and since $M$ has bounded geometry, then exists a real number $L$
		\begin{equation}
			\mu_M(supp(K(\cdot, q))) \leq Vol(B_L(q)) \leq R
		\end{equation}
		Moreover if $|K(p,q)|^2 \leq L$, we have that for all $q$ in $N$
		\begin{equation}
			\int_M |K(p,q)|^2  d\mu_M \leq C_S L^2 \leq R
		\end{equation}
		Then we have that
		\begin{equation}
			\begin{split}
				||A(\alpha)||^2 &\leq C \cdot R \int_{N}|\alpha(q)|^2 d\mu_N \\
				\leq CR||\alpha||.
			\end{split}
		\end{equation}
	\end{proof}
	An obvious consequence of this Proposition is that an integral operator with compactly supported kernel is bounded.
	\begin{lem}
		Smoothing operators with compactly supported kernel are compact operators
	\end{lem}
	\begin{proof}
		Let us consider two Riemannian manifolds $(M,g)$ and $(N,h)$ and consider an integral operator $A: \mathcal{L}^2(N) \longrightarrow \mathcal{L}^2(M)$ with a compactly supported kernel $K$. We can consider a system of coordinates $\{U_{\alpha}, x_\alpha^i\}$ and a partition of unity $\{\rho_\alpha\}$ related to this system. We will consider, moreover, for each $\alpha$, $\mu_{M}(pr_M(U_\alpha)) < +\infty$. Then we have that
		\begin{equation}
			A = \sum_{\alpha} \rho_\alpha A
		\end{equation}
		and it's easy to observe that for any $\alpha$ the operator $\rho_\alpha A$ is an integral operator with compactly supported kernel $\rho_\alpha K$. Since just a finite number  of $\rho_\alpha K \neq 0$, if we are able to show that $\rho_\alpha A$ is compact for all $\alpha$, then $A$ is compact.
		\\This means that we can consider the support of $K \subset U_\alpha$ for some $\alpha$, without loss of generality.
		\\
		\\Consider now an orthonormal frame $\{\epsilon^j\}$ of $\Lambda^*(M)$ and $\{w_i \}$ an orthonormal frame of $\Lambda(N)$, then  
		\begin{equation}
			\begin{split}
				A\alpha(p) &= \int_N K(p,q)\alpha(q)d\mu_N\\
				&= \int_N \sum_{I,J}K^I_J(p,q)\epsilon^J \otimes w_I \alpha(q) d\mu_N \\
				&= \int_N i_{pr^*_N\alpha } \sum_{I,J}K^I_J \epsilon^J \otimes w_I d\mu_N\\
				&= \int_N \sum_{I,J} i_{pr^*_N\alpha }K^I_J \epsilon^J \otimes w_I d\mu_N\\
				&= \sum_{I,J} \int_N K^I_J(p,q)\epsilon^J \otimes w_I \alpha(q) d\mu_N \\
				&= \sum_{I,J} A_{I,J},
			\end{split}
		\end{equation}
		where $A_{I,J}$ are the integral operators with kernel $K^I_J \epsilon^J \otimes w_I$. If for any $I,J$ we have that $A_{I,J}$ are compact operators, then also $A$ is a compact operator. 
		\\
		\\To show that $A_{I,J}$ is a compact operator we will show that $A_{I,J}(B_1(0))$ is relatively compact i.e. $\overline{A_{I,J}(B_1(0))}$ is a compact subset of $\mathcal{L}^2(N)$. In particular we have that since $A_{I,J}$ has compactly supported kernel, then it is a continuous operator and so $\overline{A_{I,J}(B_1(0))} = \overline{A_{I,J}(B_1(0) \cap \Omega^*(N))}$.
		\\Let us consider a $\beta \in B_1(0) \cap \Omega^*(N)$: we have that $\beta = f_I W^I$, where $\{W^i\}$ is the dual frame of $\{w_i\}$. Then we have that
		\begin{equation}
			A_{I,J} \beta = (\int_N K^I_J f_I d\mu_N) \epsilon^J =: G_J \epsilon^J.
		\end{equation}
		Consider $\{e_J\}$ is the dual frame of $\{\epsilon^J\}$. We can define the operator
		\begin{equation}
			i_{e_J}: \Omega^*_c(pr_M(U_\alpha)) \longrightarrow C^{\infty}_c(pr_M(U_\alpha)) 
		\end{equation}
		defined as
		\begin{equation}
			i_{e_K}(g_J \epsilon^J) = \begin{cases} g_K \mbox{    if     } J=K \\
				0 \mbox{    otherwise.} \end{cases}
		\end{equation}
		Then we have that
		\begin{equation}
			i_{e_J} A_{I,J} \beta (p) = \int_N K^I_J(p,q) f_I(q) d\mu_N = G_J(p) \in C^{\infty}_c(pr_M(U_\alpha)).
		\end{equation} 
		Let us consider on $C^{\infty}_c(pr_M(U_\alpha))$ the uniform norm defined as
		\begin{equation}
			|f|_{\infty} = \sup\limits_{p \in pr_M(U_\alpha)} |f(p)|
		\end{equation}
		Let us consider a sequence of compactly supported smooth functions $\{f_n \} \subset C^{\infty}_c(pr_M(U_\alpha))$. If there is a $f: pr_M(U_\alpha) \longrightarrow \numberset{C}$ such that
		\begin{equation}
			\lim\limits_{n \rightarrow + \infty} |f_n - f|_{\infty} = 0,
		\end{equation}
		then, since $\mu_{M}(pr_M(U_{\alpha})) < +\infty$, we have that
		\begin{equation}
			\lim\limits_{n \rightarrow + \infty} ||f_n - f||_{\mathcal{L}^2(pr_N(U_\alpha))} = 0.
		\end{equation}
		Let us consider a sequence 
		\begin{equation}
			A_{I,J} \beta_n = (\int_N K^I_J f_{I,n} d\mu_N) \epsilon^J =: G_{J,n} \epsilon^J 
		\end{equation}
		in $A_{I,J}(B_1(0) \cap \Omega^*(M))$. Then if the sequence $\{G_{J,n}\}_n$ uniformly converges to a function $g: pr_M(U_\alpha) \longrightarrow \numberset{C}$, then $A_{I,J} \beta_n$ converges in $\mathcal{L}^2(M)$. Indeed, since $\epsilon^I$ is an orthonormal frame, we have that
		\begin{equation}
			\begin{split}
				\lim\limits_{n \rightarrow + \infty} ||A_{I,J} \beta_n - g \epsilon^J||_{\mathcal{L}^2(M)} &= \lim\limits_{n \rightarrow + \infty} || (G_{J,n} - g) \epsilon^J||_{\mathcal{L}^2(M)}\\
				&= \lim\limits_{n \rightarrow + \infty} ||G_{J,n} - g||_{\mathcal{L}^2(pr_M(U_{\alpha}))} = 0. \label{conv}
			\end{split}
		\end{equation}
		The next step is to prove that $i_{e_J} \circ A_{I,J}(B_1(0) \cap \Omega^*(N))$ has compact closure respect to the uniform norm on $C^{\infty}_c(pr_M(U_\alpha))$.
		\\We can use the Ascoli-Arzelà theorem: to show that $i_{e_I} \circ A_{I,J}(B_1(0) \cap \Omega^*(N))$ has compact closure we have to proof that it is equicontinuous and pointwise bounded.
		\\Remember that a subset of $C^{\infty}_c(pr_M(U_\alpha))$ is equicontinuous if $\forall p \in U_\alpha \forall \epsilon  > 0$ there is a neighborhood $U_x$ of $x$ such that for all 
		\begin{equation}
			\forall y \in U_x \forall f \in i_{e_J} \circ A_{I,J}(B_1(0) \cap \Omega^*(N)) \implies |f(y) - f(x)| < \epsilon \label{equicont}
		\end{equation}
		and that $i_{e_J} \circ A_{I,J}(B_1(0) \cap \Omega^*(N))$ is pointwise bounded, i.e for all $x$ in $U_\alpha$
		\begin{equation}
			\sup\{|f(x)| | f \in i_{e_J} \circ A_{I,J}(B_1(0) \cap \Omega^*(N)) \} < +\infty. \label{pointwise}
		\end{equation}
		First we show the equicontinuity (\ref{equicont}): let us fix $\epsilon >0$ and $p$ in $U_\alpha$. Consider $i_{e_J} \circ A_{I,J} \beta$ where $\beta = f_I W^I \in B_1(0) \cap \Omega^*(N) \subset \mathcal{L}^2(N)$, we have that
		\begin{equation}
			\begin{split}
				|i_{e_J} \circ A_{I,J}\beta(y) - i_{e_J} \circ A_{I,J}\beta(p)|^2 &= |\int_N K^I_J(y, q)f_I(q) d\mu_N - \int_N  K^I_J(p, q)f_I(q) d\mu_N|^2 \\
				&\leq \int_N |K^I_J(y,q) - K^I_J(p,q)|^2 |f_I(q)|^2 d\mu_N.
			\end{split}
		\end{equation}
		Now we can observe that $K^I_J$ is continuous and with compact support. Then in particular it is uniformly continuous, i.e there is a $\delta>0$, which doesn't depend by $(p,q)$, such that for all $(y,q)$ in $B_\delta((p,q))$ we have
		\begin{equation}
			|K^I_J(y,q) - K^I_J(p,q)|^2 \leq \epsilon.
		\end{equation}
		So we obtain that
		\begin{equation}
			\begin{split}
				|i_{e_J} \circ A_{I,J}\beta(y) - i_{e_J} \circ A_{I,J}\beta(p)|^2 &\leq \int_N |K^I_J(y,q) - K^I_J(p,q)|^2 |f_I|^2 d\mu_N \\
				&\leq \epsilon \int_N |\beta|^2 d\mu_N =\epsilon.
			\end{split}
		\end{equation}
		Finally to prove the pointwise boundedness (\ref{pointwise}) it is sufficient to observe that
		\begin{equation}
			\begin{split}
				|i_{e_J} \circ A_{I,J}\beta(p)|^2 &= |\int_N  K^I_J(p, q)f_I(q) d\mu_N|^2 \\
				&\leq \int_N |K^I_J(p,q)|^2 |f_I|^2 d\mu_N \\
				&\leq M \int_N |\beta(q)|^2 d\mu_N = M,
			\end{split}
		\end{equation}
		where $M$ is the maximum of $K^I_J$ on $U_\alpha$.
		\\
		\\
		\\We can conclude that $A_{I,J}(B_1(0) \cap \Omega^*(N))$ has compact closure. Indeed for every sequence $\beta_n$ in $B_1(0) \cap \Omega^*(N)$ we have that $A_{I,J} \beta_n = G_{J,n} \epsilon^J$ for some $G_{J,n} \in C_c^\infty(pr_M(U_\alpha))$. Then, since $i_{e_J} \circ A_{I,J}(B_1(0) \cap \Omega^*(N))$ has compact closure, we know that there is a uniformly converging subsequence $\{G_{J,n_{k}}\}$ of $\{G_{J,n} \} \subset C_c^\infty(pr_M(U_\alpha))$. This, using (\ref{conv}), this means that the sequence $\{G_{J,n_{k}} \epsilon^J \}$ is a $\mathcal{L}^2$-converging subsequence of $\{G_{J,n} \epsilon^J \} \subseteq \mathcal{L}^2(M)$.
		\\
		\\This means that $A_{I,J}$ is a compact operator ans so also $A$ is a compact operator and, in general, all integral operators with compactly supported kernel are compact.
	\end{proof}
	\begin{lem}\label{impo}
		Let $A:dom(A) \subseteq \mathcal{L}^2(N) \longrightarrow \mathcal{L}^2(M)$ be a smoothing operator. Then given a multi-index $I$ and given an index $l$, we will denote by $Jl$ the multi-index defined as $Jl :=(j_1,..., j_n,l)$. 
		\\We have that $dA$ is also a smoothing operator and if the kernel of $A$ is locally given by
		\begin{equation}
			K(x,y) = K^I_J(x,y)dx^J\boxtimes \frac{\partial}{\partial y^I}
		\end{equation}
		then $dA$ is an integral operator and its kernel is locally given by
		\begin{equation}
			d_MK^I_S :=_{|_{loc}} (\sum\limits_{Jl = S} \frac{\partial}{\partial x^l} K^I_J(x,y))dx^S \boxtimes \frac{\partial}{\partial y^I}
		\end{equation}
	\end{lem}
	\begin{proof}
		Let $p$ be a point in $M$ and $q$ in $N$. Then we have that
		\begin{equation}
			\begin{split}
				dA(\alpha) &= d\int_N K(p,q)\alpha(q) d\mu_N\\
				&= d \int_N K(p,q)\alpha(q)\wedge Vol_N\\
				&= \int_N d(K(p,q)\alpha(q)\wedge Vol_N).
			\end{split}
		\end{equation}
		Now, we can observe that $K(p,q)\alpha(q)\wedge Vol_N$ is top-degree form on $N$, so $d = d_M$ in this case. Let us consider now a coordinate system $\{x^i, y^j\}$ on $M\times N$. Then we have that
		\begin{equation}
			\begin{split}
				dA(\alpha)(x) &= \int_N d_M[(\frac{\partial}{\partial y^J} \dashv \alpha(y))K^I_J(x,y)dx^J]\wedge Vol_N\\
				&= \int_N (\frac{\partial}{\partial y^J} \dashv \alpha(y)) \frac{\partial}{\partial x^l} (K^I_J(x,y))dx^l\wedge dx^J \wedge Vol_N\\
				&= \int_N \frac{\partial}{\partial x^l} (K^I_J(x,y))dx^l\wedge dx^J\otimes \frac{\partial}{\partial y^I}(\alpha(y))Vol_N \\
				&= \int_N \frac{\partial}{\partial x^l} (K^I_J(x,y))dx^l\wedge dx^J\otimes \frac{\partial}{\partial y^I}(\alpha(y)) d\mu_N.
			\end{split}
		\end{equation}
		and so we have that
		\begin{equation}
			d_MK^I_S :=_{|_{loc}} (\sum\limits_{Jl = S} \frac{\partial}{\partial x^l} K^I_J(x,y))dx^S \boxtimes \frac{\partial}{\partial y^I}.
		\end{equation}
		To conclude the proof we have to check that the local sections
		\begin{equation}
			\frac{\partial}{\partial x^l} (K^I_J(x,y))dx^l\wedge dx^J\otimes \frac{\partial}{\partial y^I}
		\end{equation}
		are the local forms of a global section of $\Lambda^*(M) \boxtimes \Lambda(N)$. To do this we can define an operator
		\begin{equation}
			Q: \Omega^*(M \times N) \longrightarrow \Gamma(pr_M^*(\Lambda(M))) \subseteq \Omega^*(M \times N)
		\end{equation}
		where
		\begin{equation}
			Q(\alpha) = \begin{cases} \frac{\alpha}{pr_N^*Vol_N} \mbox{    if    }\alpha \mbox{    is a  } (q,n)\mbox{-form where   } n = dim(N)\\
				0 \mbox{     otherwise.} \end{cases} 
		\end{equation}
		Observe that
		\begin{equation}
			\Gamma(pr_M^*(\Lambda^*(M))) \otimes_{C^{\infty}(M \times N)} \Gamma(pr_N^*(\Lambda(N))) \cong \Gamma(\Lambda^*(M) \boxtimes \Lambda(N)).
		\end{equation}
		This means that we can consider the operator
		\begin{equation}
			Q \otimes 1: \Omega^*(M \times N) \otimes_{C^{\infty}(M\times N)} \Gamma(pr_N^*(\Lambda(N))) \longrightarrow \Gamma(\Lambda^*(M) \boxtimes \Lambda(N))
		\end{equation}
		Let us now observe that
		\begin{equation}
			\Lambda^*(M) \boxtimes \Lambda(N) = pr_M^*(\Lambda^*(M)) \otimes pr_N^*(\Lambda(N)) \subset \Lambda^*(M \times N) \otimes pr_N^*(\Lambda(N))
		\end{equation}
		and so we have that
		\begin{equation}
			\Gamma(\Lambda^*(M) \boxtimes \Lambda(N)) \subseteq \Omega^*(M \times N) \otimes_{C^{\infty}(M\times N)} \Gamma(pr_N^*(\Lambda(N)))
		\end{equation}
		The vector space on the right can be see as the space of $\Gamma(pr_N^*(\Lambda(N)))$-valued differential forms on $M\times N$. This means that we can apply to $K \in \Gamma(\Lambda^*(M) \boxtimes \Lambda(N))$ the exterior derivative $d_{vec}$ of a $\Gamma(pr_N^*(\Lambda(N)))$-valued differential form. Then it's easy to check that
		\begin{equation}
			Q\otimes 1 (d_{vec} K)_{|_{loc}} = \frac{\partial}{\partial x^l} (K^I_J(x,y))dx^l\wedge dx^J\otimes \frac{\partial}{\partial y^I}
		\end{equation}
	\end{proof}
	\begin{rem}\label{derivsmooth}
		One can observe that the support of the kernel of $dA$ is strictly contained in the support of the kernel of $A$. This means that if $A$ has right-uniformly bounded support, then also $dA$ ha uniformly bounded support.
		\\Moreover if $A$ is a smoothing operator which kernel $K$ has uniformly bounded support and if in normal coordinates there is a uniform bound
		\begin{equation}
			|\frac{\partial}{\partial x^j} K^I_J(x,y)| \leq C,
		\end{equation}
		then $dA$ is a bounded operator.
	\end{rem}
	\begin{lem}\label{giga}
		Let $H: \mathcal{L}^2(N \times [0,1]) \longrightarrow \mathcal{L}^2(M)$ be a smoothing operator between manifolds of bounded geometry and denote by $K$ its kernel.  Fix, locally, some normal coordinates $\{x^i\}$ on $M$ and $\{y^j\}$ on $N$.
		\begin{equation}
			K_{|_{loc}} := K^J_I(x,t,y) dx^I \boxtimes \frac{\partial}{\partial y^J} + K^J_{I, +}(x,t,y)dx^I \wedge dt \boxtimes \frac{\partial}{\partial y^J}
		\end{equation}
		Consider $s$ and $s + \delta$ in $[0,1]$ for some $\delta>0$. Let us define the operators $A_s$ and $A_{s+\delta}: \mathcal{L}^2(N) \longrightarrow \mathcal{L}^2(M)$ as the smoothing operators which kernel is given by
		\begin{equation}
			K_s(x,y) := i^*_{s, x, y} \circ K = K^J_I(x,s,y) dx^I \boxtimes \frac{\partial}{\partial y^J}
		\end{equation}
		and
		\begin{equation}
			K_{s + \delta}(x,y) := i^*_{s+\delta, x, y} \circ K = K^J_I(x,s+\delta,y) dx^I \boxtimes \frac{\partial}{\partial y^J}
		\end{equation}
		where $i^*_{t,p,q}: \Lambda^*(M \times [0,1])_{(p,t, q)} \longrightarrow \Lambda^*(M)_{(p,q)}$ is the operator induced by the pullback of the immersion $i_t: M \longrightarrow M \times [0,1]$ defined as $i_t(p) = (p,t)$.
		\\ So, if
		\begin{enumerate}
			\item there is a number $Q$ such that for each $q$ in $N$ and for each $p$ in $M$
			\begin{equation}
				diam(K(\cdot, \cdot, q)) \leq Q,
			\end{equation}
		and
		\begin{equation}
			diam(K(p, \cdot, \cdot)) \leq Q,
		\end{equation}
			\item there is a globally bound for all $t \in [0,1]$
			\begin{equation}
				|\frac{\partial}{\partial t} K_{J}^I(p,t,q)| \leq C,
			\end{equation}
		\end{enumerate}
		then there is a constant $L\geq 0$ such that
		\begin{equation}
			||A_{s + \delta} - A_s|| \leq L\delta,
		\end{equation}
		where $L$ depends by $C$ and by the scalar curvature of $M$.
	\end{lem}
	\begin{proof}
		We have that the kernel of $A_{s + \delta} - A_s$ is locally given by
		\begin{equation}
			\begin{split}
				K_{s + \delta}(x,y) - K_s(x,y) &=  [K^J_I(x,s + \delta,y) - K^J_I(x,s,y)]dx^I \boxtimes \frac{\partial}{\partial y^J} \\
				&= [\int_{0}^{\delta} \frac{\partial}{\partial t} K^J_I(x,t,y) dt] dx^I \boxtimes \frac{\partial}{\partial y^J}
			\end{split}
		\end{equation}
		and so 
		\begin{equation}
			||K_{s + \delta}(x,y) - K_s(x,y)||^2 \leq C^2 \delta.
		\end{equation}
		Moreover, since
		\begin{equation}
			diam(supp(K(\cdot, \cdot, q))) \leq Q
		\end{equation}
		then
		\begin{equation}
			diam(supp(K_{s+ \delta} - K_s(\cdot, \cdot, q))) \leq diam(pr_{M}(supp(K(\cdot, \cdot, q))))\leq Q.
		\end{equation}
		This means that
		\begin{equation}
			\mu_M(supp(K_{s+ \delta} - K_s(\cdot, \cdot, q))) \leq R
		\end{equation}
		Then it is sufficient to apply the Proposition \ref{smoothbound} and we have that
		\begin{equation}
			||A_{s + \delta} - A_s|| \leq R C^2 \delta = L \delta.
		\end{equation}
	\end{proof}
	\begin{lem}
		Consider $A: \mathcal{L}^2(M) \longrightarrow \mathcal{L}^2(N)$ a smoothing operator with kernel $K$ and $A_\delta: \mathcal{L}^2(M) \longrightarrow \mathcal{L}^2(N)$ is another smoothing operator with kernel $K_\delta$ such that
		\begin{equation}
			||[K - K_\delta](p,q)||^2_{(p,q)} \leq \delta.
		\end{equation} 
		then if their support are left-uniformly bounded by a constant $Q$, we have that
		\begin{equation}
			||A - A_\delta|| \leq L\delta,
		\end{equation}
		where $L$ just depends by $Q$ and the infimum of the scalar curvature of $M$ 
	\end{lem}
	\begin{proof}
		It is a direct consequence of Proposition \ref{smoothbound}.
	\end{proof}

	\backmatter

	\cleardoublepage
	\phantomsection
	\bibliographystyle{sapthesis} 


\end{document}